\documentclass[11pt,letterpaper,twoside,openright,final]{thesis91e}
\title{Probabilistic Properties of Delay Differential Equations}
\author{S.\ Richard Taylor}
\discipline{Applied Mathematics}
\degree{Doctor of Philosophy}

\usepackage{enumerate}                   % enumeration environments
\usepackage{amsmath}                     % AMS math extensions
\usepackage{amsthm}                      % AMS theorem extensions
\usepackage{amssymb}                     % AMS fonts
\usepackage{amscd}                       % commutative diagrams
\usepackage{graphicx}                    % eps/pdf figures
\usepackage{lscape}                      % landscape figures

\usepackage{epic}                        % picture environment
\usepackage{eepic}
\usepackage[dotted]{minitoc}             % per-chapter tables of contents

\usepackage[refpage,norefeq]{nomencl}    % for nomenclature list
\makenomenclature
\newPrepageEnvironment{preface}{Preface}

\graphicspath{{./}{figures_eps/}}

\newlength{\figwidth}
\numberwithin{equation}{chapter}
\numberwithin{figure}{chapter}

\theoremstyle{plain}
\newtheorem{thm}{Theorem}[chapter]       % number thms within chapters
\newtheorem{cor}[thm]{Corollary}         % number corollaries like theorems
\theoremstyle{definition}
\newtheorem{defn}{Definition}[chapter]   % number defns within chapters
\newtheorem{example}{Example}[section]
\newtheorem{remark}{Remark}[section]

\newcommand{\ie}{\emph{i.e.}}
\newcommand{\eg}{\emph{e.g.}}
\newcommand{\cf}{\emph{cf.}}
\newcommand{\etc}{\emph{etc}}
\newcommand{\etal}{\emph{et al.}}
\newcommand{\viz}{\emph{viz.}}

\newcommand{\Reals}{\ensuremath{\mathrm{I\!R}}}

\newcommand{\Rn}{\ensuremath{\Reals^n}}
\newcommand{\Acal}{\ensuremath{\mathcal{A}}}

\newcommand{\Rplus}{\ensuremath{\Reals_+}}
\newcommand{\Zplus}{\ensuremath{\mathbb{Z}_+}}
\newcommand{\eps}{\ensuremath{\varepsilon}}
\newcommand{\rhostar}{\ensuremath{\rho_{\ast}}}
\newcommand{\mustar}{\ensuremath{\mu_{\ast}}}
\newcommand{\etastar}{\ensuremath{\eta_{\ast}}}
\newcommand{\Tstar}{\ensuremath{T_{\ast}}}

\begin{document}
\year=2004  % RT 4.9.2019: needed for correct copyright date
\dominitoc[e]  % build per-chapter tables of contents
%
%----------------------------------------------------------------------
% FRONT MATTER
%----------------------------------------------------------------------
\prepages
\maketitle
\sigpages
%\cleardoublepage
%\begin{center}
%\textbf{\Large Abstract}%
%\end{center}
%\chapter*{Abstract}
\begin{abstract}
\addstarredchapter{Abstract}%
Systems whose time evolutions are entirely deterministic can
nevertheless be studied probabilistically, \ie\ in terms of the
evolution of probability distributions rather than individual
trajectories.  This approach is central to the dynamics of ensembles
(statistical mechanics) and systems with uncertainty in the initial
conditions.  It is also the basis of ergodic theory---the study of
probabilistic invariants of dynamical systems---which provides one
framework for understanding chaotic systems whose time evolutions are
erratic and for practical purposes unpredictable.

Delay differential equations (DDEs) are a particular class of
deterministic systems, distinguished by an explicit dependence of the
dynamics on past states.  DDEs arise in diverse applications including
mathematics, biology and economics.  A probabilistic approach to DDEs
is lacking.  The main problems we consider in developing such an
approach are (1) to characterize the evolution of probability
distributions for DDEs, \ie\ develop an analog of the Perron-Frobenius
operator; (2) to characterize invariant probability distributions for
DDEs; and (3) to develop a framework for the application of ergodic
theory to delay equations, with a view to a probabilistic
understanding of DDEs whose time evolutions are chaotic.  We develop a
variety of approaches to each of these problems, employing both
analytical and numerical methods.

In transient chaos, a system evolves erratically during a transient
period that is followed by asymptotically regular behavior.  Transient
chaos in delay equations has not been reported or investigated before.
We find numerical evidence of transient chaos (fractal basins of
attraction and long chaotic transients) in some DDEs, including the
Mackey-Glass equation.  Transient chaos in DDEs can be analyzed
numerically using a modification of the ``stagger-and-step'' algorithm
applied to a discretized version of the DDE.
\end{abstract}

\begin{acknowledgements}
\addstarredchapter{Acknowledgements}%
\bigskip There are many individuals and organizations who have
contributed to the success of this project:

\bigskip

Foremost I wish to thank my thesis supervisor, Sue Ann Campbell, for
her insight, guidance, helpful criticism, and flexibility, and for her
careful and constructive reading of this document in its various
incarnations.

\bigskip

I am grateful for the support of my family and friends, many of whom
will understandably never read these pages but have nevertheless
been generous in their patience and encouragement.

\bigskip

This research was supported by scholarships from the Natural Sciences
and Engineering Research Council of Canada and from the University of
Waterloo.

\end{acknowledgements}

\newpage
\vspace*{\fill}

\hspace*{\fill}
\begin{minipage}[t]{4in}
I have long discovered that geologists never read each other's works,
and that the only object in writing a book is a proof of earnestness,
and that you do not form your opinions without undergoing labor of
some kind.\\
\hspace*{\fill}  ---\textsc{Charles Darwin~\cite{Dar87}}
\end{minipage}
\hspace*{\fill}

\vspace*{\fill}
\vspace*{\fill}

% ----------------------------------------------------------------------
\begin{preface}
\addstarredchapter{Preface}%
As an undergraduate physics student I discovered Ilya Prigogine's book
``The End of Certainty'', quite serendipitously, on the new arrivals
shelf at the UBC library. I signed the book out and read it through
over a very short period, inspired by Prigogine's new (to me) idea
that deterministic systems---predictable things in the world of
Newtonian mechanics---could and should be discussed in probabilistic
terms.  This idea, which pointed to a way out of the clockwork
universe that classical mechanics usually portrays, resonated with
what I had been learning about classical physics, quantum mechanics,
and ``chaos theory''.  According to Prigogine the mathematical
foundation of his ideas was called ergodic theory.  Eager to learn
more, I made my first ever visit to the mathematics library and signed
out Lasota and Mackey's ``Probabilistic Properties of Deterministic
Systems''.  Unfortunately I found that I lacked the mathematical
maturity to read the book on my own, and I soon gave up.

A few years later I met Michael Mackey at a summer school in
Montr{\'e}al, where he gave a presentation in which, as an aside, he
mentioned that a probabilistic/ergodic approach to delay differential
equations was lacking, and that such a theory might have interesting
applications.  At the time I didn't appreciate the ambitiousness of
such a project, but it was the excuse I needed to tackle these ideas
again.  The result, after some earnest labor, is the present
document---a representative subset of my present ideas on a
probabilistic approach to delay differential equations.
\end{preface}

\tableofcontents
\cleardoublepage
\addstarredchapter{List of Figures}
\listoffigures
%
%----------------------------------------------------------------------
% MAIN BODY
%----------------------------------------------------------------------
\mainbody
\chapter{Introduction} \label{ch.intro}

\begin{singlespacing}
\minitoc  % mini table of contents for this chapter
\end{singlespacing}

\vspace{0.4in}

Today, some of the most profound unanswered scientific questions are
related to the interplay between order and disorder.  The physical
basis of consciousness and intelligence (the mind-body problem), the
origins of life, the nature of turbulence, and the paradox between
between thermodynamics and deterministic microscopic dynamics, are but
a few examples.  These are fundamental problems that have plagued
scientists and philosophers for centuries, and each remains largely
unresolved.  Each of them involves the spontaneous creation of order
out of disorder, or disorder out of order, and they share among them
the difficulty of explaining how such processes can result from the
operation of deterministic physical laws.

Recent decades have seen a renewal of interest and progress in the
understanding of the nature of order and disorder, beginning with the
discovery in 1963 by Lorenz of the existence of systems that, despite
being deterministic, exhibit disordered and essentially unpredictable
behavior.  Since then, the study of ``deterministic chaos'' has
brought forth a rich corpus of experimental and theoretical results,
incorporating the insights of researchers in diverse fields spanning
mathematics, physics and engineering.  This corpus comprises what has
come to be called ``chaos theory''\footnote{Abuse and misuse of the
term ``chaos theory'' in the popular literature has led some serious
people to avoid the term.  One comprehensive reference~\cite{KH95}
manages not to use the term in its entire 800+ pages.}.

In particular, the mathematical field of \emph{dynamical systems}
(\ie, systems that evolve in time according to deterministic laws) has
contributed much to a unified understanding of chaotic systems.  This
general framework reveals that, despite the differences in origins and
physical nature of different deterministic systems, the same
underlying mechanisms operate to generate disorder.  One of the
central results of this theory is the identification of a precise
notion of what constitutes a ``chaotic system'', and the discovery of
sufficient conditions that imply the existence of chaos as
such~\cite{LY75}.

Given the random character of so-called chaotic evolutions, it is not
surprising that statistical and probabilistic ideas should be useful
tools for their analysis.  It turns out that far from being merely
descriptive, probabilistic ideas provide fundamental insights into the
behavior of dynamical systems~\cite{LM94,Pr96}.  This observation is
the basis of the ergodic theory of dynamical systems, a theory that
had its origins in foundational issues in statistical mechanics in the
late 1800's.  Despite the success of ergodic theory as a mathematical
endeavor, the physical problems that motivated its development
surprisingly remain unresolved~\cite{Ma92,Pr80,Pr96}.

It may seem, at first, that the application of probabilistic ideas to
deterministic systems is inherently contradictory.  Indeed,
probability theory is concerned with the study of inherently random
phenomena, which are antithetical to a strictly deterministic
conception of physical law.  However, experience has shown quite the
opposite: probabilistic ideas provide a new and, in the end, quite
natural way to view deterministic phenomena.  The following discussion
is intended to give the flavor of this viewpoint.  In Chapter~2 we
give a more detailed theoretical presentation.

% ---------------------------------------------------------------------
\section{Probabilistic Approach to Deterministic Systems}

Consider the following, much-studied example of a simple deterministic
system whose evolution exhibits a species of random behavior.  For a
given a real number $x_0$ (the ``initial state'' of the system)
between $0$ and $1$, let $x_1$, $x_2$, $x_3$, \etc\ be defined by
repeated application of the formula
\begin{equation} \label{eq.quadmap}
  x_{n+1} = 4 x_n (1 - x_n), \quad n=0, 1, 2, \ldots
\end{equation}
Once can view this formula as prescribing the evolution of the state
of the system, $x_n$, at discrete times $n=0,1,2,\ldots$.  The
evolution of this system is deterministic, in that once the initial
state is specified equation~\eqref{eq.quadmap} uniquely determines
the sequence of values $\{x_0, x_1, x_2, \ldots\}$ (\ie, the
\emph{trajectory} of the system) for all time. Thus for example if
$x_0=0.51$ we obtain
\begin{equation*}
  x_1 = .9996, x_2 \approx .0026, x_3 \approx .0064, x_4 \approx .025,
  x_5 \approx .099, \; \mathit{etc.}
\end{equation*}

\begin{figure}
\begin{center}
\includegraphics[width=\figwidth]{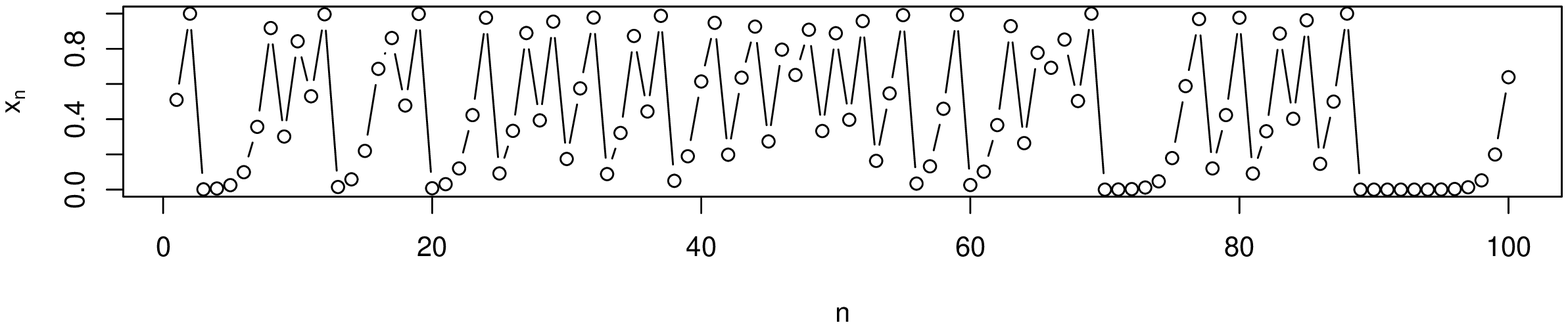}\\
\includegraphics[width=\figwidth]{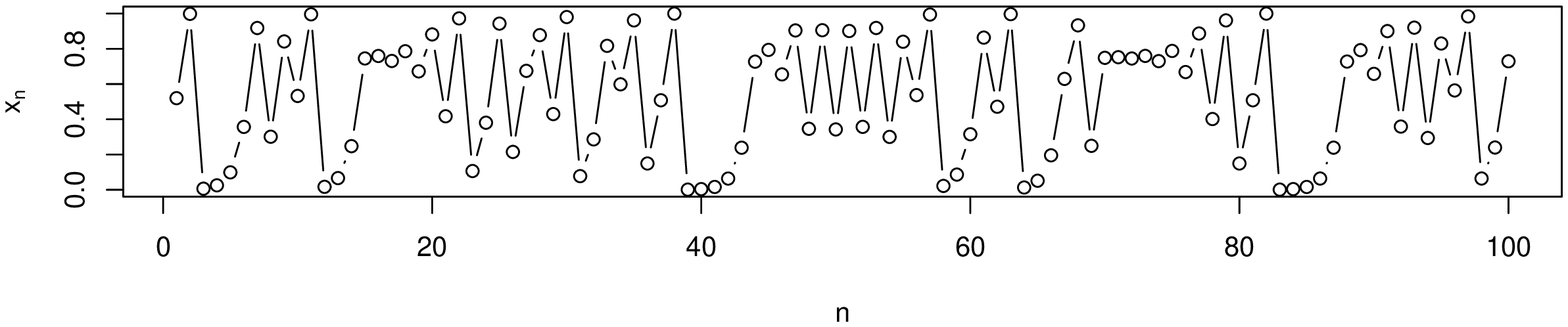}\\
\includegraphics[width=\figwidth]{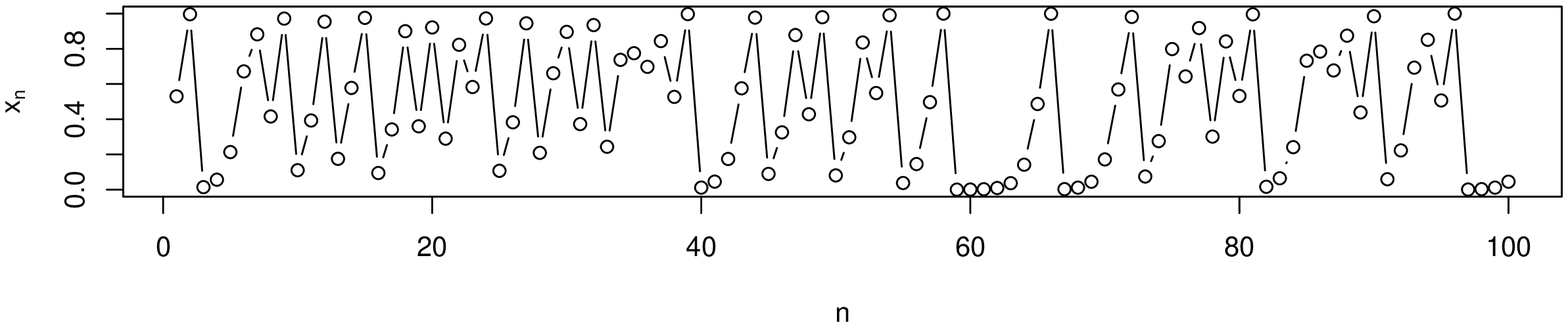}\\
\includegraphics[width=\figwidth]{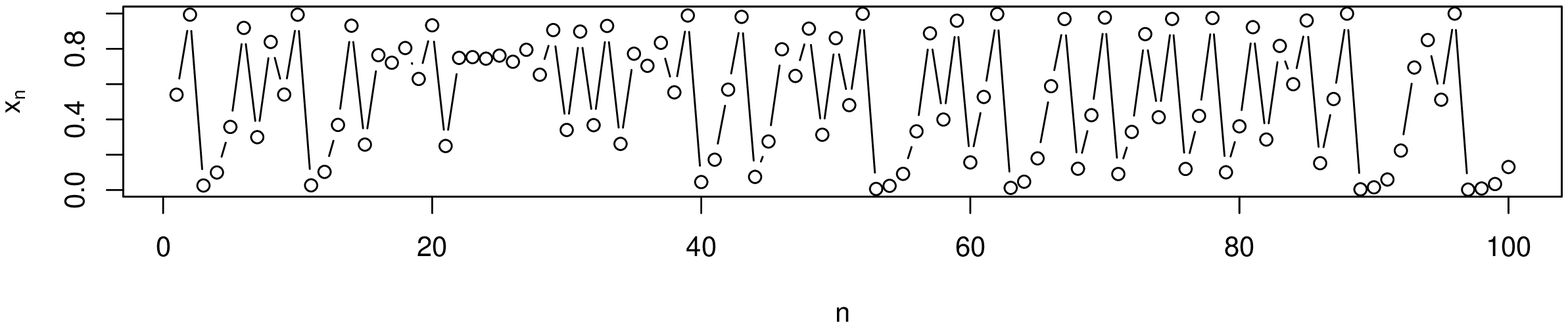}
\caption[Numerical trajectories for the map $x \mapsto 4 x (1 - x)$.]{Numerical
trajectories for the map $x \mapsto 4 x (1 - x)$.  The initial
conditions different only slightly for each trajectory.}
\label{fig.quadmapev}
\end{center}
\end{figure}
The qualitative behavior of this system is most easily appreciated
graphically, as in Figure~\ref{fig.quadmapev} which plots $x_n$ vs.\
$n$ for typical trajectories obtained for different choices of $x_0$.
Each of these trajectories is erratic, and random in the sense that no
regularity is apparent.  Furthermore, it can be seen by comparing the
graphs in Figure~\ref{fig.quadmapev} that two near-identical initial
states eventually yield radically different time evolutions.  This
phenomenon, termed ``sensitive dependence on initial
conditions''~\cite{Lor63}, imposes strong limits on the predictability
of this system over long periods of time: a small error in the
determination of the initial condition rapidly becomes amplified to
the extent that reliable prediction of the future states of the system
eventually becomes impossible.  Thus, despite being entirely
deterministic, trajectories of this simple system have some hallmarks
of essentially \emph{random} phenomena: their behavior is irregular
and unpredictable.

The mechanisms underlying the random character of this system are
reasonably well understood (see \eg~\cite{Col80}), the key notion
being sensitivity to initial conditions and its consequences.
However, within this framework it is difficult to approach questions
of the type ``what is the asymptotic behavior of a \emph{typical}
trajectory of this system?''  Indeed, the very nature of sensitivity
to initial conditions would seem to preclude any notion of ``typical''
behavior, since even very similar initial conditions eventually lead
to their own very particular, uncorrelated evolutions.

However, different conclusions are reached if one takes a
probabilistic point of view.  Suppose that instead of being precisely
determined, the initial state $x_0$ has associated with it some
uncertainty.  In particular, suppose we know the initial probability
density, $\rho$, giving the probabilities of all possible values that
$x_0$ can take.  Then it makes sense to ask, ``what will be the
probability density of $x_1$, the new state after one iteration of the
map~\eqref{eq.quadmap}?''.  A precise answer to this question can be
found using analytical methods described in Chapter 2.  For an
approximate answer, it suffices to simulate a large ensemble of
different initial states $x_0$ distributed according to $\rho$, evolve
these states forward under the map~\eqref{eq.quadmap}, and approximate
the transformed density of the ensemble by constructing a histogram of
the ensemble of values $x_1$.  One can then proceed, in the same
fashion, to determine the probability densities of subsequent states
$x_2$, $x_3$, \etc.  Thus, even if the initial state $x_0$ is not
known precisely, it is at least possible to give a probabilistic
description of the system's evolution in terms of the evolution of a
probability density.

\begin{figure}
\begin{center}
\includegraphics[width=2.5in]{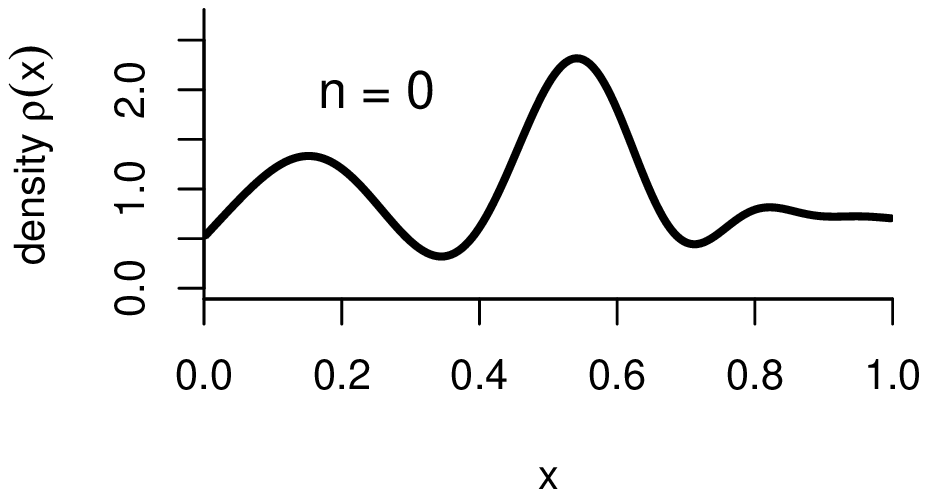}
\includegraphics[width=2.5in]{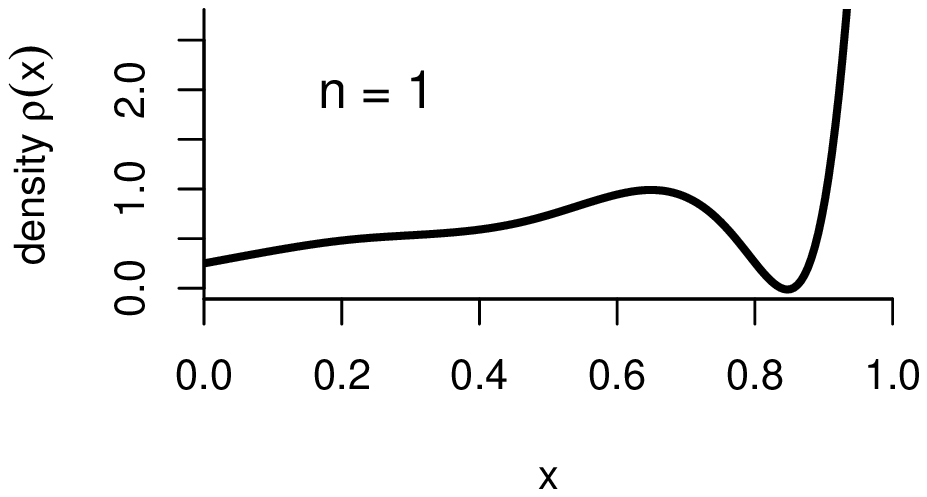}\\
\includegraphics[width=2.5in]{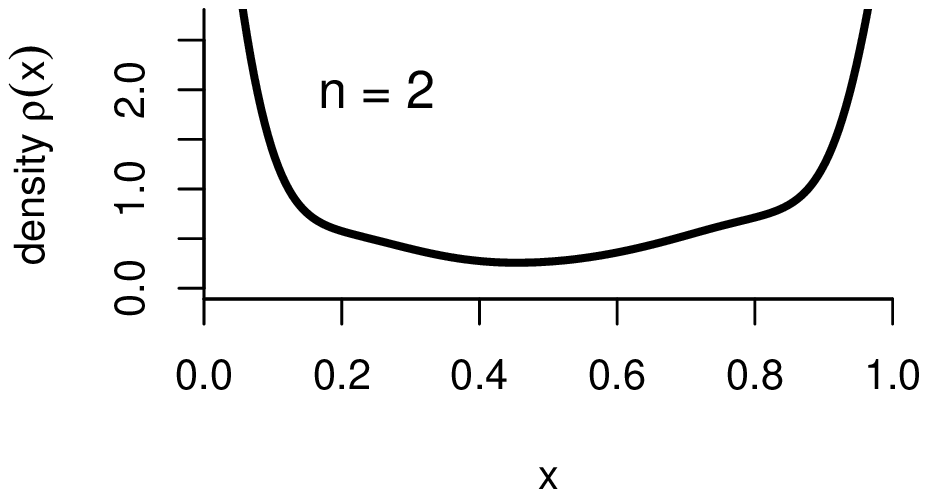}
\includegraphics[width=2.5in]{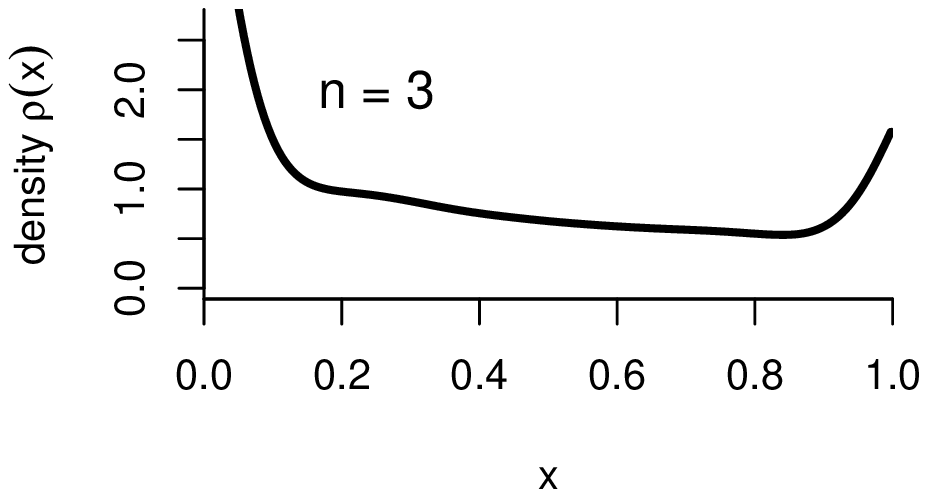}\\
\includegraphics[width=2.5in]{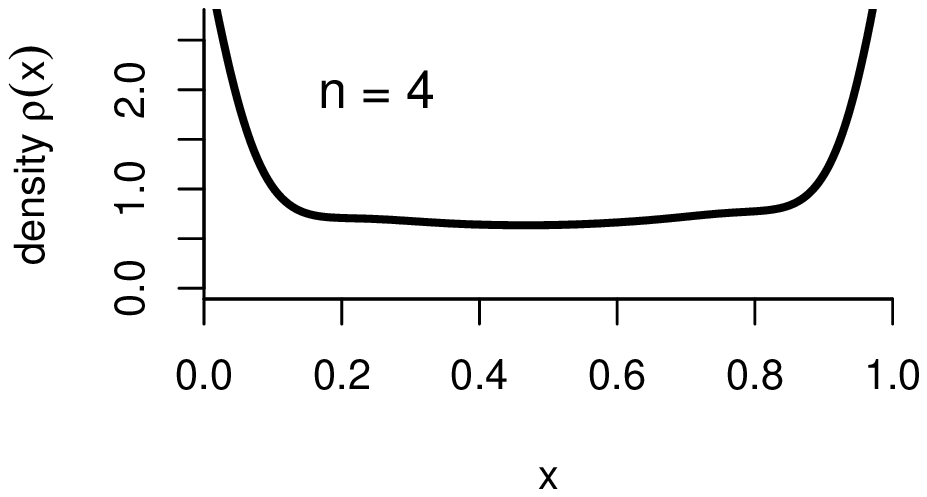}
\caption{Simulated evolution of an ensemble density $\rho$ under iterations
of the map $x \mapsto 4x (1 - x)$.}
\label{fig.quadmapdensev1}
\end{center}
\end{figure}
\begin{figure}
\begin{center}
\includegraphics[width=2.5in]{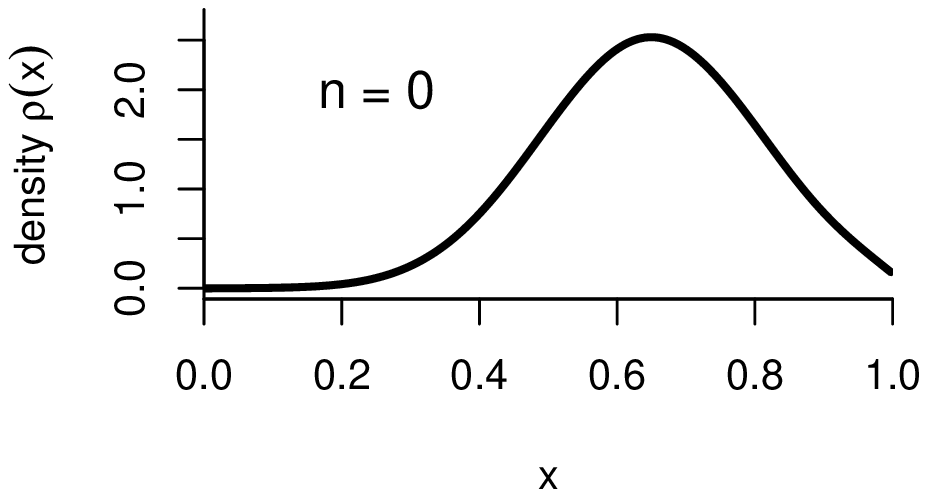}
\includegraphics[width=2.5in]{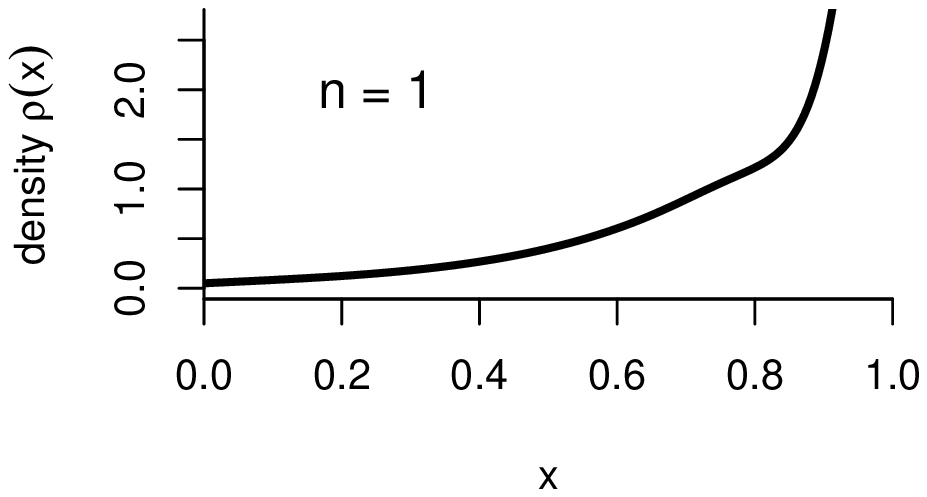}\\
\includegraphics[width=2.5in]{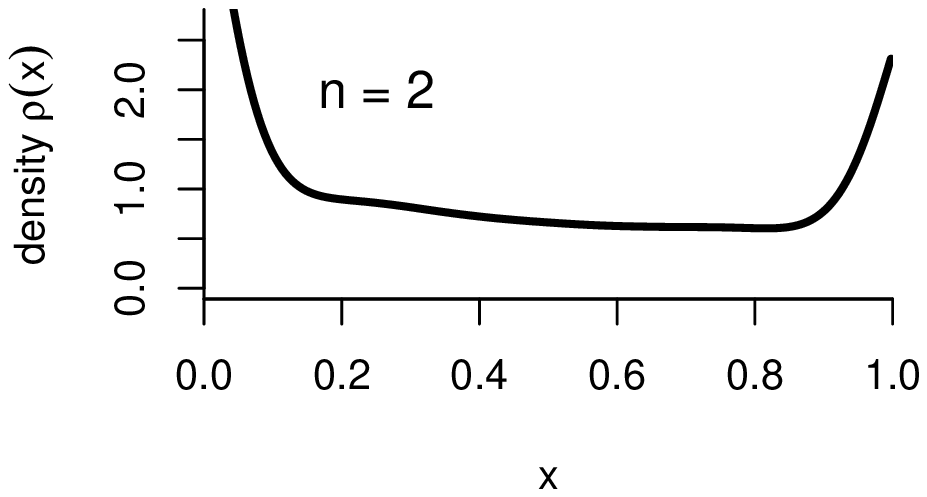}
\includegraphics[width=2.5in]{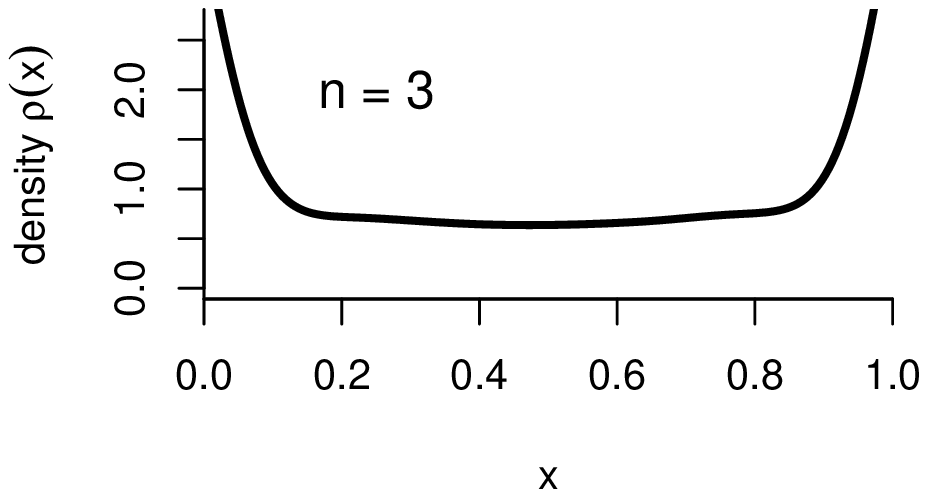}\\
\includegraphics[width=2.5in]{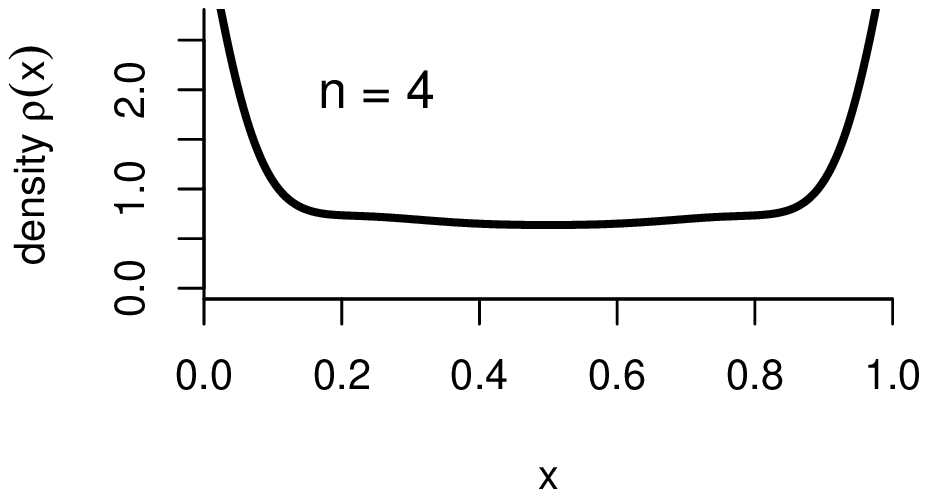}
\caption{As Figure~\ref{fig.quadmapdensev1}, with a different initial
density.}
\label{fig.quadmapdensev2}
\end{center}
\end{figure}
The graphs in Figure~\ref{fig.quadmapdensev1} show a particular choice
for the probability density $\rho$ of the initial state $x_0$,
together with the subsequent densities of states $x_1$, $x_2$, $x_3$,
$x_4$, obtained by numerical simulation of an ensemble of $10^6$
initial values distributed according to $\rho$, iterated forward under
the map~\eqref{eq.quadmap}.  The striking feature of this figure is
that the sequence of densities rapidly approaches an equilibrium or
invariant density that does not change under further iteration.
Moreover, the invariant density appears to be unique.  This is
supported by Figure~\ref{fig.quadmapdensev2}, which shows how a
different choice of initial density evolves toward the same
equilibrium density as before.

\begin{figure}
\begin{center}
\includegraphics[width=3.5in]{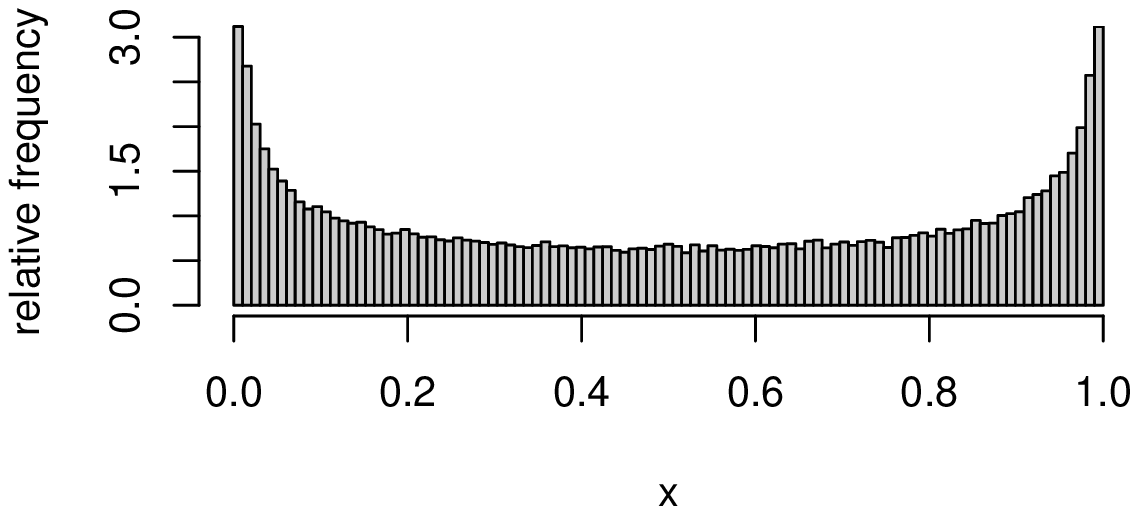}
\caption{Histogram of a trajectory $\{x_n\}$ generated by $10^6$
iterates of the map $x_{n+1}=4 x_n (1-x_n)$.}
\label{fig.quadmapinvdens}
\end{center}
\end{figure}
A different but related statistical approach to this system is to
focus on the statistics of a single trajectory.  For a given initial
state $x_0$, by iterating $x_{n+1} = 4 x_n (1-x_n)$ we obtain an
arbitrarily long sequence $\{x_n\}$ like the one illustrated in
Figure~\ref{fig.quadmapev}.  A histogram of this sequence reveals the
long-term frequency with which the trajectory visits various parts of
the interval $[0,1]$.  Figure~\ref{fig.quadmapinvdens} shows such a
histogram, for a trajectory of length $10^6$.  Remarkably, this
histogram reproduces the invariant density shown in
Figures~\ref{fig.quadmapdensev1} and~\ref{fig.quadmapdensev2}, which
arises in a different context.  Moreover, the same histogram is
obtained for almost \emph{any} choice of initial state.\footnote{There
are exceptions, such as $x_0=0$, that yield trajectories with
different (periodic) asymptotic behavior.  These exceptions are very
rare: in fact they constitute a set of Lebesgue measure $0$ (\cf\
Chapter~\ref{ch.ergth}).}  Thus the invariant density describes the
behavior of ``typical'' trajectories, \ie\ those whose statistics are
described by this particular probability density.

A probabilistic or ensemble treatment of dynamical systems provides a
point of view complementary to one given in terms of the evolution of
individual trajectories.  The iterated map~\eqref{eq.quadmap} is just
one example of a system that behaves erratically on the level of
individual trajectories, but has very regular asymptotic properties
when considered at the level of probability densities.  This
observation appears to hold for many other systems.  Moreover, it
turns out that the converse holds as well: various regularity
properties at the level of probability densities imply various degrees
of disorder in the evolution of individual trajectories.
Chapter~\ref{ch.ergth} explores these connections further.

% ======================================================================
\section{Delayed Dynamics}

The aim of the present work is to develop a probabilistic (\ie\
ergodic) approach to systems with delayed dynamics, particularly those
systems whose evolution can be described by a \emph{delay differential
equation}.  Delay differential equations (DDEs) arise in the
mathematical description of systems whose time evolution depends
explicitly on a past state of the system, as for example in the case
of delayed feedback.  Neural systems~\cite{adH79}, respiration
regulation~\cite{GM79}, agricultural commodity markets~\cite{AM},
nonlinear optics~\cite{Gibbs}, and neutrophil populations in the
blood~\cite{GM79} are but a few systems in which delayed feedback
leads naturally to a description in terms of a delay differential
equation.

We will restrict our attention to systems modeled by evolutionary
delay equations that can be expressed in the form
\begin{equation} \label{eq.dde0}
  x'(t) = f\big( x(t), x(t-\tau) \big), \quad x(t) \in \Rn, t \geq 0.
\end{equation}
Here the ``state'' of the system at time $t$ is $x(t)$, whose rate of
change depends explicitly, via the function $f$, on the past state
$x(t-\tau)$ where $\tau$ is a fixed time delay.  More general delay
equations might be considered: multiple time delays, variable time
delays, continuously distributed delays, and higher derivatives all
arise in applications and lead to more complicated evolution
equations.  Nevertheless, equations of the form~\eqref{eq.dde0}
constitute a sufficiently broad class of systems to be of practical
importance, and they will provide adequate fodder for the types of
problems we wish to consider.

Some delay equations give rise to erratic time evolutions; there are
numerous examples in the
literature~\cite{adH85,HM82,HW83,Farm82,Hale88,Ikeda87,Ucar02,Wal81}.
In some cases it is possible to describe precisely the sense in which
these systems are chaotic, and to identify generic mechanisms
responsible~\cite{adH85,HM82,HW83,Wal81}.  These works focus on the
application of topological notions of chaos, but little has been done
on interpreting delay equations in an ergodic or measure theoretic
context.  We are aware of a handful of published
works~\cite{adH85,Cap91,Ersh91,LM92,LM95}, among which there is no
consistent framework for discussing ergodic notions in the context of
DDEs.  One aim of the present work is to develop such a framework.

A probabilistic approach to delay equations in particular is desirable
for a variety of reasons.  It has been suggested~\cite{DH95,MM00} that
brains encode information at the level of large neuron ensembles,
rather than at the level of individual neurons.\label{neurons}
Statistical periodicity~\cite{LM94}, in which the ensemble density
cycles periodically rather than settling down to an invariant density,
provides one possible mechanism for representing information at this
level.  Statistical periodicity does not (and cannot) arise in
\emph{ordinary} differential equations, but it does occur in some
delay equations~\cite{LM95,Ma92}.  Moreover, since synaptic and
conduction delays introduce explicit delays into the dynamics of
neurons, delay equations arise naturally in models of neuron
dynamics~\cite{adH81,adH79,CR76}.  Since the dynamics of large
ensembles can be treated with probabilistic methods, a probabilistic
approach to delay equations (\ie\ a statistical mechanics of systems
with delayed dynamics) would provide the mathematical tools for
further developing this theory of brain functioning.

Delay equations also serve as relatively simple models for the study
of infinite dimensional systems.  Much of the complex dynamics that we
experience directly is extended in space as well as in time, for
example fluid turbulence and other systems modeled by partial
differential equations (PDEs).  These systems are infinite
dimensional, in that the state at any given time cannot be specified
precisely by a finite set of data.  There is a rich literature on
chaotic dynamics in PDEs, including some rigorous results on the
existence of compact attractors and Li and Yorke type chaos, as well
as numerous numerical observations, including methods for estimating
Lyapunov exponents and dimensions for chaotic attractors (see \eg\
\cite{CHZ98,Gh88,GT96,HM81,Luce95,TG95,YD97,Zub95} and references
therein).  However, to our knowledge, with the exception
of~\cite{Rud85,Rud87,Rud88} and \cite{CHZ98} little has been said of
infinite-dimensional systems within the context of ergodic theory.
Delay equations are also infinite dimensional, but in some respects
they are much simpler than PDEs.\footnote{For example, a delay
equation is much easier to solve, since it can often be represented as
a sequence of ordinary differential equations (method of steps), \cf\
Section~\ref{sec.methsteps}.}  Thus delay equations are a natural
place to begin formulating ergodic concepts applicable to a broader
class of spatially extended systems.

\section{Probabilistic Questions for Delay Equations}

The following chapters explore various problems related to a
probabilistic or ergodic-theoretic treatment of delay differential
equations.  Questions that arise naturally in this context, and which
we propose to consider, are the following:
\begin{itemize}
\item How can the language and concepts of ergodic theory be adapted and
applied in the context of delay equations?
\item How can the evolution of a probability distribution or ensemble
density for a given DDE be determined; \ie, can one find an evolution
equation for the probability density?
\item How can one determine invariant probability distributions for DDEs?
\end{itemize}
To some extent these problems can be considered independently, and
each is the subject of one of the chapters that follow.

Chapter~\ref{ch.ergth} presents the elements of ergodic theory needed
in our subsequent discussion.  It also serves to illuminate the scope
of ergodic concepts that might be considered in the context of DDEs,
and to further motivate the considerations of subsequent chapters.

Chapter~\ref{ch.framework} examines the fundamental theoretical
questions posed by an ergodic approach to delay equations, and
infinite dimensional dynamical systems in general.  Alternative
representations of a delay equation as a dynamical system are
considered, and the measure-theoretic framework needed for a
probabilistic treatment is developed.  As it turns out, even these
fundamental questions lead to unresolved ambiguities that are endemic
to infinite dimensional systems.

Chapter~\ref{ch.densev} explores the practical issue of how the
evolution of probability distributions for delay equations might be
determined.  In particular it would be nice to have an evolution
equation governing the evolution of a probability density under the
action of a delay equation.  Such an equation would be a useful tool,
for example, in making probabilistic predictions for systems governed
by DDEs, and for modeling ensembles (such as neuron populations,
\cf\ page~\pageref{neurons}) whose microscopic dynamics are governed by
DDEs.  A number of different approaches to this problem are explored,
including analytical methods and computer algorithms.

Invariant measures, which determine the long-term statistical
properties of deterministic systems, are important quantities in
statistical mechanics and also play a fundamental role in ergodic
theory.  Chapter~\ref{ch.invdens} considers the problem of determining
invariant measures for delay equations.  Here the focus is on finding
computer algorithms for computing invariant measures and their
densities.

Transient chaos is a phenomenon in which a deterministic system has a
chaotic time evolution for a transient period of time after which it
asymptotically becomes regular (\eg\ periodic).  Such systems can be
described in ergodic theoretic terms, but due to the instability of
the ``chaotic set'' specialized techniques are required for their
numerical analysis.  Chapter~\ref{ch.transient} illustrates the
existence of transient chaos in some delay equations, and shows how
numerical methods can be adapted to their analysis.  Until recently,
limits on computational resources made it impractical to numerically
investigate transient chaos in infinite dimensional systems such as
DDEs.

\chapter{Ergodic Theory} \label{ch.ergth}

\begin{singlespacing}
\minitoc   % mini table of contents for this chapter
\end{singlespacing}

\newpage
This chapter introduces the basic concepts of the ergodic theory of
dynamical systems.  Following a brief review of measure theoretic
probability (those familiar with measure theory might like to skip
directly to Section~\ref{sec.probev}), we consider the problem of
determining the evolution of a probability measure or density under
the evolution of a dynamical system.  This serves to formalize some of
the intuitive ideas presented in Chapter~\ref{ch.intro}.  In
section~\ref{sec.ergth} we then turn to the characterization, using
ergodic concepts, of dynamical systems that exhibit irregular
behavior.  The general discussion given here sets the context in which
we will consider the problem of an ergodic/probabilistic treatment of
delay differential equations.

% ============================================================================
\section{Dynamical Systems Formalism} \label{sec.dynsys}

The essential properties of a deterministic evolutionary system are
encapsulated in the mathematical notion of a \emph{semigroup}.  At any
particular time
\nomenclature[t]{$t$}{time variable ($\in \Rplus$ or \Zplus) for a dynamical system}%
$t$ the state of such a system is identified with an element $x_t$
\nomenclature[xt]{$x_t$}{phase point in $X$ at time $t$ for a
dynamical system}%
(the current \emph{phase point}) of a \emph{phase space} $X$.  In
practice $X$ is often simply \Rn\ or \Reals, and the phase point or
state $x$ represents the numerical value of some physical quantity,
\eg\ a voltage, displacement, population,
\etc.  In some of the following we will require only that $X$ have a topology,
but occasionally we will want $X$ to be equipped with some further
structure.

The term ``deterministic'' implies that the phase point $x_t$ is, for
all times $t>0$, uniquely determined by the initial phase point $x_0$.
That is, for each $t>0$ there is a rule that determines $x_t$ from
$x_0$---in other words, a transformation $S_t: X \to X$ that takes the
initial state $x_0$ to the future state $x_t$:
\begin{equation} \label{eq.sg1}
  x_0 \mapsto S_t(x_0) = x_t.
\end{equation}

In many systems, the dynamical law that determines $S_t$ does not
change with time (\eg\ autonomous systems evolving under
time-invariant physical laws such as Newton's laws of motion).  Then
the result of the time evolution depends only on the initial phase
point and the amount of time elapsed, not on the particular moment
identified as the initial time.  The identification of $t=0$ as the
initial time in equation~\eqref{eq.sg1} is then somewhat arbitrary.
Consequently, the evolution of an initial phase point $x_0$ to time
$s+t$ can be accomplished by a sequence of two evolutions, one from
time $0$ to time $s$ carried out by the transformation $S_s$, followed
by the evolution from time $s$ to time $t$ carried out by the
transformation $S_t$.  This can be expressed by the relation
\begin{equation} \label{sg.sgprop1}
  S_{s+t} = S_t \circ S_s.
\end{equation}
That is, the family of transformations $\{S_t: t \geq 0\}$ form a
semigroup:

\begin{defn}[semigroup of transformations]\label{def.sg}
Let either $G = \Zplus = \{0,1,2,\dots\}$ or $G=\Rplus=\{t \in
\Reals: t \geq 0\}$.
\nomenclature[R]{$\Rplus$}{the set $\{x \in \Reals: x \geq 0\}$}
\nomenclature[Z]{$\Zplus$}{the set $\{x \in \mathbb{Z}: x \geq 0\} = \{0,1,2,\ldots\}$}%
A one-parameter \emph{semigroup of transformations} $\{S_t: t \in
G\}$
\nomenclature[S]{$S_t$}{dynamical system or semigroup of transformations}
\nomenclature[Xa]{$X$}{phase space}
on $X$ is a family of transformations $S_t: X \to X$ satisfying
\begin{enumerate}
  \item \label{sg1} $S_0(x) = x \quad \forall x \in X$,
  \item \label{sg2} $S_t \circ S_{t'} = S_{t+t'} \quad \forall t, t' \in G$.
\end{enumerate}
\end{defn}
 
Thus the evolution of a deterministic system under a time-invariant
dynamical law can be described by an evolution equation $x_t = S_t
x_0$ where the family of transformations $\{S_t: t \in G\}$ forms a
semigroup.  For discrete-time systems $G=\Zplus$; for continuous-time
systems $G=\Rplus$.

If the mapping $x_0 \mapsto x_t$ is also
continuous with respect to both $x_0$ and $t$, the semigroup is
called a semidynamical system:

\begin{defn}[semidynamical system] \label{def.semidyn}
Let $X$ be a Banach space with $\{S_t: t \in G\}$ a semigroup of
transformations on $X$.  Then $\{S_t\}$ is a
\emph{semidynamical system} if the mapping
\begin{equation}
  (x,t) \mapsto S_t(x)
\end{equation}
from $X \times G$ into $X$ is continuous.
\end{defn}

\begin{defn}[orbit; trajectory]
Let $\{S_t: t \in G\}$ ($G=\Rplus$ or \Zplus) be a semidynamical
system on $X$, and let $x_0 \in X$.  The set
\begin{equation}
  \{S_t (x_0): t \geq 0\} \subset X
\end{equation}
is called the \emph{orbit} or \emph{trajectory} originating at $x_0$.
\end{defn}

In general, the evolution of a semidynamical system cannot be extended
uniquely to all times in the past.  However, if each of the $S_t$ is
one-to-one (hence invertible), the identification $S_{-t} = S_t^{-1}$
extends the family of transformations to all $t<0$, and the initial
phase point does indeed uniquely determine $x_t$ for all times in the
past.  In this case the resulting \emph{group} (rather than semigroup)
of transformations $\{S_t\}$ is called a
\emph{dynamical system}.  In the present work we will be concerned
with systems that in general are not invertible.  Nevertheless, for
simplicity's sake we will frequently abuse notation and refer to these
systems simply as dynamical systems, as the technical distinction
between dynamical and semi-dynamical systems will not play an important
role.

% ============================================================================
\section{Measure Theoretic Probability} 

The evolution of a probability distribution under the action of a
dynamical system requires the structures of measure theoretic
probability.  To this end, the elements of this theory are reviewed in
the following.  For a more complete review see \eg\ the standard
reference~\cite{Hal74}.

\subsection{Measure theory} \label{sec.mtheory}

Basically, a measure is a function that assigns a ``size'' to a subset
$A \subset X$.  It turns out that it is not possible, in general, to
do this consistently for just any subset of $X$ (see
\eg\ \cite{Roy68} for a construction of an ``unmeasurable'' set).  Rather, such
an assignment can only be made consistently for a collection of
subsets called a $\sigma$-algebra.

\begin{defn}[$\sigma$-algebra, measurable space]
Let $X$ be a set.  A non-empty collection \Acal\ of subsets of $X$ is
called a \emph{$\sigma$-algebra},
\nomenclature[A]{\Acal}{$\sigma$-algebra}
and the pair $(X,\Acal)$
\nomenclature[Xb]{$(X,\Acal)$}{measurable space}
is called a \emph{measurable space}, if each of the following holds.
\begin{enumerate}
\item If $A \in \Acal$ then $X \setminus A \in \Acal$.
\item If $\{A_i\} \subset \Acal$ is a finite or countable collection then
$\bigcup_i A_i \in \Acal$.
\item\label{alg.int} If $\{A_i\} \subset \Acal$ is a finite or countable collection then
$\bigcap_i A_i \in \Acal$.
\item If $A, B \in \Acal$ then $A \setminus B \in \Acal$.
\item\label{alg.X} $X \in \Acal$ and $\emptyset \in \Acal$.
\end{enumerate}
\end{defn}

The elements of a $\sigma$-algebra are sometimes called
\emph{measurable sets}.  In the context of probability theory they are
sometimes called \emph{events}.  Note that the properties above are
not independent, \eg\ properties 3--5 follow from others.  We include
all five here to give a more thorough and intuitive characterization
of $\sigma$-algebras.

Finite $\sigma$-algebras are easy to construct.  For example $\{X,
\emptyset\}$ is a trivial $\sigma$-algebra.  Given $A \subset X$, the
collection $\{X, \emptyset, A, X \setminus A\}$ is a $\sigma$-algebra.
Although less trivial $\sigma$-algebras are sometimes extremely
difficult to characterize explicitly, the existence of a
$\sigma$-algebra containing certain sets can be asserted using the
following theorem.
\begin{thm}
Let $X$ be a set and let $\mathcal{B}$ be a class of subsets of $X$.
Then there is a unique smallest $\sigma$-algebra, denoted
$\sigma(\mathcal{B})$%
\nomenclature[S]{$\sigma(\mathcal{B})$}{$\sigma$-algebra generated by a class
$\mathcal{B}$ of sets}
(the $\sigma$-algebra \emph{generated} by $\mathcal{B}$), that
contains every set in $\mathcal{B}$.  That is,
\begin{itemize}
\item $\mathcal{B} \subset \sigma(\mathcal{B})$, and
\item if $\mathcal{A}$ is any $\sigma$-algebra such that $\mathcal{B} \subset \mathcal{A}$,
then $\sigma(\mathcal{B})\subset \mathcal{A}$.
\end{itemize}
\end{thm}
\begin{proof}
See \eg\ \cite{Hal74}.
\end{proof}

We frequently work with a phase space $X$ that comes already equipped
with a topology (\eg, a metric space).  An important application of
the preceding theorem is the identification of a $\sigma$-algebra that
is compatible with a given topology:

\begin{defn}[Borel $\sigma$-algebra]
Let $X$ be a topological space.  The $\sigma$-algebra \Acal\ generated
by the collection of all open subsets of $X$ is called the \emph{Borel
$\sigma$-algebra} on $X$, and the elements of \Acal\ are called
\emph{Borel sets}.
\end{defn}

The Borel $\sigma$-algebra is a natural choice since the sets we
typically want to work with---open and closed sets, as well as their
finite and countable unions and intersections---are all measurable.
In the following, whenever a topology is given or implied the
corresponding Borel $\sigma$-algebra will be implied also.

\begin{defn}[measure, measure space]
Let $(X,\Acal)$ be a measurable space.  A function $\mu: \Acal \to
\Rplus$
\nomenclature[M]{$\mu(A)$}{measure of a set $A$}
is called a \emph{measure}, and the triple $(X,\Acal,\mu)$
\nomenclature[Xb]{$(X,\Acal,\mu)$}{measure space or probability space}
is called a \emph{measure space}, if each of the following holds.
\begin{enumerate}
\item $\mu(\emptyset)$ = 0.
\item $\mu(A) \geq 0 \quad \forall A \in \Acal$.
\item If $\{A_i\} \subset \Acal$ is a finite or countable collection of pairwise
disjoint sets then\\ $\mu\big(\bigcup_i A_i\big) = \sum_i \mu(A_i)$.
\item\label{measure.minus} If $A, B \in \Acal$ and $B \subset A$ then
$\mu(A \setminus B) = \mu(A) - \mu(B)$.
\end{enumerate}
\end{defn}
Note that~\ref{measure.minus} is not an independent property, but is a
useful consequence of the first three.

In the following, ``a measure space $X$'' will be taken to mean ``a
measure space $(X,\Acal,\mu)$'' wherever this is unlikely to cause
confusion.

\begin{defn}[Borel measure]
Given a topological space $X$, a \emph{Borel measure} on $X$ is a
measure defined on the Borel $\sigma$-algebra generated by the
topology on $X$.
\end{defn}

On the real line, the notion of length provides a natural Borel
measure $\mu$, defined on intervals according to
\begin{equation}
  \mu\big( [a,b] \big) = b - a.
\end{equation}
Similarly, in \Rn\ the notion of volume also provides a natural Borel
measure, defined on $n$-dimensional ``rectangles'' according to
\begin{equation}
  \mu\big( [a_1,b_1] \times \cdots \times [a_n,b_n] \big) = (b_1 - a_1) \cdots (b_n - a_n).
\end{equation}
It can be proved~\cite{Hal74} that this formula can be uniquely
extended to a measure on the respective Borel $\sigma$-algebra of \Rn.
\emph{Lebesgue measure}, $\lambda$,
\nomenclature[L]{$\lambda$}{Lebesgue measure on \Rn}
is the completion of this measure.\footnote{A measure $\lambda$ is the
completion of a Borel measure $\mu$ if
\begin{itemize}
\item $\lambda$ is defined on the smallest $\sigma$-algebra containing both the Borel
sets and all subsets of Borel sets having Borel measure $0$.
\item $\lambda$ agrees with $\mu$ wherever $\mu$ is defined; in particular, $\lambda(A)=0$ if $A \subseteq B$ and $\mu(B)=0$.
\end{itemize}
}

\begin{defn}[null set]
Let $(X,\Acal,\mu)$ be a measure space.  A set $A \subset X$ is called
a \emph{null set} if it is contained in a set $B \in \Acal$ with
$\mu(B)=0$.
\end{defn}

\begin{defn}[almost everywhere]
Let $(X,\Acal,\mu)$ be a measure space.  A property is said to hold
$\mu$-almost everywhere if it holds on the complement of a null set.
\end{defn}

% ------------------------------------------------------------------------------------
\subsection{Probability measures}

With the basic concepts of measure theory outlined above, the elements
of probability theory that we will need can be expressed in measure
theoretic terms.

\begin{defn}[finite and probabilistic measures]
Let $X$ be a measure space.  If $\mu(X) < \infty$ then $\mu$ is said
to be \emph{finite}.  If $\mu(X) = 1$ then the finite measure $\mu$ is
said to be \emph{probabilistic} and $(X,\Acal,\mu)$ is called a
\emph{probability space}.
\end{defn}

Note that any finite measure $\mu$ can be normalized to yield a
corresponding probability measure $\nu$ according to $\nu(A) =
\mu(A)/\mu(X)$.

Depending on context a probability measure has at least two different
interpretations, owing to the similarity in structure between
probability and statistics and the fact that probability measures
embody this structure:

\begin{description}
\item[Probabilistic interpretation] A probability measure $\mu$ can represent the
``probability distribution'' of a random variable $x \in X$.  For a
given measurable set $A \subset X$, the quantity $\mu(A)$ represents
$\text{Prob}(x \in A)$.  The measurable sets are called \emph{events},
since the statement ``$x \in A$'' asserts the occurrence of a
particular event or outcome.  The defining properties of a probability
measure can be interpreted in this context as the axioms of
probability theory.
\item[Statistical ensemble interpretation] A probability measure $\mu$ can represent the
distribution of an idealized, infinite ensemble of points $\{x_i\}
\subset X$.  The quantity $\mu(A)$ describes the fraction of the $x_i$ for
which $x_i \in A$.
\end{description}

This dual interpretation of probability measures is helpful in solving
problems, since we are free to use whichever interpretation is most
convenient in a given context.

In the following we will require the notion of measurable transformation:

\begin{defn}[measurable transformation]
Let $(X,\Acal)$ and $(Y,\mathcal{B})$ be measurable spaces.  A
transformation $S:X \to Y$ is \emph{measurable} if
\begin{equation}
  S^{-1}(B) \in \Acal \quad \forall B \in \mathcal{B},
\end{equation}
where the pre-image $S^{-1}(B)$ is defined as the set
\begin{equation}
  S^{-1}(B) = \{x \in X: S(x) \in B\}.
\end{equation}
\end{defn}
\nomenclature[S]{$S^{-1}(A)$}{pre-image of a set $A$ under a transformation $S$}

In particular, a transformation $S: X \to X$ is measurable if
$S^{-1}(A) \in \Acal$ $\forall A \in \Acal$.  A functional $f: X \to
\Reals$ is measurable if $f^{-1}(B) \in \Acal$ for every Borel set
$B \subset \Reals$.  Note that the set $S^{-1}(B)$ is well defined
even if $S$ does not have an inverse.

Continuity is a sufficient condition for measurability:

\begin{thm} \label{thm.contmeas}
Let $X$, $Y$ be topological spaces and \Acal, $\mathcal{B}$ their
respective Borel $\sigma$-algebras.  If $S: X \to Y$ is continuous
then $S$ is measurable.
\end{thm}
\begin{proof}
See \eg~\cite{Hal74}.
\end{proof}

% --------------------------------------------------------------------------------
\subsection{Lebesgue integral} \label{sec.integral}

We will require a notion of integration on measure spaces.  One such
notion that can be defined on an arbitrary measure space (and requires
no other structure) is the \emph{Lebesgue integral}.  Various
developments of the theory of this integral have been given, all
somewhat lengthy.  The following, based on the presentation given
in~\cite{LM94}, gives the basic idea; for a more complete treatment
see \eg\ \cite{Hal74}.

An intuitive way to approach the Lebesgue integral is to define it
first for ``simple functions''.  Given a subset $A \subset X$, denote
by $1_A$ the ``indicator function''
\begin{equation}
  1_A(x) = \begin{cases}
	     1 & \text{if $x \in A$}\\
             0 & \text{if $x \notin A$}.
	   \end{cases}
\end{equation}
\nomenclature{$1_A$}{indicator function of a set $A$}

\begin{defn}[simple function]
Let $(X,\Acal,\mu)$ be a measure space, and suppose the sets $A_i \in
\Acal$, $i = 1,\ldots,n$, are pairwise disjoint.  Then a functional
$g: X \to \Reals$ of the form
\begin{equation} \label{eq.simpfn}
  g(x) = \sum_{i=1}^n \lambda_i 1_{A_i}(x), \quad \lambda_i \in \Reals
\end{equation}
is called a \emph{simple function}.
\end{defn}

\begin{defn}
The Lebesgue integral of the simple function~\eqref{eq.simpfn} is
defined by
\begin{equation}
  \int_X g(x) \, d\mu(x) \equiv \sum_{i=1}^n \lambda_i \mu(A_i).
\end{equation}
\end{defn}

The definition of Lebesgue integral proceeds by approximating such
more general functions with simple functions.  If $f: X \to \Reals$ is
a non-negative bounded measurable function, then $f$ can be
approximated by simple functions $f_n$ such that $\{f_n(x)\}$
converges to $f(x)$ uniformly in $x$~\cite[p.\,20]{LM94}.

\begin{defn}
If $f: X \to \Reals$ is a non-negative bounded measurable function,
and $\{f_n\}$ is a sequence of simple functions converging uniformly
to $f$ then
\begin{equation}
  \int_X f(x) \, d\mu(x) \equiv \lim_{n \to \infty} \int_X f_n(x) \, d\mu(x).
\end{equation}
(This limit exists and is independent of the particular sequence
$\{f_n\}$.)
\end{defn}

The preceding definition can be extended unambiguously to define the
integral over $X$ of an arbitrary measurable functional $f:X
\to \Reals$.  If the integral is finite then $f$ is said to be \emph{integrable}.
For arbitrary $A \in \Acal$ we define
\begin{equation}
  \int_A f(x)\,d\mu(x) \equiv \int_X 1_A(x) f(x) \, d\mu(x).
\end{equation}
Where confusion is unlikely we will sometimes write simply $\int
f\,d\mu$ to mean $\int_X f(x)\,d\mu(x)$.

While the Lebesgue integral is more general than the usual Riemann
integral, the two notions of integral agree for any Riemann-integrable
function if the integral is taken with respect to Lebesgue
measure~\cite[p.\,323]{Ru76}.

Some important properties of the Lebesgue integral, which we give
without proof, are the following.
\begin{itemize}
\item $\int |f|\,d\mu = 0$ if and only if $f=0$ almost everywhere.
\item If $\int_A f\,d\mu = \int_A g\,d\mu$ $\forall A \in \Acal$ then
$f=g$ almost everywhere.
\item If $f:X \to \Reals$ is integrable and the finite or countable
collection of sets $\{A_i\} \subset \Acal$ is disjoint with $\bigcup_i
A_i = A$, then
\begin{equation}
  \sum_i \int_{A_i} f \, d\mu = \int_A f\,d\mu.
\end{equation}
\end{itemize}

\begin{thm}[change of variables] \label{thm.cov}
Let $(X,\mu)$ be a measure space, $S: X \to Y$ a measurable
transformation.  If $f: Y \to \Reals$ is measurable then $\mu \circ
S^{-1}$ is a measure on $Y$, and for any measurable $A \subset Y$,
\begin{equation}
  \int_A f \, d(\mu \circ S^{-1}) = \int_{S^{-1}(A)} (f \circ S) \, d\mu.
\end{equation}
\end{thm}
\begin{proof}
See \eg\ \cite{Hal74}.
\end{proof}

The Lebesgue integral leads to an important class of normed vector
spaces:
\begin{defn}[$L^p$ space]
Let $(X,\Acal,\mu)$ be a measure space and $p$ a real number, $1 \leq
p < \infty$.  The set of functionals $f: X \to \Reals$ such that
$|f|^p$ is integrable, \ie,
\begin{equation}
  \int_X{|f(x)|^p \, d\mu(x)} < \infty,
\end{equation}
is denoted by $L^p(X,\Acal,\mu)$,
\nomenclature[L]{$L^p$}{$L^p$ space of functions}
or simply $L^p(X)$ if \Acal\ and $\mu$ are understood.  The norm
$\|\cdot\|$ on $L^p$ is defined by
\begin{equation}
  \|f\| = \left[ \int{|f|^p \, d\mu} \right]^{1/p}.
\end{equation}
\end{defn}

% --------------------------------------------------------------------------------
\subsection{Densities} \label{sec.densities}

Densities provide a convenient and intuitive way of prescribing
probability measures on \Rn.  Indeed, applied probability and
statistics is concerned largely with certain specific types of
densities (\eg, the uniform, Gaussian, and Poisson densities, among
others).  Densities have the additional advantage that they allow many
problems in measure theory and probability to be solved using
calculus.

Suppose $(X,\Acal,\mu)$ is a measure space.  If $f: X \to \Reals$ is a
non-negative integrable function then
\begin{equation} \label{eq.densdef1}
  \mu_f(A) = \int_A f \, d\mu
\end{equation}
\nomenclature[M]{$\mu_f$}{probability measure with density functional $f$}%
defines a finite measure on $A$.  For a certain class of measures,
this observation can be reversed:

\begin{defn}[absolutely continuous measure]
Let $(X,\Acal,\mu)$ be a measure space.  A measure $\nu$ on \Acal\ is
\emph{absolutely continuous} with respect to $\mu$ (written $\nu \ll \mu$)
\nomenclature[0]{$\ll$}{$\nu \ll \mu$ means measure $\nu$ is absolutely continuous with respect to $\mu$}
if for every set $A \in \Acal$ for which $\mu(A) = 0$ we have that
$\nu(A) = 0$.
\end{defn}

\begin{thm}[Radon-Nikodym] \label{thm.RN}
Let $X$ be a $\sigma$-finite measure space\footnote{A measure space
$X$ is $\sigma$-finite if $X=\bigcup_i A_i$ with $\mu(A_i) <
\infty$~\cite{Hal74}.  Any measure we might wish to consider for
practical purposes is $\sigma$-finite.} and let $\mu_f$ be a finite
measure on $\Acal$ with $\mu_f \ll \mu$.  Then there exists a unique,
non-negative integrable functional $f \in L^1(X)$ such that
\begin{equation} \label{eq.densdef2}
  \mu_f(A) = \int_A f \, d\mu \quad \forall A \in \Acal.
\end{equation}
\end{thm}
\begin{proof}
See \eg\ \cite[p.\,69]{Fr70}.
\end{proof}

Absolute continuity with respect to a given measure $\mu$ defines an
important class of measures $\mu_f$ that can be expressed in the
form~\eqref{eq.densdef2}.  If $\mu_f$ is a probability measure, the
functional $f$ in equation~\eqref{eq.densdef2} is called the
\emph{probability density} or simply the \emph{density} of $\mu_f$.
\begin{defn}[Set of densities]
Let $(X,\Acal,\mu)$ be a measure space.  The space
\begin{equation}
D(X) = \{ f \in L^1(X): f \geq 0 \text{ and } \|f\| = 1 \}
\end{equation}
\nomenclature[D]{$D(X)$}{set of densities on a phase space $X$, a
subset of $L^1(X)$}%
is the set of \emph{densities on $X$}, and an element $f \in D(X)$ is
called a \emph{density}.
\end{defn}
On \Rn\ densities are usually given with respect to Lebesgue measure.

% ============================================================================
\section[Evolution of Probabilities]{Evolution of Probabilities for Dynamical Systems} \label{sec.probev}

\subsection{Evolution of probability measures}

Suppose we have a dynamical system described by a semigroup $\{S_t: t
\in G\}$ ($G = \Rplus$ or $\Zplus$) of transformations $S_t: X \to X$,
such that $x_0 \mapsto S_t(x_0) = x_t$.  Suppose further that $X$ is
equipped with a $\sigma$-algebra \Acal, and that a particular measure
$\nu_0$ on \Acal\ has the interpretation that $\nu_0(A) =
\text{Prob}(x_0 \in A)$.  One can then ask, what is the probability
measure $\nu_t$ that describes $x_t$, \ie, such that
$\nu_t(A)=\text{Prob}\big(x_t \in A\big)$?

Consider that for arbitrary $A \in \Acal$ we have
\begin{equation}
\begin{split}
  \nu_t(A) &= \text{Prob}\big( S_t\big( x_0 \big) \in A \big) \\
         &= \text{Prob}\big( x_0 \in S_t^{-1}(A) \big) \\
         &= \nu_0 \big( S_t^{-1}(A) \big).
\end{split}
\end{equation}
Thus the probability measure $\nu_t$ describing $x_t$ is given by the
formula
\begin{equation} \label{eq.muev1}
  \nu_t = \nu_0 \circ S_t^{-1}.
\end{equation}
It is readily verified that this expression does indeed define a
probability measure $\nu_t$ on \Acal.  This measure is said to be the
\emph{image} of $\nu_0$ under the transformation $S_t$ and (abusing
notation somewhat) we write $S_t \nu_0 \equiv \nu_0 \circ S_t^{-1}$.
Note that each of the $S_t$ must be a measurable transformation.
Since $\{S_t\}$ is assumed to be a semidynamical system (\cf\
Definition~\ref{def.semidyn}), the transformation $x \mapsto S_t(x)$
is indeed measurable since it is continuous.

% ----------------------------------------------------------------------------
\subsection{Evolution of densities: Perron-Frobenius operator} \label{sec.pfop}

If the initial phase point $x_0$ of a dynamical system is described
by a probability density rather than a probability measure, we can ask
``what is the probability density that describes $x_t = S_t\big( x_0
\big)$?  For this question to be well-posed, we require that $S_t$ be nonsingular:

\begin{defn}[nonsingular transformation]
Let $(X,\Acal,\mu)$ be a measure space.  A measurable transformation
$S:X \to X$ is \emph{nonsingular} (with respect to $\mu$) if for every
set $A \in \Acal$ for which $\mu(A)=0$ we have that
$\mu\big(S^{-1}(A)\big)=0$.
\end{defn}

Suppose the density of $x_0$ (with respect to some measure $\mu$) is
given by $f_0$, with corresponding probability measure $\nu_0$ (\cf\
equation~\eqref{eq.densdef1}).  If $S_t$ is nonsingular then from
equation~\eqref{eq.muev1} we have that
\begin{equation}
  \nu_t \ll \nu_0 \ll \mu.
\end{equation}
Therefore, by Theorem \ref{thm.RN} (assuming the measure space is
$\sigma$-finite) there is a unique element $f_t \in L^1(X)$ such that
\begin{equation}
  \nu_t(A) = \int_A f_t \, d\mu \quad \forall A \in \Acal.
\end{equation}
Using equation~\eqref{eq.muev1} we obtain
\begin{equation}
  \int_A f_t \, d\mu = \int_{S_t^{-1}(A)} f_0 \, d\mu \quad \forall A \in \Acal.
\end{equation}
In fact, for any given $f_0 \in L^1(X)$, this relationship uniquely
determines an element $f_t \in L^1(X)$~\cite[p.\,42]{LM94} and thus
defines a mapping $P_t: f_0 \mapsto f_t$.

\begin{defn}[Perron-Frobenius operator]
Let $(X,\Acal,\mu)$ be a measure space, $S: X \to X$ a nonsingular
transformation.  The operator $P:L^1(X) \to L^1(X)$ defined by
\begin{equation} \label{eq.pfdef1}
  \int_A (Pf) \, d\mu = \int_{S^{-1}(A)} f \, d\mu \quad \forall A \in \Acal
\end{equation}
is called the \emph{Perron-Frobenius} operator
\nomenclature[P]{$P$}{Perron-Frobenius operator}
\nomenclature[P]{$P_t$}{Perron-Frobenius operator corresponding to $S_t$}
corresponding to $S$.
\end{defn}

Note that if $P_t$ is the Perron-Frobenius operator corresponding to
$S_t$ then the action of $P_t$ on a density $f_0$ carries out the
evolution of this density under the action of $S_t$.

Some important properties of the Perron-Frobenius operator, which can
be verified directly from its definition, are the following~\cite{LM94}.
\begin{itemize}
\item $P: L^1(X) \to L^1(X)$ is linear.
\item $P f \geq 0$ if $f \geq 0$.
\item $\int (Pf) \, d\mu = \int f \, d\mu$.
\item If $P$ is the Perron-Frobenius operator with respect to $S$, then $P^n$ is the
Perron-Frobenius operator with respect to $S^n$.
\item If $\{S_t: t \in G\}$ is a semigroup then the corresponding family of
Perron-Frobenius operators $\{P_t: t \in G\}$ is also a semigroup.
\end{itemize}

From the first three of these properties it follows that $P$ is a
\emph{Markov operator}, \ie\ a linear transformation of $L^1(X)$ that maps
densities to densities.

% ------------------------------------------------------------------------------
\subsubsection{Explicit Perron-Frobenius operators}

In some cases equation~\eqref{eq.pfdef1} can be used to find an
explicit representation of the Perron-Frobenius operator.  If $S:
\Reals \to \Reals$ and $A = [a,x]$ then~\eqref{eq.pfdef1} becomes
\begin{equation}
  \int_a^x (Pf)(s) \, ds = \int_{S^{-1}([a,x])} f(s) \, ds.
\end{equation}
Differentiating then yields
\begin{equation} \label{eq.PFexpl1}
  Pf(x) = \frac{d}{dx} \int_{S^{-1}([a,x])} f(s) \, ds.
\end{equation}

Consider for example the dynamical system discussed in
Chapter~\ref{ch.intro}, defined by iterates of
\begin{equation}
  S: x \mapsto 4 x (1 - x).
\end{equation}
Since
\begin{equation}
  S^{-1}([0,x]) = [0,\tfrac{1}{2} - \tfrac{1}{2}\sqrt{1-x}] \cup
                  [\tfrac{1}{2} + \tfrac{1}{2}\sqrt{1-x}, 1],
\end{equation}
equation~\eqref{eq.PFexpl1} becomes
\begin{equation}
\begin{split}
  Pf(x) &= \frac{d}{dx} \int_0^{1/2 - 1/2\sqrt{1-x}} f(s) \, ds +
          \frac{d}{dx} \int_{1/2+1/2\sqrt{1-x}}^1 f(s) \, ds \\
   &= \frac{1}{4\sqrt{1-x}} \Big[ f\big(\tfrac{1}{2} - \tfrac{1}{2}\sqrt{1-x}\big) +
                                  f\big(\tfrac{1}{2} + \tfrac{1}{2}\sqrt{1-x}\big) \Big].
\end{split}
\end{equation}
Figure~\ref{fig.quaddens} illustrates the sequence of densities
$P^n(f)$, $n=0,\ldots,3$, obtained by equation~\eqref{eq.PFexpl1} for
a uniform initial density $f=1$ on $[0,1]$; \cf\
Figures~\ref{fig.quadmapdensev1}--\ref{fig.quadmapdensev2}.
\begin{figure}
\begin{center}
\includegraphics[width=\figwidth]{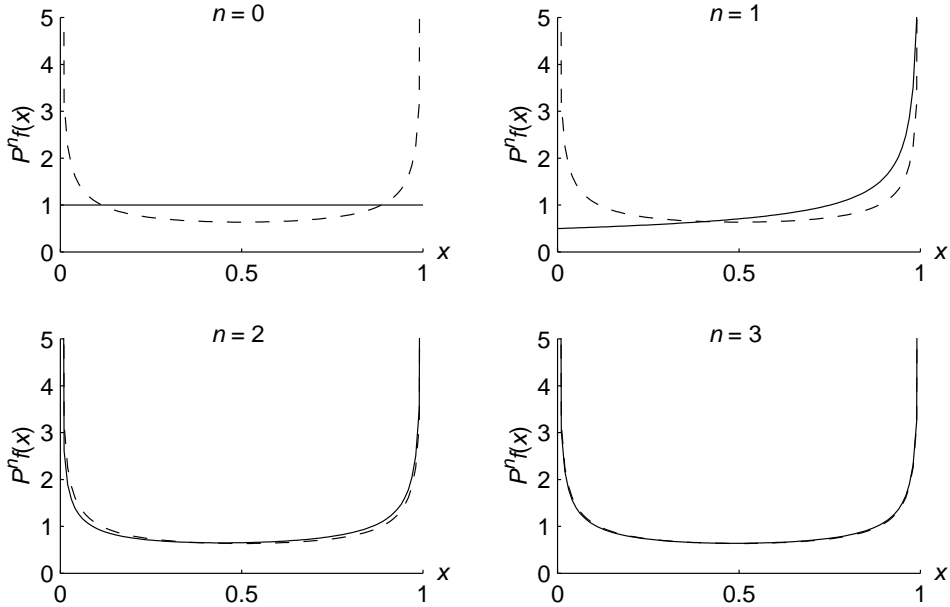}
\caption[Iterates $P^n(f)$ of the Perron-Frobenius operator
corresponding to the quadratic map $x \mapsto 4x(1-x)$, with
initial density $f = 1$.]{Iterates $P^n(f)$ of the Perron-Frobenius
operator corresponding to the quadratic map $x \mapsto 4x(1-x)$, for
the initial density $f = 1$.  Also shown, for comparison, is the
invariant density $f_{\ast}(x)$ (dashed curve~--~--~--) given by
equation~\eqref{eq.quadinvdens}.}
\label{fig.quaddens}
\end{center}
\end{figure}

% ----------------------------------------------------------------------------
\subsubsection{Perron-Frobenius operator for flows} \label{sec.fpode}

Autonomous ordinary differential equations consitute an important
class of continuous-time dynamical systems.  If the initial value problem
\begin{equation}
\begin{split}
  &\frac{dx}{dt} = F(x), \quad t \geq 0, \quad F: \Rn \to \Rn \\
  &x(0) = x_0
\end{split}
\end{equation}
has a unique solution for all $t>0$ \footnote{\eg\ it suffices to have
$F$ continuous with $|F(x)| \leq \alpha + \beta |x|$ for constants
$\alpha$, $\beta$~\cite{Dr77}.} for each initial value $x_0$, then the
family of solution maps $S_t: x_0 \mapsto x(t)$ constitute a
differentiable semigroup or \emph{flow} $\{S_t: t \in \Rplus\}$ on
\Rn.

Suppose $u(t) \in L^1(\Rn)$ is the probability density of the solution
variable $x$ at time $t$.  If the vector field $F$ is smooth and we
interpret $u(x,t)$ as a scalar quantity transported by the flow, then
by analogy with the continuity equation for scalar transport in fluid
mechanics, we have the following evolution equation for $u$,
\begin{equation} \label{eq.pfflow}
  \frac{\partial u}{\partial t} = - \nabla\cdot(u F).
\end{equation}
(See \eg~\cite[p.\,210]{LM94} for a measure-theoretic justification of
this evolution equation.)  The semigroup $\{P_t\}$ of Perron-Frobenius
operators corresponding to $\{S_t\}$ is then defined by
\begin{equation}
  (P_t f)(x) = u(x,t)
\end{equation}
where $u(x,t)$ is the solution of~\eqref{eq.pfflow} with initial data
$u(x,0)=f(x)$.

% ============================================================================
\section{Ergodic Theory} \label{sec.ergth}

Ergodic theory can be described as the study of measure theoretic
invariants of dynamical systems.  Invariant measures play a
fundamental role.

\subsection{Invariant measures}

Chapter~\ref{ch.intro} gave an example of a discrete-time dynamical
system, defined by iterates of the map
\begin{equation} \label{eq.quadmap2}
  S: x \mapsto 4 x (1 - x).
\end{equation}
Numerical evidence suggests the existence of a unique density that is
unchanged by evolution under the action the map~\eqref{eq.quadmap2}
(\cf\ Figures~\ref{fig.quadmapdensev1} and~\ref{fig.quadmapdensev2}).
The concepts introduced in the previous section provide the following
characterization of invariant densities and the more general notion of
invariant measure.

\begin{defn}[invariant measure]
Let $(X,\Acal)$ be a measurable space, $\{S_t: t \in G\}$ ($G=\Rplus$
or \Zplus) a dynamical system on $X$.  A measure $\mu$ on \Acal\ is
\emph{invariant} under $\{S_t\}$ if $S_t \mu = \mu$.  That is,
\begin{equation}
  \mu\big( S_t^{-1}(A) \big) = \mu(A) \quad \forall A \in \Acal, \; \forall t \in G.
\end{equation}
We say that $S_t$ preserves $\mu$, or that $S_t$ is $\mu$-preserving.
\end{defn}

Absolutely continuous invariant measures can be characterized as fixed
points of the Perron-Frobenius operator:

\begin{defn}[invariant density]
Let $(X,\Acal,\mu)$ be a measure space, $\{S_t: t \in G\}$ a dynamical
system such that each $S_t$ nonsingular.  A probability density $f$ is
\emph{invariant} under $\{S_t\}$ if $P_t f = f$, where $P_t$ is the
Perron-Frobenius operator with respect to $S_t$.  That is,
\begin{equation}
  \int_A f \, d\mu = \int_{S_t^{-1}(A)} f \, d\mu \quad \forall A \in \Acal,
 \; \forall t \in G.
\end{equation}
\end{defn}

For example, with the Perron-Frobenius operator~\eqref{eq.PFexpl1}
corresponding to the map \eqref{eq.quadmap2}, it is straightforward to
check that the density
\begin{equation} \label{eq.quadinvdens}
  f_{\ast}(x) = \frac{1}{\pi \sqrt{x (1-x)}}
\end{equation}
is invariant under iterates of $S$.  The graph of this density is
shown in Figure~\ref{fig.quaddens}; it agrees with the numerically
obtained invariant density shown in Figures~\ref{fig.quadmapdensev1}
and~\ref{fig.quadmapdensev2}.

One immediate consequence of the existence of an invariant measure is
the following.

\begin{thm}[Poincar{\'e} Recurrence]
Let $\{S_t: t \in G\}$ ($G = \mathbb{Z}_+$ or \Rplus) be a
$\mu$-preserving dynamical system on a measure space $X$.  Then
$\forall A \in \Acal$ and $\forall T > 0$, the set
\begin{equation}
  \{ x \in A: S_t x \notin A \; \; \forall t > T\}
\end{equation}
has measure $0$.  In other words, if $\mu(A)>0$ then under the action
of $\{S_t\}$, $\mu$-almost every initial phase point in $A$ returns to
$A$ infinitely often.
\end{thm}
\begin{proof}
  See \eg\ \cite{KH95}, or~\cite{Ma92} in the simpler case
  $\mu(X) < \infty$.
\end{proof}

This theorem already has interesting consequences for the orbit
structure of the dynamical system $x \mapsto 4 x (1 - x)$ (typical
orbits are shown in Figure~\ref{fig.quadmapev}): since the measure
with density~\eqref{eq.quadinvdens} is invariant, under iteration by
$S$, Lebesgue-almost every $x \in [0,1]$ returns arbitrarily close to
$x$ infinitely often.  This imples a certain kind of irregularity, in
fact aperiodicity, thus confirming the intuitive impression of
Figure~\ref{fig.quadmapev}.  However, for some systems the
consequences of the Poincar{\'e} Recurrence Theorem are quite vacuous.
For example the identity map preserves Lebesgue measure on $[0,1]$,
but the consequences of the Poincar{\'e} Recurrence Theorem in this
case are trivial.

% ---------------------------------------------------------------------
\subsection{Statistical regularity} \label{sec.statreg}

From the examples above it seems that, for some systems, existence of
an invariant measure implies strong statistical properties, while for
other systems is does not.  It would be nice to have tools for
discerning between dynamical systems that exhibit varying degrees of
irregularity and have different statistical properties.

As observed in Chapter~\ref{ch.intro}, an orbit $\{x_n\}$ of the map
$x \mapsto 4 x (1-x)$ appears to have a well defined asymptotic
distribution, as seen from a histogram of the trajectory (\cf\
Figure~\ref{fig.quadmapinvdens}).  The frequency with which an orbit
visits a given histogram bin is asymptotically regular, despite the
irregularity of the orbit itself.  A somewhat stronger notion of
statistical regularity requires that an arbitrary continuous function,
evaluated along an orbit of $\{S_t\}$, has a well-defined time
average~\cite{Ruelle89}:
\begin{defn}[statistical regularity]
Let $\{S_t: t \in G\}$ ($G=\Rplus$ or \Zplus) be a dynamical system on
$X$.  An orbit $\{S_t x: t > 0\}$ is \emph{statistically regular} if
the time average
\begin{equation} \label{eq.timeav1}
  \bar{\phi}(x) = \lim_{n \to \infty} \frac{1}{n} \sum_{k=1}^n \phi( S_k x ) \qquad
  \text{(if $G=\Zplus$)}\\
\end{equation}
or
\begin{equation} \label{eq.timeav2}
  \bar{\phi}(x) = \lim_{T \to \infty} \frac{1}{T} \int_0^T \phi( S_t x ) \, dt \qquad
  \text{(if $G=\Rplus$)}
\end{equation}
exists for every bounded continuous function $\phi: X \to \Reals$.
\end{defn}

Statistical regularity implies the existence of an invariant
probability measure:

\begin{thm} \label{thm.invmeas}
Let $\{S_t: t \in G\}$ be a dynamical system on a metrizable space
$X$.  If the time average $\bar{\phi}(x)$ exists for every bounded
continuous function $\phi: X \to \Reals$ then there is a unique
invariant Borel probability measure $\mu_x$ such that
\begin{equation} \label{eq.phibar}
  \bar{\phi}(x) = \int \phi \, d\mu_x.
\end{equation}
\end{thm}
\begin{proof}
The mapping $T: \phi \mapsto \bar{\phi}(x)$ defined by
equations~\eqref{eq.timeav1}--\eqref{eq.timeav2} is a bounded, linear,
positive functional on $C(X)$.  Therefore by the Riesz
Representation Theorem~\cite{Hal74} there is a unique Borel
probability measure $\mu_x$ such that
\begin{equation}
  T(\phi) = \int \phi \, d\mu_x.
\end{equation}
Furthermore (in the discrete-time case),
\begin{equation}
\begin{split}
  \int \phi \; d(\mu_x \circ S_t^{-1}) &= \int (\phi \circ S_t) \; d\mu_x \\
  &= \overline{\phi \circ S_t}(x) \\
  &= \lim_{n \to \infty} \frac{1}{n}\sum_{k=1}^n (\phi \circ S_t) (S_k x) \\
  &= \lim_{n \to \infty} \frac{1}{n}\sum_{k=1}^n \phi( S_{t+k} x ) \\
  &= \lim_{n \to \infty} \frac{1}{n} \Big[ \sum_{k=1}^{n} \phi(S_k x)
                          + \sum_{k=n+1}^{n+t} \phi(S_k x)
                          - \sum_{k=1}^t \phi(S_k x) \Big] \\
  &= \bar{\phi}(x).
\end{split}
\end{equation}
It follows that $\mu_x \circ S_t^{-1} = \mu_x$, by uniqueness of
$\mu_x$.  A similar argument holds in the continuous-time case.
\end{proof}

% -----------------------------------------------------------------------------
\subsection{Ergodicity, mixing, exactness} \label{sec.ergmixexact}

A century ago, motivated by fundamental problems in statistical
mechanics, Boltzmann and Gibbs raised the
\emph{ergodic problem}: to determine sufficient conditions under which
the time average $\bar{\phi}(x)$ (equations
\eqref{eq.timeav1}--\eqref{eq.timeav2}) exists and is essentially
independent of $x$.  The answer, given by Birkhoff in 1931, is that it
is both necessary and sufficient that the dynamical system $\{S_t\}$
be ergodic.

\begin{defn}[invariant set]
Let $S: X \to X$ be a measurable transformation.  A set $A \in \Acal$
is \emph{invariant} under $S$ if\footnote{The ``symmetric difference'' of two sets is defined by $A
\; \Delta \; B = (A \setminus B) \cup (B \setminus A)$.}  $S^{-1}(A) \; \Delta \;
A$ has measure $0$ (that is, $S^{-1}(A)=A$ mod $\mu$).  If $\{S_t\}$
is a semigroup then $A$ is invariant under $\{S_t\}$ if $A$ is
invariant under $S_t$ $\forall t > 0$.
\end{defn}

\begin{defn}[ergodic dynamical system]
Let $\{S_t: t \in G\}$ ($G=\Rplus$ or \Zplus) be a dynamical system on
a probability space $X$.  $\{S_t\}$ is \emph{ergodic} (alternatively
$\mu$ is ergodic) if for every $S_t$-invariant set $A$, either $\mu(A)
= 0$ or $\mu(A) = 1$.
\end{defn}

Ergodicity is a non-decomposability condition: having $\{S_t\}$
ergodic means there is no non-null set that is invariant under
$\{S_t\}$.  If it were possible to decompose $X$ into invariant sets
$A \cup B = X$ then $\{S_t\}$ could be studied separately on either
$A$ or $B$.  An ergodic dynamical system must be studied on
essentially the entire space $X$.

Ergodicity is a sufficient condition for statistical regularity:

\begin{thm}[Birkhoff Ergodic Theorem] \label{thm.birk}
Let $\{S_t: t \in G\}$ be an ergodic $\mu$-preserving dynamical system
on a measure space $X$, and $\phi \in L^1(X)$. Then for $\mu$-almost
all $x$,
\begin{equation}
  \int \phi \, d\mu =
  \begin{cases}
     \displaystyle
     \lim_{n \to \infty} \frac{1}{n} \sum_{k=1}^n \phi( S_k x ) &  \text{if $G = \Zplus$} \\
     \displaystyle
     \lim_{T \to \infty} \frac{1}{T} \int_0^T \phi( S_t x )\,dt & \text{if $G = \Rplus$.}
  \end{cases}
\end{equation}
\end{thm}
\begin{proof}
The proof is somewhat technical, and can be found in any of the
standard references on ergodic theory---see \eg\ \cite{KH95,PY98}.
\end{proof}

In other words, if $\{S_t\}$ is ergodic then the time average of
$\phi$ along $\mu$-almost every orbit is just a ``spatial'' average or
expectation of $\phi$, weighted with respect to $\mu$.  Ergodicity
also gives an explicit formula for the ergodic measure $\mu$ in terms
of time averages.  Applying the ergodic theorem with $\phi = 1_A$
yields
\begin{equation} \label{eq.binav}
\begin{split}
  \mu(A) &= \int 1_A d\mu \\
         &= \lim_{n \to \infty} \frac{1}{n} \sum_{k=1}^n 1_A(S_k x) \\
         &= \lim_{n \to \infty} \frac{1}{n} \#\{k > 0: S_k x \in A\},
\end{split}
\end{equation}
which is just the fraction of time that the orbit $\{S_k x: k > 0\}$
spends in the set $A$.  Thus an ergodic measure describes the
asymptotic distribution in phase space of $\mu$-almost every orbit.
This helps explain the observation made in Chapter~\ref{ch.intro} that
the long-run distribution of points on any given orbit agrees with the
invariant density.

From Theorem~\ref{thm.birk} we have that ergodicity implies the
existence of the time average of any $L^1$ function $\phi$ evaluated
along $\mu$-almost every trajectory, hence ergodicity implies
statistical regularity.  However, ergodicity remains a difficult
property to prove for the dynamical systems that arise in statistical
mechanics, and for practical purposes the ergodic problem is still
unresolved.

Ergodicity (and hence statistical regularity) does not necessarily
imply any kind of ``random'' behavior of individual orbits.
Irrational rotations are a classic example.  Let points on the
circumference of a circle be parametrized by points $x \in [0,2\pi)$,
and denote rotation through angle $\phi$ by the map
\begin{equation}
  S: x \mapsto x + \phi \quad (\textrm{mod}\; 2\pi).
\end{equation}
This system is ergodic if (and only if) $\phi/2\pi$ is irrational (see
\eg\ \cite[p.\,75]{LM94} and~\cite[Prop 4.2.1]{KH95}).  The ergodic
invariant measure is Lebesgue measure on $[0,2\pi)$, hence Lebesgue
almost every orbit of this system is statistically regular and is,
asymptotically, distributed uniformly on $[0,2\pi)$.

To distinguish between degrees of irregularity there are other,
stronger statistical properties that we can ask of a given dynamical
system.  Two such properties are mixing and exactness.

\begin{defn}[mixing]
Let $\{S_t: t \in G\}$ ($G=\Rplus$ or \Zplus) be a $\mu$-preserving
dynamical system on a probability space $(X,\Acal,\mu)$.  $\{S_t\}$ is
called \emph{mixing} if $\forall A, B \in \Acal$
\begin{equation}
  \mu\big( A \cap S_t^{-1}(B) \big) \to \mu(A) \mu(B) \quad \text{as } t \to \infty.
\end{equation}
\end{defn}

In particular if $\mu(A) \not= 0$ then for a mixing dynamical system,
\begin{equation} \label{eq.mixingdef}
  \frac{ \mu\big( A \cap S_t^{-1}(B) \big) }{ \mu(A) } \to \mu(B) \quad \text{as } t \to \infty.
\end{equation}
In other words, the fraction (with respect to $\mu$) of phase points
originating in $A$ that end up in $B$ after time $t$ (for sufficiently
large $t$) is equal to the size of $B$.  This result is independent of
the particular choice of sets $A$, $B$.

The mixing property comes closer than ergodicity to characterizing
what one would call a ``random'' process, as the
condition~\eqref{eq.mixingdef} implies that $S_t x$ eventually becomes
uncorrelated with the initial state $x$.  Thus mixing plays much the
same role as ``sensitivity to initial conditions'' in the topological
approach to chaotic dynamics.  Mixing distinguishes between merely
statistically regular systems and systems with stronger statistical
properties.  Irrational rotations, for example, are ergodic but not
mixing.  A yet stronger statistical property is exactness:

\begin{defn}[exactness] \label{defn.exact}
Let $\{S_t: t \in G\}$ be a $\mu$-preserving dynamical system on a
probability space $(X,\Acal,\mu)$, such that $S_t(A) \in \Acal$
$\;\forall A \in \Acal,\; t\geq0$.  $\{S_t\}$ is called \emph{exact} if
\begin{equation} \label{eq.exactdef}
  \lim_{t \to \infty} \mu\big( S_t(A) \big) = 1.
\end{equation}
\end{defn}

Non-invertibility is a necessary condition for exactness, since for
invertible $S_t$ we have
\begin{equation}
  \mu \big( S_t(A) \big) = (\mu \circ S_t^{-1}) \big( S_t(A) \big) = \mu(A),
\end{equation}
so the condition~\eqref{eq.exactdef} cannot hold.  Thus there are
mixing systems (\eg\ the two-dimensional Baker
transformation~\cite{LM94}, which is invertible) that are not exact.
The transformation $x \mapsto 4 x (1-x)$ is
exact~\cite[p.\,167]{LM94}.

The implications between the various ergodic properties of dynamical
systems, and their topological counterparts, are summarized in the
following diagram (proofs can be found in~\cite{KH95,LM94};
definitions of topological notions of complex dynamics are given
in~\cite{KH95}).
\begin{equation*}
\begin{CD}
\text{exactness} \\ 
@VVV \\
\text{mixing} @>>> \text{topological mixing} \\
@VVV @VVV \\
\text{ergodicity} @>>> \text{transitivity} \\
@VVV @VVV \\
\text{statistical regularity} @>>> \text{orbit recurrence} \\
\end{CD}
\end{equation*}

For systems having any of the ergodic properties discussed above, an
arbitrary initial density converges---that is, iterates of the
Perron-Frobenius operator converge---to a uniform density with respect
to the invariant measure.  A classification can be made in terms of
the strength of this convergence, and the Perron-Frobenius operator
beomes an important classification tool:

\begin{thm}
Let $\{S_t: t \in G\}$ ($G=\Rplus$ or \Zplus) be a dynamical system on
a probability space $(X,\mu)$, and let $\{P_t: t \in G\}$ be the
corresponding semigroup of Perron-Frobenius operators.  If $\{S_t\}$
has a unique absolutely continuous measure $\mu_{f_{\ast}}$ with
positive density $f_{\ast}$ then
\begin{enumerate}
  \item $\{S_t\}$ is exact iff $\,\forall f \in L^1(X)$, $\{P_t f\}$
  is strongly convergent\footnote{$\{f_n\}$ is strongly convergent to
  $f$ if $\| f_n - f \| \xrightarrow{n \to \infty} 0$ (in the $L^1$
  norm).}\ to $f_{\ast}$;
  \item $\{S_t\}$ is mixing iff $\, \forall f \in L^1(X)$, $\{P_t f\}$
  is weakly convergent\footnote{$\{f_n\}$ is weakly convergent to $f$
  if $\phi(f_n - f) \xrightarrow{n \to \infty} 0$ for any
  bounded linear $\phi: L^1(X) \to \Reals$.}\ to $f_{\ast}$;
  \item $\{S_t\}$ is ergodic iff $\,\forall f \in L^1(X)$, $\{P_t f\}$
  is C{\'e}saro convergent\footnote{$\{f_n\}$ is C{\'e}saro convergent
  to $f$ if $\frac{1}{n} \sum_{k=1}^n \phi(f_k - f) \xrightarrow{n \to
  \infty} 0$ for every bounded linear $\phi: L^1(X) \to \Reals$.}\ to
  $f_{\ast}$.
\end{enumerate}
\end{thm}
\begin{proof}
See \eg\ \cite[p.\,72]{LM94}, \cite[pp.\,56, 63, 92]{Ma92}.
\end{proof}

% ----------------------------------------------------------------------------
\subsection{Natural and physical measures} \label{sec.SRB}

According the Birkhoff Ergodic Theorem~\ref{thm.birk}, an ergodic
measure $\mu$ describes the asymptotic statistics of $\mu$-almost
every trajectory of a dynamical system $\{S_t\}$.  Unfortunately, this
statement need not have much dynamical relevance.  For example, the
map $x \mapsto 4 x (1-x)$ has a fixed point at the origin.  The
probability measure concentrated at the origin is ergodic, and
trivially reflects the statistics of an orbit originating at this
point, but it has nothing to say about any other orbit.  The same
applies to an ergodic measure concentrated on any periodic orbit.  The
problem here is that the Birkhoff Ergodic Theorem gives a description
of orbits only on a set of Lebesgue measure zero.

In general, a dynamical system may have many ergodic measures---in
fact uncountably many~\cite{ER85,Ruelle89}---only some of which imply
something about the statistical behavior of typical orbits, if by
``typical'' we mean ``Lebesgue almost every''.  It would be nice to
have a way to select, among the ergodic measures present, which
measure is natural in the sense of reflecting the dynamics of typical
orbits.

Ergodic measures that are absolutely continuous with respect to
Lebesgue measure are a good candidate for a class of ``natural''
ergodic measures.  If $\mu_f$ is an ergodic measure with strictly
positive density $f$ (with respect to Lebesgue measure) then the
statement ``$\mu_f$-almost every'' implies ``Lebesgue almost
every''.\footnote{The relation ``$\ll$'' is transitive and reflexive,
so the condition ``$\mu \ll \nu$ and $\nu \ll \mu$'' gives an
equivalence relation $\mu \sim \nu$.  Equivalent measures share the
same null sets.  All measures with positive density (\eg, with respect
to Lebesgue measure) are equivalent.  Thus we seek ``natural''
measures in the equivalence class containing Lebesgue measure.}  Then
the Birkhoff Ergodic Theorem says that $\mu_f$ describes the
statistics of Lebesgue almost every orbit, hence $\mu_f$ is the
dynamically relevant ergodic measure.  Moreover, if such a natural
measure exists then it is unique:

\begin{thm}
Let $\{S_t: t \in G\}$ be a nonsingular dynamical system on a measure
space $X$.  If $\{S_t\}$ is ergodic then $\{S_t\}$ has at most one
invariant density with respect to $\mu$.  Furthermore, if $\{S_t\}$
has a unique strictly positive invariant density then $\{S_t\}$ is
ergodic.
\end{thm}
\begin{proof}
See \eg\ \cite[Thm 4.2.2]{LM94}; also \cite[Prop 5.1.2]{KH95}.  The
result follows from the fact that if $\nu_1$, $\nu_2$ are distinct
ergodic measures then $\nu_1 \perp \nu_2$\footnote{The relation
$\perp$ is the antithesis of absolute continuity.  If $\mu \perp \nu$
then essentially the only sets on which $\mu$ does not vanish are
those on which $\nu$ does, and vice versa.  See \eg\
\cite{Hal74}.}~\cite[p.\,94]{PY98}.  Thus $\nu_1$, $\nu_2$ cannot both be
absolutely continuous, hence at most one of them can have a density.
\end{proof}

However, absolute continuity is not an adequate criterion for the
selection of a natural ergodic measure for a dissipative dynamical
system.  In a dissipative system, phase space volumes are contracted
by the time evolution, typically onto a compact invariant set, or
attractor~\cite{Ruelle81,Ruelle89}.  In this case the relevant ergodic
measure is expected to be concentrated on a set of Lebesgue measure
zero, and therefore will not be absolutely continuous.  Thus, even for
the dynamically relevant ergodic measure, the Birkhoff Ergodic Theorem
gives a statement only about orbits originating on a set of Lebesgue
measure zero.

Nevertheless, in physical experiments and computer simulations there
is typically just one invariant measure---the so-called \emph{physical}
measure---that describes typical orbits of the system.  The existence
of such a measure motivates the following definition.

\begin{defn}[SRB (Sinai-Ruelle-Bowen) measure] \label{def.SRB}
Let $\{S_t: t \in G\}$ be a dynamical system on a measure space $X$.
Then $\mu$ is an SRB or physical measure for $\{S_t\}$ if for any
bounded continuous $\phi: X \to \Reals$ and for all $x$ in a set of
positive Lebesgue measure,
\begin{equation}
  \int \phi \,d\mu = \lim_{n \to \infty} \frac{1}{n} \sum_{k=1}^n \phi( S_k x ) \qquad
  \text{(if $G = \Zplus$)}
\end{equation}
or
\begin{equation}
   \int \phi \,d\mu = \lim_{T \to \infty} \frac{1}{T} \int_0^T \phi( S_t x )\,dt \qquad
   \text{(if $G = \Rplus$)}.
\end{equation}
\end{defn}

Thus an SRB measure is one for which the conclusion of the Birkhoff
Ergodic Theorem~\ref{thm.birk} holds, not just on a set of positive
$\mu$ measure, but on a set of positive \emph{Lebesgue} measure.  Thus
an SRB measure describes the statistics of Lebesgue almost every orbit
originating in some nontrivial set.  By definition, existence of an
SRB measure requires statistical regularity for all orbits in a set of
positive Lebesgue measure.  An SRB measure is necessarily invariant
under $S_t$, as can be seen from the proof of
Theorem~\ref{thm.invmeas}.  An SRB measure need not be
ergodic~\cite{BB03}.

Other candidates for the notion of ``physical'' measure have been
given.  For example, suppose a dynamical system with random
perturbations of amplitude \eps\ has a stationary measure
$\mu_{\eps}$.  The zero-noise limit ($\eps \to 0$) of $\mu_{\eps}$, if
it exists, is the \emph{Kolmogorov measure}~\cite{Ruelle89}.  For some
systems (\eg\ Axiom A systems\footnote{Axiom A systems are a fairly
restrictive class of dynamical systems with strong chaotic properties.
For details see \eg\ \cite{Ruelle89}.  Examples are Anosov flows,
Smale's horseshoe map, and the solenoid~\cite{Smale67}.})  the
Kolmogorov measure is known to coincide with the SRB measure.  Because
of this equivalence, a number of different definitions of SRB measure
appear in the literature; see \eg~\cite{BB03,Ruelle89}.  The
definition above is commonly preferred because it is motivated by
physical considerations.

Existence of an SRB measure is a strong condition that, although quite
natural to define, turns out to be very difficult to prove.  Existence
is known for Axiom A systems and $C^2$ flows with hyperbolic
attractors~\cite{Bow75,Ruelle76}.  The notion of strange attractor, of
which much has been made in chaos studies, has been defined as an
attracting invariant set that supports an SRB measure that is
mixing~\cite{Sin00}.  Proving the existence of such a measure is
widely recognized as one of the most important outstanding problems in
dynamical systems theory~\cite{ER85,Via01}.  Computer methods are
showing promise in this direction, and have recently been used to show
the existence of an SRB measure supported on the famous Lorenz
attractor~\cite{Lor63,Tucker00,Tucker02}.

% ============================================================================
\section{Dimensions and Lyapunov Exponents} \label{sec.dimexp}

In numerical studies of chaotic systems various numerical
parameters---for example Lyapunov exponents, dimensions, and
entropy---are frequently used to quantify the degree of ``randomness''
exhibited by typical orbits.  The Birkhoff Ergodic Theorem and its
generalizations make it possible to define these quantities rigorously
in terms of time averages along trajectories, and to establish that
these quantities are identical for almost every trajectory with
respect to the ergodic measure:

\begin{description}
\item[Lyapunov characteristic exponents] are the time averages of the
exponential rates of divergence of nearby trajectories, along
orthogonal directions in the tangent space of a given
orbit~\cite{Ben80a,Ben80b,Osel68}.
\item[Dimensions] help quantify the geometric structure of invariant sets
(\eg\ attractors) of dynamical systems.  For sets that support an
invariant measure, one can define the information dimension,
correlation dimension, and Hausdorff dimension (see \eg\
\cite{Ruelle89}) which quantify the average number of ``independent directions''
on the invariant set.
\item[Entropy] is a measure of the average rate of \emph{information creation}
along an orbit of a dynamical system (see \eg\ \cite{Shaw81}).
\end{description}

% ============================================================================
\section{Reliability of Numerical Simulations} \label{sec.shadowing}

Because of sensitivity to initial conditions, orbits of chaotic
dynamical systems cannot be reliably computed using finite precision
arithmetic, as in computer simulations.  A numerically computed
pseudo-orbit approximating a true orbit typically loses any relation
to the true orbit after only a few iterations of the system dynamics.
The exact map $S: x \mapsto 2 x \textrm{ mod } 1$ on $[0,1]$ makes
this point especially clear, since $S$ effects a left-shift on the
binary representation of $x$.  That is,
\begin{equation}
S: \; 0.x_1 x_2 x_3\ldots \; \mapsto \; 0.x_2 x_3 x_4\ldots
\end{equation}
where $x_1$, $x_2$, $\ldots$ are the digits of the binary
representation of $x$.  If the computer stores $16$ binary digits in
its representation of $x$, then the numerical approximation of any
particular orbit becomes meaningless after only $16$ iterations of the
dynamics.

The situation therefore seems hopeless when we come to using computer
simulations to approximate statistical properties of dynamical systems,
since very long and reasonably accurate pseudo-orbits are required.
Surprisingly, this does not pose a problem for sufficiently
well-behaved systems.

\begin{defn}[pseudo-orbit; shadowing]
Let $\{S_n = S^n: n \in \Zplus\}$ be a dynamical system on a metric
space $X$.
\begin{itemize}
\item A pseudo-orbit $\{x_n: a \leq n \leq b\}$ is an
\emph{$\alpha$-pseudo-orbit} if
\begin{equation}
  d\big(x_{n+1}, S(x_n)\big) < \alpha \quad \forall a \leq n \leq b.
\end{equation}
\item A point $y \in X$ \emph{$\beta$-shadows} $\{x_n: a \leq n \leq b\}$ if
\begin{equation}
  d\big(S^n(y),x_n\big) < \beta \quad \forall n.
\end{equation}
\end{itemize}
\end{defn}

One can think of an $\alpha$-pseudo-orbit as an orbit of $\{S_n\}$ that
is perturbed by an amount smaller than $\alpha$ after each iteration;
this models, for example, the round-off error in a numerical
simulation.  A shadow orbit $\{S^n(y)\}$ is a \emph{true} orbit that
is approximated by $\{x_n\}$ within accuracy $\beta$ at all times.

\begin{defn}[shadowing property]
A dynamical system $\{S_n\}$ has the \emph{shadowing property} if
$\forall \beta > 0$, $\exists \alpha > 0$ such that every
$\alpha$-pseudo-orbit is $\beta$-shadowed by a point $y$.
\end{defn}

For a system with the shadowing property, a pseudo-orbit (\eg, one
found by numerical simulation) is always an accurate representation of
\emph{some} nearby true orbit, although perhaps not the particular
orbit one was trying to approximate.  It is known that Anosov
systems\footnote{An Anosov system is one for which the entire phase
space is a hyperbolic set~\cite{KH95}.} have this
property~\cite{KH95}.  More generally a smooth dynamical system has
the shadowing property in a neighborhood of a hyperbolic invariant
set; this is the celebrated Shadowing Lemma~\cite{Guck83,KH95}.

It can be shown~\cite{Ben78a} that for any
uniformly continuous functional $\phi$ and a given $\delta>0$, one can
choose a sufficiently small $\alpha>0$ so that if $\{x_n\}$ is an
$\alpha$-pseudo-orbit then, $\forall n > 0$,
\begin{equation}
  \bigg| \frac{1}{n} \sum_{k=1}^n \phi(x_k) -
         \frac{1}{n} \sum_{k=1}^n \phi(S^k y) \bigg| < \delta.
\end{equation}
where $\{S^k(y)\}$ $\beta$-shadows $\{x_k\}$.  Thus the time average
of $\phi$ along the pseudo-orbit differs by less than $\delta$ from
the time average of $\phi$ along some true orbit $\{S^k y\}$, hence
time averages computed from numerical simulations are in principle
reliable.

The considerations above apply only to systems with the shadowing
property.  Unfortunately this property is difficult to prove except
under fairly restrictive conditions, as in Anosov systems.
Nevertheless, numerical studies suggest that numerical simulations of
many non-Anosov systems are statistically reliable as
well~\cite{Ben78a}; the shadowing property provides a plausible
mechanism that might account for this phenomenon.

\chapter{Elements of an Ergodic Theory of Delay Equations}\label{ch.framework}
\markboth{CHAPTER \thechapter.\ \ ERGODIC THEORY OF DELAY EQUATIONS}{}

\begin{singlespacing}
\minitoc
\end{singlespacing} 

%==========================================================================
\newpage
\section{Introduction}

The central aim of this thesis is to apply probabilistic concepts,
\eg, from ergodic theory, to the dynamics of delay differential
equations (DDEs).  Before such a project can proceed, a number of
foundational questions must be addressed.  For instance,
\begin{itemize}
\item In what sense can a DDE be interpreted as a dynamical system,
\ie, with a corresponding evolution semigroup?
\item What is the phase space for such a system?
\item What semigroup of transformations governs the phase space dynamics of
a DDE?
\end{itemize}

The dynamical systems approach to delay equations is well established,
and provides standard answers to these questions.  This theory is
discussed in Sections~\ref{sec.ddesys}--\ref{sec.ddesg} below.  As it
happens the phase space for a delay differential equation is
infinite dimensional, which complicates matters considerably.

An ergodic approach to DDEs will require a theory of probability in
infinite dimensional spaces.  The elements of this theory are
discussed in Sections~\ref{sec.ddefp}--\ref{sec.infprob}.  Naturally,
there are technical and interpretational difficulties with doing
probability in infinite dimensions.  Indeed, the available
mathematical machinery proves to be inadequate to deal with some of
the problems that arise.  This has important consequences for the
remainder of the thesis.

%==========================================================================
\section{Delay Differential Equations} \label{sec.ddetheory}

Delay differential equations, which are representative of the more
general class of \emph{functional} differential equations~\cite{HV93},
take a great variety of forms.  Delay equations having multiple time
delays, time-dependent delays, and even continuous distributions of
delays all arise in mathematical models of evolutionary
systems~\cite{Dr77}.  To simplify matters we will restrict our
attention to delay equations of the form
\begin{equation} \label{eq.dde1}
  x'(t) = f\big( x(t), x(t-\tau) \big),
\end{equation}
where $x(t) \in \Rn$, $f: \Rn \times \Rn \to \Rn$, and $\tau>0$ is a
single fixed time delay.  Despite this restriction, the class of delay
equations of the form~\eqref{eq.dde1} provides more than a sufficient
arena for the considerations that follow.

% ---------------------------------------------------------------------------
\subsection{Definition of a solution}

By a solution of the DDE~\eqref{eq.dde1} we mean the following: if for
some $t_0 \in \Reals$ and $\beta > t_0$, the function $x: [t_0-\tau,
\beta]
\to \Rn$ satisfies~\eqref{eq.dde1} for $t \in [t_0,\beta]$,
then we say $x$
\nomenclature[xt]{$x(t)$}{solution of a delay differential equation}%
is a solution of~\eqref{eq.dde1} on $[t_0-\tau,\beta]$.  If $\phi:
[t_0-\tau,t_0] \to \Rn$ and $x$ is a solution that coincides with
$\phi$ on $[t_0-\tau,t_0]$, we say $x$ is a solution through
$(t_0,\phi)$.\footnote{The difficulty that arises if $\phi$ does not
satisfy~\eqref{eq.dde1} at $t_0$ is avoided if $x'$ is interpreted as
a right-hand derivative.}

Because equation~\eqref{eq.dde1} is autonomous (\ie, the right-hand
side does not depend explicitly on $t$), it is invariant under time
translation.  That is, if $x(\cdot)$ is a solution then, for any $T
\in \Reals$, $x(\cdot+T)$ is also a solution.  Consequently the choice
of initial time $t_0$ is arbitrary, and for the sake of convenience we
can take $t_0 = 0$.  Let $C = C([-\tau,0])$ be the space of continuous
functions from $[-\tau,0]$ into \Rn.  Then if $\phi \in C$ and $x:
[-\tau,\beta] \to \Rn$, we say $x$ is a solution of~\eqref{eq.dde1}
with \emph{initial function} $\phi$,
\nomenclature[f]{$\phi$}{initial function for a delay differential equation}%
or simply a solution through $\phi$, if $x$ is a
solution through $(0,\phi)$.

Our intention to consider delay equations as models of deterministic
processes imposes some constraints on the equations it makes sense to
consider.  In order that a given DDE describes an evolutionary process
at all, we require the existence of solutions, at least for some
subset of initial functions $\phi \in C$.  Moreover, since ergodic
theory is concerned largely with asymptotic properties, we require
\emph{global} existence, \ie\ existence of solutions on the entire interval
$[-\tau,\infty)$.  To ensure that the process is deterministic we
require that, for given $\phi \in C$, the solution through $\phi$
should be unique.

These constraints are in fact met under fairly mild restrictions on the
right-hand side of~\eqref{eq.dde1}.  The following section presents
the basic results of this theory that we will require.  For further
details of the existence and uniqueness theory for delay equations see
for example~\cite{Dr77,HV93}.

% ----------------------------------------------------------------------
\subsection{Existence and uniqueness theory}

The existence and uniqueness theory for delay equations can be derived
from the more general theory of functional differential equations.
Since we intend to consider only equations of the form~\eqref{eq.dde1}
we will not make use of the full generality available.  Nevertheless,
the more general theory leads to a presentation that is simpler and
also benefits from an analogy with similar results in the theory of
ordinary differential equations.

In the following, $C([a,b])$
\nomenclature[C]{$C([a,b])$}{space of continuous functions from $[a,b]$
into \Rn, with sup norm}
denotes the Banach space of continuous functions from $[a,b]$ into
\Rn, equipped with the sup norm, and $C$ denotes the space
$C([-\tau,0])$.
\nomenclature[Ct]{$C$}{space of continuous functions from $[-\tau,0]$
into \Rn, with sup norm}%
If $\beta>0$ and $x \in C([-\tau,\beta])$, let $x_t \in C$ be defined
by
\begin{equation}
  x_t(s) = x(t+s), \quad s \in [-\tau,0].
\end{equation}
Suppose $F: \Reals \times C \to \Rn$.  Then the equation
\begin{equation} \label{eq.rfde}
  x'(t) = F(t,x_t),
\end{equation}
where $x'$ denotes the right-hand derivative, is called a
\emph{retarded functional differential equation} (RFDE) on $D$.

Equation~\eqref{eq.rfde} provides for a very general dependence of
$x'(t)$ on the retarded values of $x$ on the interval $[t-\tau,t]$.
The DDE~\eqref{eq.dde1} is a special case, with $F$ given by
\begin{equation} \label{eq.ddefnl}
  F(t,\phi) = f\big(\phi(0),\phi(-\tau)\big).
\end{equation}

A solution $x$ of~\eqref{eq.rfde} is defined in the same manner as for
the DDE~\eqref{eq.dde1} (see the preceding section).  The basic
results on existence and uniqueness of solutions are presented below.
We omit the proofs, which are somewhat lengthy and technical, and
refer the reader to~\cite{Dr77,HV93} for details.

\begin{thm}[Local existence] \label{thm.ddeloc}
Suppose $\Omega \subset \Reals \times C$ is open, and $F:\Omega \to
\Rn$ is continuous.  If $(t_0,\phi) \in \Omega$ then there is a
solution of the RFDE~\eqref{eq.rfde} through $(t_0,\phi)$.
\end{thm}
\begin{proof}
See \eg~\cite[p.\,43]{HV93}.
\end{proof}

\begin{cor}
If $f: \Rn \times \Rn \to \Rn$ is continuous then for any $\phi \in C$
there is a solution of~\eqref{eq.dde1} through $\phi$.
\end{cor}
\begin{proof}
Equation~\eqref{eq.dde1} is equivalent to~\eqref{eq.rfde} with $F$
defined as in~\eqref{eq.ddefnl}.  Since $f$ is continuous, $F$ is a
composition of continuous functions and is therefore continuous on
$\Reals \times C$.  The conclusion follows from
Theorem~\ref{thm.ddeloc}.
\end{proof}

\begin{defn}[Lipschitzian]
Let $\Omega \subset \Reals \times C$ and $F: \Omega \to \Rn$.  $F$ is
\emph{Lipschitzian} in $\phi$ (on $\Omega$) if, for some $K \geq 0$,
\begin{equation}
  \big| F(t,\phi_1) - F(t,\phi_2) \big| \leq K | \phi_1 - \phi_2 | \quad
  \forall (t,\phi_1), (t,\phi_2) \in \Omega.
\end{equation}
\end{defn}

\begin{thm}[Uniqueness] \label{thm.ddeuniq}
Suppose $\Omega \subset \Reals \times C$ is open, $F: \Omega \to
\Rn$ is continuous, and $F(t,\phi)$ is Lipschitzian in $\phi$ on every
compact set in $\Omega$.  If $(t_0,\phi) \in \Omega$ then there is a
unique solution of the RFDE~\eqref{eq.rfde} through $(t_0,\phi)$.
\end{thm}
\begin{proof}
See \eg~\cite[p.\,44]{HV93}.
\end{proof}

\begin{cor}
If $f: \Rn \times \Rn \to \Rn$ is Lipschitzian then for any $\phi \in
C$ there is a unique solution of~\eqref{eq.dde1} through $\phi$.
\end{cor}
\begin{proof}
Equation~\eqref{eq.dde1} is equivalent to~\eqref{eq.rfde} with $F$
defined as in~\eqref{eq.ddefnl}.  Since $f$ is Lipschitzian,
$F(t,\phi)$ is Lipschitzian in $\phi$~\cite[p.\,292]{Dr77}.  The
conclusion follows from Theorem~\ref{thm.ddeuniq}.
\end{proof}

\begin{thm}[Global existence] \label{thm.rfdeglobal}
Suppose $F: [t_0,\beta) \times C \to \Rn$ is continuous, and that
$F(t,\phi)$ is Lipschitzian in $\phi$.  If
\begin{equation}
  \big| F(t,\phi) \big| \leq M(t) + N(t) \big| \phi \big|, \quad
  \forall t \in [t_0,\beta), \; \phi \in C,
\end{equation}
for some positive continuous functions $M,N$ on $[t_0,\beta)$, then
there is a unique solution of the RFDE~\eqref{eq.rfde} through $\phi$
on $[t_0-\tau,\beta)$.
\end{thm}
\begin{proof}
See \eg~\cite[p.\,308]{Dr77}.
\end{proof}

\begin{cor} \label{cor.ddeglobal}
If $f: \Rn \times \Rn \to \Rn$ is Lipschitzian and satisfies
\begin{equation}
  \big| f(u,v) \big| \leq N(t) \max \{ |u|, |v| \}
\end{equation}
for some positive continuous function $N(t)$ on $[0,\beta)$, then
$\forall \phi \in C$ there is a unique solution of~\eqref{eq.dde1}
through $\phi$ on $[-\tau,\beta)$.
\end{cor}
\begin{proof}
With $F$ defined as in equation~\eqref{eq.ddefnl}, equation~\eqref{eq.dde1} is
equivalent to the RFDE~\eqref{eq.rfde}.  Since $f$ is Lipschitzian it
is also continuous, hence $F$ is continuous and $F(t,\phi)$ is
Lipschitzian in $\phi$~\cite[p.\,292]{Dr77}.  Since
\begin{equation}
\begin{split}
  \big| F(t,\phi) \big| &= \big| f\big( \phi(0), \phi(-\tau) \big) \big| \\
   &\le N(t) \max \{ |\phi(0)|, |\phi(-\tau)| \} \\
   &\le N(t) \sup_{s \in [-\tau,0]} |\phi(s)| \\
   &= N(t) |\phi|,
\end{split}
\end{equation}
the conclusion follows from Theorem~\ref{thm.rfdeglobal}.
\end{proof}

The hypotheses of the preceding theorems can be considerably weakened.
In particular, continuity of $F$ can be weakened to continuity of
$F(t,x_t)$ with respect to $t$ for every continuous $x$.  The
Lipschitz condition on $F$ can also be weakened to a local Lipschitz
condition, for which it suffices that $f$ in the DDE~\eqref{eq.dde1}
have continuous first partial derivatives~\cite[p.\,261]{Dr77}.

In the following we also require continuous dependence of solutions on
initial conditions, for which the following theorem gives a result
analogous to that for ordinary differential equations.

\begin{thm}[Continuous dependence] \label{thm.rfdecont}
Suppose $x$ is a solution through $(t_0,\phi)$ of the
RFDE~\eqref{eq.rfde} and that it is unique on $[t_0-\tau,\beta]$.  If
$\{(t_n,\phi_n)\} \subset \Reals \times C$ is a sequence such that
$(t_n,\phi_n) \to (t_0,\phi)$ as $n \to \infty$, then for all
sufficiently large $n$ every solution $x_n$ through $\phi_n$ exists on
$[t_n-\tau,\beta]$, and $x_n \to x$ uniformly on $[t_0-\tau,\beta]$.
\end{thm}
\begin{proof}
See Theorem 2.2 of~\cite[p.\,43]{HV93}, which proves a stronger
result giving continuous dependence on $(t_0,\phi,F)$.  The version
given here is a simpler special case.
\end{proof}

% ----------------------------------------------------------------------
\subsection{Method of steps} \label{sec.methsteps}

Existence and uniqueness for a given DDE can sometimes be shown
indirectly, by representing the DDE as a sequence of ordinary
differential equations.  This approach, known as the method of
steps~\cite{Dr77}, also furnishes a method of finding explicit
solutions.

The DDE problem
\begin{equation} \label{eq.methsteps1}
\begin{split}
  & x'(t) = f\big( x(t), x(t-\tau) \big), \quad t \geq 0 \\
  & x_0 = \phi,
\end{split}
\end{equation}
when restricted to the interval $[0,\tau]$, becomes the \emph{ordinary}
differential equation
\begin{equation}
  x'(t) = f\big( x(t), x_0(t-\tau) ) \equiv g_0\big(t,x(t)\big),
  \quad t \in [0,\tau],
\end{equation}
since $x_0 = \phi$ is a known function.  Under suitable hypotheses on
$g$, existence and uniqueness of a solution of this equation (hence a
solution of~\eqref{eq.methsteps1}) on $[0,\tau]$ can be established.
Denoting this solution by $x_1$ and restricting
equation~\eqref{eq.methsteps1} to the interval $[\tau,2\tau]$, we
obtain the ordinary differential equation
\begin{equation}
  x'(t) = f\big( x(t), x_1(t-\tau) \big) \equiv g_1\big(t,x(t)\big),
  \quad t \in [\tau,2\tau],
\end{equation}
for which we can again establish existence and uniqueness of a
solution $x_2$.

Proceeding inductively, considering equation~\eqref{eq.methsteps1} as
an ordinary differential equation on a sequence of intervals $[n \tau,
(n+1)\tau]$, it is sometimes possible to show existence and uniqueness
of a solution of the DDE on $[-\tau,\infty)$.  This approach is
especially simple if $f\big(x(t),x(t-\tau)\big)=f\big(x(t-\tau)\big)$
is independent of $x(t)$, since existence and uniqueness of $x_{n+1}$
then requires only integrability of $x_{n}$, hence almost-everywhere
continuity of $\phi$ is sufficient to guarantee existence and
uniqueness of a solution on $[-\tau,\infty)$.

% ----------------------------------------------------------------------
\section{Delay Equation as a Dynamical System} \label{sec.ddesys}

As noted above, to make sense of the DDE~\eqref{eq.dde1} as
prescribing the evolution of a deterministic system, we require that
for any $\phi$ in $C$, a solution $x$ through $\phi$ exists and is
unique on $[-\tau,\infty)$.  We will also require that $x(t)$ depend
continuously on $\phi$.  Thus from now on we will simply assume that
sufficient conditions are satisfied to guarantee that these
constraints are met, for example the hypotheses of
Corollary~\ref{cor.ddeglobal}.

By a simple rescaling of the time variable in~\eqref{eq.dde1}, the
delay time $\tau$ can be made equal to $1$.  For the sake of
convenience, and wherever it seems natural, we will assume in the
following that such a rescaling has been done.  Thus the generic DDE
``initial data problem'' we consider is the following,
\begin{equation} \label{eq.ddeivp}
\begin{split}
  & x'(t) = f\big( x(t), x(t-1) \big), \quad t \ge 0 \\
  & x(t) = \phi(t), \quad t \in [-1,0],
\end{split}
\end{equation}
where $\phi \in C = C([-1,0])$.
\nomenclature[Cu]{$C$}{space of continuous functions from $[-1,0]$
into \Rn, with sup norm}

Since equation~\eqref{eq.ddeivp} specifies the evolution of a variable
$x(t) \in \Rn$, it might seem that such a DDE could be regarded simply
as a dynamical system on \Rn.  However, $x(t)$ alone is inadequate as
a ``phase point'', since the initial value $x(0)$ does not provide
sufficient information to determine a solution.  Indeed, in order that
the right-hand side $f\big(x(0),x(-1)\big)$ is well defined for all $t
\in [0,1]$, initial data consisting of values of $x(t)$ for $t \in [-1,0]$
must be supplied, as in~\eqref{eq.ddeivp}.

In general, to determine a unique solution of~\eqref{eq.ddeivp} for
all $t \geq T$, it is necessary and sufficient to know the retarded
values of $x(t)$ for all $t$ in the ``delay interval'' $[T-1,T]$.
Thus equation~\eqref{eq.dde1} can only be considered as a dynamical
system if the phase point at time $t$ contains information about the
solution $x(t)$ on the entire interval $[t-1,t]$.  That this is in
fact sufficient to define a dynamical system corresponding to the
initial value problem~\eqref{eq.ddeivp} is shown in the following.

As before, let $C$ be the Banach space of bounded continuous functions
from $[-1,0]$ into \Rn, supplied with the sup norm.  For each $t \geq
0$ define a transformation $S_t: C \to C$ by
\begin{equation} \label{eq.ddesg}
  (S_t\phi)(s) \equiv x_t(s) = x(t+s), \quad s \in [-1,0],
\end{equation}
where $x(t)$ is the solution of~\eqref{eq.ddeivp}.  Then we have
(\cf\ \cite{HV93}):
\begin{thm} \label{thm.ddesg}
The family of transformations $S_t$, $t \geq 0$, defined by
equation~\eqref{eq.ddesg}, is a semidynamical system on $C$ (\cf\
Definition~\ref{def.semidyn}).  That is,
\begin{enumerate}[(a)]
\item $S_0 \phi = \phi \quad \forall \phi \in C$,
\item $(S_t \circ S_{t'})\phi = S_{t+t'}\phi \quad \forall \phi \in C,
\; t_1, t_2 \geq 0$,
\item $(t,\phi) \mapsto S_t(\phi)$ is continuous $\forall t \geq 0$.
\end{enumerate}
\end{thm}
\begin{proof}
(a) is obvious from equations~\eqref{eq.ddeivp} and~\eqref{eq.ddesg},
since by definition
\begin{equation}
  (S_0 \phi)(s) = x(s) = \phi(s), \quad s \in [-1,0].
\end{equation}

To prove (b), let $x(t)$ be the solution of \eqref{eq.ddeivp}.  Then
by definition of $S_t$,
\begin{equation} \label{eq.Stprime}
\begin{gathered}
  (S_{t+t'} \phi)(s) = x(t+t'+s), \\
  (S_{t'} \phi)(s) = x_{t'}(s) = x(t'+s).
\end{gathered}
\end{equation}
By translation invariance of the DDE, $x(t+t')$ is also a solution,
corresponding to the initial function $x_{t'}$.  Thus by definition
of $S_t$,
\begin{equation}
  (S_t x_{t'})(s) = x(t+t'+s).
\end{equation}
Combining with~\eqref{eq.Stprime}, we have
\begin{equation}
  (S_t \circ S_{t'}) \phi = S_t x_{t'} = S_{t+t'}.
\end{equation}

(c) follows from Theorem~\ref{thm.rfdecont}, which asserts continuity
of $(t_0,\phi) \mapsto x \in C([t_0-1,\beta])$.  Since $x_{t+t_0}$ is
just the restriction of $x$ to $[t+t_0-1,t+t_0] \subset
[t_0-1,\beta]$, we also have continuity of $(t_0,\phi) \mapsto
x_{t+t_0} = S_{t+t_0}
\phi$.
\end{proof}

In terms of the evolution semigroup just defined, the initial data problem
\eqref{eq.ddeivp} can be written as an abstract initial value problem,
\begin{equation} \label{eq.absivp}
\begin{cases}
  x_t = S_t(x_0) \\
  x_0 = \phi.
\end{cases}
\end{equation}
In accordance with the terminology of Section~\ref{sec.dynsys} we call
the function $x_t$ the ``phase point'' at time $t$ of the
corresponding DDE~\eqref{eq.ddeivp}.  The trajectory
\begin{equation}
  \{x_t = S_t \phi: t \geq 0\}
\end{equation}
is a continuous curve in the function space $C$.  The relationship of
the DDE solution $x(t)$ to this trajectory is simple, and is given by
\begin{equation}
  x(t) = x_t(0).
\end{equation}
That is, the solution $x(t)$ ``reads off'' the right endpoint of the
phase point $x_t$.  In other words, $x(t)$ can be interpreted as the
projection of $x_t$ under the map $\pi: C \to \Rn$ defined by
$\pi(x_t) = x_t(0)$.

The action of $S_t$ has a simple geometric interpretation.  Since
$(S_t \phi)(\cdot) = x(t+\cdot)$, $S_t$ consists of a translation of
the solution $x$ followed by a restriction to the interval
$[-1,0]$.  Figure~\ref{fig.ddestate} illustrates this action, together
with the relationship of the state $x_t$ to the DDE solution $x(t)$.
\begin{figure}
% RT 4.9.2019: use standalone figure for submission to arxiv
%\psfrag{mt}{$-\tau$}
%\psfrag{tmt}{$t-\tau$}
%\psfrag{t}{$t$}
%\psfrag{xt}{$x_t(s)$}
%\psfrag{x}{$x(t)$}
%\psfrag{if}{$x_0 = \phi$}
%\psfrag{0}{$0$}
%\psfrag{s}{$s$}
%\begin{center}\includegraphics[width=5in]{ddestate}
\begin{center}\includegraphics[width=5in]{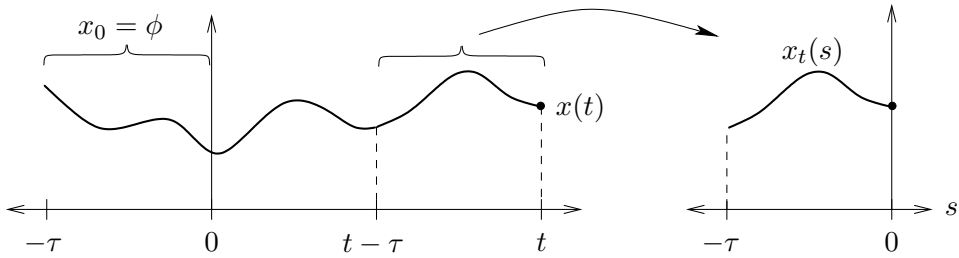}
\caption[Relationship between the solution $x(t)$ of a delay
equation and the phase point $x_t \in C$ of the corresponding
dynamical system.]{Relationship between the solution $x(t)$ of the
delay equation~\eqref{eq.ddeivp} and the phase point $x_t \in C$ of
the corresponding dynamical system.}
\label{fig.ddestate}
\end{center}
\end{figure}

The phase space of the dynamical system $\{S_t\}$ (and hence the phase
space of the corresponding DDE~\eqref{eq.ddeivp}), being the space of
continuous functions on the interval $[-1,0]$, is
infinite dimensional.  The infinite dimensionality of the phase space
for delay equations complicates their analysis dramatically, and as we
will see, it proves to be a serious barrier to developing a
probabilistic treatment.

%==========================================================================
\section{Representations of the Semigroup} \label{sec.ddesg}

The previous section illustrates how the delay
equation~\eqref{eq.ddeivp} can be viewed as a dynamical system in an
infinite dimensional phase space.  However, the definition of the
corresponding semigroup $S_t$ (\cf\ equation~\eqref{eq.ddesg}) is given
implicitly in terms of a particular solution of the DDE.  That is, in
order to evaluate $S_t(\phi)$ we must have the corresponding solution
$x$ of~\eqref{eq.ddeivp} already in hand.  Consequently, the present
definition of $S_t$ provides little insight as to how $S_t$ operates
as a transformation on $C$.  It is illuminating to consider
alternative representations of the semigroup, with a view to making
its action on the phase space of continuous functions more
transparent.  Sections~\ref{sec.explmap}--\ref{sec.absde} explore some
of the possibilities.

% ------------------------------------------------------------------------
\subsection{Explicit solution map} \label{sec.explmap}

For some delay equations it is possible to write the semigroup
operator $S_t$ explicitly as an iterated map on $C$.  For example,
consider delay equations of the form~\eqref{eq.ddeivp} where $f$ is
linear in its first argument, \viz,
\begin{equation} \label{eq.semilin}
  x'(t) = -\alpha x(t) + g\big( x(t-1) \big).
\end{equation}
Using the notation introduced in the previous section, let $x_t(\cdot)
= x(t+\cdot) \in C$ represent the phase point at time $t$ for the
corresponding dynamical system $\{S_t\}$.  It is simplest to construct
just the time-one map $S = S_1$ for this system, for which the only
relevant phase points are those at discrete times,
\begin{equation} \label{eq.ddestate2}
  x_n(\cdot) = x(n + \cdot), \quad n = 0,1,2,\ldots.
\end{equation}
In this notation the DDE \eqref{eq.semilin} becomes
\begin{equation}
  x_{n+1}'(s) = -\alpha x_{n+1}(s) + g\big( x_n(s) \big),
\end{equation}
an ordinary differential equation for $x_{n+1}$ in terms of the
(known) previous phase point $x_n$.  Its solution defines the time-one
map $S: x_n \mapsto x_{n+1}$.  Explicitly (\cf\ \cite{Ersh91}),
\begin{equation}
  (S u)(s) = u(0) e^{-\alpha(s+1)} + \int_{-1}^s {e^{\alpha(t-s)}
   g\big( u(t) \big) \, dt}, \quad s \in [-1,0].
\end{equation}

This map gives a representation of the DDE \eqref{eq.semilin} as a
discrete-time dynamical system,
\begin{equation}
   x_{n+1} = S x_n.
\end{equation}
Together with an initial function $x_0=\phi$, this system defines a
trajectory $\{x_n: n=0,1,\ldots\} \subset C$.  From this trajectory,
the solution $x(t)$ of the original delay equation
\eqref{eq.semilin} can be recovered according to equation
\eqref{eq.ddestate2}.

It is interesting that, although $\{S_t\}$ is a continuous-time
dynamical system, a trajectory of the discrete-time system $\{S^n: n
\in \Zplus\}$ is sufficient to construct the solution $x(t)$ of the
original DDE for all $t>0$.  The continuous-time family of maps $S_t$
does not provide any additional information about the solution, so it
is reasonable to treat the DDE as a truly discrete-time dynamical
system in $C$.  This observation does not depend on the special form
of the DDE~\eqref{eq.semilin}, as the same conclusion can be drawn for
the more general DDE~\eqref{eq.ddeivp} where, although we do not have
an explicit formula for the time-one map, $S$ can be defined using the
method of steps (\cf\ page~\pageref{sec.methsteps}).

% -----------------------------------------------------------------------
\subsection{Initial boundary value problem}  \label{sec.ddeivp}

The semigroup of operators $S_t$ on $C$ also has a representation in
terms of the solution of an initial boundary value problem.  Again,
this representation may be more illuminating than an implicit
definition of $S_t$ in terms of solutions $x(t)$ of the DDE, and it
applies even if an explicit solution map like that in the previous
section cannot be obtained.

If the right-hand side $f$ of the DDE \eqref{eq.ddeivp} is continuous,
then the solution $x$ is continuously differentiable on $(0,\infty)$.
Therefore, at least for $t>1$, the phase point $x_t(s)$ is
differentiable in both $t$ and $s$.  It follows that $x_t$, considered
as a function
\begin{equation} \label{eq.udef}
  u(s,t) = x_t(s) = x(t+s),
\end{equation}
satisfies the partial differential equation
\begin{equation} \label{eq.bvp}
  \frac{\partial u(s,t)}{\partial t} = \frac{\partial u(s,t)}{\partial s},
                 \quad s \in [-1,0], \; t > 1. \\
\end{equation}
The DDE \eqref{eq.ddeivp} implies a boundary condition on $u$,
\begin{equation} \label{eq.bvpbc}
  \frac{\partial u(s,t)}{\partial s}\Big|_{s=0} = f\big( u(0,t),u(-1,t) \big).
\end{equation}
Equations \eqref{eq.bvp}--\eqref{eq.bvpbc}, together with initial
data
\begin{equation} \label{eq.bvpic}
  u(s,0)=\phi(s),
\end{equation}
constitute an initial boundary value problem describing the evolution
of $x_t$.  If the initial function $\phi$ is differentiable and
satisfies the ``splicing condition''
\begin{equation} \label{eq.splicing}
  \phi'(0) = f(\phi(0),\phi(-1)),
\end{equation}
then the domain of~\eqref{eq.bvp} can be extended to $[-1,0] \times
[0,\infty)$.\footnote{If the splicing condition does not hold,
$u(s,t)$ can be interpreted as a weak solution on $[-1,0] \times
[0,\infty)$~\cite{BM00}.}

This initial boundary value problem evidently has a solution
$u(s,t)=x(t+s)$.  Moreover, since any solution $u(s,t)$
of~\eqref{eq.bvp}--\eqref{eq.bvpic} defines a solution $x$ of the
corresponding DDE via~\eqref{eq.udef}, uniqueness of the solution of
the DDE implies uniqueness for the initial boundary value problem.

% --------------------------------------------------------------------------
\subsection{Abstract differential equation}\label{sec.absde}

The connection between the initial boundary value problem
\eqref{eq.bvp}--\eqref{eq.bvpic} and the evolution semigroup $\{S_t\}$ can
be made more explicit by re-interpreting the initial boundary value
problem as an ``abstract Cauchy problem'', \ie, an initial value
problem on the function space $C$.

Recall that the phase point $x_t$ for the DDE \eqref{eq.ddeivp} is
given by
\begin{equation}
  x_t(s) = x(t+s),
\end{equation}
where $x$ is a solution of the DDE.  The phase space trajectory
corresponding to this solution is a continuous curve $\{x_t: t \geq
0\} \subset C$.  Under suitable hypotheses on the function $f$ in the
DDE~\eqref{eq.ddeivp} this curve is differentiable.  That is, the time
derivative
\begin{equation}
  \frac{d}{dt} x_t = \lim_{h \to 0}\frac{x_{t+h} - x_t}{h}
\end{equation}
exists, where the limit is taken in the strong sense of convergence in
$C$.  In fact, we have:

\begin{thm} \label{thm.absde}
Suppose that $f: \Rn \times \Rn \to \Rn$ is continuous, and that $\phi
\in C$ is continuously differentiable and satisfies the splicing
condition~\eqref{eq.splicing}.  Let $x$ be the corresponding solution
of the DDE~\eqref{eq.ddeivp}, and let $x_t(s) = x(t+s)$, $s
\in [-1,0]$.  Then the strong derivative $\tfrac{d}{dt}x_t$
exists and satisfies $\tfrac{d}{dt}x_t=\Acal x_t$ where $\Acal: C \to
C$ is given by $\Acal: u \mapsto u'$.
\end{thm}
\begin{proof}
\begin{equation}
\begin{split}
\lim_{h \to 0} \Big\| \frac{u_{t+h} - u_t}{h} - \Acal u_t \Big\| & \\
= \lim_{h \to 0} \sup_{s \in [-1,0]} &\Big| \frac{x(t+h+s) - x(t+s)}{h} -
 x'(t+s) \Big|.
\end{split}
\end{equation}
Then by the mean value theorem,
\begin{equation}
  \lim_{h \to 0} \Big\| \frac{u_{t+h} - u_t}{h} - \Acal u_t \Big\|
  = \lim_{h \to 0} \sup_{s \in [-1,0]} | x'(t+c(h)+s) - x'(t+s) |
\end{equation}
for some $|c(h)| < |h|$.  Under the given hypotheses, $x$ is
continuously differentiable on $[-\tau,\infty)$.  Thus $x'$ is
continuous and hence uniformly continuous on any closed interval
containing $[t-1,1]$, so the limit above is zero.
\end{proof}

Thus, at least for continuously differentiable initial functions
satisfying the splicing condition, the trajectory $\{x_t: t \geq 0\}$
corresponding to $x_0=\phi$ is differentiable and satisfies
\begin{equation} \label{eq.absde}
  \frac{d}{dt} x_t = \Acal x_t.
\end{equation}
The operator \Acal\ is called the \emph{infinitesimal generator} of
the semigroup~\cite[Ch.\,7]{LM94}.  Equation~\eqref{eq.absde} can be
regarded as an ``abstract differential equation'', with the mapping $u
\mapsto \mathcal{A}u$ acting like a vector field on $C$.  Together
with the initial condition $x_0 = \phi$, it constitutes an ``abstract
Cauchy problem'', or initial value problem, on $C$.  The DDE semigroup
$\{S_t\}$ furnishes a solution of this initial value problem, $x_t =
S_t(\phi)$, $t \geq 0$.\footnote{If $\phi$ does not satisfy the
splicing condition, $x_t=S_t \phi$ can be interpreted as a mild
solution of~\eqref{eq.absde}, \ie, there is a sequence of functions
$\phi_n \in C$, converging to $\phi$, that do satisfy the splicing
condition, such that $S_t \phi_n$ converges uniformly to $S_t
\phi$~\cite{BM00,DGLW95}.}  Thus the action of the semigroup $S_t$ can be
interpreted as carrying the initial function $\phi$ along a trajectory
in $C$ that is an integral curve of the differential
equation~\eqref{eq.absde}.

The theory of abstract differential equations such as~\eqref{eq.absde}
is most fully developed in the case where the corresponding semigroup
turns out to be a family of \emph{linear} operators.  This is the
case, for instance, when the DDE~\eqref{eq.ddeivp} is
linear~\cite[p.\,194]{HV93}.  Then there is an existence and
uniqueness theory for initial value problems satisfying differential
equations like~\eqref{eq.absde} (the Hille-Yosida Theorem and its
relatives~\cite[Ch.\,2]{CP78}).  For our purposes a detailed
discussion of this theory is unwarranted.  Instead, in the following
we merely sketch its relevance to linear delay equations.

Note that the infinitesimal generator $\mathcal{A}$ is not defined on
all of $C$, so that it is not strictly valid to consider $\mathcal{A}$
as a vector field on $C$.  In fact, it is clear from the proof of
Theorem \ref{thm.absde} that $\mathcal{A}$ is defined only on the
domain
\begin{equation}
  D(\mathcal{A}) = \{ u \in C: u' \in C \text{ and } u'(0) =
  f\big(u(0),u(-1)\big) \},
\end{equation}
However, $D(\mathcal{A})$ is dense in $C$~\cite[p.\,194]{HV93}.  This,
together with restrictions on $\mathcal{A}$ that are satisfied if the
delay equation~\eqref{eq.dde1} is linear, implies that the initial
value problem
\begin{equation} \label{eq.absdeivp}
\begin{split}
  &\frac{d}{dt} x_t = \mathcal{A} x_t, \quad x_t \in D(\mathcal{A}),\\
  &x_0 = \phi \in D(\mathcal{A}),
\end{split}
\end{equation}
has a unique solution $\{x_t=S_t(\phi): t \geq 0\}$~\cite{DGLW95}.

%==========================================================================
\section{Perron-Frobenius Operator} \label{sec.ddefp}

Having determined how a delay differential equation defines a
dynamical system, we are in a position to approach one of the
fundamental problems posed in this thesis.  That is, given a system
whose evolution is determined by a DDE
\begin{equation} \label{eq.ddefp1}
  x'(t) = f\big(x(t), x(t-1)\big),
\end{equation}
and whose initial phase point $\phi \in C$ is not known but is given
instead by a probability distribution over all possible initial
states, how does the probability distribution for the phase point
evolve in time?  Alternatively, we could consider the statistical
formulation of the problem: given a large ensemble of independent
systems, each governed by~\eqref{eq.ddefp1}, and whose initial
functions are distributed according to some density over $C$, how does
this ensemble density evolve in time?  It is of particular interest to
characterize those probability distributions that are invariant under
the action of the DDE.

In a sense, the answer to this problem is simple and is provided by
the Perron-Frobenius operator formalism, introduced in
Chapter~\ref{ch.ergth}.  Suppose the initial distribution of phase
points is described by a probability measure $\mu$ on $C$.  That is,
the probability that the initial function $\phi$ is an element of a
given set $A \subset C$ (correspondingly, the fraction of the ensemble
whose initial functions are elements of $A$), is given by $\mu(A)$.
Then, after evolution by time $t$, the new distribution is described
by the measure $\nu$ given by
\begin{equation} \label{eq.pullback1}
  \nu = \mu \circ S_t^{-1},
\end{equation}
provided $S_t$ is a measurable transformation on $C$
(\cf\ Section~\ref{sec.ddesigalg}).  That is, after time $t$ the
probability that the phase point is an element of $A \subset C$ is
$\nu(A) = \mu( S_t^{-1}(A))$.  If the initial distribution of states
$u$ can be described by a density $\rho(u)$ with respect to some
measure $\lambda$, then after time $t$ the density will have evolved
to $P_t\rho$, where the Perron-Frobenius operator $P_t$ corresponding
to $S_t$ is defined by
\begin{equation} \label{eq.fpdde1}
  \int_A P_t \rho(u)\,d\lambda(u) = \int_{S_t^{-1}(A)}
  \rho(u)\,d\lambda(u) \quad \forall \text{ $\lambda$-measurable } A
  \subset C.
\end{equation}

Equations \eqref{eq.pullback1}--\eqref{eq.fpdde1} might appear to
answer the problem of the evolution of probability measures for DDEs.
However, they amount only to a formal answer---essentially a symbolic
restatement of the problem.  In fact, everything that is specific to a
given DDE is contained in the symbol $S_t^{-1}$.

Although the DDE can be expressed in terms of an evolution semigroup,
in none of its representations (\cf\ Section~\ref{sec.ddesg}) is there
an apparent way to invert the resulting transformation $S_t$.  It is
almost certain that such an inversion will be non-trivial, since
solutions of delay equations frequently cannot be uniquely extended
into the past~\cite{Dr77}, so that $S_t$ will not have a unique
inverse.  That is, $S_t^{-1}$ may have numerous branches that need to
be accounted for when evaluating $S_t^{-1}(A)$ in the Perron-Frobenius
equation~\eqref{eq.fpdde1}.  This is a serious barrier to deriving a
closed-form expression for the Perron-Frobenius operator $P_t$.

There are other subtle issues raised by
equations~\eqref{eq.pullback1}--\eqref{eq.fpdde1}.  The most apparent
difficulty is that the integrals in~\eqref{eq.fpdde1} are over sets in
a function space, and it is not immediately apparent how such
integrals can be carried out.  More fundamentally, it is not clear
what family of measures we are considering, and in particular what
subsets $A \subset C$ are measurable (\ie, what is the relevant
$\sigma$-algebra on $C$?).  Also, in equation~\eqref{eq.fpdde1} what
should be considered a natural choice for the measure $\lambda$ with
respect to which probability densities are to be defined?  For that
matter, does it make sense to talk about probability densities over
the function space $C$?  These issues are explored in the following
section.

%==========================================================================
\section{Probability in Infinite Dimensional Spaces} \label{sec.infprob}

Any discussion of an ergodic theory of delay equations will require a
theory of measure and integration on function spaces.  In particular
we need to discuss probability measures on the space $C$ of continuous
functions on the interval $[-1,0]$, since this is a natural phase
space for the DDE \eqref{eq.ddeivp}.  Colloquially speaking, we need
to make precise the somewhat non-intuitive notion of selecting a
\emph{random function} from $C$.

Measure-theoretic probability provides a sufficiently abstract setting
to accomplish this.  Recall from Chapter~\ref{ch.ergth} that we can
represent a random variable $x \in X$ by its associated probability
measure $\mu$, with the interpretation that for a given subset $A
\subset X$, $\mu(A)$ expresses the probability that $x \in A$.  To
ensure consistency with the axioms of probability, we cannot assign a
probability to just any subset of $X$.  Rather, $\mu$ must be defined
on an appropriate $\sigma$-algebra---a collection of so-called
\emph{measurable sets} (\cf\ Section~\ref{sec.mtheory}).  So choosing an
appropriate $\sigma$-algebra on $C$ is a necessary starting point.

\subsection{Appropriate $\sigma$-algebra}\label{sec.ddesigalg}

In real Euclidean spaces, the notion of measure derives from our
physical intuition of length, area, volume, and their generalizations
to higher dimensions.  Thus line segments in one dimension, and
rectangles in two dimensions, are natural candidates for inclusion in
the $\sigma$-algebras of choice for these spaces.  The natural choice
of $\sigma$-algebra would seem to be the smallest $\sigma$-algebra
that contains all such sets---that is, the $\sigma$-algebra generated
by these sets.  This is the so-called \emph{Borel $\sigma$-algebra},
which happens to coincide with the smallest $\sigma$-algebra that
contains all open subsets.

A similar approach leads to a natural choice of $\sigma$-algebra for
infinite dimensional spaces such as $C$.  That is, we take the Borel
$\sigma$-algebra generated by the metric topology on $C$.  With this
choice, many important subsets of $C$ such become measurable, \ie\ we
can assign meaningful probabilities to them:
\begin{itemize}
\item any open set in $C$
\item $\{u \in C:  u(s) \in (a,b) \;\forall s \in [-1,0]\}$; $a,b \in \Reals$
\item any $\epsilon$-ball $B_{\epsilon}(v) = \{u \in C: \|u-v\| < \epsilon\}$; $v \in C$, $\epsilon \in \Reals$ 
\end{itemize}

Besides achieving the measurability of important sets for analysis,
there is a more fundamental reason for choosing the Borel
$\sigma$-algebra.  Recall that studying the evolution of probability
measures under a given transformation makes sense only if the
transformation is \emph{measurable}.  Therefore, for our study of DDEs
it is essential to choose a $\sigma$-algebra on which the semigroup
$S_t$ defined by equation~\eqref{eq.ddesg} is measurable.  The
following establishes that the Borel $\sigma$-algebra accomplishes
this.
\begin{thm} \label{thm.ddemeas}
For every $t \geq 0$, $S_t: C \to C$ (\cf\ equation~\eqref{eq.ddesg})
is a measurable transformation on the Borel $\sigma$-algebra on $C$.
\end{thm}
\begin{proof}
$S_t$ is continuous on $C$, by Theorem~\ref{thm.ddesg}(c), hence
measurable, by Theorem~\ref{thm.contmeas}.
\end{proof}

It may be the case that the Borel $\sigma$-algebra on $C$ is in fact
not the most natural choice in the context of a probabilistic approach
to DDEs.  Certainly, as demonstrated in the following sections,
measures on the Borel sets of infinite dimensional spaces do not
behave as we might like.  However, in light of the preceding
considerations, from now on we will consider only measures defined on
the Borel sets of $C$.

% ---------------------------------------------------------------------
\subsection{Densities} \label{sec.infdimdens}

Recall that if a measure $\mu$ is absolutely continuous with respect to
a measure $\lambda$, then it can be expressed as
\begin{equation} \label{eq.densdef}
  \mu(A) = \int_A \rho \,d\lambda,
\end{equation}
where the integral is in the sense of Lebesgue, and $\rho \in
L^1(X,\lambda)$ is the density of $\mu$ with respect to $\lambda$.
Furthermore, any Lebesgue integrable function $\rho \in L^1(X,\lambda)$
with
\begin{equation} \label{eq.densnorm}
  \int \rho\,d\lambda = 1
\end{equation}
uniquely determines an absolutely continuous measure $\mu$
(\cf\ Section~\ref{sec.densities}).

Since the relations~\eqref{eq.densdef}--\eqref{eq.densnorm} require
only a $\sigma$-algebra and a measure $\lambda$ on $X$, they apply
equally well in the more abstract setting of infinite dimensional
spaces such as $C$.  That is, if $C$ is equipped with a
$\sigma$-algebra \Acal\ and measure $\lambda$ on \Acal, then the
function space $L^1(C)=L^1(C,\Acal,\lambda)$ is unambiguously defined
(\cf\ Section~\ref{sec.integral}), and any functional $\rho \in
L^1(C)$ determines an absolutely continuous measure on $C$.  However,
in this context the intuitive appeal of densities is lacking: it is
impossible to draw the graph of such a density functional.  Even
imagining a density on $C$ seems beyond the power of one's
imagination.

The analytical benefits of using densities also appear to be quite
limited in infinite dimensional spaces.  The connection between
measure theory and calculus in finite dimensions owes much to the
theory of integration, notably the fundamental theorem of calculus and
other theorems that facilitate calculations with integrals.  There is
no adequate theory of integration on function spaces that makes it
possible to evaluate integrals like~\eqref{eq.densdef} on $C$ (\cf\
comments in~\cite{LM92}).  A notable exception to this is Wiener
measure, although this does not seem to be adequate for our purposes;
see Section~\ref{sec.wien}, page~\pageref{sec.wien}.

Even allowing that a more powerful theory of integration may be
available in the future, there remain some inherent difficulties with
using densities to specify probability measures on
infinite dimensional spaces.  Equation~\eqref{eq.fpdde1} for the
evolution of a probability density $\rho$ under the action of a
semigroup $S_t$ is valid only if $S_t$ is non-singular.  That is,
pre-images under $S_t$ of $\lambda$-measure-zero sets must have
$\lambda$-measure zero.  It turns out to be difficult to guarantee
this.  In fact, on an infinite dimensional space, every absolutely
continuous measure fails to remain absolutely continuous under
arbitrary translations~\cite{Yam85}.  That is, for \emph{any} measure
$\lambda$ on $C$, there is some $v \in C$ for which the translation
\begin{equation}
  T: u \mapsto u + v
\end{equation}
is singular (in the measure-theoretic sense), and hence does not map
densities to densities.  If even translations do not lead to
well-defined density evolution, there is little hope of studying delay
equations with density functionals.

\subsection{Lack of a ``natural'' measure on $C$}

As if the preceding did not complicate matters enough, if we are to
work with densities on $C$ there remains the problem of choosing a
basic measure $\lambda$ with respect to which densities are to be
defined (\cf\ equation~\eqref{eq.densdef}).  This too turns out to be
problematic.

In real Euclidean spaces we are accustomed to taking Lebesgue measure
as the ``natural'' measure with respect to which densities are
defined.  That is, ``a random number distributed uniformly on the
interval $[0,1]$'' means ``a random variable on $[0,1]$ distributed
according to Lebesgue measure''.  Why is Lebesgue measure---of all
possible measures---the gold standard for representing the concept of
``uniformly distributed''?

The property of Lebesgue measure that selects it uniquely as the
natural measure on Euclidean spaces is its translation invariance.
Given a random variable $x$ uniformly distributed on $[0,1]$, we
expect that adding a constant $a$ to $x$ should result in a new random
variable, $x+a$, that is uniformly distributed on $[a,a+1]$, at least
according to what seems to be the common intuitive notion of
``uniformly distributed''.  More generally, a random variable
uniformly distributed on any set in \Rn\ should remain uniformly
distributed if translated by a constant vector.  Formally, the measure
$\lambda$ on \Rn\ that encapsulates uniform distribution should
satisfy
\begin{equation}
  \lambda(A) = \lambda(A+a), \quad \forall a \in \Rn, \; \forall
  \text{ measurable } A.
\end{equation}
Another way to say this is that $\lambda$ is invariant under the
translation group
\begin{equation}
  T_a: x \mapsto x - a.
\end{equation}
That is,
\begin{equation} \label{eq.transinv}
  \lambda = \lambda \circ T_a^{-1} \quad \forall a \in \Rn.
\end{equation}
Equation \eqref{eq.transinv} uniquely defines the Borel measure
$\lambda$ on the Borel $\sigma$-algebra on \Rn (which agrees with
Lebesgue measure on the Borel sets).  This is a specific instance of
Haar measure: every locally compact topological group (\eg, the
translation group just considered on \Rn) has a unique group-invariant
measure on the Borel $\sigma$-algebra, called the Haar measure, that
is non-zero on any open set~\cite[p.\,313]{Lang93}.

In light of these considerations, in choosing a natural measure on $C$
it seems reasonable to seek a translation-invariant measure.  After
all, we would like that a uniformly distributed ensemble of functions
in the unit ball in $C$ should remain uniformly distributed under
translation by any function in $C$.  Unfortunately the existence of a
Haar measure on $C$ is not guaranteed, since $C$ is not locally
compact.\footnote{A normed vector space is locally compact iff it is
finite dimensional~\cite[p.\,39]{Lang93}.}  In fact the situation is
worse than that, as the following theorem demonstrates.

\begin{thm}
Let $X$ be an infinite dimensional separable Banach space.  If
$\lambda$ is a non-zero translation-invariant measure on the Borel
sets of $X$, then $\lambda(A)=\infty$ for every open $A \subset X$.
\end{thm}
\begin{proof} (after~\cite{HSY92}.)
Let $B \subset X$ be an open ball of radius $\epsilon>0$, and suppose
$\lambda(B)>0$ is finite.  Because $X$ is infinite dimensional, there
is an infinite sequence $B_i$, $i=1,2,\ldots$ of disjoint open balls
$B_i \subset B$, each of radius $\epsilon/4$ (\cf\ the proof of Theorem
4.3.3 in~\cite[p.\,134]{Fr70}).  Because $\{B_i\}$ is a countable
disjoint collection with $\cup_i B_i \subset B$, we have
\begin{equation}
  \lambda(B) \geq \lambda(\cup_i B_i)
             =\sum_{i=1}^{\infty} \lambda(B_i),
\end{equation}
where $\lambda(B_i) = \lambda(B_1)$ by translation invariance.  Since
$\lambda(B)$ is finite, this implies that $\lambda(B_i)=0$ $\forall
i$.  Separability of $X$ implies that $X$ can be covered by a
countable collection of $\epsilon/4$-balls, each of which we have just
shown must have measure 0.  Hence $\lambda(X)=0$, a contradiction.
\end{proof}

Since we expect any reasonable measure to be non-zero at least on
some open sets, we can conclude that translation-invariance will not
suffice to select a natural measure on $C$.

Aside from making the definition of densities on $C$ ambiguous, the
absence of a natural measure undermines one of the most important
concepts in ergodic theory.  Recall from Section~\ref{sec.SRB} that an
SRB measure $\mu$ for a dynamical system $S_t$ on $X$ is one such
that, for any functional $\varphi \in L^1(X)$,
\begin{equation} \label{eq.srb1}
  \lim_{T \to \infty}\frac{1}{T}\int_0^T \varphi(S_t x)\,dt = \int
   \varphi\,d\mu
\end{equation}
for Lebesgue almost every $x$.  Thus the time average of $\varphi$
along almost every trajectory is equal to the spatial average of
$\varphi$ weighted with respect to $\mu$.  Because $\varphi(x)$
represents an arbitrary observable of the system, and $\mu$
encapsulates the asymptotic statistical behavior of $\varphi(x(t))$ on
almost every orbit of the system, is it widely accepted that an SRB
measure is \emph{the} relevant physical measure---the one that nature
reveals to the experimentalist.

The notion of ``almost every'' in \eqref{eq.srb1} is always
unquestioningly taken to mean ``Lebesgue almost every''.  As we have
seen, for infinite dimensional systems, and for delay equations in
particular, we have no natural analog of ``Lebesgue almost every'',
since there is no translation invariant measure to take the place of
Lebesgue measure.

That this ambiguity emerges at all is somewhat amusing, since the
notion of SRB measure was introduced on purely \emph{physical}
grounds.  The very definition of SRB measure requires that we make
precise the notion of ``physically relevant''---but for DDEs this
leads to considerations in the decidedly non-physical setting of
infinite dimensional geometry, where it appears to be an inherently
ambiguous term.

% --------------------------------------------------------------------------
\subsection{Genericity and prevalence} \label{sec.prevalence}

Without a natural measure on $C$ to characterize a physically relevant
notion of ``almost every'', the definition of SRB measure for a delay
differential equation is problematic.  One way out of this dilemma is
to introduce a notion of almost every that does not depend on a
specific measure, such as the topological concept of
\emph{genericity}.  A property is said to be generic if it holds on a
residual set, that is a countable intersection of open dense sets.
The complement of a residual set is a set of ``first category'', hence
first category sets are topological analogs of sets of measure zero.
Although genericity provides one way to quantify the notion of almost
every in infinite dimensional spaces, it lacks the probabilistic
interpretation that we would like to have in the context of ergodic
theory.  More importantly, even in \Rn\ residual sets can have measure
zero~\cite{HSY92}, so using genericity in the definition of SRB
measure would be inconsistent with the accepted definition for finite
dimensional systems.

A more promising alternative is a translation-invariant probabilistic
notion of almost every called \emph{prevalence}~\cite{HSY92}:
\begin{defn}
Let $X$ be a Banach space equipped with its Borel $\sigma$-algebra
\Acal.  A Borel set $A \in \Acal$ is called \emph{shy} if there is a measure
$\mu$ on \Acal\ such that
\begin{itemize}
\item $0 < \mu(U) < \infty$ for some compact $U \subset X$, and
\item $\mu(A+x)=0$ $\;\forall x \in X$.
\end{itemize}
$A$ is called \emph{prevalent} if it is the complement of a shy set.
\end{defn}

Roughly speaking, a set is shy if for some nontrivial measure on $X$,
every translate of $A$ has measure zero.  Two key properties make
prevalence an attractive candidate for a notion of ``almost every''
appropriate to a definition of SRB measure for infinite dimensional
systems (for proofs see~\cite{HSY92}):

\begin{enumerate}
\item If $A$ is prevalent then any translate of $A$ is prevalent;
\ie\ prevalence is a translation-invariant property.
\item $A \subset \Rn$ is shy if and only if $A$ has Lebesgue measure
zero.
\end{enumerate}

The first property means prevalence is a natural or physical notion of
almost every in the sense discussed in the previous section.  The
second property guarantees that, in finite dimensions, a property
holds on a prevalent set if and only if it holds on a set of positive
Lebesgue measure.  Thus for finite dimensional systems the definition
of SRB measure (\cf\ Definition~\ref{def.SRB}, page~\pageref{def.SRB})
is unchanged is we substitute ``a prevalent set'' for ``a set of
positive Lebesgue measure''.  The novelty and significance of this
alternative definition is that it applies equally well to
\emph{infinite} dimensional systems.

Tools for proving shyness and prevalence are developed
in~\cite{HSY92}.  The following interesting results have been proved
(here we use ``almost every'' in the sense of ``in a prevalent set''):

\begin{itemize}
\item If $X$ is infinite dimensional then every compact subset of $X$ is shy.
\item Almost every element of $C$ is nowhere differentiable.
\item For $1 \leq p \leq \infty$ almost every $C^p$ map on \Rn\ has the property
that all of its periodic points are hyperbolic.
\item Almost every $f \in C([0,1],\Reals)$ satisfies
$\int_0^1 f(x)\,dx \neq 0$.
\end{itemize}

We are unaware of any applications of prevalence to the concept of SRB
measure.  This appears to be a promising direction for further
investigation.

% -------------------------------------------------------------------------
\subsection{Wiener measure}\label{sec.wien}

As already noted, an adequate theory of integration on infinite
dimensional spaces in lacking.  Such a theory is needed if we are to
further develop the Perron-Frobenius operator formalism to
characterize the evolution of densities for DDEs, which requires a
theory of integration of functionals on the space $C$.  This
difficulty also arises in~\cite{LM92}, in the context of a different
approach to the evolution of densities for DDEs.

However, there is a notable exception worth mentioning.  There is one
probability measure (or family of measures) on a function space,
called Wiener measure, for which there is a substantial theory of
integration~\cite{Kac}.  This measure plays an important role in
quantum field theory (see \eg~\cite{Ryd85}), and is central to the
theory of stochastic differential equations~\cite[Ch.\,11]{LM94}.

Let
\begin{equation}
  C_0 = \{ u \in C([0,1],\Rn): u(0) = 0 \}.
\end{equation}
A Brownian motion\footnote{A Brownian motion is a continuous-time
analog of a random walk starting at the origin.  See \eg~\cite{LM94}.}
is a stochastic process that generates a random path or ``random
function'' $w \in C_0$ such that for a given $t \in [0,1]$, $w(t)$ has
Gaussian probability density~\cite{LM92}
\begin{equation}
  \rho(x_1,\ldots,x_n) = \frac{1}{(\sqrt{2 \pi t})^n} \exp\big[
    -(x_1^2 + \cdots + x_n^2) / (2 t) \big].
\end{equation}
Then, roughly speaking, Wiener measure $\mu_w$ assigns to a given
subset $A \subset C_0$ a measure equal to the probability that a
Brownian motion generates an element of $A$.

With Wiener measure it is possible to prove strong ergodic properties
(\eg~exactness) for a certain class of partial differential
equations~\cite{BK84,Rud85,Rud87,Rud88}.  The success of these
investigations, together with the considerable machinery that has been
developed around the Wiener measure, suggests that Wiener measure
might be a good choice for the measure of integration in the study of
other infinite dimensional systems such as delay equations.  However,
in contrast with the quantum field equations and the PDEs mentioned
above, the dynamical system $\{S_t\}$ corresponding to a delay
equation does not leave the space $C_0$ invariant.  That is, we cannot
study $S_t$ on $C_0$ alone.  Thus Wiener measure does not seem to be
adequate for our purposes.  Nevertheless, an approach based on Wiener
measure might be still possible, and this suggests a fruitful avenue
for further investigation.

%==========================================================================
\section{Conclusions}

In this chapter we have developed a framework in which an ergodic
treatment of delay differential equations might be developed.  This
provides a setting and terminology that will be needed for our
subsequent discussions of the ergodic properties of DDEs.

However, as far as the possibilities for the rigorous development of
an ergodic theory of DDEs are concerned, the main results of this
chapter are somewhat pessimistic.  The picture that emerges is a
characterization of DDEs as infinite dimensional dynamical systems on
the phase space $C$ of continuous functions on the interval $[-1,0]$.
With this characterization, an ergodic theory of DDEs is possible in
principle.  In such a theory the mathematical objects of primary
interest are probability measures on $C$.  This entails a theory of
measure and probability on infinite dimensional spaces.  As we have
seen, the foundations of this theory run aground on a number of
technical and interpretational difficulties including the following.
\begin{itemize}
\item Non-invertibility of the evolution semigroup $\{S_t\}$.
\item Likely singularity of $S_t$ with respect to most measures on $C$.
\item Lack of an adequate theory of integration on
infinite dimensional spaces.
\item Non-existence of a natural (\ie~translation-invariant) measure on $C$.
\item Ambiguity in the definition of SRB measure for infinite dimensional systems.
\end{itemize}

Some of these difficulties (\eg, with integration in infinite
dimensions) appear to require significant new mathematical tools that
are beyond the scope of this thesis.  Others (\eg, with the choice of
a natural measure on $C$ and the definition of SRB measure) are simply
ambiguities that arise when dynamical systems theory developed with
only finite-dimensional systems in mind is carried over to an
infinite dimensional setting.  Nevertheless, in the absence of
criteria by which these ambiguities could be resolved, we must content
ourselves with having carefully discussed the available alternatives.

In light of the foregoing the following chapters focus less on ergodic
formalism, in order to pursue more fruitful lines of inquiry.  In the
next chapter we turn to the practical problem of computing the
evolution of probability densities for the state $x(t) \in \Rn$ rather
than an abstract phase point in $C$.

\chapter{Density Evolution for Delay Equations}\label{ch.densev}

\begin{singlespacing}
\minitoc   % mini table of contents for this chapter
\end{singlespacing}
\newpage

In light of the results of the previous chapter, a comprehensive
treatment of delay equations within the ergodic theory of dynamical
systems is out of reach.  Nevertheless, a probabilistic treatment is
feasible if the dynamical systems formalism is abandoned, and this is
the approach taken in the present chapter.

This chapter again considers systems that can be modeled by a DDE of
the form
\begin{equation} \label{eq.dde2}
  x'(t) = f\big( x(t), x(t-1) \big), \quad t \geq 0,  \quad x(t) \in \Rn,
\end{equation}
where without loss of generality the ``delay time'' has been scaled to
one.  In contrast with the previous chapter, we now take the point of
view of an experimentalist, interpreting this equation as prescribing
the evolution of an observable quantity $x(t)
\in \Rn$, rather than a phase point in an abstract function space.  Thus we
imagine an experimental setting in which an ensemble of independent
systems evolves according to~\eqref{eq.dde2}, and seek a probabilistic
description of this ensemble in terms of the evolution of the density
$\rho(x,t)$ of the ensemble of solution values $x(t)$.  (Alternatively
we can think of $\rho(x,t)$ as a \emph{probability} distribution that
quantifies our uncertain knowledge of the state of a single system
governed by~\eqref{eq.dde2}.)

\begin{figure}
\begin{center}
\includegraphics[width=\figwidth]{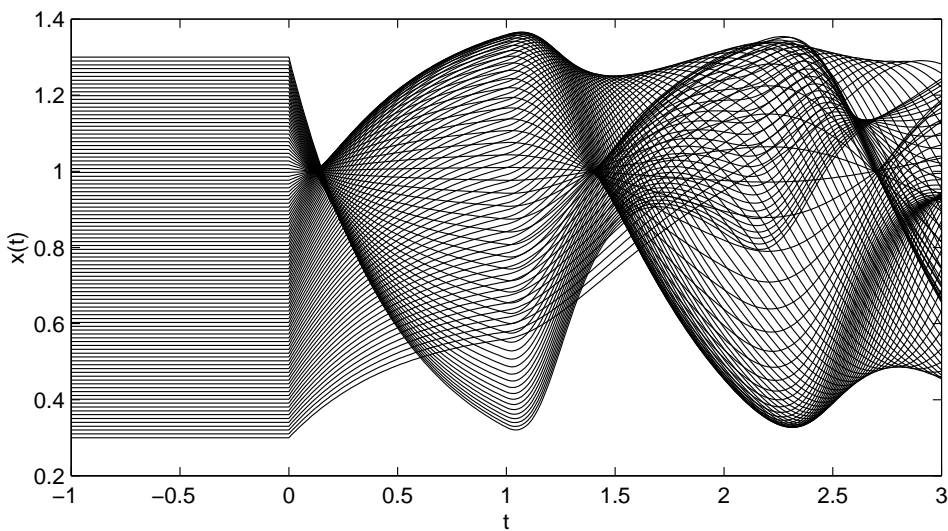}
\caption[An ensemble of $100$ solutions of the Mackey-Glass
equation, corresponding to an ensemble of constant initial
functions.]{An ensemble of $100$ solutions of the Mackey-Glass
equation~\eqref{eq.mg2}, corresponding to an ensemble of constant
initial functions with values uniformly distributed on the interval
$[0.3,1.3]$.}
\label{fig.densidea}
\end{center}
\end{figure}
Figure~\ref{fig.densidea} illustrates the problem we wish to consider.
The figure depicts an ensemble of $100$ solutions\footnote{Numerical
solutions were computed here using the solver DDE23~\cite{ST00}} of
the Mackey-Glass equation~\cite{MG77},
\begin{equation} \label{eq.mg2}
\begin{gathered}
  x'(t) = -\alpha x(t) + \beta \frac{x(t-1)}{1 + x(t-1)^{10}},\\
  \alpha=2, \quad \beta=4, \quad n=10,
\end{gathered}
\end{equation}
which was originally introduced to model oscillations in neutrophil
populations.  This equation has been the subject of much study because
of the variety of dynamical phenomena it exhibits.  The parameter
values chosen here correspond to the existence of a chaotic attractor.
The solutions shown in Figure~\ref{fig.densidea} correspond to an
ensemble of $100$ constant initial functions, whose values are
uniformly distributed on the interval $[0.3,1.3]$.  From the density
of solution curves on this graph, one can form an idea of the density
$\rho(x,t)$ of solution values $x(t)$ at any given time $t$.  For
example, at $t=3$ solutions are particularly dense near $x=0.45$,
$x=0.6$, $x=0.9$ and $x=1.2$.

The main question this chapter attempts to answer in the context of
Figure~\ref{fig.densidea} is the following.  If the density
$\rho_0(x)=\rho(x;0)$ of constant initial values $x$ at $t=0$ is
known, how can one determine (\ie, predict) the density $\rho(x;t)$
for times $t>0$?  The following section develops an appropriate
framework for the analysis of this problem.  In
Sections~\ref{sec.fpexplicit}--\ref{sec.approxfp} this framework is
used to develop various approaches to the evolution of densities.
Each of these approaches is essentially independent of the others, but
they are presented in an order that takes advantage of the interplay
between them.  Analytical techniques are considered in
Sections~\ref{sec.fpexplicit} and~\ref{sec.stepsdens};
Sections~\ref{sec.histos} and~\ref{sec.approxfp} focus on
computational approaches.

% =====================================================================
\clearpage
\section{Probabilistic Framework}\label{sec.nonauton}

Although equation~\eqref{eq.dde2} describes the evolution of a
finite-dimensional vector $x(t) \in \Rn$, the space of initial
conditions for this equation---the space $C$ of continuous functions
from $[-1,0]$ into \Rn---is infinite dimensional.  This is the main
source of difficulty in our attempts so far to develop a probabilistic
approach.  To make the problem more intuitive as well as
mathematically tractable, it is necessary to somehow restrict the
dimension of the set of ``allowable'' initial conditions.

% ----------------------------------------------------------------------
\subsection{Restricted initial value problem}

The simplest such restriction would be to allow only initial functions
from some $n$-dimensional subspace of $C$, such as the space of
constant initial functions (\cf\ Figure~\ref{fig.densidea}).  Given
the plethora of different finite-dimensional subspaces available in
$C$, this restriction might seem excessive.  However, there is a
physical justification for such a restriction, since in an
experimental setting the initial preparation of the ensemble is
typically in an equilibrium state.  In this case we expect each of the
units in the ensemble will have a constant initial history, and thus
the subspace of constant initial functions is naturally selected
by the experiment.

There are a number of other ways that an experimental setting might
naturally select a finite-dimensional set of allowable initial
functions for~\eqref{eq.dde2}.  Since our hypothetical ensemble has
presumably not been in existence for all time, there must be some
process by which the individual initial histories are generated.
Since this process cannot be described by the governing delay
equation, it is reasonable to posit some other process that does
govern the initial histories $x(t)$ on the interval $[-1,0]$, and to
describe this process by an \emph{ordinary} differential equation.

For convenience, let the initial time for the DDE~\eqref{eq.dde2} be
$t=1$ rather than $t=0$.  Thus we consider the DDE
\begin{equation} \label{eq.ddet1}
  x'(t) = f\big( x(t), x(t-1) \big), \quad t \in [1,\infty)
\end{equation}
with initial function specified on the interval $[0,1]$.  Then the
corresponding initial value problem can be written
\begin{equation} \label{eq.augdde}
\begin{split}
  &x'(t) = \begin{cases}
              g\big( x(t) \big) & \text{if $t \in [0,1)$} \\
              f\big( x(t), x(t-1) \big) & \text{if $t \geq 1$},
           \end{cases} \\
  &x(0) = x_0,
\end{split}
\end{equation}
where $g: \Rn \to \Rn$ describes the process by which the initial
function is determined by the initial value $x(0)$.  Of course we
require that $x'=g(x)$ with initial value $x(0)=x_0$ have a unique
solution\footnote{\eg, it suffices to have $g$ bounded and
continuously differentiable~\cite{Dr77}.} on $[0,1]$, so
that~\eqref{eq.augdde} describes a deterministic process on
$[0,\infty)$.  In~\eqref{eq.augdde} the set of allowable initial
functions selected by $g$ is just the set of solutions of the ODE
$x'=g(x)$ on $[0,1]$.  This is a one-dimensional set parametrized by
the initial value $x_0$.  For example the space of constant initial
functions corresponds to $g=0$.

Even if the initial function is not determined by an ODE, we still
would like the set of allowable initial functions to be parametrized
by the initial value $x_0$, since specifying an ensemble of initial
values $x_0$ then determines an ensemble of initial functions, and
hence an ensemble of solutions of the given DDE.  Thus in the most
general case we wish to consider the DDE~\eqref{eq.ddet1} with an
initial function specified by
\begin{equation}
  x(t) = \psi(t,x_0), \quad t \in [0,1]
\end{equation}
for some function $\psi: \Reals \times \Rn \to \Rn$ which should have
the following properties:
\begin{itemize}
\item The function $t \mapsto \psi(t,x_0)$ (\ie~the initial
function corresponding to the initial value $x_0$) is continuous.
\item The mapping $x_0 \mapsto \psi(t,x_0)$ is a measurable, non-singular
transformation of \Rn, so if $x_0$ is distributed with density $\rho_0$
then the density of $x(t)=\psi(t,x_0)$ is well defined for each $t \in
[0,1]$ (\cf\ Section~\ref{sec.pfop}, page~\pageref{sec.pfop}).
\item $\psi(0,x_0)=x_0$, so the parameter $x_0$ defines
the initial value $x(0)$.
\end{itemize}
Every such $\psi$ determines a particular one-parameter family of
allowable initial functions in $C$.  If $\psi(t,x_0)$ is the solution
map for an ordinary differential equation $x'=g(x)$ (\ie, the function
$t \mapsto \psi(t,x_0)$ is the solution of the ODE with $x(0)=x_0$),
then it satisfies the conditions above.  For example the family of
constant initial functions corresponds to $\psi: (t,x) \mapsto x$,
which in turn corresponds to $g=0$.

Having parametrized the set of allowable initial functions according
to a particular function $\psi$, the initial value problem
corresponding to~\eqref{eq.ddet1} becomes
\begin{equation} \label{eq.ddeparam}
\begin{split}
  &x'(t) = f\big( x(t), x(t-1) \big), \quad t \geq 1 \\ &x(t) =
  \psi(t,x_0), \quad t \in [0,1], \\ &x(0) = x_0.
\end{split}
\end{equation}
Under suitable mild restrictions on $f$ (\cf\
Chapter~\ref{ch.framework}), for each $x_0 \in \Rn$ this problem
uniquely determines the evolution of $x(t)$ for $t \in [0,\infty)$.

In the following we restrict our attention to systems in which the
initial function is determined by an ordinary differential equation.
Thus the remainder of this chapter is concerned with probabilistic
approaches to the initial value problem~\eqref{eq.augdde} which we
will call the ``augmented DDE'', as distinguished from the
corresponding DDE~\eqref{eq.ddet1} with no restriction on the set of
allowable initial functions.

% ----------------------------------------------------------------------
\subsection{Perron-Frobenius operator}\label{sec.augfp}

Let $S_t: \Rn \mapsto \Rn$ be the solution map for the augmented
DDE~\eqref{eq.augdde}.  That is,
\begin{equation} \label{eq.nonautonsg}
  S_t: x_0 \mapsto x(t), \quad t \in [0,\infty),
\end{equation}
where $x(t)$ is the (presumed unique) solution of~\eqref{eq.augdde}.
If an ensemble of initial values $x_0$ is specified with density
$\rho_0$, then the evolution of this density under the action of $S_t$
is given, in principle, by the corresponding Perron-Frobenius operator
$P_t: L^1(\Rn) \to L^1(\Rn)$ (\cf\ Section~\ref{sec.pfop}).  This
operator carries the initial density $\rho_0$ to the density $P_t\rho_0$
at time $t$, and is defined by the relation
\begin{equation}\label{eq.ddefp}
  \int_A P_t \rho_0(x) \, dx = \int_{S_t^{-1}(A)} \rho_0(x)\,dx \quad
  \text{$\forall$ Borel $A \subset \Rn$.}
\end{equation}

Recall that $P_t$ is well defined only if $S_t$ is a measurable,
nonsingular transformation.  In fact measurability is guaranteed
because for each $t$, $S_t$ is a continuous map on \Rn\ (this follows
from continuity with respect to initial conditions for both the ODE
and the DDE, \cf\ Theorem~\ref{thm.ddemeas}
page~\pageref{thm.ddemeas}).  However, non-singularity of $S_t$ is not
guaranteed for all $t$---indeed, Section~\ref{sec.fpexplicit} presents
a counter-example.

Note that the family of transformations $\{S_t: t \geq 0\}$ is not a
semigroup.  Consequently, neither is the family of Perron-Frobenius
operators $\{P_t\}$.  This is not a consequence of restricting the
allowable set of initial functions, but rather comes from viewing the
DDE as specifying an evolution in \Rn\ (rather than the function space
$C$).  In \Rn\ the augmented DDE~\eqref{eq.augdde} is non-autonomous,
in that an explicit time dependence appears in the term $x(t-1)$
(which acts as a forcing term).  This destroys the time-invariance
required by the semigroup property.  In short, the value of $x(t)$ at
a particular time is not sufficient to uniquely determine its
subsequent evolution---an obvious consequence of delayed dynamics.

The absence of the semigroup property for $S_t$ and $P_t$ has
important consequences.  For instance, it is not possible to express
$P_{t+t'}$ as a composition $P_{t} \circ P_{t'}$.  With the semigroup
property, to find $P_n$ for any integer $n$ it suffices to find $P_1$
and then express $P_n = (P_1)^n$.  Without the semigroup property this
construction fails, and finding $P_t$ for arbitrarily large $t$
becomes far less trivial.

The remainder of this chapter is concerned with the evolution of
densities for the augmented DDE~\eqref{eq.augdde}.  This amounts to
finding the corresponding Perron-Frobenius operator $P_t$.
Sections~\ref{sec.fpexplicit} and~\ref{sec.stepsdens} are concerned
with finding an analytical formula for $P_t$.
Sections~\ref{sec.histos} and~\ref{sec.approxfp} present numerical
approaches to approximating $P_t\rho_0$ for given initial densities
$\rho_0$.

% =====================================================================
\section{Explicit Solution Map}\label{sec.fpexplicit}

For some delay equations it is possible to find an explicit formula
for the Perron-Frobenius operator defined by
equations~\eqref{eq.augdde}
and~\eqref{eq.nonautonsg}--\eqref{eq.ddefp}.  This can be accomplished
by first finding an explicit formula for the transformation $S_t: x_0
\mapsto x(t)$, which requires that the general solution to the given
DDE be found.  Equation~\eqref{eq.ddefp} is then used to derive a
formula for $P_t$.  The following examples illustrate this procedure.

\begin{example} \label{ex.lin}
Consider the linear DDE
\begin{equation} \label{eq.fpexp1a}
  x'(t) = \alpha x(t-1), \quad t \geq 1,
\end{equation}
with the set of allowable initial functions on $[0,1]$ restricted to
constant functions, \ie,
\begin{equation} \label{eq.fpexp1b}
  x(t) = x_0, \quad t \in [0,1],
\end{equation}
with $x_0$ distributed according to a given initial density $\rho_0$.
Define the family of solution maps $\{S_t: t \geq 0\}$ by
\begin{equation}
  S_t: x_0 \mapsto x(t),
\end{equation}
where $x(t)$ is the solution
of~\eqref{eq.fpexp1a}--\eqref{eq.fpexp1b}.  Since the DDE does not
depend explicitly on $x(t)$, the method of steps (\cf\
Section~\ref{sec.methsteps} page~\pageref{sec.methsteps}) reduces to
iterating the following integral for $n=1,2,\ldots$,
\begin{equation} \label{eq.itint}
  x(t) = x(n) + \int_n^t \alpha x(s-1)\,ds, \quad t \in [n,n+1].
\end{equation}
Thus we obtain
\begin{equation}
\begin{split}
  S_t(x_0) &= \begin{cases}
    x_0   &  t \in [0,1] \\
    (\alpha t - \alpha + 1) x_0 &  t \in [1,2] \\
    (\tfrac{1}{2}\alpha^2 t^2 - 2 \alpha^2 t + \alpha t + 2 \alpha^2
       - \alpha + 1) x_0 &  t \in [2,3] \\
    \quad\vdots \\
    \beta_n(t) x_0 & t \in [n,n+1]
  \end{cases} \\
  &= \beta(t) x_0,
\end{split}
\end{equation}
where, from equation~\eqref{eq.itint}, $\beta_n(t)$ is a polynomial of
degree $n$.  Recall that the Perron-Frobenius operator corresponding
to $S_t$ is defined by
\begin{equation}\label{eq.fpexp1c}
  \int_A P_t \rho_0(x)\,dx = \int_{S_t^{-1}(A)} \rho_0(x)\,dx.
\end{equation}
Taking $A=[0,x]$ we have
\begin{equation}
  S_t^{-1}(A) = \begin{cases}
    [0,x/\beta(t)] & \text{if $\beta(t) > 0$} \\
    [x/\beta(t),0] & \text{if $\beta(t) < 0$},
  \end{cases}
\end{equation}
and equation~\eqref{eq.fpexp1c} becomes
\begin{equation}
  \int_0^x P_t \rho_0(s)\,ds = \begin{cases}
    \displaystyle \int_0^{x/\beta(t)} \rho_0(s)\,ds & \text{if $\beta(t) > 0$} \\
    \displaystyle \int_{x/\beta(t)}^0 \rho_0(s)\,ds & \text{if $\beta(t) < 0$.}
  \end{cases}
\end{equation}
Differentiating on both sides yields the explicit formula
\begin{equation} \label{eq.fpans1}
  (P_t \rho_0)(x) = \rho(x,t) = \frac{1}{|\beta(t)|} \rho_0\Big( \frac{x}{\beta(t)} \Big).
\end{equation}
Notice that $S_t$ is non-singular (hence $P_t$ is well defined) if and
only if $\beta(t)\neq 0$, which does not necessarily hold for all $t$.
In particular, for any $\alpha \leq -1/2$ there is a time
$t_{\ast}=1-1/\alpha
\in [1,2]$ at which $\beta(t_{\ast})=0$ and therefore
\begin{equation}
  S_{t_{\ast}}(x) = 0 \quad \forall x.
\end{equation}
That is, all solutions of~\eqref{eq.fpexp1a}--\eqref{eq.fpexp1b} pass
through $0$ at $t=t_{\ast}$.  In general this occurs whenever
$\beta(t)=0$.  At these times $S_t$ is singular and $P_t$ is
undefined, though it is clear that the ensemble of solutions is
described by a point mass concentrated at $x=0$.  In such cases it is
possible to give the interpretation $P_t \rho_0 \to \delta$ (the Dirac
delta function) as $\beta(t) \to 0$~(\cf\ \cite{NCV}
and~\cite[p.\,398]{LM94}).
\end{example}

In the previous example the Perron-Frobenius operator was easy to
construct because the solution map $S_t$ was one-to-one and easy to
invert.  The following example shows what happens for even slightly
more interesting DDEs, where the solution map is not necessarily
one-to-one.

\begin{example} \label{ex.quad}
Consider the DDE
\begin{equation}
  x'(t) = -x(t-1)^2, \quad t \geq 1,
\end{equation}
where again we allow only constant initial functions on $[0,1]$, so that
\begin{equation}
  x(t) = x_0, \quad t \in [0,1].
\end{equation}
With the solution map $S_t: x_0 \mapsto x(t)$ defined as before, the
method of steps yields
\begin{equation}
  S_t(x) = \begin{cases}
    x   &  t \in [0,1] \\
    x - (t-1) x^2  &  t \in [1,2] \\
    x - (t-1) x^2 + (t^2 - 4 t + 4) x^3 \\
   \quad +(-\tfrac{1}{3}t^3 + 2 t^2 - 4 t + \tfrac{8}{3}) x^4 &  t \in [2,3] \\
    \quad\vdots
  \end{cases}
\end{equation}
For $t \in [0,1]$, we have simply $P_t \rho_0 = \rho_0$ since $S_t$ is the
identity transformation.  For $t \in [1,2]$, take $A=(-\infty,x]$.
Then if $x \leq \tfrac{1}{4(t-1)}$,
\begin{equation}
  S_t^{-1}(A) = \bigg(-\infty, \frac{1-\sqrt{1-4(t-1)x}}{2(t-1)}\bigg]
      \cup      \bigg[\frac{1+\sqrt{1-4(t-1)x}}{2(t-1)}, \infty \bigg),
\end{equation}
and otherwise $S_t^{-1}(A) = \Reals$.  Differentiating with respect to
$x$ on both sides of~\eqref{eq.fpexp1c} then yields
\begin{equation} \label{eq.quadans}
\begin{split}
  (P_t \rho_0)(x) = \frac{1}{\sqrt{1-4(t-1)x}} \Big[
  & \rho_0\Big( \frac{1-\sqrt{1-4(t-1)x}}{2(t-1)} \Big) \\
  +& \rho_0\Big( \frac{1+\sqrt{1-4(t-1)x}}{2(t-1)} \Big) \Big]
\end{split}
\end{equation}
if $x \leq \tfrac{1}{4(t-1)}$, and $(P_t \rho_0)(x)=0$ otherwise.
Inverting $S_t$ becomes extremely difficult for $t \in [2,3]$, and
impossible for $t>3$ (since it would require explicit roots of a
fifth-order polynomial), so that it is not possible to derive an
explicit formula for $P_t$.
\end{example}

In each of the preceding examples, the solution map $S_t$ is a
differentiable transformation on \Reals.  For such transformations the
corresponding Perron-Frobenius operator $P_t$ can be expressed as
\begin{equation} \label{eq.fpline}
  (P_t \rho_0)(x) = \sum_{y \in S_t^{-1}\{x\}} \frac{\rho_0(y)}{|S_t'(y)|},
\end{equation}
where $S_t'(y)$ is understood to mean $\tfrac{d}{dy}S_t(y)$). Indeed,
this is frequently given as the \emph{definition} of the
Perron-Frobenius operator in studies of transformations of the real
line (see \eg~\cite{KH95,PY98}).  The results of
examples~\ref{ex.lin}--\ref{ex.quad} are in fact specific cases of
this result.

For more complicated delay equations than those considered in the
examples above, the difficulties in finding an explicit formula for
the Perron-Frobenius operator are twofold:
\begin{itemize}
\item It can be difficult to determine the general solution, and hence
the solution map $S_t$.  This was made easier in the examples by lack
of explicit dependence on $x(t)$, but in general the problem can be
difficult.
\item Determining pre-images $S_t$, as in equation~\eqref{eq.fpexp1c}
and~\eqref{eq.fpline}, can be quite difficult.
\end{itemize}
The second of these difficulties is the more imposing, especially as
$S_t$ is generally not a one-to-one transformation.  Even the simple
example~\ref{ex.quad} results in a solution map $S_t$ for which it is
impossible to find an expression for the pre-image $S_t^{-1}\{x\}$
that occurs in~\eqref{eq.fpline}.  For these reasons it is not
practical, in general, to construct Perron-Frobenius operators for
DDEs by the direct means of first constructing the solution map.

For many applications, an analytical solution of the problem will not
be possible, whereas a numerical approximation of the density
$\rho(x,t) \equiv (P_t\rho_0)(x)$ might suffice.  The following section
presents a simple method of computing such an approximation, by
directly simulating an ensemble of solutions.

% =====================================================================
\section{Ensemble Simulation}\label{sec.histos}

The simplest approach to approximating $P_t\rho_0$ for particular
initial densities $\rho_0$ is the ``brute force'' method of simulating
an actual ensemble of solutions.  That is, a large ensemble of initial
values $\{x_0^{(1)},\ldots,x_0^{(N)}\}$ is chosen at random from a
distribution with density $\rho_0$.  For each $x_0^{(i)}$ the
corresponding solution $x^{(i)}(t)=S_t(x_0^{(i)})$
of~\eqref{eq.augdde} is constructed (numerically, or by some
analytical formula).  Then for any given $t$ the density
$\rho(x,t)=(P_t\rho_0)(x)$ is approximated by a histogram of the set of
values $\{x^{(i)}(t),\ldots,x^{(N)}(t)\}$.  With reference to
Figure~\ref{fig.densidea}, this amounts to constructing a histogram of
solution values $x(t)$ plotted above a given value of $t$.

\begin{figure}
\begin{center}
\includegraphics[height=1.5in]{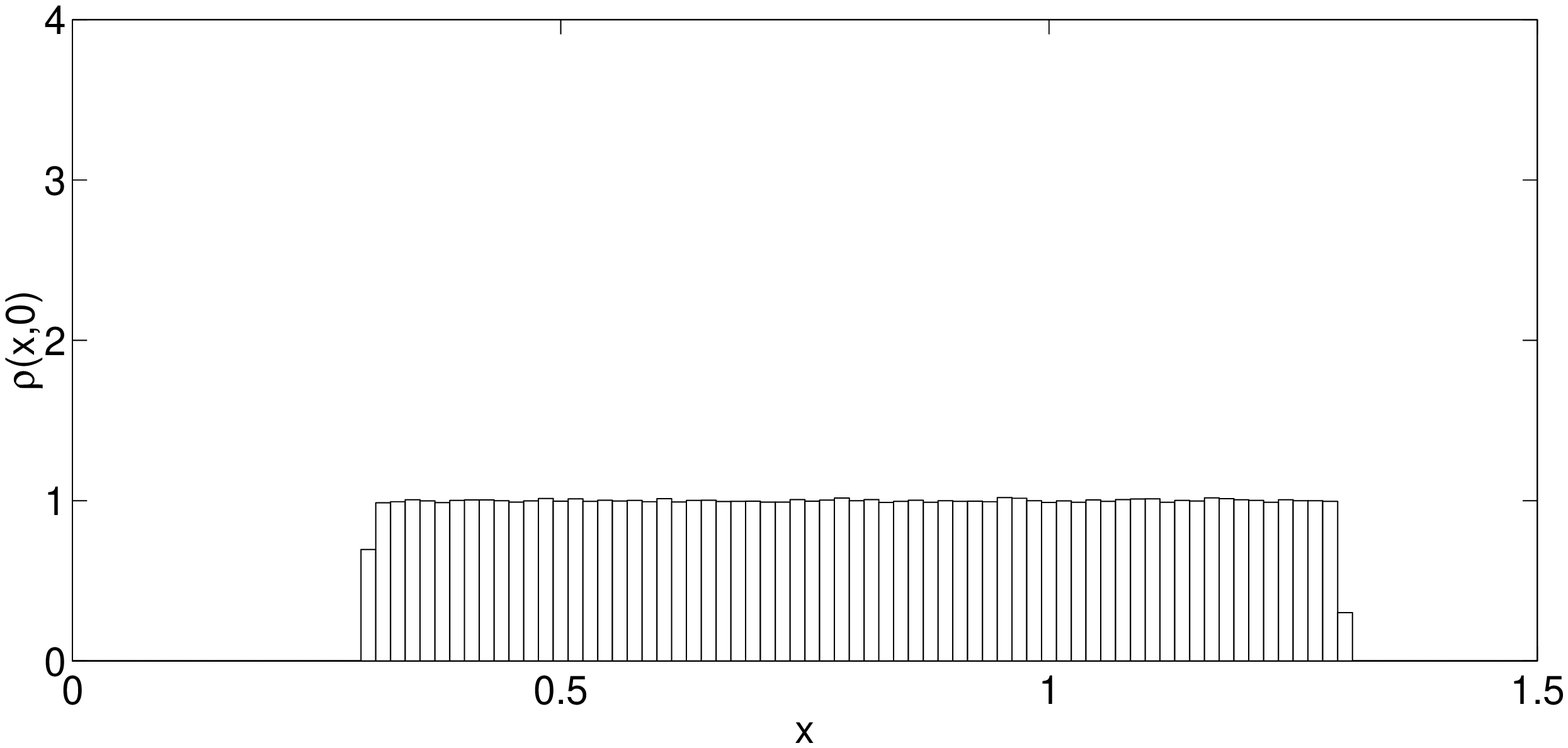} \\[0.1in]
\includegraphics[height=1.5in]{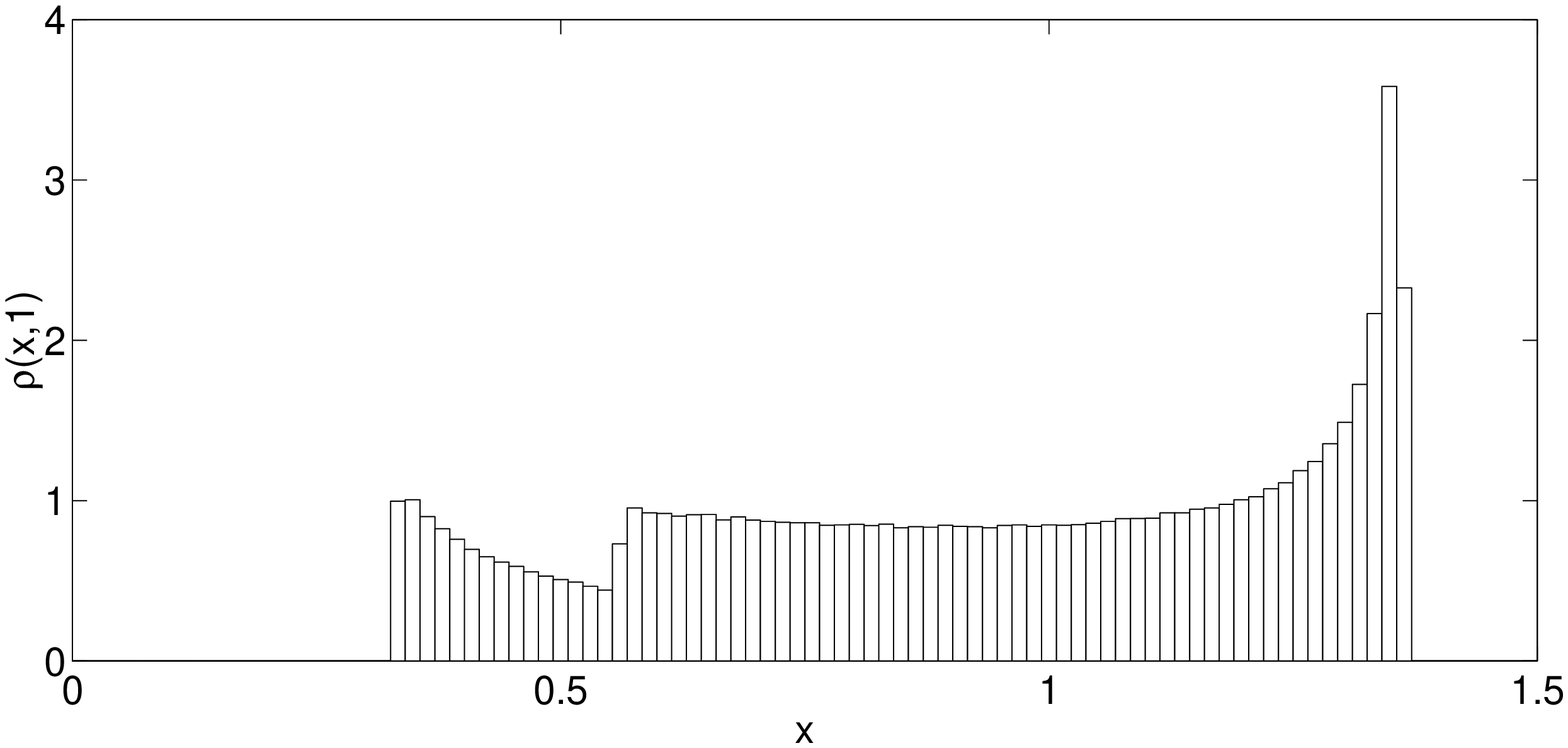} \\[0.1in]
\includegraphics[height=1.5in]{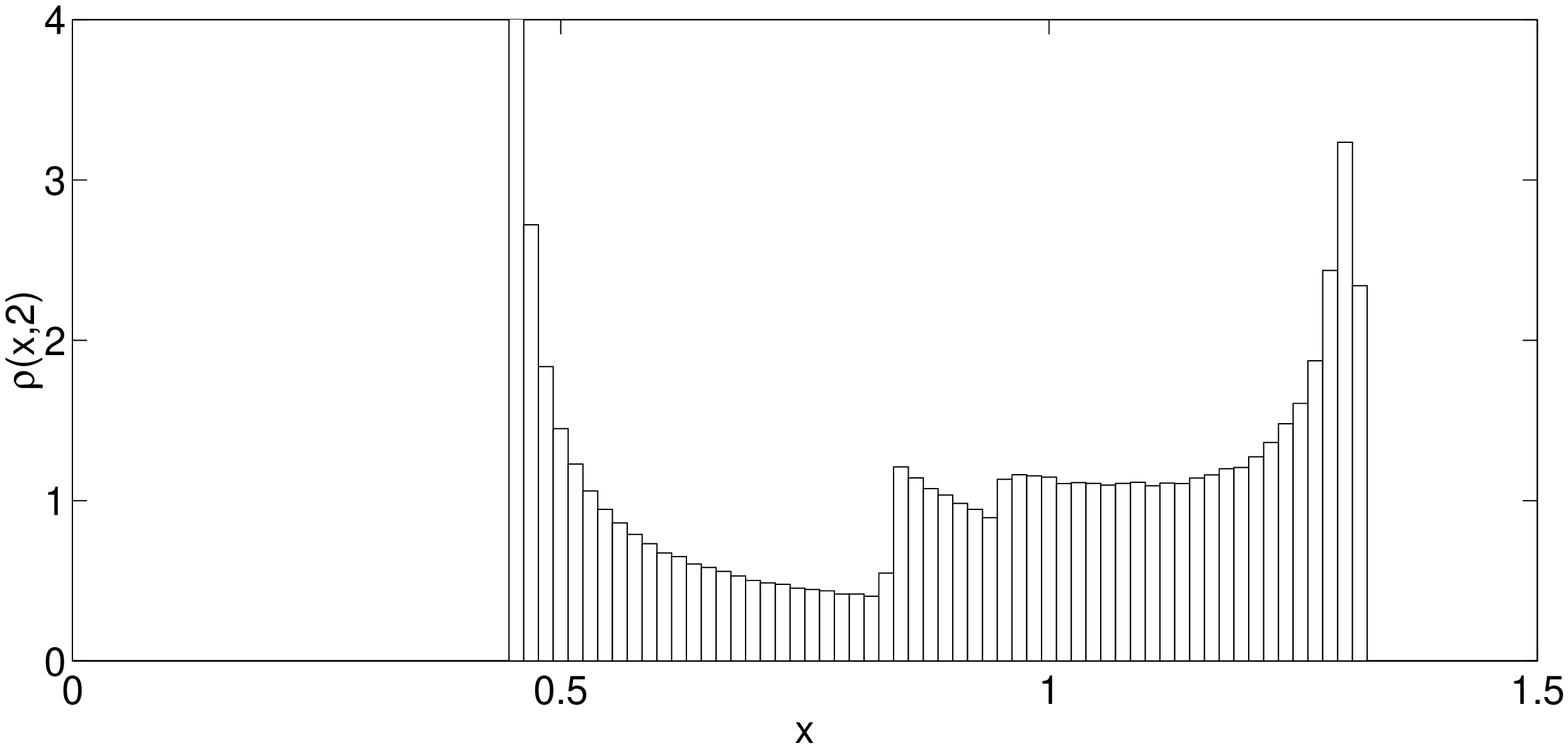} \\[0.1in]
\includegraphics[height=1.5in]{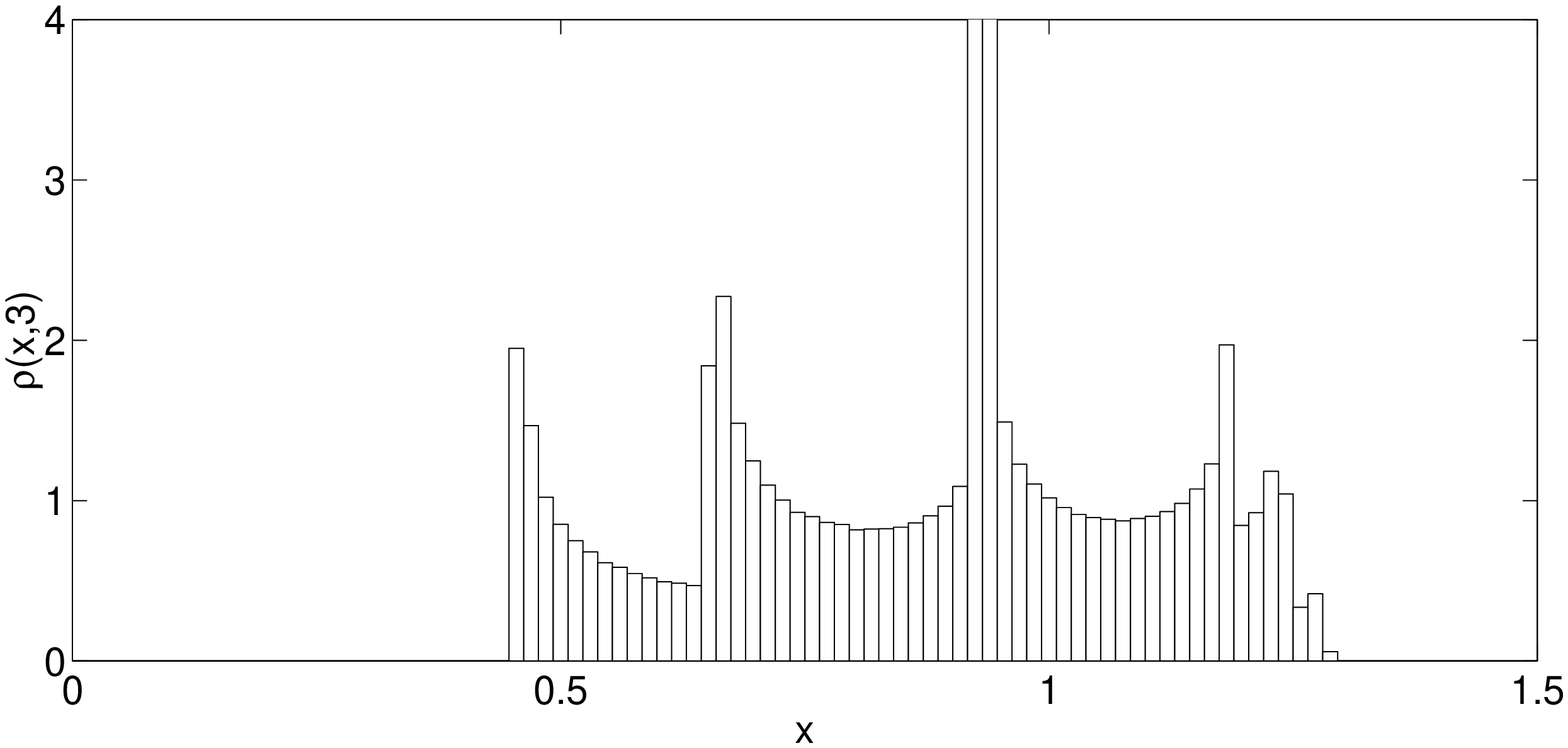}
\caption[Normalized histograms of $x(t)$, at times $t=0,1,2,3$, for an
ensemble of $10^6$ solutions of the Mackey-Glass
equation.]{Normalized histograms of $x(t)$, at times
$t=0,1,2,3$, for an ensemble of $10^6$ solutions of the Mackey-Glass
equation~\eqref{eq.mg2}.  The solutions correspond to an ensemble of
constant initial functions with values uniformly distributed on
$[0.3,1.3]$ (\cf\ Figure~\ref{fig.densidea}).}
\label{fig.densmg1}
\end{center}
\end{figure}
Figure~\ref{fig.densmg1} shows the results of such a computation
applied to the Mackey-Glass equation~\eqref{eq.mg2}.  As in
Figure~\ref{fig.densidea}, the initial ensemble consists of constant
functions (hence $g=0$ in~\eqref{eq.augdde}), with values uniformly
distributed on the interval $[0.3,1.3]$ (all initial values in this
interval are eventually attracted to the same chaotic attractor). That
is,
\begin{equation}
  \rho_0(x) = \begin{cases}
              1 & \text{if $x \in [0.3,1.3]$} \\
              0 & \text{otherwise}.
	    \end{cases}
\end{equation}
For each of $10^6$ initial values $x_0=x^{(i)}(0)$ sampled from this
distribution, an approximate solution $x^{(i)}(t)$ was computed
numerically.\footnote{Numerical solutions were computed using the
solver DDE23~\cite{ST00}.}  The sequence of graphs shown in
Figure~\ref{fig.densmg1} depict the resulting histograms of the
solution values $\{x^{(i)}(t): i = 1, \ldots, 10^6\}$ at times
$t=0,1,2,3$ (here we take the initial time for the DDE to be $t=0$, so
the initial function is specified on the interval $[-1,0]$).  The
relationship between these densities and the corresponding ensemble of
solutions shown in Figure~\ref{fig.densidea} is apparent on brief
inspection.  For example, jump discontinuities in the densities occur
at boundaries where the solutions in Figure~\ref{fig.densidea} overlay
one another, \eg, near $x=0.5$ at $t=1$.  High peaks, apparently
integrable singularities in the density, occur where the ensemble of
solutions in Figure~\ref{fig.densidea} ``folds over'' on itself,
\eg~near $x=1.3$ at $t=1$.  Some of these features are artifacts
resulting from discontinuities in the initial density, but others are
not.  See \eg\ Figure~\ref{fig.densmgb} which illustrates the
evolution of a Gaussian initial density.

\begin{figure}
\begin{center}
\includegraphics[height=1.5in]{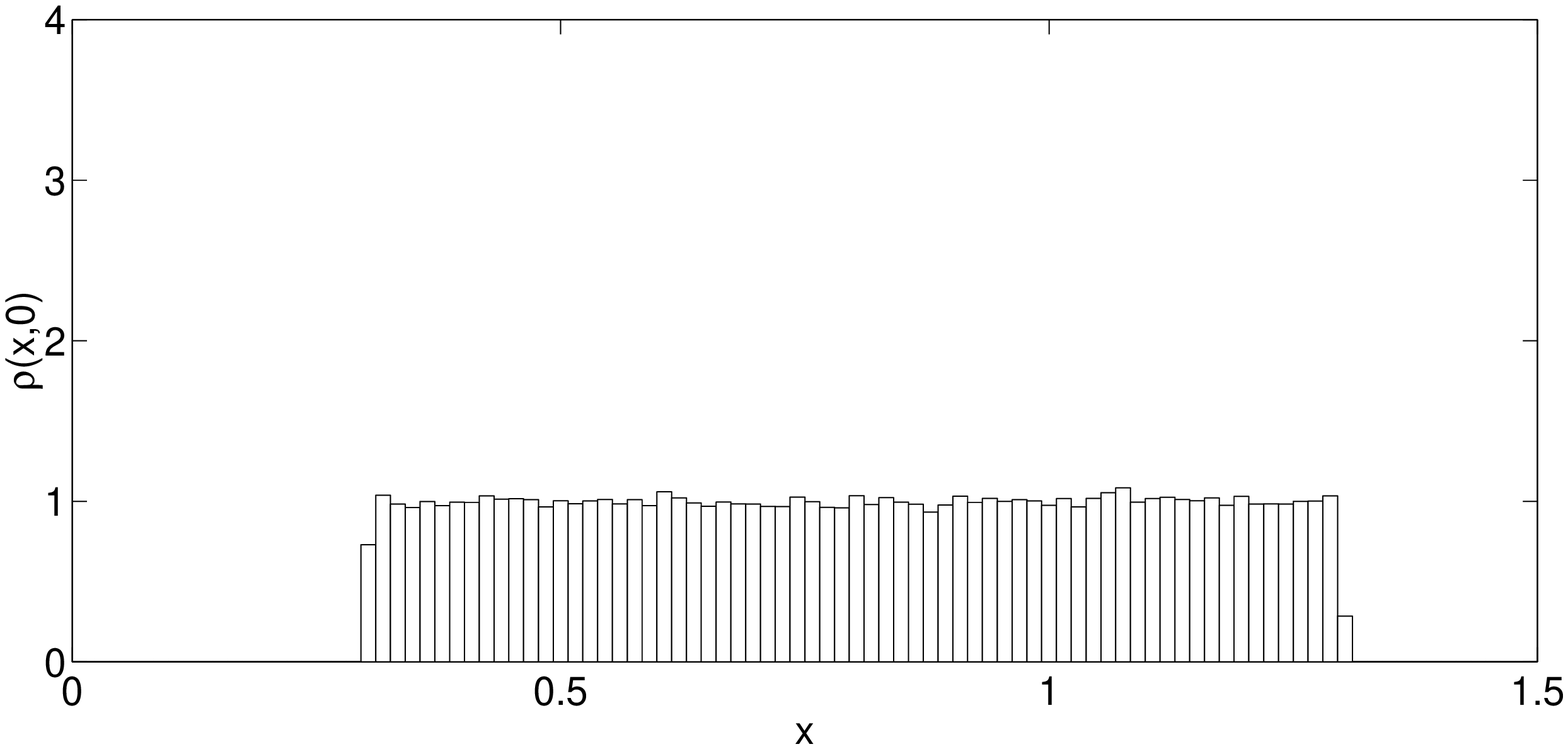} \\[0.1in]
\includegraphics[height=1.5in]{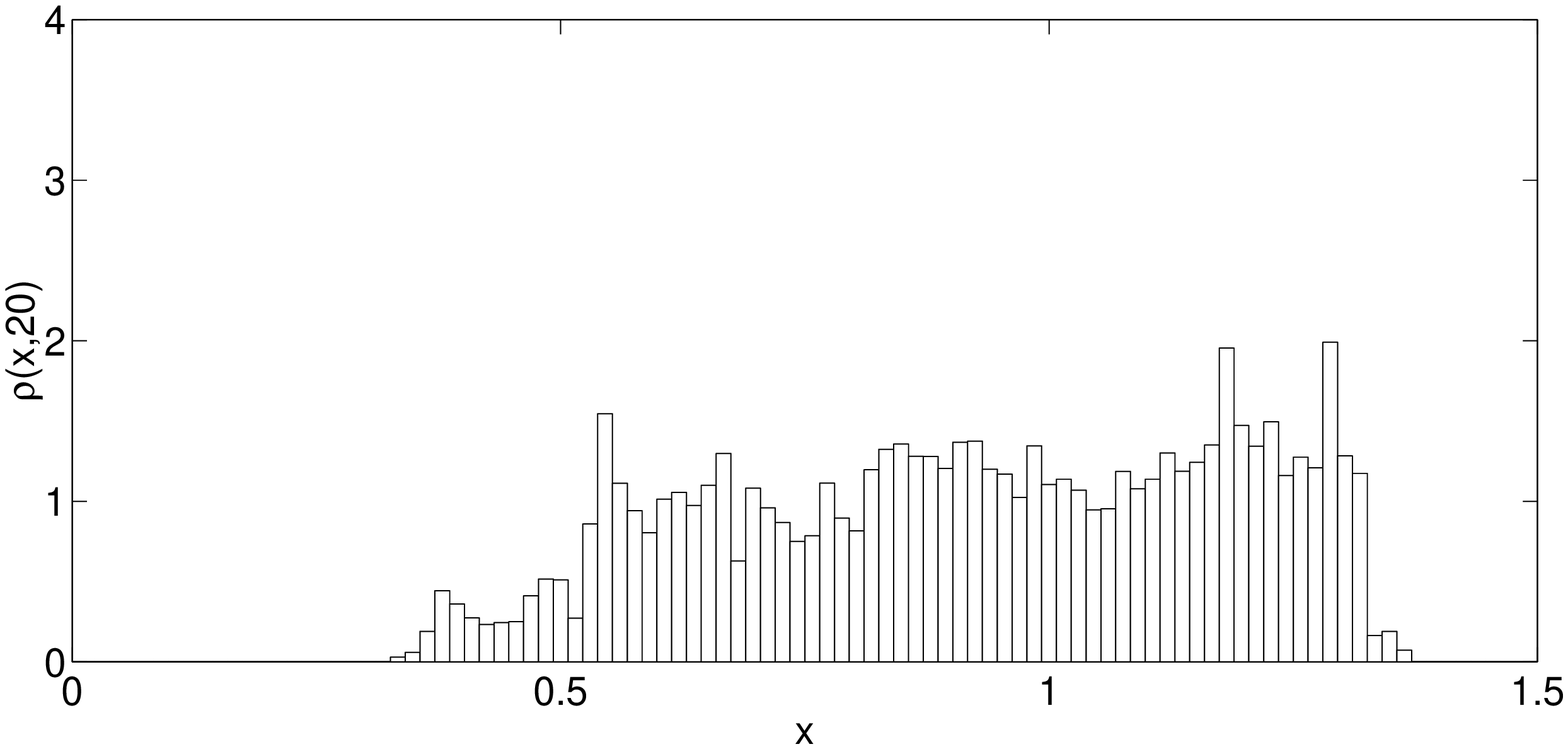} \\[0.1in]
\includegraphics[height=1.5in]{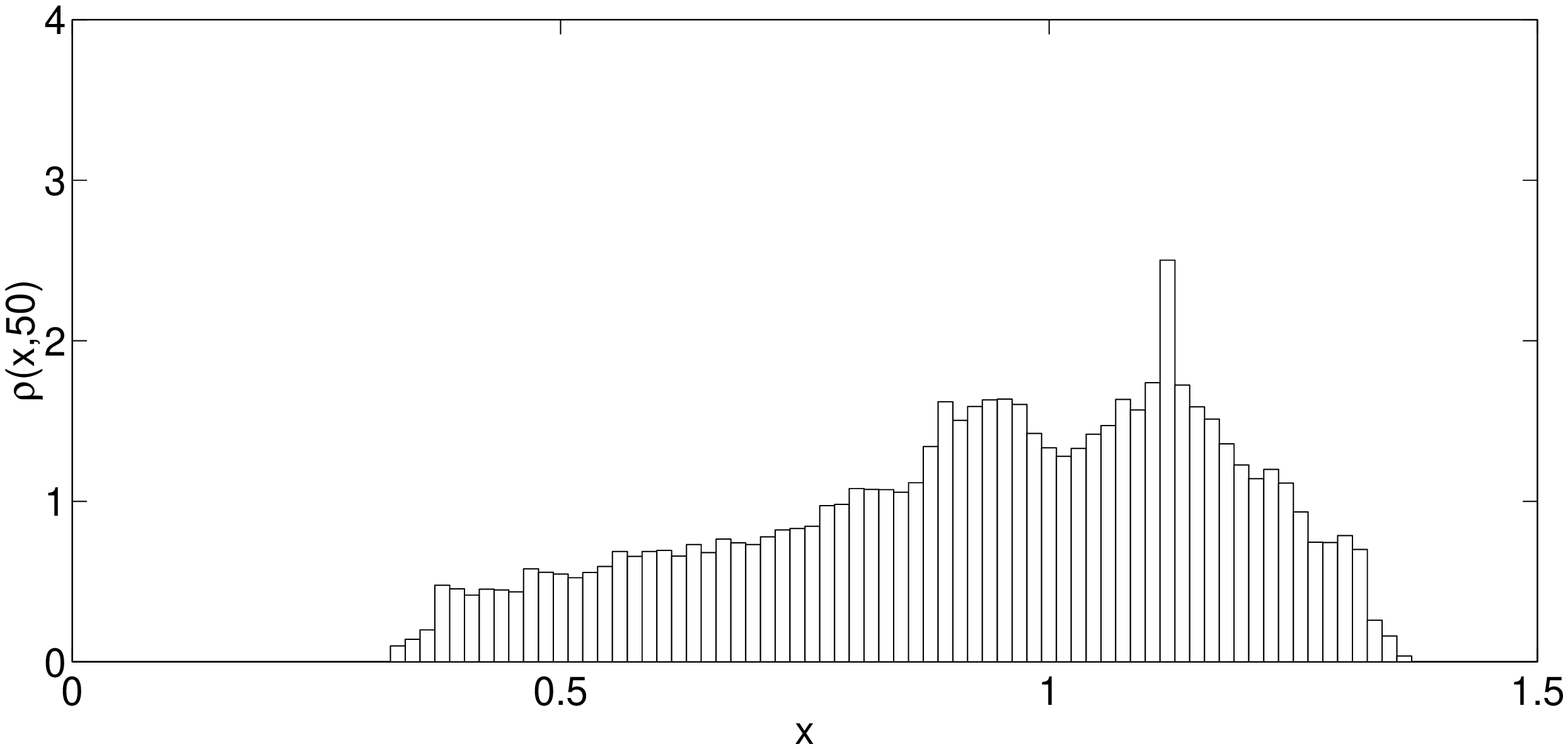} \\[0.1in]
\includegraphics[height=1.5in]{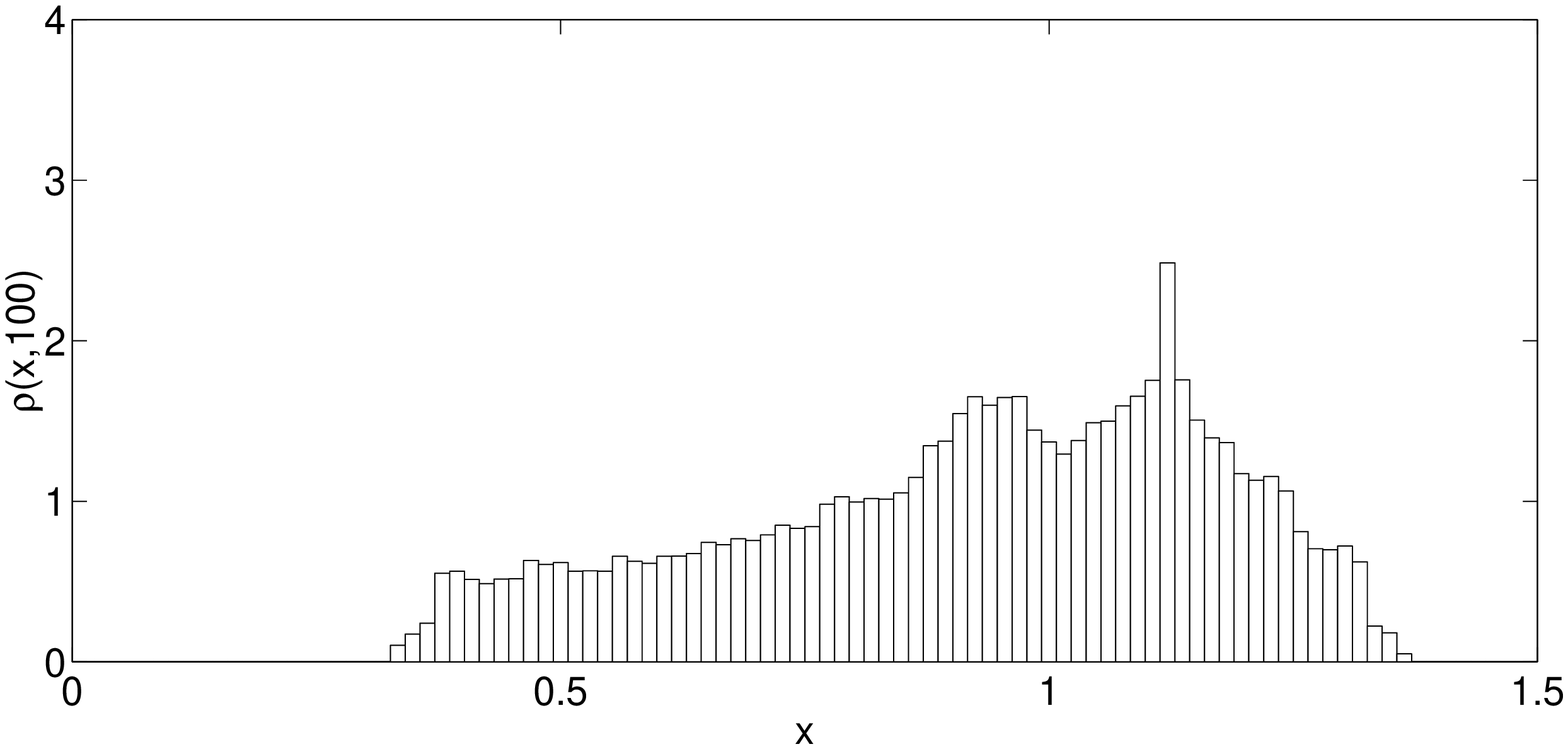}
\caption{Normalized histograms of $x(t)$, at times $t=0,20,50,100$,
for the same ensemble of solutions considered in
Figure~\ref{fig.densmg1}.}
\label{fig.densmg2}
\end{center}
\end{figure}
An interesting property of the Mackey-Glass equation is revealed when
the evolution of densities is carried to large times.
Figure~\ref{fig.densmg2} shows a sequence of histograms constructed at
times $t=0,20,50,100$, for the same ensemble considered in
Figure~\ref{fig.densmg1}.  It appears from this sequence that the
density $\rho(x,t)$ approaches a limiting density $\rho_{\ast}(x)$ as
$t \to \infty$.  That is, there appears to be an asymptotically stable
invariant density for this system.  The invariant density observed is
in fact independent of the initial density.  We will return to the
problem of characterizing such invariant densities for DDEs in
Chapter~\ref{ch.invdens}.

Note that convergence to the invariant density is relatively slow, for
example compared to maps on the interval where statistical convergence
occurs after only a few iterations of the Perron-Frobenius operator
(\cf\ Figures~\ref{fig.quadmapdensev1}--\ref{fig.quadmapdensev2}).
The behavior seen here is not typical of dynamical systems considered
elsewhere, and may have implications for the statistical mechanics of
systems with delayed dynamics, for example the neural ensemble
encoding mechanism proposed in~\cite{MM00} where rapid statistical
convergence plays an important role.

The brute force approach to densities has the tremendous advantage of
being easy to implement---it requires only a method for numerically
solving DDEs---and it is the obvious ``quick and dirty'' solution to
the problem.  However, it is a na\"{\i}ve approach, in that it
provides no insight into the process by which $\rho(x,t)$ evolves.
For example, the method offers only a heuristic explanation of the
discontinuities and singularities that appear in
Figure~\ref{fig.densmg1}.  Moreover, as shown below, constructing an
accurate histogram can require millions of samples $\{x^{(i)}(t)\}$,
hence millions of solutions of the DDE must be computed.  Especially
when approximating the evolution of densities for large $t$, the
amount of computation necessary can render the method practically
useless.

% -------------------------------------------------------------------
\subsubsection{Sampling requirements} \label{sec.sampling}

In order for the histogram of a solution ensemble
$\{x^{(1)}(t),\ldots,x^{(N)}(t)\}$ to provide an accurate
approximation of the actual density $\rho(x,t)$, the ensemble must be
sufficiently large.  Suppose an interval $B \subset \Reals$ is one of
the histogram bins.  The height of the histogram on $B$ is given by
\begin{equation}
  y = \#\{x^{(i)} \in B\},
\end{equation}
\ie, the number of the $x^{(i)}$ that lie in $B$.  The $x^{(i)}$ are
independent samples from a distribution with density $\rho$, so for
sufficiently large $n$ the fraction
\begin{equation}\label{eq.pest}
  \frac{y}{N} = \frac{1}{N} \sum_{i=1}^{N} 1_B(x_i)
\end{equation}
estimates (by the weak law of large numbers) the probability
\begin{equation}
  p=\int_B{\rho(x,t)\,dx}
\end{equation}
that a random number $x$ selected from this distribution will lie in
$B$.

The random variable $y$ takes integer values between $0$ and $N$, with
binomial probability distribution
\begin{equation}
  P(y) = \frac{N!}{y! (N-y)!} p^y (1-p)^{N-y}.
\end{equation}
For sufficiently large $N$, $P(y)$ can be approximated by a Gaussian
density with mean $Np$ and standard deviation $\sqrt{Np(1-p)}$.  Thus
the quantity $y/N$ (equation~\eqref{eq.pest}) will be distributed with
mean $p$ and standard deviation
\begin{equation} \label{eq.stdev}
  \sigma = \sqrt{\frac{p (1-p)}{N}}.
\end{equation}

Equation~\eqref{eq.stdev} predicts $O(1/\sqrt{N})$ convergence of the
``sample mean'' $y/N$ to the ``population mean'' $p$---a standard
result in sampling and measurement theory~\cite[p.\,36]{Baird}.
Suppose we wish $y/N$ to approximate $p$ within fractional error
$\delta$, with $95\%$ confidence.  Then we require that $2\sigma <
\delta p$, yielding (via equation \eqref{eq.stdev}) the sampling
requirement
\begin{equation}
  N > \frac{4(1-p)}{\delta^2 p}.
\end{equation}
Thus, for a moderately high-resolution histogram (say, with $100$
bins, so $p$ is on the order $10^{-2}$), $95\%$ confidence of accuracy
within fractional error $\delta=10^{-2}$ would require a sample of
size
\begin{equation}
  N \gtrsim \frac{4}{(10^{-2})^2 10^{-2}} = 4 \times 10^6.
\end{equation}

For some delay equations, and particularly when densities are to be
obtained for large $t$, this sample size requirement entails a
prohibitive computational cost (for example, constructing
Figure~\ref{fig.densmg2} required about 10 hours of computer time).
In such situations, the brute force approach to density evolution
becomes impractical.  This motivates the following section, which
develops a more efficient numerical method for computing the evolution
of densities for DDEs.

% =====================================================================
\section{Approximate Solution Map}\label{sec.approxfp}

The brute force approach to approximating densities is computationally
expensive, so a more efficient numerical method is desirable.
Developing such a method is the aim of the present section.  To
simplify the development, the method is presented only in the context
of one-dimensional delay equations (\ie, with solution variable
$x(t)\in\Reals$).  The generalization to higher dimensions is
straightforward, but requires more elaborate notation.

% ----------------------------------------------------------------------
\subsection{Approximate Perron-Frobenius operator}

Consider the DDE initial value problem~\eqref{eq.augdde} for
$x(t)\in\Reals$, with solution map $S_t: x_0 \mapsto x(t)$, and
suppose an ensemble of initial values $x_0$ is specified with density
$\rho_0$.  Since $S_t$ is a continuous transformation of \Reals, it can
be approximated by a piecewise linear function.  Thus, suppose
$I=(a,b)$ is an interval containing the support of $\rho_0$, and define
a mesh of points $a=x_0 < x_1 < \cdots < x_k=b$ spanning $I$.  Let
$\tilde{S}_t:I\to\Reals$ be the transformation whose graph is a
straight line on each interval $(x_i,x_{i+1})$, and satisfies
\begin{equation}
  \tilde{S}_t(x_i) = S_t(x_i) \equiv y_i, \quad i=1,\ldots,k.
\end{equation}
Then $\tilde{S}_t$ furnishes a piecewise linear approximation of
$S_t$, and agrees with $S_t$ at each of the $x_i$.

Because $\tilde{S}_t$ is piecewise linear, it is almost-everywhere
differentiable.  Therefore the corresponding Perron-Frobenius operator
$\tilde{P}_t$ can be expressed as~\cite[Ch.\,12]{PY98}
\begin{equation} \label{eq.pfpwlin}
  (\tilde{P}_t \rho_0)(x) = \sum_{z \in \tilde{S}_t^{-1}\{x\}}
                          \frac{\rho_0(z)}{|\tilde{S}_t'(z)|},
\end{equation}
Since only those pre-images of $\{x\}$ that lie in the support of
$\rho_0$ give a non-zero contribution to the sum, we need consider only
those $z \in \tilde{S}_t^{-1}\{x\}$ that lie in some interval
$[x_i,x_{i+1})$.  On each such interval we have simply
\begin{equation}
  \tilde{S}_t'(z) = \frac{y_{i+1}-y_i}{x_{i+1}-x_i}.
\end{equation}
Furthermore, since $\tilde{S}_t$ is piecewise linear each $z \in
\tilde{S}_t^{-1}\{x\}$ can be found by linear interpolation.
Thus, for each interval $[y_i,y_{i+1})$ that contains $x$, there is
exactly one element $z \in \tilde{S}_t^{-1}\{x\}$, given by
\begin{equation} \label{eq.zdef}
  z = x_i + \frac{x_{i+1}-x_i}{y_{i+1}-y_i}(x-y_i).
\end{equation}
Figure~\ref{fig.pfpwlin} illustrates this procedure for determining
the set of pre-images of $\{x\}$ under a piecewise linear
transformation.
\begin{figure}
\begin{center}
% RT 4.9.2019: use standalone figure for submission to arxiv
%\psfrag{x1}[][]{$\hat{x}_1$}
%\psfrag{x2}[][]{$\hat{x}_2$}
%\psfrag{x3}[][]{$\hat{x}_3$}
%\psfrag{xy1}{$(x_1,y_1)$}
%\psfrag{xy2}{$(x_2,y_2)$}
%\psfrag{xy3}{$(x_3,y_3)$}
%\psfrag{y1}[][]{$\hat{y}$}
%\psfrag{eq}{$y=\tilde{S}(x)$}
%\psfrag{ylab}[][]{$y$}
%\psfrag{xlab}{$x$}
%\psfrag{Sinv}[][]{$\tilde{S}^{-1}\{\hat{y}\} = \{ \hat{x}_1, \hat{x}_2, \hat{x}_3 \}$}
%\includegraphics[width=4in]{pfpwlin}
\includegraphics[width=4in]{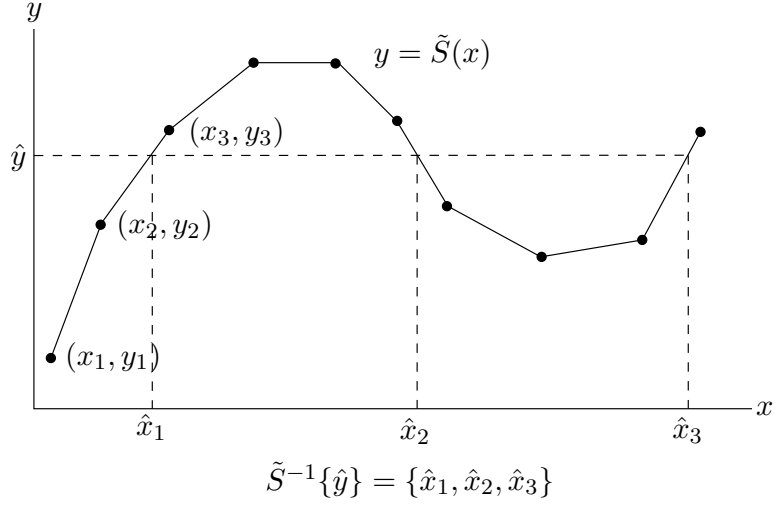}
\caption{Construction of the set of pre-images of a point $\hat{y}$ under
a piecewise linear transformation $\tilde{S}$.}
\label{fig.pfpwlin}
\end{center}
\end{figure}

% ----------------------------------------------------------------------------
\subsection{Algorithm}

The considerations above suggest the following algorithm for computing
an approximation of the transformed density $\rho(x,t)=(P_t\rho_0)(x)$.
\begin{enumerate} \label{densalg}
\item Specify a grid of closely spaced points $x_1 < x_2 < \cdots <
x_k \in \Reals$, such that the interval $(x_1,x_k)$ contains the
support of $\rho_0$.
\item Compute (\eg, by numerical solution of~\eqref{eq.augdde}) 
the sequence of values $\{y_i=S_t(x_i): i=1,\ldots,k\}$.
\item \label{densalg2}Specify a grid of points $\tilde{x}_1 <
\tilde{x}_2 < \cdots < \tilde{x}_p  \in \Reals$, at which the density
$\rho_i=\rho_0(\tilde{x}_i,t)$ is to be approximated.
\item Initialize $\rho_i=0$ for $i=1,\ldots,p$.
\item For each $i\in\{1,2,\ldots,(k-1)\}$ determine which, if any, of
the $\tilde{x}_j$ lie in the interval $[y_i,y_{i+1})$ (or the interval
$(y_{i+1},y_i]$ if $y_{i+1}<y_i$).  For each such $\tilde{x}_j$,
compute
\begin{equation}
  z = x_i + \frac{x_{i+1}-x_i}{y_{i+1}-y_i}(\tilde{x}_j - y_i),
\end{equation}
and increment $\rho_j$ by the quantity
\begin{equation}
  \rho_0(z) \frac{|x_{i+1}-x_i|}{|y_{i+1}-y_i|}.
\end{equation}
\end{enumerate}

This algorithm constructs a vector $(\rho_i: i=1,\ldots,p)$
approximating the density $\rho(x,t)$ at points $\{\tilde{x}_i\}$.
Steps 1--2 define the piecewise linear approximation $\tilde{S}_t$.
Steps 3--5 evaluate~\eqref{eq.pfpwlin} at each of the points
$\tilde{x}_i$.  Notice that steps 1--2 are decoupled from 3--5 in that
the initial density $\rho_0$ enters only in steps 3--5, after the
approximating transformation $\tilde{S}_t$ has already been
determined.

\subsubsection{Example}

Figures~\ref{fig.mgmap}--\ref{fig.denscompare} illustrate the results
of applying this algorithm to the Mackey-Glass
equation~\eqref{eq.mg2}.  As before, the equation is restricted to
constant initial functions (hence $g=0$ in~\eqref{eq.augdde}), and the
initial density $\rho_0$ corresponds to an ensemble of initial values
uniformly distributed on the interval $[0.3,1.3]$.
Figure~\ref{fig.mgmap} shows graphs of the approximating
transformation $\tilde{S}_t\approx S_t$, at times $t=1,2,3,4$.  These
were obtained by using the numerical solver DDE23~\cite{ST00} to
compute values $y_i=S_t(x_i)$ for a uniform grid of 1000 initial
values $x_i$ in the interval $[0,1.5]$.
\begin{figure}
\begin{center}
\includegraphics[height=1.6in]{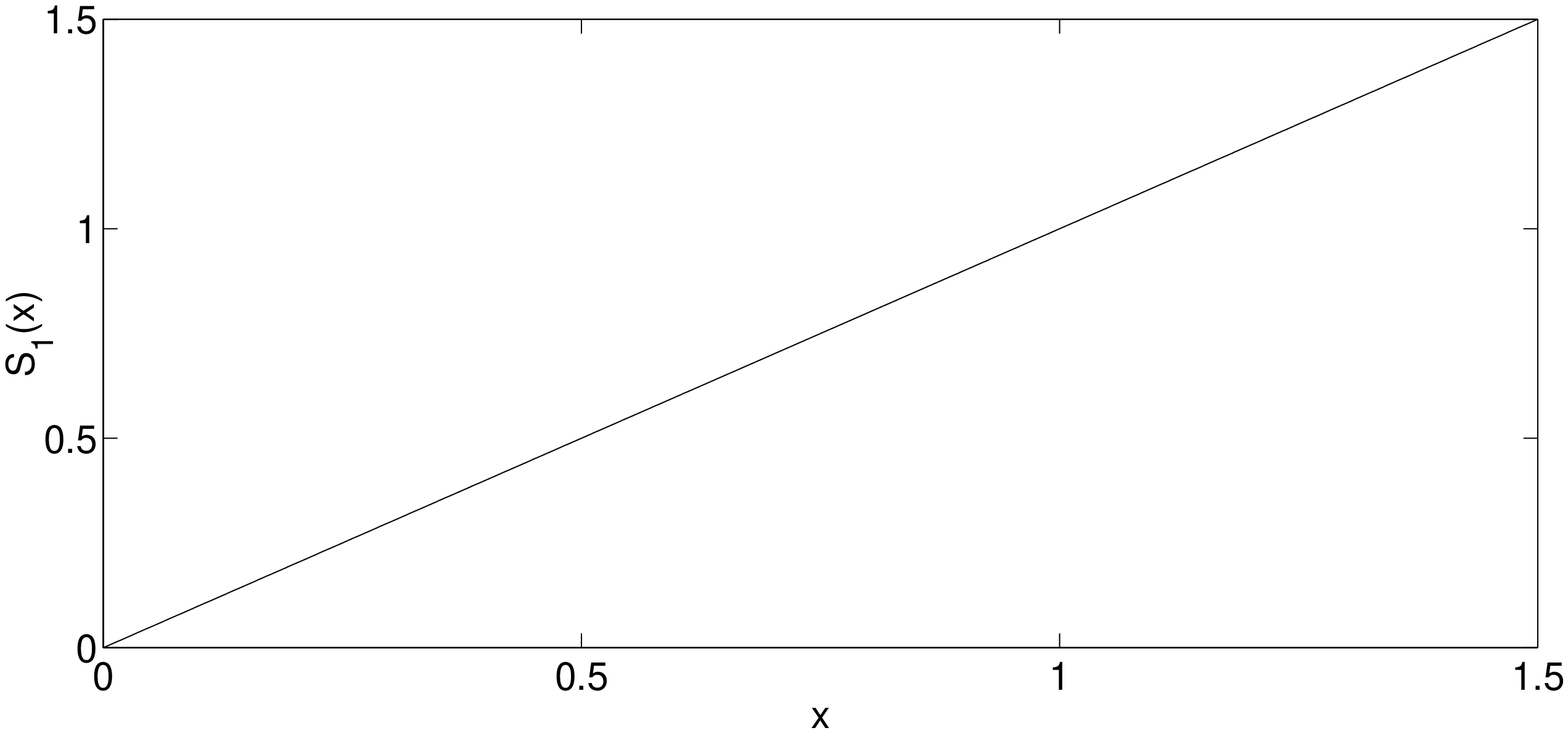} \\[0.1in]
\includegraphics[height=1.6in]{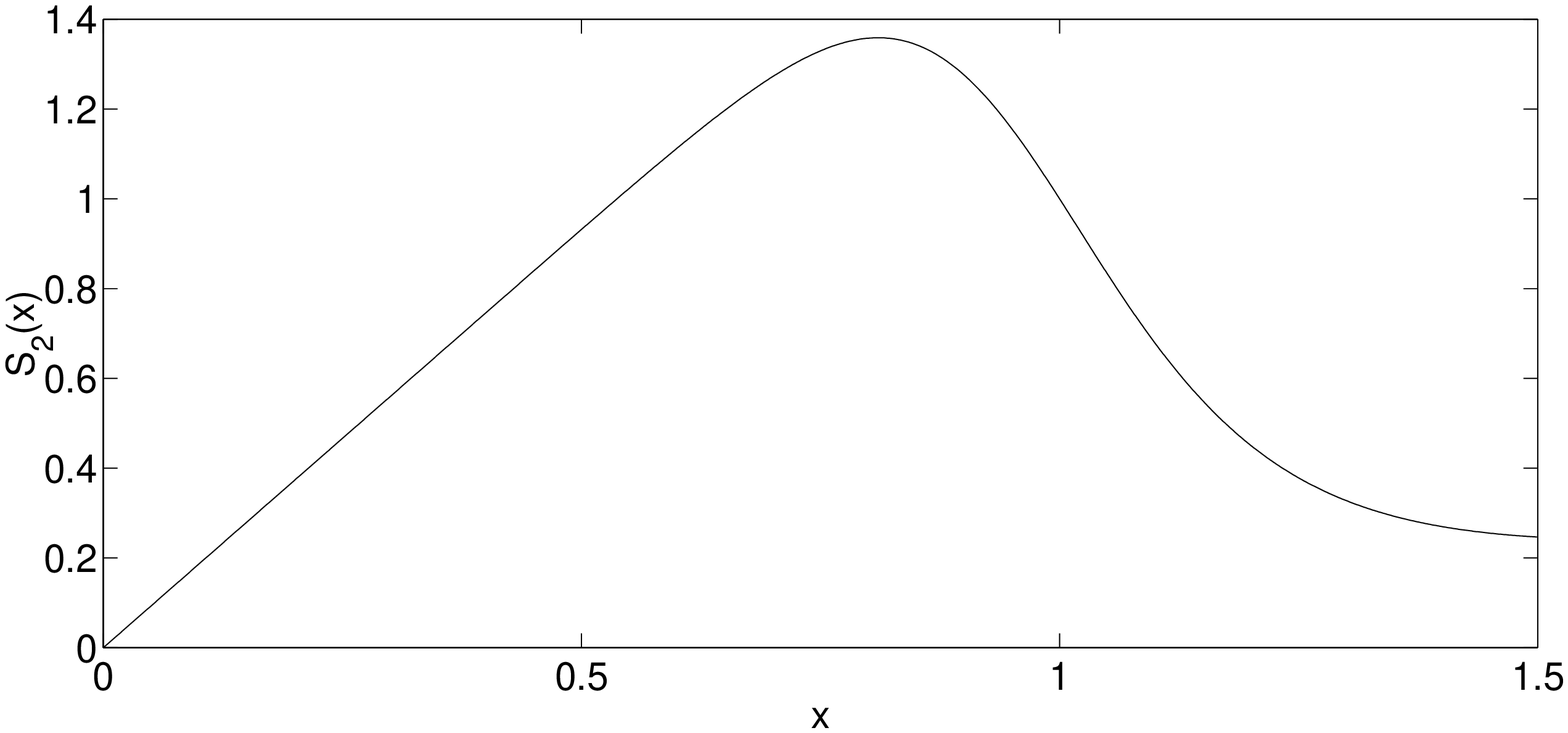} \\[0.1in]
\includegraphics[height=1.6in]{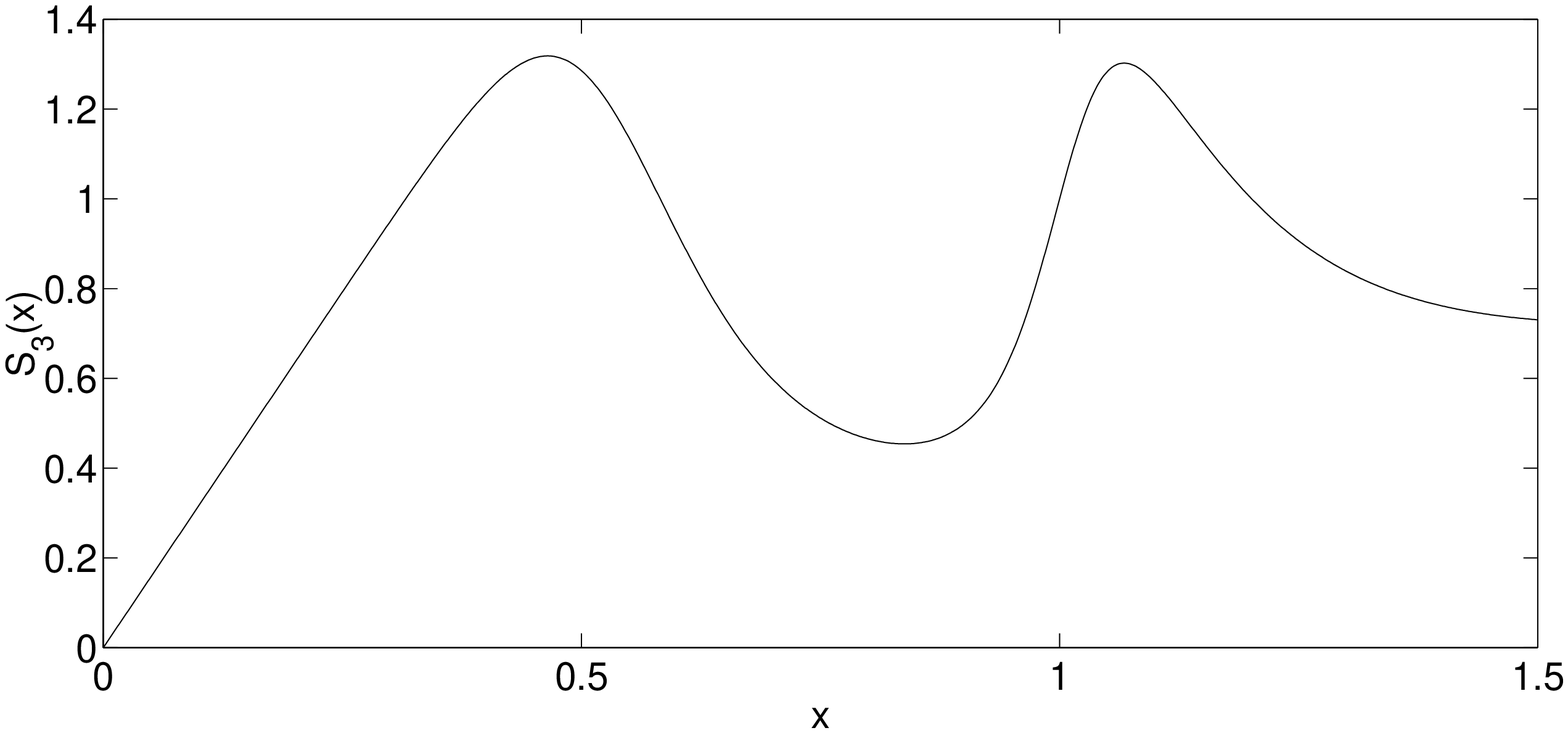} \\[0.1in]
\includegraphics[height=1.6in]{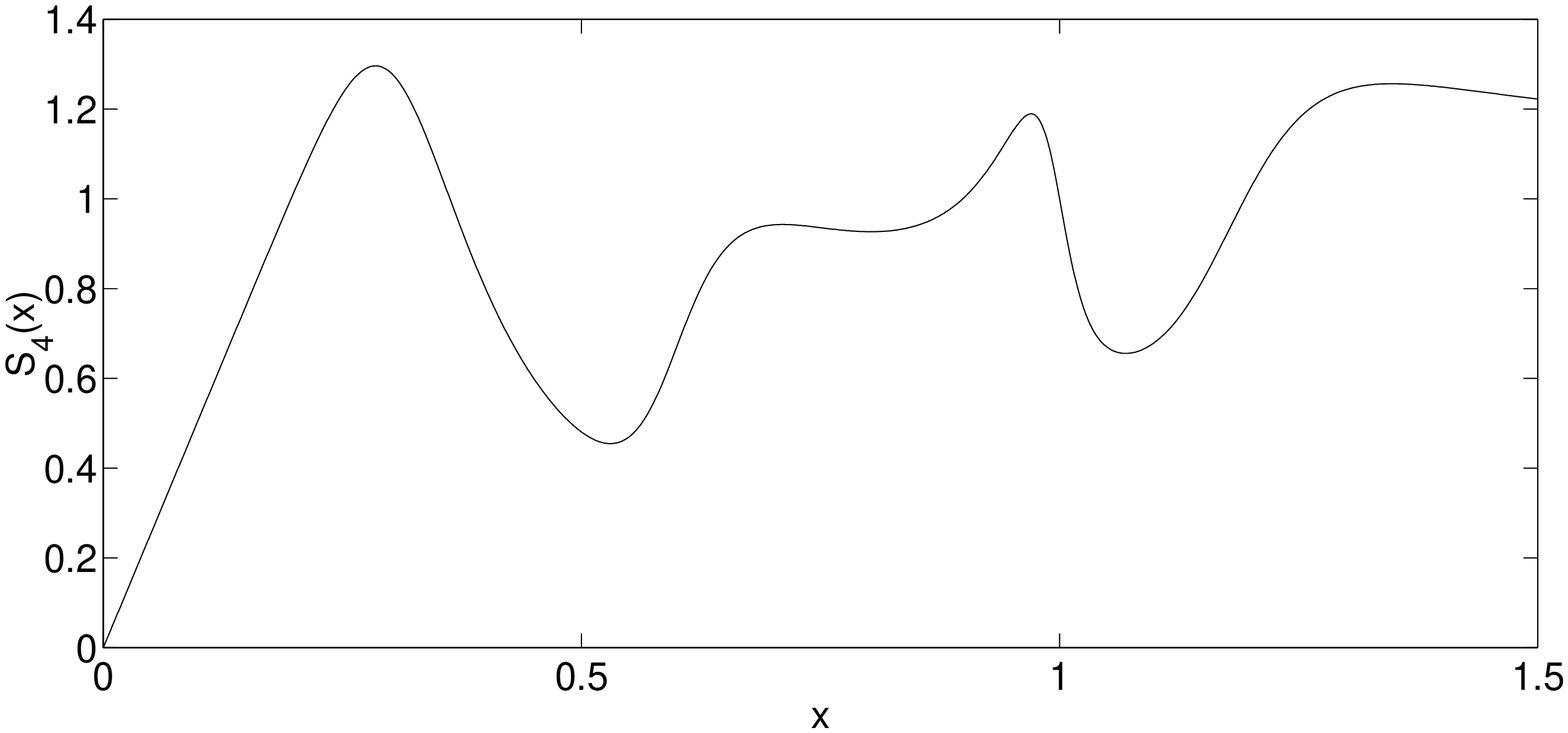}
\caption[Approximate solution maps for the Mackey-Glass equation restricted
to constant initial functions.]{Approximate solution maps $\tilde{S}_t
\approx S_t$ at times $t=1,2,3,4$, for the Mackey-Glass
equation~\eqref{eq.mg2} restricted to constant initial functions on
$[0,1]$.}
\label{fig.mgmap}
\end{center}
\end{figure}

\begin{figure}
\begin{center}
\includegraphics[height=1.5in]{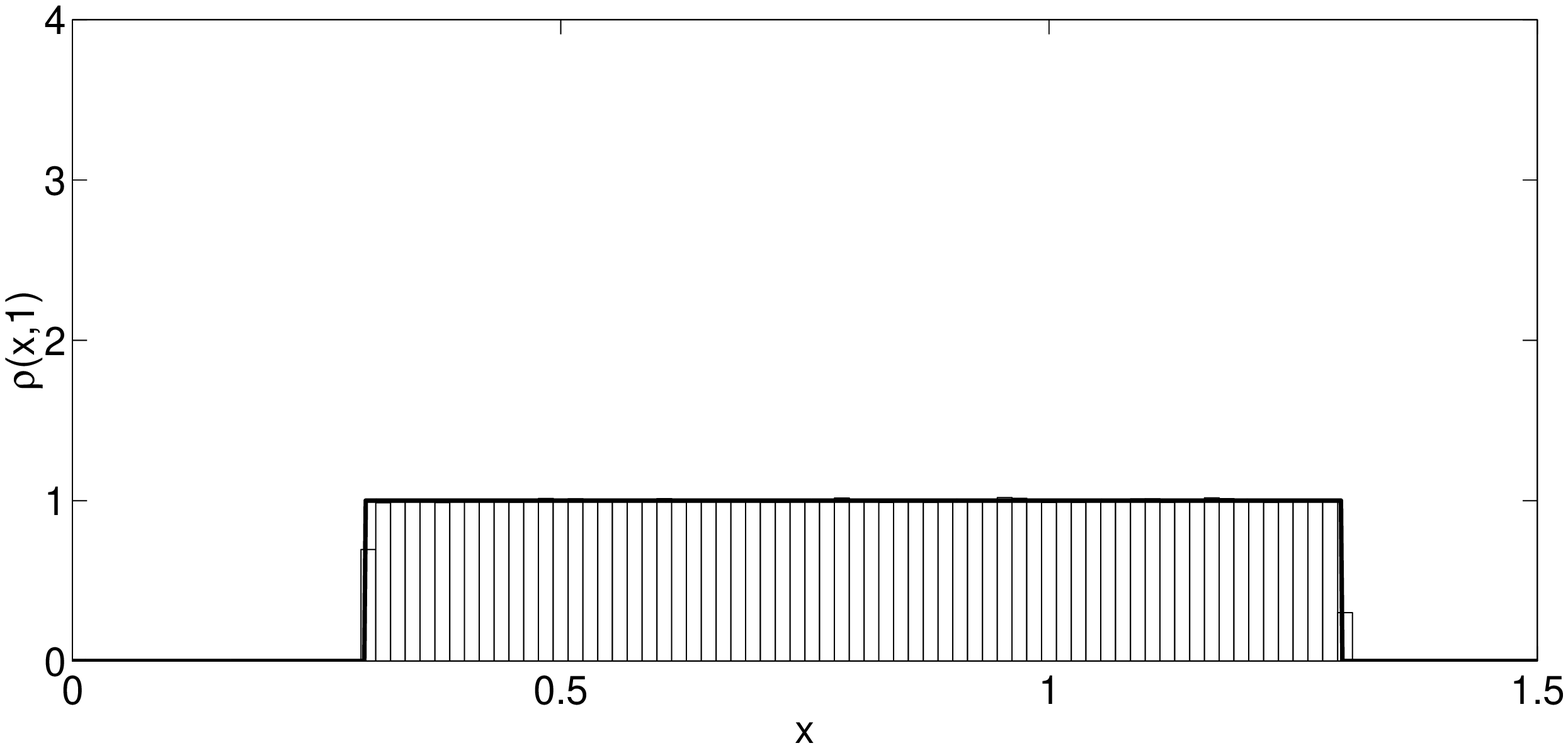} \\[0.1in]
\includegraphics[height=1.5in]{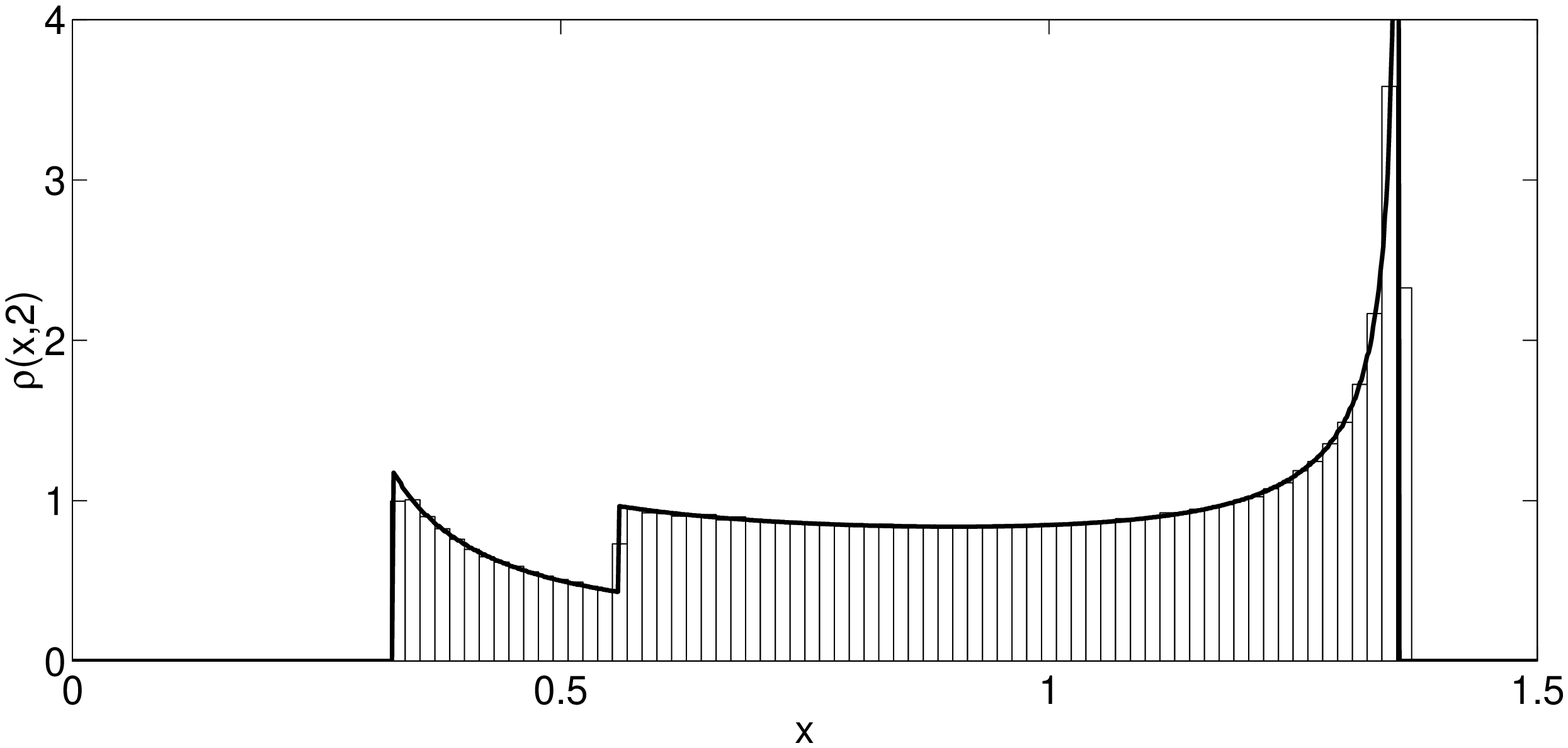} \\[0.1in]
\includegraphics[height=1.5in]{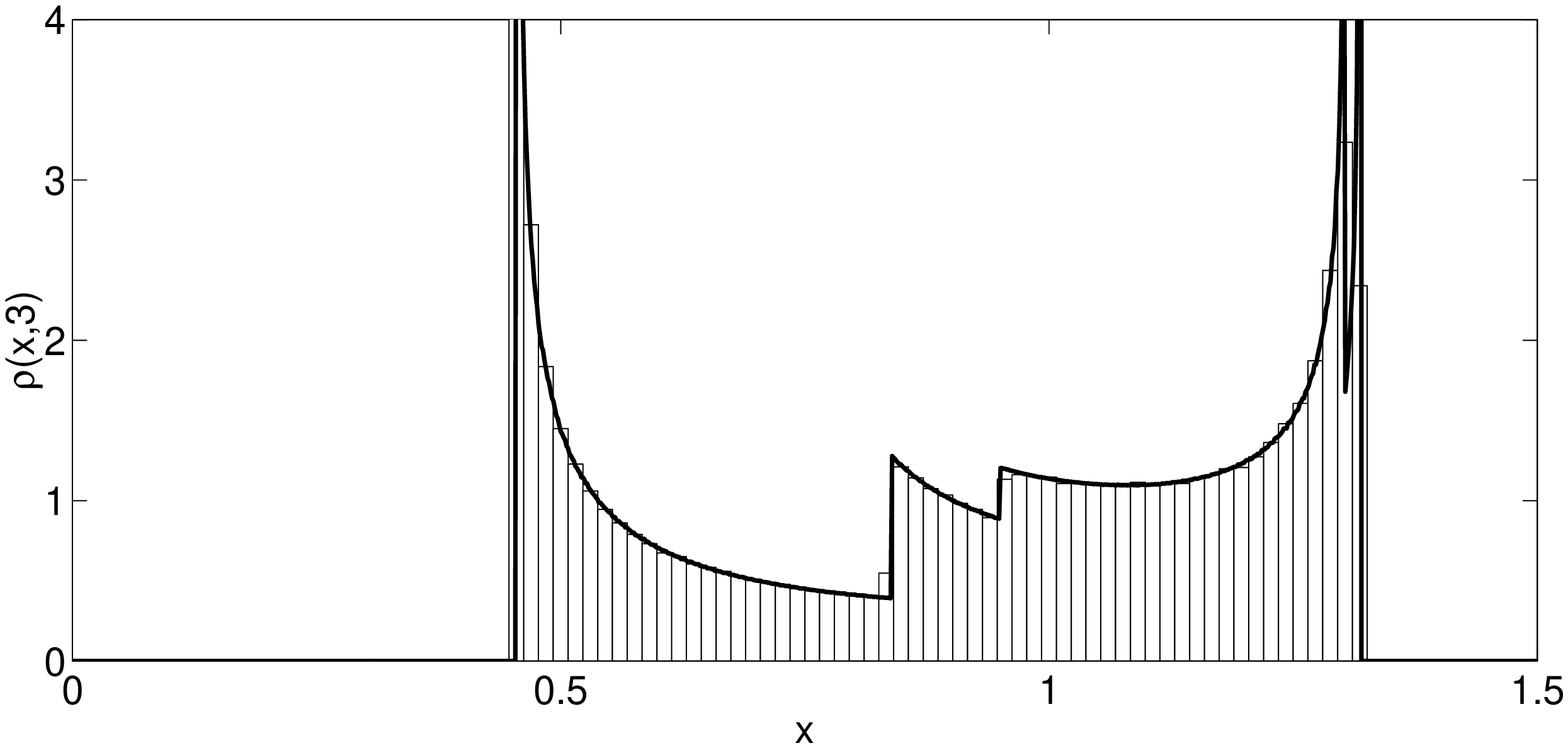} \\[0.1in]
\includegraphics[height=1.5in]{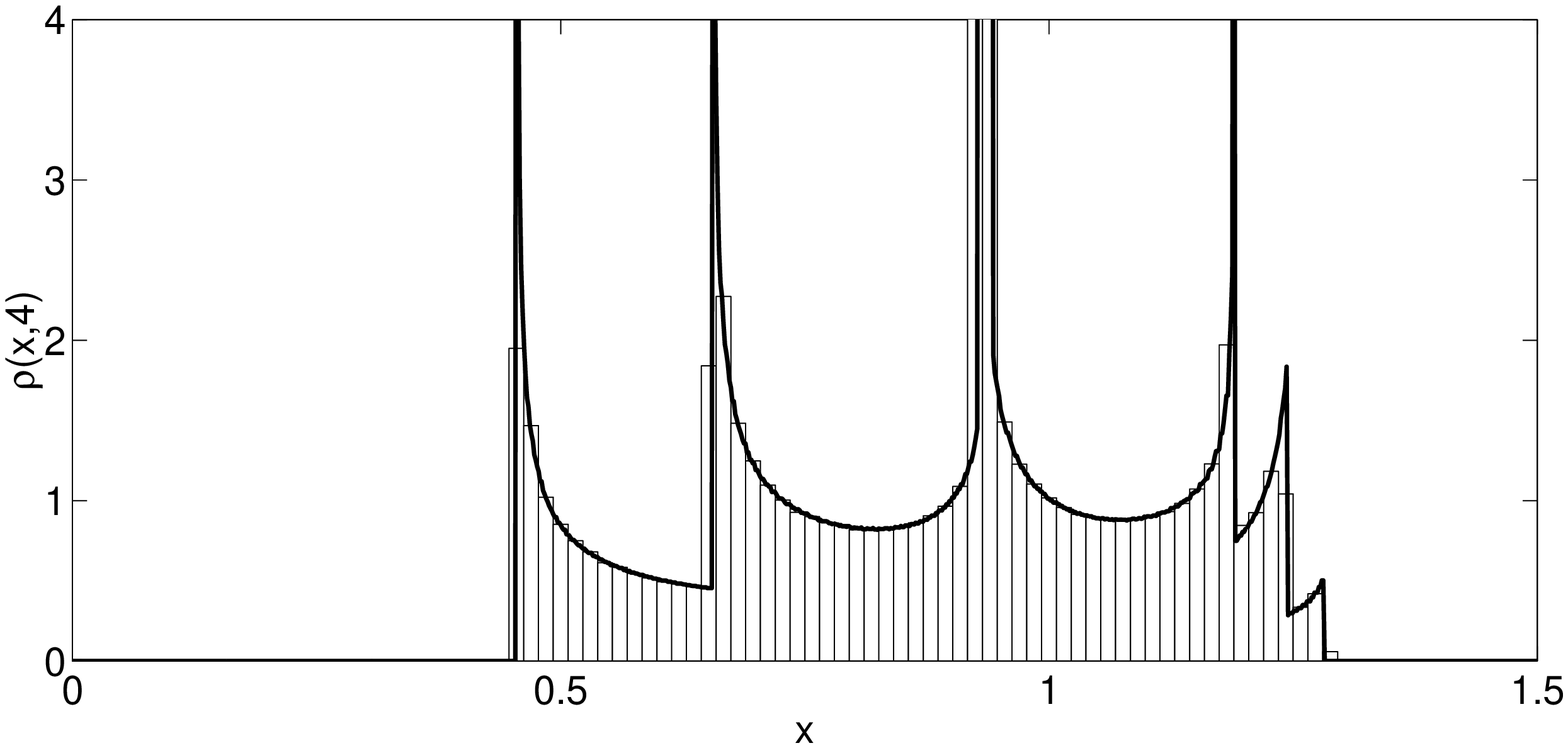}
\caption[Computed densities $\rho(x,t)$ for the Mackey-Glass
equation restricted to constant initial functions.  Densities were
computed using the algorithm on page~\pageref{densalg}.]{Computed
densities $\rho(x,t)$ for the Mackey-Glass equation~\eqref{eq.mg2}
restricted to constant initial functions.  Densities were computed
using the algorithm on page~\pageref{densalg}, and are shown (heavy
curves) together with the corresponding histograms from
Figure~\ref{fig.densmg1}, obtained by ``brute force'' ensemble
simulation.}
\label{fig.denscompare}
\end{center}
\end{figure}

\begin{figure}
\begin{center}
\includegraphics[height=1.24in]{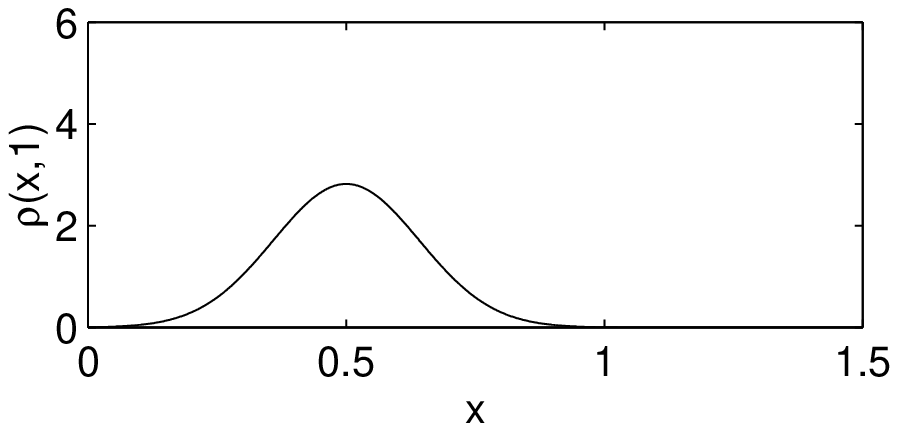} \\[0.1in]
\includegraphics[height=1.24in]{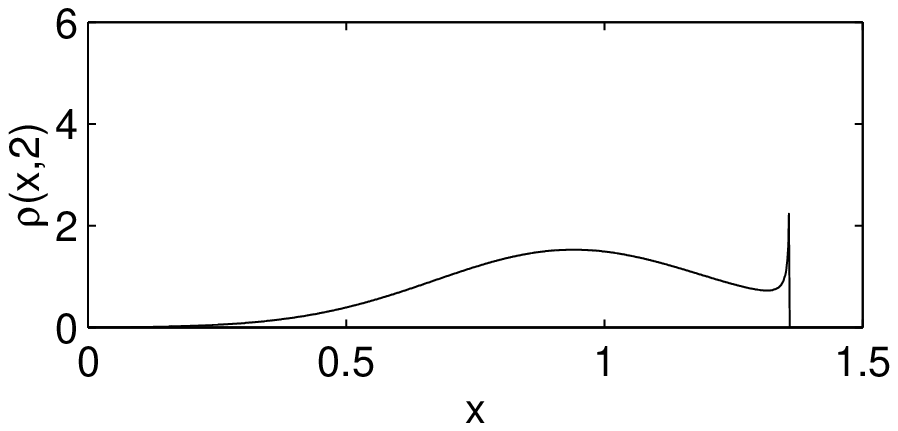} \\[0.1in]
\includegraphics[height=1.24in]{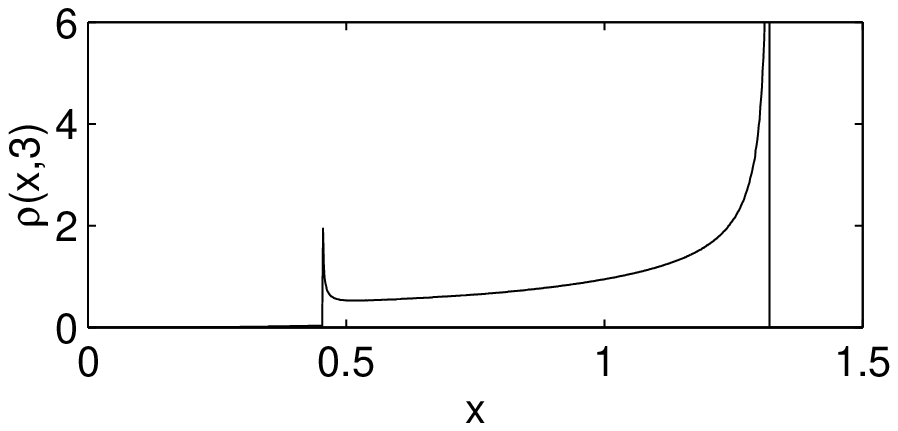} \\[0.1in]
\includegraphics[height=1.24in]{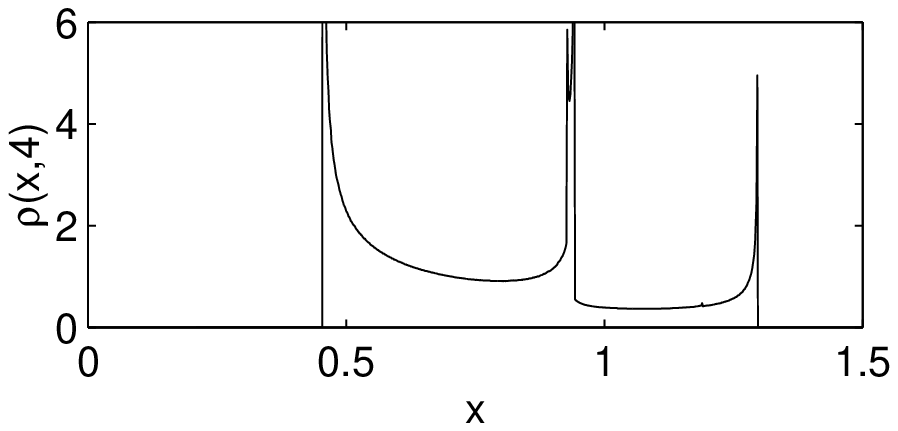} \\[0.1in]
\includegraphics[height=1.24in]{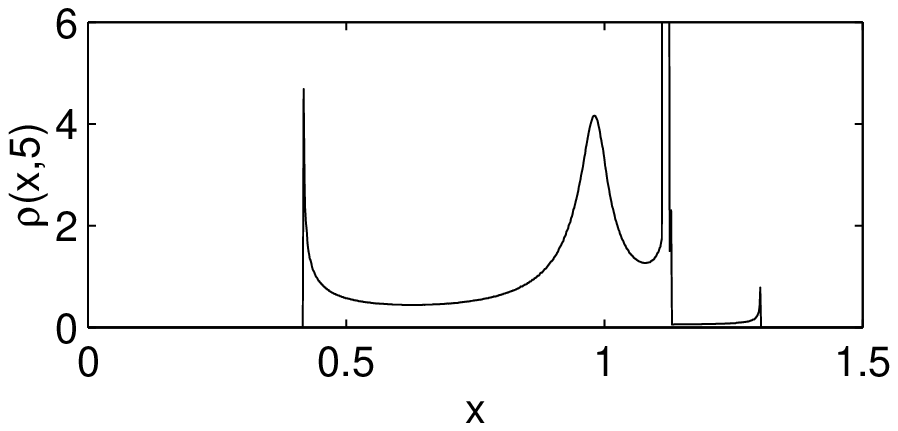}
\caption{As Figure~\ref{fig.denscompare}, with a different initial density.}
\label{fig.densmgb}
\end{center}
\end{figure}

% --------------------------------------------------------------------
\subsection{Discussion}

The algorithm presented here is superior in a number of respects to
the ``brute force'' ensemble simulation approach of
Section~\ref{sec.histos}.  In the brute force approach, a sufficient
number of solutions of the DDE must be computed to ensure adequate
statistical sampling.  Here, one solution of the given DDE is computed
for each point on the grid $\{x_i\}$ used to define the piecewise
linear approximation of $S_t$.  Figure~\ref{fig.denscompare} was
generated using $k=1000$ such points.  In fact even $k=100$ yields an
approximate $\rho(x,t)$ with accuracy on the order of that obtained by
the brute force approach with an ensemble of $10^6$ solutions.  This
is a dramatic computational saving, and is the primary benefit of the
method developed here.

The present method has the further advantage that the computation of
$\rho(x,t)$ (page~\pageref{densalg2}, steps 3--5) is decoupled from
the solution of the DDE (steps 1--2).  Since steps 1--2 are
independent of the initial density, a set of solutions of the DDE only
needs to be computed once to construct the approximation
$\tilde{S}_t$.  Subsequently, the evolution of any number of different
initial densities can be computed by steps 3--5.  By contrast, in the
brute force approach, computing the evolution of each different
initial density requires the computation of a new ensemble of
solutions of the DDE.

\begin{figure}
\begin{center}
\includegraphics[width=\figwidth]{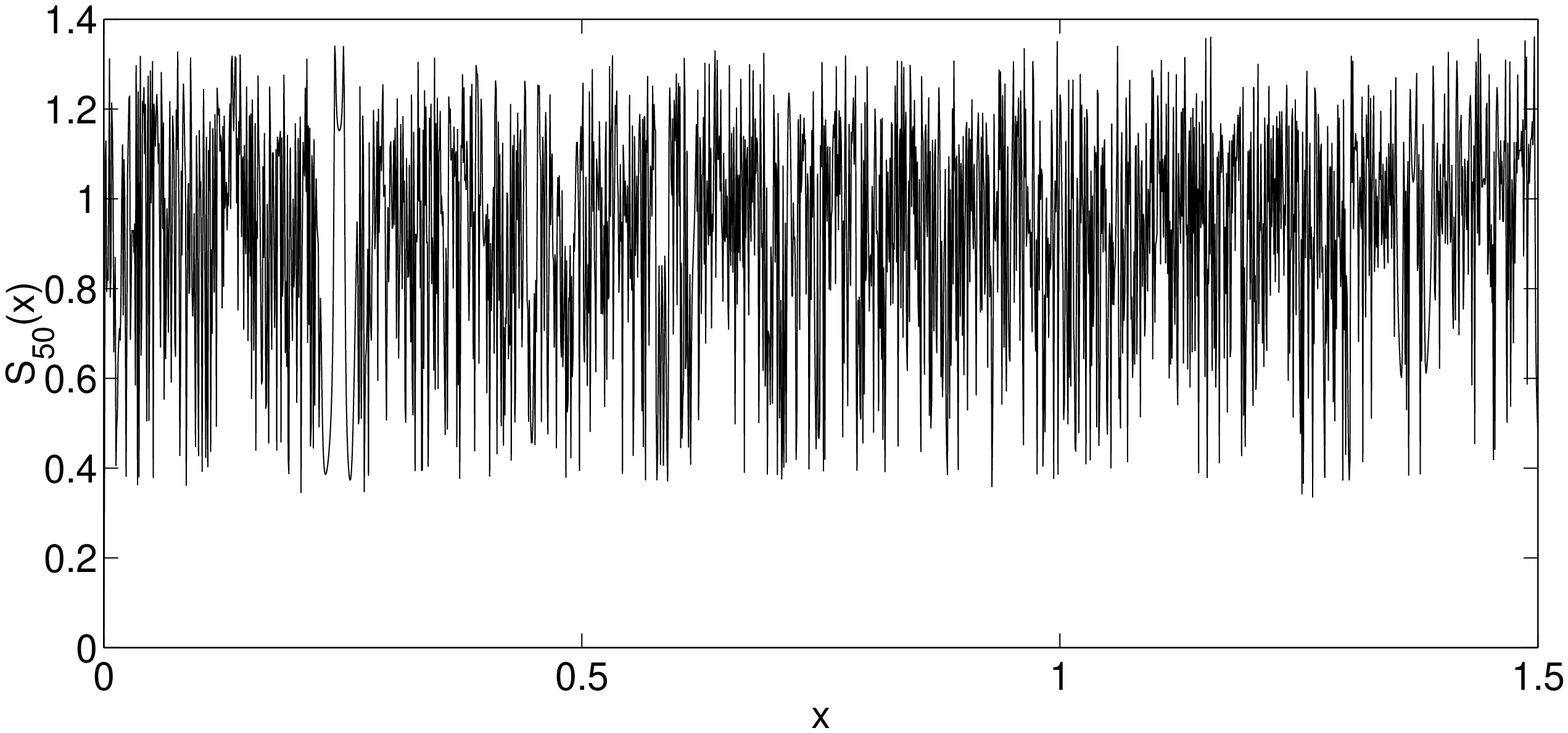}
\caption[Approximate solution map $\tilde{S}_{50}$ for the
Mackey-Glass equation restricted to constant initial
functions.]{Approximate solution map $\tilde{S}_{50}$ for the
Mackey-Glass equation~\eqref{eq.mg2} restricted to constant initial
functions (evaluated on a grid of $2000$ initial values $x$).}
\label{fig.mgmaplong}
\end{center}
\end{figure}

\begin{figure}
\begin{center}
\includegraphics[width=3in]{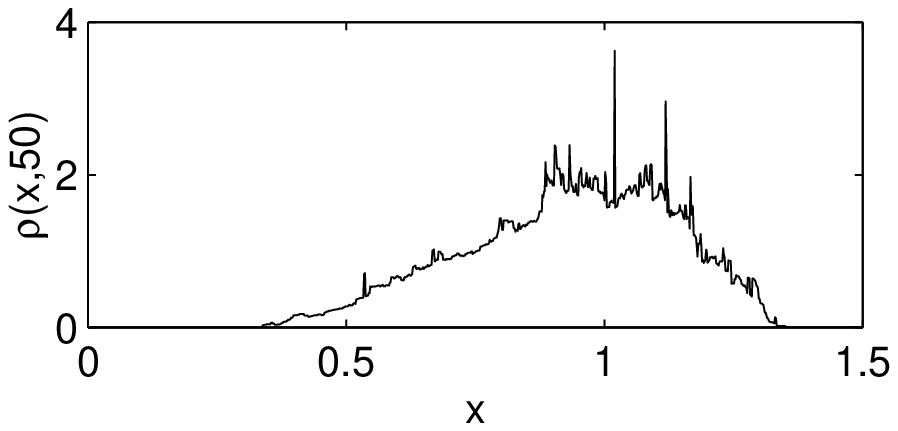}
\caption[Computed density $\rho(x,t)$ at time $t=50$, for the
Mackey-Glass equation restricted to constant initial
functions.]{Computed density $\rho(x,t)$ at time $t=50$, for the
Mackey-Glass equation~\eqref{eq.mg2} restricted to constant initial
functions.}
\label{fig.densmglong}
\end{center}
\end{figure}

Unfortunately it is not possible to evolve densities arbitrarily far
forward in time in this way, at least for delay equations with chaotic
dynamics.  Because of the stretching and folding of phase space
typical of chaotic systems, for large $t$ the solution map $S_t$
acquires a very complex structure.  This is illustrated in
Figure~\ref{fig.mgmaplong}, which shows the approximate solution map
$\tilde{S}_{50}$ for the Mackey-Glass equation~\eqref{eq.mg2}.  Owing
both to this fine structure and sensitivity to initial conditions, as
$t$ increases it eventually becomes impossible to obtain a reasonable
approximation of $S_t$ using finite-precision arithmetic.  This
difficulty is not a function of the accuracy of the numerical method
for integrating the DDE, but is rather a consequence of the complex
dynamics of the DDE itself.

Surprisingly, even though $\tilde{S}_t$ is a poor approximation of
$S_t$ for large $t$, it nevertheless appears to retain information
about the ensemble dynamics.  Figure~\ref{fig.densmglong} shows a
density $\rho(x,t)$ evolved forward to time $t=50$, again for the
Mackey-Glass equation~\eqref{eq.mg2} restricted to constant initial
functions.  This density was computed using the algorithm above, for
the solution map $\tilde{S}_{50}$ shown in Figure~\ref{fig.mgmaplong},
and the same initial density as in Figure~\ref{fig.densmgb}.  The
result shows a remarkable agreement with the corresponding density
computed by ``brute force'' ensemble simulation, shown in
Figure~\ref{fig.densmg2}, page~\pageref{fig.densmg2}.  Thus it appears
that the present method can provide an approximation of the same
asymptotic density as that found by direct ensemble simulation, while
requiring about $2$ orders of magnitude less computation time than the
ensemble simulation approach.

% =====================================================================
\clearpage

\section{Evolution Equation for Densities} \label{sec.stepsdens}

In section~\ref{sec.fpexplicit} Perron-Frobenius operators were
derived by first finding an explicit formula for the solution map
$S_t$.  This is an awkward and difficult intermediate step.  Rather it
would be nice if, in the spirit of equation~\eqref{eq.pfflow} for the
evolution of a density under the action of a flow defined by an
ordinary differential equation, one could derive an evolution equation
for the density itself.  This approach to density evolution for DDEs
is the subject of the present section.

Consider the augmented DDE initial value problem
\begin{equation} \label{eq.augdde2}
\begin{split}
  &x'(t) = \begin{cases}
    g\big( x(t) \big) & t \in [0,1) \\
    f\big( x(t), x(t-1) \big) & t \geq 1
          \end{cases} \\
  &x(0) = x_0,
\end{split}
\end{equation}
with $x(t) \in \Reals$, and suppose that an ensemble of initial values
$x_0$ is specified with density $\rho_0$.  We would like to derive an
evolution equation for the density $\rho(x,t)$ of the corresponding
ensemble of solutions $x(t)$.

There is an important preliminary observation to be made.  Ideally, we
would like to derive an evolution equation of the form
\begin{equation} \label{eq.ideal}
  \frac{d\rho}{dt} = \{\text{some operator}\}(\rho).
\end{equation}
However, $\rho$ cannot satisfy such an equation.  This is because the
family of solution maps $\{S_t\}$ for equation~\eqref{eq.augdde2} does
not form a semigroup (\cf\ remarks at the end of
Section~\ref{sec.augfp}).  That is, the density $\rho$ cannot be
sufficient to determine its own evolution, as in~\eqref{eq.ideal},
because the values $x(t)$ in the ensemble it describes are
insufficient to determine \emph{their} own evolution.  This difficulty
arises because $\rho$ does not contain information about the past
states of the ensemble, which is necessary to determine the evolution
of the ensemble under~\eqref{eq.augdde2}.  Thus, any solution to the
problem must take a form other than~\eqref{eq.ideal}.

% ---------------------------------------------------------------------
\subsection{ODE system}

The method of steps is sometimes used to write a DDE as a system of
ordinary differential equations.  This is a promising connection, as
we already know how densities evolve for ODEs (\cf\
Section~\ref{sec.fpode}, page~\pageref{sec.fpode}).

% ----------------------------------------------------------------------
\subsubsection{Method of steps}

Let $x(t)$ be a solution of~\eqref{eq.augdde2}, and define for
$n=0,1,2,\ldots$ the functions
\begin{equation} \label{eq.odesys1}
  y_n(s) = x(n + s), \quad s \in [0,1].
\end{equation}
Figure~\ref{fig.methsteps1} illustrates this relationship between the
$y_n$ and $x$.
\begin{figure}
\begin{center}
% RT 4.9.2019: use standalone figure for submission to arxiv
%\psfrag{y0}{$y_0(s)$}
%\psfrag{y1}{$y_1(s)$}
%\psfrag{y2}{$y_2(s)$}
%\psfrag{xt}{$x(t)$}
%\psfrag{x0}{$x_0$}
%\psfrag{s}{$s$}
%\psfrag{t}{$t$}
%\includegraphics[width=4in]{methsteps1}
\includegraphics[width=4in]{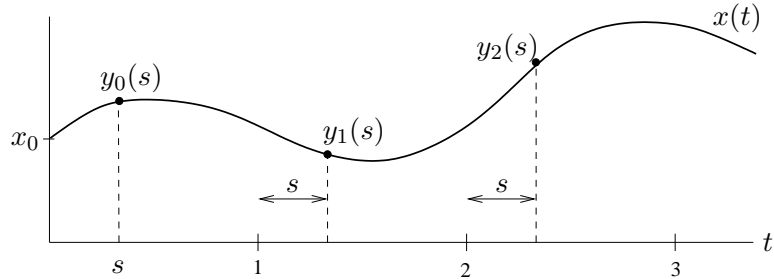}
\caption{Relationship between the DDE solution $x(t)$ and the
variables $y_n(s)$ defined in the method of steps,
equation~\eqref{eq.odesys1}.}
\label{fig.methsteps1}
\end{center}
\end{figure}
Since $x(t)$ satisfies~\eqref{eq.augdde2}, it follows that for $n \geq
1$,
\begin{equation} \label{eq.odesys2}
  y_n'(s) = f\big(y_n(s),y_{n-1}(s)\big), \quad n=1,2,\ldots,
\end{equation}
and $y_0$ satisfies
\begin{equation} \label{eq.odesys3}
  y_0' = g(y_0).
\end{equation}
Thus the augmented DDE becomes a system of evolution equations for the
$y_n$, together with the set of compatibility or boundary conditions
\begin{equation} \label{eq.compat}
\begin{split}
  &y_0(0) = x_0 \\
  &y_n(0) = y_{n-1}(1), \quad n=1,2,\ldots
\end{split}
\end{equation}
The ODE system \eqref{eq.odesys2}--\eqref{eq.odesys3}, together with
these compatibility conditions, can be solved sequentially to yield
the solution $x(t)$ of the DDE up to any finite time.  This is
essentially the method of steps for solving the DDE (\cf\
Section~\ref{sec.methsteps}, page~\pageref{sec.methsteps}).

If the system of ODEs~\eqref{eq.odesys2}--\eqref{eq.odesys3} could be
taken together as a vector field $F$ in $\Reals^{N+1}$, then an
ensemble of solutions of the DDE could be represented
via~\eqref{eq.odesys1} as an ensemble of vectors $y =
(y_0,\ldots,y_N)$, each carried along the flow induced by $F$.  The
density $\eta(y,t)$ of such an ensemble would evolve according to a
continuity equation (\cf\ Section~\ref{sec.fpode})
\begin{equation}
  \frac{\partial \eta}{\partial t} = -\nabla \cdot (\eta F).
\end{equation}
\cite{Mpriv} suggests this as an avenue to a probabilistic treatment of DDEs.
However, it is unclear how to ensure the compatibility
conditions~\eqref{eq.compat} are satisfied by every vector in the
ensemble, or how to determine the initial $(N+1)$-dimensional density
$\eta(y,0)$ of this ensemble in terms of a given density of initial
values $x_0$ in~\eqref{eq.augdde2}.  In short there is no obvious way
to treat equations~\eqref{eq.odesys2}--\eqref{eq.odesys3}
simultaneously rather than sequentially.  The following modified setup
is one way to avoid these difficulties.

% ----------------------------------------------------------------------
\subsubsection{Modified method of steps}

Any solution of the DDE problem~\eqref{eq.augdde2} can be extended
unambiguously to all $t < 0$ by setting
\begin{equation}
  x(t) = x_0, \quad t < 0,
\end{equation}
so that $x'(t)=0$ for all $t<0$.  For $n=0,1,2,\ldots$ let functions
$y_n$ be defined by
\begin{equation} \label{eq.yndef}
  y_n(t) = x(t-n), \quad t \geq 0.
\end{equation}
Figure~\ref{fig.methsteps2} illustrates the relationship between the
$y_n$ and $x$.  On substitution into equation~\eqref{eq.augdde2} we
find that for $t \in [m,m+1]$, $m=0,1,2,\ldots$, the $y_n(t)$ satisfy
\begin{equation} \label{eq.vf}
  y_n' = \begin{cases}
            f(y_n,y_{n+1})   & \text{if $n<m$} \\
            g(y_m)           & \text{if $n=m$} \\
            0                & \text{if $n>m$}.
	 \end{cases}
\end{equation}
Thus, for fixed $N$ the vector $y(t)=\big(y_0(t),\ldots,y_N(t)\big)$
satisfies an ordinary differential equation
\begin{equation} \label{eq.stepsys}
  y' = F(t,y),
\end{equation}
where the vector field $F: \Reals \times \Reals^{N+1} \to \Reals^{N+1}$ is
given by
\begin{equation} \label{eq.stepfield}
\begin{gathered}
  F(t,y_0,\ldots,y_N) = (y_0',\ldots,y_N'), \\
  \text{with} \quad y_n' = \begin{cases}
           0                & \text{if $t < n$} \\
           g(y_n)           & \text{if $t \in [n,n+1)$} \\
           f(y_n,y_{n+1})   & \text{if $t \geq n+1$}.
         \end{cases}
\end{gathered}
\end{equation}
\begin{figure}
\begin{center}
% RT 4.9.2019: use standalone figure for submission to arxiv
%\psfrag{y1}{$y_1(t)$}
%\psfrag{y2}{$y_2(t)$}
%\psfrag{y3}{$y_3(t)$}
%\psfrag{xt}{$x(t)=y_0(t)$}
%\psfrag{x0}{$x_0$}
%\psfrag{t}{$t$}
%\psfrag{tm1}{$t-1$}
%\psfrag{tm2}{$t-2$}
%\psfrag{tm3}{$t-3$}
%\includegraphics[width=\figwidth]{methsteps2}
\includegraphics[width=\figwidth]{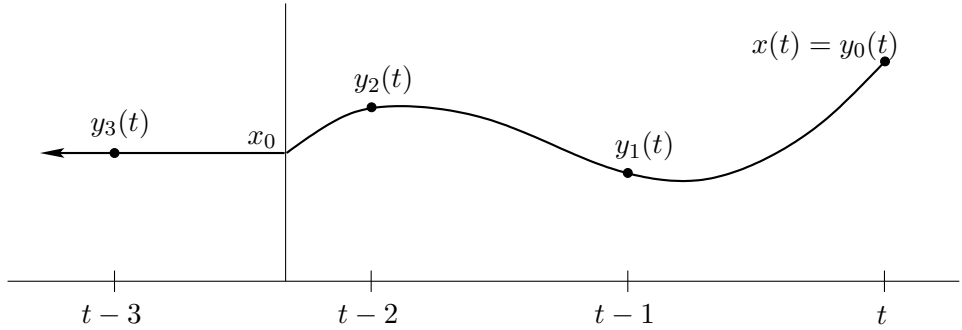}
\caption{Relationship between the DDE solution $x(t)$ and the
variables $y_n(t)$ defined in the modified method of steps,
equation~\eqref{eq.yndef}.}
\label{fig.methsteps2}
\end{center}
\end{figure}

Some remarks about the ODE
system~\eqref{eq.stepsys}--\eqref{eq.stepfield} are in order:
\begin{itemize}
\item For fixed $N>0$, the right-hand side $F(t,y)$, and hence the
solution $y(t)$, is defined only for $0 \leq t \leq N+1$.
\item For any given initial vector $y(0)=(y_0(0),\ldots,y_N(0))$,
equation~\eqref{eq.stepsys} has a unique solution $y(t)$ defined for
$0 \leq t \leq N+1$, provided we have existence and uniqueness for the
original DDE problem~\eqref{eq.augdde2}.
\item $F(t,y)$ is piecewise constant in time.  That is, for each
$m=0,1,2,\ldots$, the vector field $F(t,y)=F(y)$ is independent of $t
\in [m,m+1)$, and induces a flow in $\Reals^{N+1}$ that carries the
solution $y(t)$ forward from $t=m$ to $t=m+1$.  Thus we can speak
of~\eqref{eq.stepsys} as defining a sequence of flows in
$\Reals^{N+1}$.
\end{itemize}

The ODE system~\eqref{eq.stepsys} gives a representation of the method
of steps as an evolution equation in $\Reals^{N+1}$.  Indeed, the
solution $x(t)$ of the DDE is given, up to time $t=N+1$, by
$x(t)=y_0(t)$ where $y(t)$ is the solution of~\eqref{eq.stepsys}
corresponding to the initial condition $y(0)=(x_0,\ldots,x_0)$.  Thus
the DDE problem~\eqref{eq.augdde2} is equivalent to the initial value
problem
\begin{equation} \label{eq.odeivp}
\begin{split}
  &y' = F(y,t), \quad y(t) \in \Reals^{N+1}, \;0 \leq t \leq N+1, \\
  &y(0) = (x_0,\ldots,x_0),
\end{split}
\end{equation}
with the identification $x(t)=y_0(t)$.

% ---------------------------------------------------------------------
\subsection{Continuity equation}

Having established the equivalence of the DDE
system~\eqref{eq.augdde2} with the ODE system~\eqref{eq.odeivp}, we can
proceed to the probabilistic treatment of DDEs using techniques
developed for ODEs.

Suppose an ensemble of initial vectors $y \in \Reals^{N+1}$ is given,
with $(N+1)$-dimensional density $\eta(y,t)$.  Then under the sequence
of flows induced by the vector field $F(y,t)$, this density evolves
according to a continuity equation (\cf\ Section~\ref{sec.fpode}),
\begin{equation} \label{eq.cont1}
  \frac{\partial \eta(y,t)}{\partial t} = -\nabla \cdot \big(\eta(y,t) F(y,t)\big),
\end{equation}
where $\nabla=(\partial/\partial y_0,\ldots,\partial/\partial y_N)$.
The initial density $\eta(y,0)$ derives from the density $\rho_0$ of
initial values $x_0$ for the DDE~\eqref{eq.augdde2}.  That is, an
ensemble of initial values $x_0$ with density $\rho_0$ corresponds to an
ensemble of initial vectors $y=(x_0,\ldots,x_0)$ in $\Reals^{N+1}$,
with ``density''
\begin{equation} \label{eq.deltaic}
  \eta(y_0,\ldots,y_N;0) = \rho_0(y_0) \delta(y_0-y_1) \delta(y_0-y_2) \cdots
                                     \delta(y_0-y_N),
\end{equation}
where $\delta$ is the Dirac delta function.  This corresponds to a
line mass concentrated on the line $y_0=y_1=\cdots=y_N$ with linear
density $\rho_0(y_0)$.  Thus the problem of density evolution for DDEs
becomes a problem of determining how this line mass is redistributed
by the flow induced by~\eqref{eq.stepsys}.

With singular initial data such as~\eqref{eq.deltaic}, strong
solutions of the continuity equation~\eqref{eq.cont1} do not exist.
However, \eqref{eq.cont1} can be interpreted in a weak sense that
makes it possible to define ``solutions'' that satisfy initial
conditions like~\eqref{eq.deltaic} (\cf\ Section~\ref{sec.fpode}).
Such a weak solution can be obtained using the method of
characteristics.

% ---------------------------------------------------------------------
\subsection{Method of characteristics} \label{sec.ddemethchar}

Consider the initial value problem
\begin{equation} \label{eq.cont3}
\begin{gathered}
  \frac{\partial\eta(y,t)}{\partial t} + \nabla \cdot
  \big( F(y,t) \eta(y,t) \big)=0, \quad t \geq 0, \\
  \eta(y,0) = g(y)
\end{gathered}
\end{equation}
for the unknown function $\eta(y,t)$, with $y=(y_0,\ldots,y_N)$.
Assume provisionally that $F$ and $\eta$ are differentiable in $y$, so
the divergence operator can be expanded (by the product rule) to yield
\begin{equation} \label{eq.cont2}
  \frac{\partial\eta}{\partial t} + \eta \nabla\cdot F + F \cdot
  \nabla \eta = 0.
\end{equation}

Let a curve $\Gamma \subset \Reals^{N+2}$ (that is, in
$(y_0,\ldots,y_N,t)$-space) be parametrized by smooth functions
\begin{equation}
  y_0=y_0(t), \;\ldots\; , y_N=y_N(t)
\end{equation}
defined for all $t \geq 0$, and parametrize the value of $\eta$ on
$\Gamma$ by
\begin{equation} \label{eq.etaparam}
  \eta(t) = \eta\big( y_0(t),\ldots,y_N(t), t \big).
\end{equation}
(This slight abuse of notation helps clarify the following
development.)  Differentiating~\eqref{eq.etaparam} yields
\begin{equation}
  \frac{d\eta}{dt} = \nabla \eta \cdot \frac{dy}{dt}
                   + \frac{\partial \eta}{\partial t}.
\end{equation}
Thus, if the functions $y(t)$, $\eta(t)$ satisfy
\begin{gather}
  \frac{dy}{dt} = F\big( y(t),t \big) \label{eq.chardef} \\
  \frac{d\eta}{dt} = -\eta(t) \nabla \cdot F\big( y(t),t \big)
                    \label{eq.charev}
\end{gather}
for all $t \geq 0$, then $\eta$ as given by
equation~\eqref{eq.etaparam} satisfies the PDE~\eqref{eq.cont2} at
every point on $\Gamma$.  In fact any solution of the ODE
system~\eqref{eq.chardef}--\eqref{eq.charev} furnishes a solution of
the PDE~\eqref{eq.cont2} on a particular curve $\Gamma$.  In
particular, if $y$, $\eta$ are solutions of this system corresponding
to initial values
\begin{equation}
\begin{gathered}
  y(0) = (y_0, \ldots, y_N) \\
  \eta(0) = g\big( y_0, \ldots, y_N \big),
\end{gathered}
\end{equation}
then $\Gamma$ intersects the hyperplane $t=0$ at the point
$y=(y_0,\ldots,y_N)$, where $\eta(0)$ agrees with the initial data
given by $g$, so that $\eta(t)$ gives the solution of~\eqref{eq.cont3}
at every point of $\Gamma$.

The family of curves $\Gamma$ that satisfy~\eqref{eq.chardef} are
called the \emph{characteristics} of the PDE \eqref{eq.cont2}.  As we
have seen, the characteristics have the geometric interpretation that
an initial datum specified at $(y_0,\ldots,y_N,0)$ is propagated along
the characteristic that passes through this point.  Not surprisingly,
the characteristic curves of~\eqref{eq.cont3} coincide with the
integral curves of the vector field $F$ (\cf\
equation~\eqref{eq.chardef}).  That is, initial data are propagated
along streamlines of the induced flow.

If the characteristics foliate $\Reals^{N+2}$, every point
$(y_0,\ldots,y_N,t)\in\Reals^{N+2}$ has a characteristic curve passing
through it.  Then the solution of~\eqref{eq.cont3} can found at any
point, by using~\eqref{eq.chardef}--\eqref{eq.charev} to obtain the
solution on the characteristic curve through that point.  The usual
procedure for obtaining the solution $\eta(y,t)$ at
$P=(y_0,\ldots,y_N,t)$ is as follows.

\begin{enumerate} \label{charalg}
\item Determine the characteristic curve $\Gamma$ through $P$ and
follow it ``backward'' in time to find the point $Q$ on $\Gamma$ at
$t=0$.
\item Evaluate $g$ at $Q$ to determine the initial value $\eta(0)$ on
$\Gamma$.
\item\label{intstep} With this value of $\eta(0)$, integrate
equation~\eqref{eq.charev} forward along $\Gamma$ to $P$, at which
point the value of $\eta(t)$ is the solution of~\eqref{eq.cont3}.
\end{enumerate}
Figure~\ref{fig.charidea} gives a schematic illustration of the
method.
\begin{figure}
\begin{center}
% RT 4.9.2019: use standalone figure for submission to arxiv
%\psfrag{ylab}[][]{$y$}
%\psfrag{tlab}{$t$}
%\psfrag{Q}{$Q=(0,y(0))$}
%\psfrag{eta0}{$\eta(0)=g(y(0))$}
%\psfrag{eq}{$y'=F(y,t)$}
%\psfrag{eq2}{$\eta'=-\nabla\cdot F(y,t)$}
%\psfrag{P}{$P=(t,y(t))$}
%\psfrag{etat}{$\eta=\eta(t)$}
%\includegraphics[width=\figwidth]{charidea}
\includegraphics[width=\figwidth]{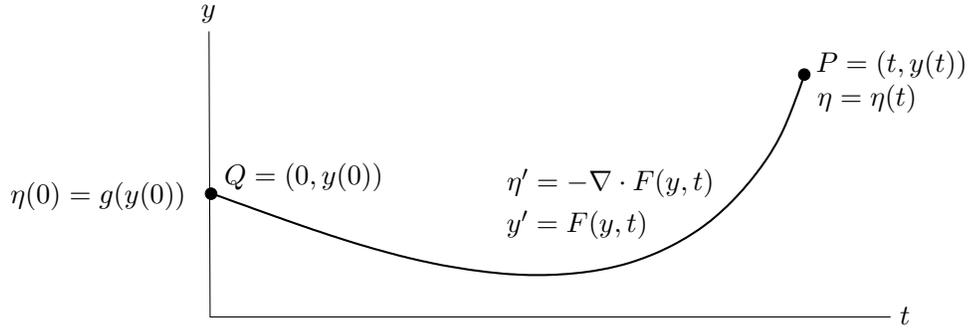}
\caption{Method of characteristics for the continuity
equation~\eqref{eq.cont3}: propagation of initial data along an
integral curve of $y'=F(y,t)$.}
\label{fig.charidea}
\end{center}
\end{figure}

Although the derivation assumes differentiability of $\eta$ (hence
also of $g$), the method itself does not rely on any special
properties of these functions---it requires only integration of the
vector field $F$ and evaluation of $g$.  Hence the method can be
applied even if $g$ is discontinuous, or singular as
in~\eqref{eq.deltaic}.  However, in such cases the resulting function
$\eta$ must be interpreted as a weak solution~\cite{CP88,Zaud83}.

Supposing the solution $\eta(y,t)$ of the initial value
problem~\eqref{eq.cont1}--\eqref{eq.deltaic} to have been found---\eg,
by the method of characteristics---the corresponding density
$\rho(x,t)$ of DDE solutions $x(t)$ can be determined as follows.
Since $x(t)=y_0(t)$ (\cf\ equation~\eqref{eq.yndef}), the density of
$x$ is identified with the density of $y_0$.  This density is
determined by integrating $\eta$ over all components except $y_0$,
\ie,
\begin{equation} \label{eq.rhoint}
  \rho(y_0,t) = \idotsint \eta(y_0,y_1,\ldots,y_N;t)\,dy_1 \cdots dy_N.
\end{equation}

% -----------------------------------------------------------------------
\subsubsection{Alternative formulation}

There is another way to formulate the evolution of the density
$\eta(y,t)$, that turns out to be equivalent to the method of
characteristics and serves to illuminate the method above.
It also provides an explicit formula (actually a codification of the
algorithm on page~\pageref{charalg}) for the solution $\eta(y,t)$.

Recall that the vector $y(t)$ evolves according to a system of
ODEs~\eqref{eq.stepsys}--\eqref{eq.stepfield}.  Let $y(t)$ be the
solution of this system with initial value $y(0)$, and define the
corresponding family of solution maps
$\hat{S}_t:\Reals^{N+1}\to\Reals^{N+1}$ by
\begin{equation}
  \hat{S}_t: y(0) \mapsto y(t)
\end{equation}
(to be distinguished from the solution map $S_t$ for the DDE, defined
by~\eqref{eq.nonautonsg}).  As $y(t)$ evolves under the action of
$\hat{S}_t$, the density $\eta(y,t)$ evolves according to the
corresponding Perron-Frobenius operator $\hat{P}_t: \eta(y,0) \mapsto
\eta(y,t)$, defined by
\begin{equation} \label{eq.fpdefsys}
  \int_A \hat{P}_t \eta(y,0)\,d^{N+1}y = \int_{\hat{S}_t^{-1}(A)}
  \eta(y,t)\,d^{N+1}y \quad \text{$\forall$ Borel $A \subset
  \Reals^{N+1}$}.
\end{equation}

Recall that $\hat{S}_t$ can be represented as a composition of flows
on the intervals $[m,m+1]$, $m=0,1,2,\ldots$, so it is one-to-one on
$\Reals^{N+1}$ and has an inverse (which can be found by reversing the
sequence of flows).  This makes possible the change of variables
$z=S_t(y)$ in~\eqref{eq.fpdefsys}, which by Theorem~3.2.1
of~\cite{LM94} becomes
\begin{equation}
  \int_A \hat{P}_t \eta(y,0)\,d^{N+1}y = \int_A
  \eta\big( \hat{S}_t^{-1}(z),t \big) J_t^{-1}(z)\,d^{N+1}z.
\end{equation}
Since $A$ is arbitrary, this implies the following explicit formula
for $P_t$,
\begin{equation} \label{eq.etafp}
  P_t \eta(y,0) = \eta(y,t) = \eta\big( \hat{S}_t^{-1}(y), 0 \big)
             J_t^{-1}(y).
\end{equation}
Here $J_t^{-1}$ is the density of the measure $\lambda\circ S_t^{-1}$
with respect to Lebesgue measure $\lambda$~\cite[p.\,46]{LM94}
(also~\cite[Section 5.1]{KH95}).  If $\hat{S}_t$ and $\hat{S}_t^{-1}$
are differentiable transformations\footnote{It suffices that the
vector field $F(t,y)$ be smooth in $y$~\cite[p.\,19]{Kuz95}.} of
$\Reals^{N+1}$ then $J_t^{-1}$ is just the determinant of the Jacobian
matrix $D\hat{S}_t^{-1}$,
\begin{equation}
\begin{gathered}
  J_t^{-1}(y) = \det \big( D\hat{S}_t^{-1}(y) \big), \\
  J_t(y) = \det \big( D\hat{S_t}(y) \big).
\end{gathered}
\end{equation}
In this case the formula~\eqref{eq.etafp} can be seen as a
multi-dimensional analog of~\eqref{eq.fpline}, in the case of
invertible $S_t$.

Notice that $\hat{S}_t$ effects the translation of a point $(y(0),0)$
along a characteristic curve $\Gamma$ to $(y(t),t)$.  Similarly
$\hat{S}_t^{-1}$ effects a translation backward along $\Gamma$ to
$t=0$.  This draws the connection between~\eqref{eq.etafp} and the
method of characteristics (page~\pageref{charalg}): the point $Q$
(where the initial density is evaluated) is identified with the point
$(\hat{S}_t^{-1}(y),0)$ in~\eqref{eq.etafp}.

The factor $J_t^{-1}(y)$ also has a geometric interpretation: it is
the factor by which the volume of an infinitesimal volume element at
$y(t)$ increases under transportation by $\hat{S}_t^{-1}$.  This
factor can equivalently be understood as resulting from
step~\ref{intstep} of the method of characteristics algorithm, since
$\nabla\cdot F(y,t)$ is the instantaneous growth rate of an
infinitesimal volume at $y$ as it is transported by the flow $S_t$
induced by $F$.  Conservation of mass requires that the density
supported on an infinitesimal volume element decrease in proportion to
the volume growth, \ie, by the factor $J_t^{-1}(y)$.  This provides a
geometrical explanation of the term $J_t^{-1}(y)$ in~\eqref{eq.etafp}.

\subsubsection{Examples}

To illustrate this approach to the evolution of densities for DDEs, we
revisit examples~\ref{ex.lin} and~\ref{ex.quad}, for which analytical
solutions are obtainable for the densities $\eta(y,t)$ and
$\rho(x,t)$.

\begin{example} \label{ex.linb}
Consider again the linear DDE of example~\ref{ex.lin}
(page~\pageref{ex.lin}),
\begin{equation}
  x'(t) = \alpha x(t-1), \quad t \geq 1,
\end{equation}
with initial data restricted to constant initial functions on $[0,1]$,
and the initial value $x(0)$ distributed with density $\rho_0(x)$.  The
vector $y=(y_0,y_1,y_2)$ defined by~\eqref{eq.yndef} satisfies a
differential equation $y'=F(y,t)$ for $t \in [0,3]$, with
\begin{equation} \label{eq.sysex1}
\begin{gathered}
  y_n' = \begin{cases}
           0                & \text{if $t < n+1$} \\
           \alpha y_{n+1}   & \text{if $n+1 \leq t \leq 3$}.
         \end{cases}
\end{gathered}
\end{equation}
Under the action of this system, the density $\eta(y,t)$ of an
ensemble of vectors $y \in \Reals^3$ evolves according to the
continuity equation~\eqref{eq.cont1}. whose characteristic curves are
streamlines of the flow induced by~\eqref{eq.sysex1}.  The solution of
this system is readily obtained (\eg, using Maple), and the solution
map $\hat{S}_t:y(0) \mapsto y(t)$ found to be
\begin{equation} \label{eq.linbSt}
  \hat{S}_t(y) = \begin{cases}
    (y_0, y_1, y_2)                                  & \text{if $t \in [0,1)$} \\
    \big( y_0 + \alpha (t-1) y_1, \;y_1, \;y_2 \big) & \text{if $t \in [1,2)$} \\
    \big( y_0 + \alpha (t-1) y_1 +
      \big( \tfrac{1}{2}\alpha^2 t^2 -
       2 \alpha^2 (t-1) \big) y_2, \\
      \quad y_1 + \alpha (t-2) y_2, \; y_2 \big)        & \text{if $t \in [2,3]$.}
  \end{cases}
\end{equation}
This linear transformation is easily inverted to yield
\begin{equation}
  \hat{S}_t^{-1}(y) = \begin{cases}
    (y_0,y_1,y_2)                                 & \text{if $t \in [0,1)$} \\
    \big( y_0 - \alpha(t-1)y_1,\;y_1,\;y_2 \big)  & \text{if $t \in [1,2)$} \\
    \big( y_0 - \alpha(t-1)y_1 +
      \alpha^2 (\tfrac{1}{2} t^2 - t) y_2,\\
     \quad y_1 - \alpha (t-2) y_2, y_2 \big)      & \text{if $t \in [2,3]$.}
  \end{cases}
\end{equation}
The Jacobian of this transformation is
\begin{equation}
  J_t^{-1}(y) = \det\big( D \hat{S}_t^{-1}(y) \big) = 1 \quad \forall t \in [0,3]
\end{equation}
(\ie, $\hat{S}_t^{-1}$ is volume-preserving).  The initial ensemble of
vectors $y$ has ``density'' $\eta(y,0)$ is given by
\begin{equation}
  \eta(y_0,y_1,y_2,0) = \rho_0(y_0) \delta(y_0-y_1) \delta(y_0-y_2),
\end{equation}
hence equation~\eqref{eq.etafp} gives
\begin{equation}
  \eta(y,t) = \begin{cases}
    \rho_0(y_0) \delta(y_0-y_1) \delta(y_0-y_2)  & \text{if $t \in [0,1)$} \\
    \rho_0\big( y_0-\alpha(t-1)y_1 \big)
       \delta\big( y_0 - \alpha(t-1) y_1 - y_1 \big) \cdot \\
 \quad \delta\big( y_0 - \alpha(t-1) y_1 - y_2 \big) & \text{if $t \in [1,2)$} \\
    \rho_0\big( y_0 - \alpha(t-1)y_1 + \alpha^2(\tfrac{1}{2} t^2 - t) y_2 \big)\cdot\\
 \quad \delta\big( [y_0 - \alpha(t-1)y_1 + \alpha^2(\tfrac{1}{2} t^2 - t) y_2] -\\
      \qquad       [y_1 - \alpha (t-2) y_2] \big) \cdot\\
 \quad \delta\big(  [y_0 - \alpha(t-1)y_1 + \alpha^2(\tfrac{1}{2} t^2 - t) y_2] - y_2 \big)
  & \text{if $t \in [2,3]$.}
  \end{cases}
\end{equation}
Integrating over $y_1$, $y_2$ with $x=y_0$ we obtain for $t \in
[0,1)$,
\begin{equation}
\begin{split}
  \rho(x,t) &= \iint  \rho_0(x) \delta(x-y_1) \delta(x-y_2) \,dy_1\,dy_2 \\
            &= \rho_0(x),
\end{split}
\end{equation}
(hence $P_t\rho=\rho$ as expected, since $S_t$ is just the identity
transformation).  For $t \in [1,2)$,
\begin{equation}
\begin{split}
  \rho(x,t) &= \iint \rho_0\big( x-\alpha(t-1)y_1 \big) \delta\big(
       x - \alpha(t-1) y_1 - y_1 \big) \cdot \\
            & \hspace{0.75in} \delta\big( x - \alpha(t-1) y_1 - y_2 \big)\,dy_1\,dy_2 \\
   &= \int \rho_0\big( x-\alpha(t-1)y_1 \big) \delta\big(
       x - \alpha(t-1) y_1 - y_1 \big) \,dy_1 \\
   &= \frac{1}{|1+\alpha(t-1)|}\int\rho_0\Big( x - \alpha(t-1)\frac{x-z}{1 + \alpha(t-1)} \Big)
       \delta(z) \,dz \\
   &= \frac{1}{|1+\alpha(t-1)|}\rho_0\Big( x - \alpha(t-1) \frac{x}{1 + \alpha(t-1)} \Big) \\
   &= \frac{1}{|1+\alpha(t-1)|}\rho_0\Big( \frac{x}{1 + \alpha(t-1)} \Big).
\end{split}
\end{equation}
This agrees with the result~\eqref{eq.fpans1} of example~\ref{ex.lin},
which was obtained by a different method.  For $t \in [2,3]$ the
integral for $\rho(x,t)$ becomes too complicated to be worth writing
out fully here, but its result also agrees with~\eqref{eq.fpans1}.
\end{example}

\begin{example} \label{ex.quadb}
Consider again the DDE of example~\ref{ex.quad} (page~\pageref{ex.quad}),
\begin{equation}
  x'(t) = -x(t-1)^2, \quad t \geq 1,
\end{equation}
with initial data restricted to constant initial functions on $[0,1]$,
and the initial value $x(0)$ distributed with density $\rho_0(x)$.  The
vector $y=(y_0,y_1,y_2)$ defined as in~\eqref{eq.yndef} satisfies a
differential equation $y'=F(y,t)$ for $t \in [0,3]$, with
\begin{equation} \label{eq.sysex2}
\begin{gathered}
  y_n' = \begin{cases}
           0            & \text{if $t < n+1$} \\
           -y_{n+1}^2   & \text{if $n+1 \leq t \leq 3$}.
         \end{cases}
\end{gathered}
\end{equation}
Solving this system (\eg, using Maple), and defining the solution map
$\hat{S}_t: y(0) \mapsto y(t)$ yields
\begin{equation} \label{eq.quadbSt}
\hat{S}_t(y) = \begin{cases}
   (y_0,\;y_1,\;y_2)                 & \text{if $t \in [0,1)$} \\
   \big( y_0 - (t-1) y_1^2, \;y_1, \;y_2 \big) & \text{if $t \in [1,2)$} \\
   \big(y_0 - (t-1) y_1^2 + (t^2 - 4 t + 4) y_1 y_2^2 - \\
    \quad (\tfrac{1}{3}t^3 - 2 t^2 + 4 t - \tfrac{8}{3}) y_2^4, \;
   y_1 - (t-2)y_2^2, \; y_2 \big)    & \text{if $t \in [2,3]$}.
  \end{cases}
\end{equation}
Inverting this transformation yields
\begin{equation}
  \hat{S}_t^{-1}(y) = \begin{cases}
    (y_0, \;y_1, \;y_2)   & \text{if $t \in [0,1)$} \\
    \big( y_0 + (t-1) y_1^2, \;y_1, \;y_2 \big)   & \text{if $t \in [1,2)$} \\
    \big(y_0 + (t-1) y_1^2 + (t^2 - 2 t) y_1 y_2^2 +
       ( \tfrac{1}{3}t^3 - t^2 + \tfrac{4}{3} ) y_2^4, \\
    \qquad y_1 + (t-2) y_2^2, \; y_2 \big) & \text{if $t \in [2,3]$.}
  \end{cases}
\end{equation}
The Jacobian of this transformation is again
\begin{equation}
  J_t^{-1}(y) = \det\big(D\hat{S}_t^{-1}(y)\big) = 1 \quad \forall t \in [0,3].
\end{equation}
The initial density $\eta(y,0)$ is given by
\begin{equation}
  \eta(y,0) = \rho_0(y_0) \delta(y_0-y_1) \delta(y_1-y_2),
\end{equation}
so equation~\eqref{eq.etafp} gives
\begin{equation}
  \eta(y,t) = \begin{cases}
    \rho_0(y_0) \delta(y_0-y_1) \delta(y_1-y_2)  & \text{if $t \in [0,1)$} \\
    \rho_0\big(y_0 + (t-1)y_1^2\big) \delta\big(y_0 + (t-1)y_1^2 - y_1\big) \delta(y_1-y_2) &
        \text{if $t \in [1,2)$} \\
    \rho_0\big( y_0 + (t-1) y_1^2 + (t^2 - 2 t) y_1 y_2^2 +
       ( \tfrac{1}{3}t^3 - t^2 + \tfrac{4}{3} ) y_2^4 \big) \cdot \\
    \quad \delta\big( [y_0 + (t-1) y_1^2 + (t^2 - 2 t) y_1 y_2^2 +
       ( \tfrac{1}{3}t^3 - t^2 + \tfrac{4}{3} ) y_2^4] \\
         \qquad - [y_1 + (t-2) y_2^2] \big) \cdot \\
    \quad \delta\big(  [y_1 + (t-2) y_2^2] - y_2 \big) & \text{if $t \in [2,3]$}.
  \end{cases}
\end{equation}
Finally, $\rho(x,t)$ is obtained by integrating $\eta(y,t)$ over
$y_1$, $y_2$.  For $t \in [0,1)$,
\begin{equation}
\begin{split}
  \rho(x,t) &= \iint  \rho_0(x) \delta(x-y_1) \delta(y_1-y_2) \,dy_1\,dy_2 \\
            &= \rho_0(x),
\end{split}
\end{equation}
as expected.  For $t \in [1,2)$,
\begin{equation}
\begin{split}
  \rho(x,t) &= \iint \rho_0\big(x + (t-1)y_1^2\big) \delta\big(x + (t-1)y_1^2 - y_1\big)
                      \delta(y_1-y_2)\,dy_1\,dy_2 \\
  &= \int \rho_0\big(x + (t-1)y_1^2\big) \delta\big(x + (t-1)y_1^2 - y_1\big) \,dy_1 \\
  &= \sum_{\{y_1: x + (t-1)y_1^2 - y_1 = 0\}}
      \frac{\rho_0\big(x + (t-1)y_1^2\big)}{|2 (t-1) y_1 - 1|} \\
  &= \frac{1}{\sqrt{1-4(t-1)x}} \Big[
  \rho_0\Big( \frac{1-\sqrt{1-4(t-1)x}}{2(t-1)} \Big) \\
  &\hspace{1.5in}+ \rho_0\Big( \frac{1+\sqrt{1-4(t-1)x}}{2(t-1)} \Big) \Big].
\end{split}
\end{equation}
This is identical to the result of example~\ref{ex.quad}.  For $t \in
[2,3)$ it would be extremely difficult to find an explicit formula for
$\rho(x,t)$, since the final integration requires solving a quartic
equation.  For $t > 3$, a quintic equation must be solved, so finding
an explicit formula for $\rho(x,t)$ appears to be impossible.
\end{example}

\begin{remark}
In each of the examples above, the transformation $\hat{S}_t$ was
found to be volume-preserving, so that the Jacobian $J_t^{-1}(y)=1$ in
equation~\eqref{eq.etafp}.  This could have been anticipated from
equation~\eqref{eq.charev} for the evolution of the density along a
characteristic curve of the continuity equation, since by
equation~\eqref{eq.vf},
\begin{equation}
  \nabla \cdot F(y,t) = D_1 f(y_0,y_1) + \cdots + D_1 f(y_{m-1},y_m) +
                      g'(y_m)
\end{equation}
for $t \in [m,m+1)$.  Thus, if $f$ is independent of its first
argument and $g'=0$ (as in the examples above), then $\nabla \cdot
F(y,t)=0$, so that the sequence of flows constituting $S_t$ are all
volume-preserving.
\end{remark}

% ---------------------------------------------------------------------
\subsection{Geometric interpretation} \label{sec.geom}

As the examples above suggest, for all but the simplest delay
equations an analytical treatment of the density evolution problem is
difficult, and perhaps impossible.  Nevertheless, the approach
developed above does provide some geometrical insight even when an
analytical approach fails.

Recall that up to any finite time $t \leq N+1$ the DDE
problem~\eqref{eq.augdde} can be represented by an ordinary
differential equation
\begin{equation} \label{eq.stepsys2}
  y'=F(y,t), \quad y(t) \in \Reals^{N+1}, \quad t \geq 0,
\end{equation}
with $F$ defined by~\eqref{eq.stepfield}.  An ensemble of initial
values with density $\rho_0$ corresponds to an ensemble of initial
vectors $y$ with $(N+1)$-dimensional ``density''
\begin{equation} \label{eq.deltaic2}
  \eta(y_0,\ldots,y_N;0) = \rho_0(y_0) \delta(y_0-y_1) \delta(y_1-y_2) \cdots
                                     \delta(y_{N-1}-y_N),
\end{equation}
representing a line mass concentrated on the line $y_0=y_1=\cdots=y_N$
in $\Reals^{N+1}$.  Under evolution by~\eqref{eq.stepsys2}, \ie, under
transformation by the solution map~$\hat{S}_t$, this line mass is
redistributed.  This transportation of a line mass under $\hat{S}_t$
is illustrated in Figure~\ref{fig.denscurveidea}.  After evolution by
time $t$, $\eta(y,t)$ is supported on a one-dimensional curve that is
the image of this line under $\hat{S}_t$.  We will call this curve the
``density support curve''.  It is a continuous, non-self-intersecting
curve in $\Reals^{N+1}$, owing to continuity and invertibility of
$\hat{S}_t$.

\begin{figure}
\begin{center}
\includegraphics[width=2.55in]{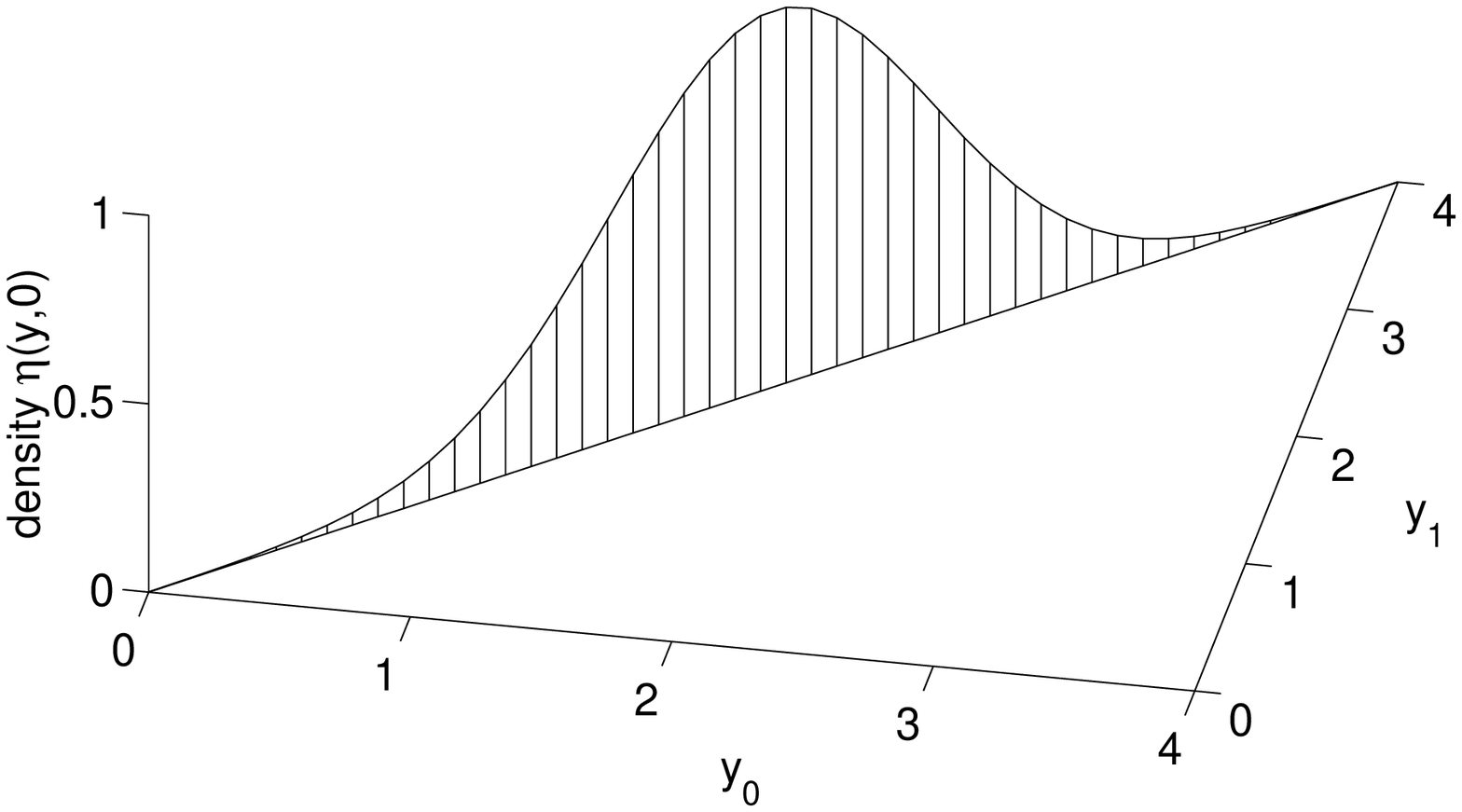}\hfill
\includegraphics[width=2.55in]{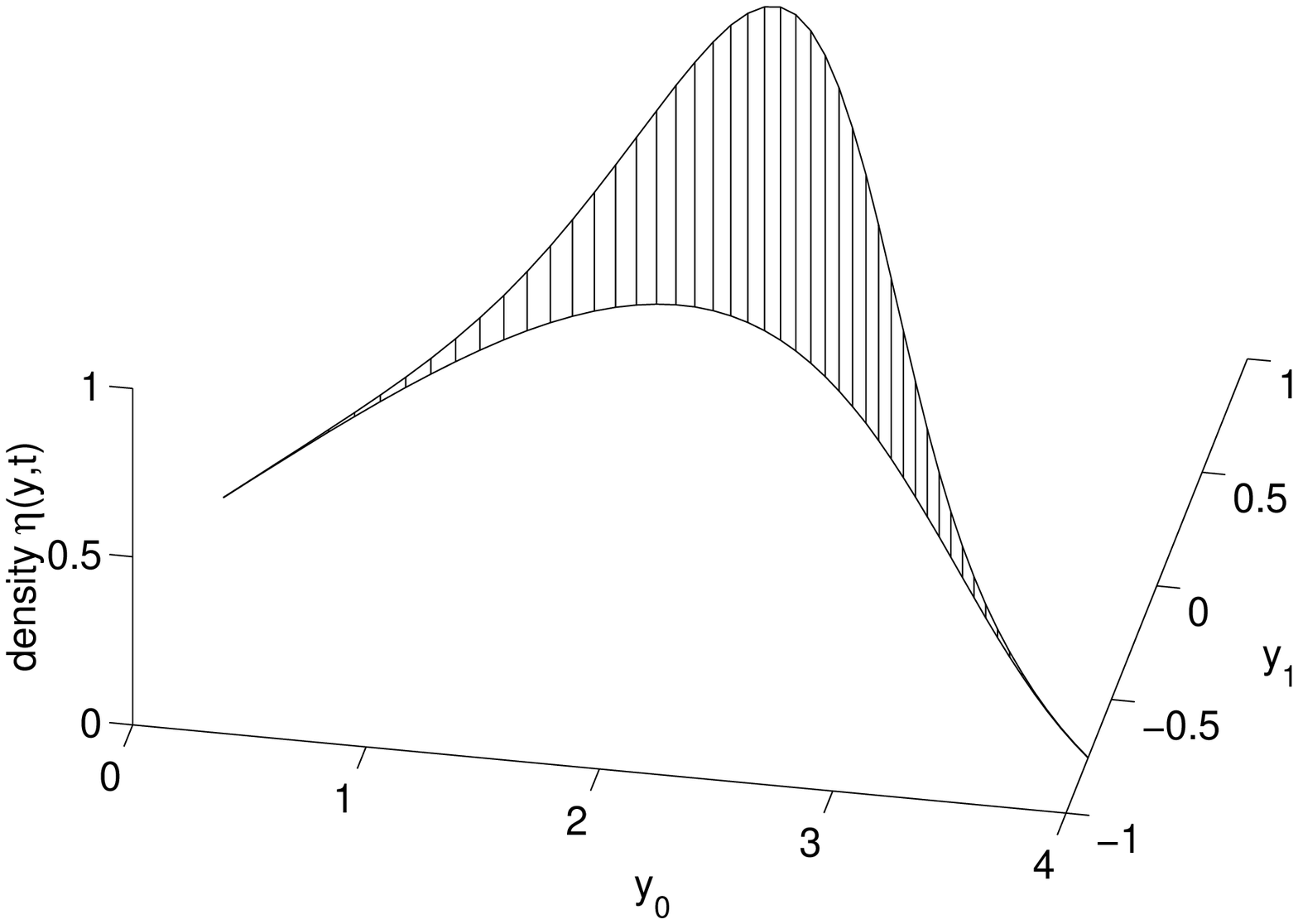}\\
\hfill(a)\hfill\hfill(b)\hspace*{\fill}
\caption[Transportation of a line mass under a transformation of the
$(y_0,y_1)$-plane.]{Transportation of a line mass under a
transformation of the $(y_0,y_1)$-plane.  (a) Initial mass distributed
on the line $y_0=y_1$.  (b) Mass distribution after transformation.}
\label{fig.denscurveidea}
\end{center}
\end{figure}

\begin{example}
Explicit representations can be found for the density support curves
in the previous examples.  Consider the linear DDE of example
\ref{ex.linb} (page \pageref{ex.linb}), for which the transformation
$\tilde{S}_t$ is given explicitly by equation~\eqref{eq.linbSt}.  The
initial density support curve $y_0=y_1=y_2$ can be represented
parametrically as
\begin{equation}
  \Gamma = \{ (s,s,s): s \in \Reals \}.
\end{equation}
Then
\begin{equation}
\begin{split}
  \tilde{S}_t(\Gamma) &= \big\{ \tilde{S}_t(y): y \in \Gamma \big\} \\
    &= \big\{ y(s) = (y_1, y_2, y_3)(s) : s \in \Reals \big\}
\end{split}
\end{equation}
where
\begin{equation}
  ( y_1, y_2, y_3 )(s) = \begin{cases}
    (s, s, s)                                  & t \in [0,1) \\
    \big( s + \alpha (t-1) s, \;s, \;s \big)   & t \in [1,2) \\
    \big( s + \alpha (t-1) s +
      \big( \tfrac{1}{2}\alpha^2 t^2 -
       2 \alpha^2 (t-1) \big) s, \\
      \quad s + \alpha (t-2) s, \; s \big)     & t \in [2,3).
       \end{cases}
\end{equation}
For this DDE, at any given time the density support curve is a
straight line, a consequence of linearity of $\tilde{S}_t$.
\end{example}

\begin{example}
Consider the quadratic DDE of example~\ref{ex.quadb} (page
\pageref{ex.quadb}), for which the transformation $\tilde{S}_t$ is given
explicitly by equation~\eqref{eq.quadbSt}.  Representing the initial
density support curve $\Gamma$ parametrically as in the previous
example, we have
\begin{equation}
  \tilde{S}_t(\Gamma) = \big\{ y(s) = (y_1, y_2, y_3)(s) : s \in \Reals \big\}
\end{equation}
where
\begin{equation}
  (y_1,y_2,y_3)(s) = \begin{cases}
   (s,\;s,\;s)                 & t \in [0,1) \\
   \big( s - (t-1) s^2, \;s, \;s \big) & t \in [1,2) \\
   \big( s - (t-1) s^2 + (t^2 - 4 t + 4) s^3 - \\
    \quad (\tfrac{1}{3}t^3 - 2 t^2 + 4 t - \tfrac{8}{3}) s^4, \;
   s - (t-2)s^2, \; s \big)    & t \in [2,3].
  \end{cases}
\end{equation}
\end{example}

\bigskip

Thus, the density evolution method developed in
section~\ref{sec.ddemethchar} amounts to keeping track of the
evolution of the density support curve under the action
of~\eqref{eq.stepsys2}.  For the purposes of a numerical
implementation, this curve can be represented by a set of points
$\{y^{(1)}(t),y^{(2)}(t),\ldots,y^{(k)}(t)\}\subset\Reals^{N+1}$ (\eg,
representing a piecewise-linear approximation of the curve).  The
$y^{(i)}(t)$ are images under $\hat{S}_t$ of points $y^{(i)}(0)$ of
the form
\begin{equation} \label{eq.yidef}
  y^{(i)}(0) = \big(x_0^{(i)}, \ldots, x_0^{(i)}\big) \in \Reals^{N+1},
\end{equation}
which lie on the initial density support curve $y_0=y_1=\cdots=y_N$.
Thus the points $y^{(i)}(t)$ can be determined by integrating
(numerically) each of these initial points forward
under~\eqref{eq.stepsys2}.  With sufficiently closely spaced points
$x_0^{(i)}$ in the support of the initial density $\rho_0(x)$, the
resulting set of $y^{(i)}(t)$ should provide a good approximation of
the density support curve, and the mass distribution on it (see
Figure~\ref{fig.denscurve1}).
\begin{figure}
\begin{center}
% RT 4.9.2019: use standalone figure for submission to arxiv
%\psfrag{L1}[][]{{\small initial support curve}}
%\psfrag{L2}[][]{{\small transformed support curve}}
%\psfrag{y0}{$y_0$}
%\psfrag{y1}{$y_1$}
%\psfrag{y01}{$y^{(1)}(0)$}
%\psfrag{y02}{$y^{(2)}(0)$}
%\psfrag{y03}{$y^{(3)}(0)$}
%\psfrag{yt1}{$y^{(1)}(t)$}
%\psfrag{yt2}{$y^{(2)}(t)$}
%\psfrag{yt3}{$y^{(3)}(t)$}
%\psfrag{St}{$\hat{S}_t$}
%\includegraphics[width=\figwidth]{denscurve1}
\includegraphics[width=\figwidth]{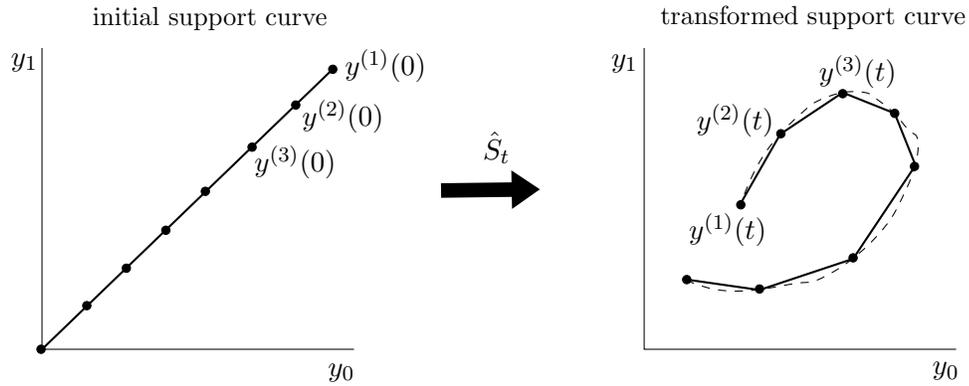}
\caption[Approximating the image under $\hat{S}_t$ of the density
support curve.]{Approximating the image under $\hat{S}_t$ of the
density support curve, by following the evolution
under~\eqref{eq.stepsys2} of a set of points $\{y^{(i)}(t)\}$ that
represent the nodes of a piecewise linear curve.  Here the curve is
shown projected onto the $(y_0,y_1)$-plane.}
\label{fig.denscurve1}
\end{center}
\end{figure}

Figure~\ref{fig.mgdenscurve} illustrates the results of applying this
idea to the Mackey-Glass equation~\eqref{eq.mg2}, for an ensemble of
constant initial functions with values distributed on the interval
$[0.3,1.3]$, as in Figures~\ref{fig.densidea}--\ref{fig.densmg1}.
Thus the initial density support curve is the part of the line
$y_0=y_1=\cdots$ with $0.3 \leq y_0 \leq 1.3$.

The first row of figure~\ref{fig.mgdenscurve} shows the sequence of
density support curves obtained at times $t=1,2,3,4$, projected onto
the $(y_0,y_1)$-plane (as a result of this projection, some of the
curves intersect themselves).  The second row shows the corresponding
densities $\rho(x,t)$ from Figure~\ref{fig.densmg1}.  These densities
can be interpreted as resulting from projecting the mass supported on
the corresponding density support curve onto the $y_0$-axis.  With
this interpretation, the density support curves provide an obvious
geometrical interpretation of the structures observed in the
corresponding densities $\rho(x,t)$.  Discrete jumps in the density
occur at the endpoints of the transformed density support curve, and
the maxima (singularities) correspond to turning points of the
transformed density support curve.

\begin{landscape}
\begin{figure}
\vspace*{0.5in}
\includegraphics[width=1.85in]{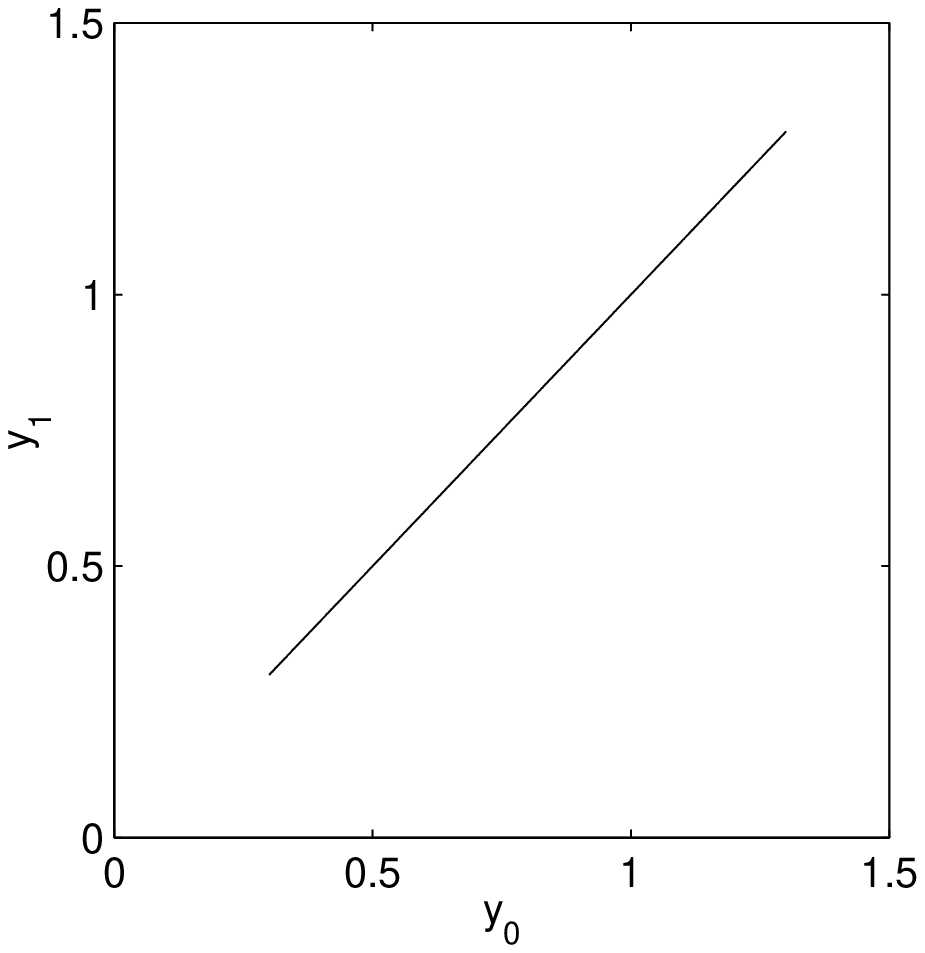}\hfill
\includegraphics[width=1.85in]{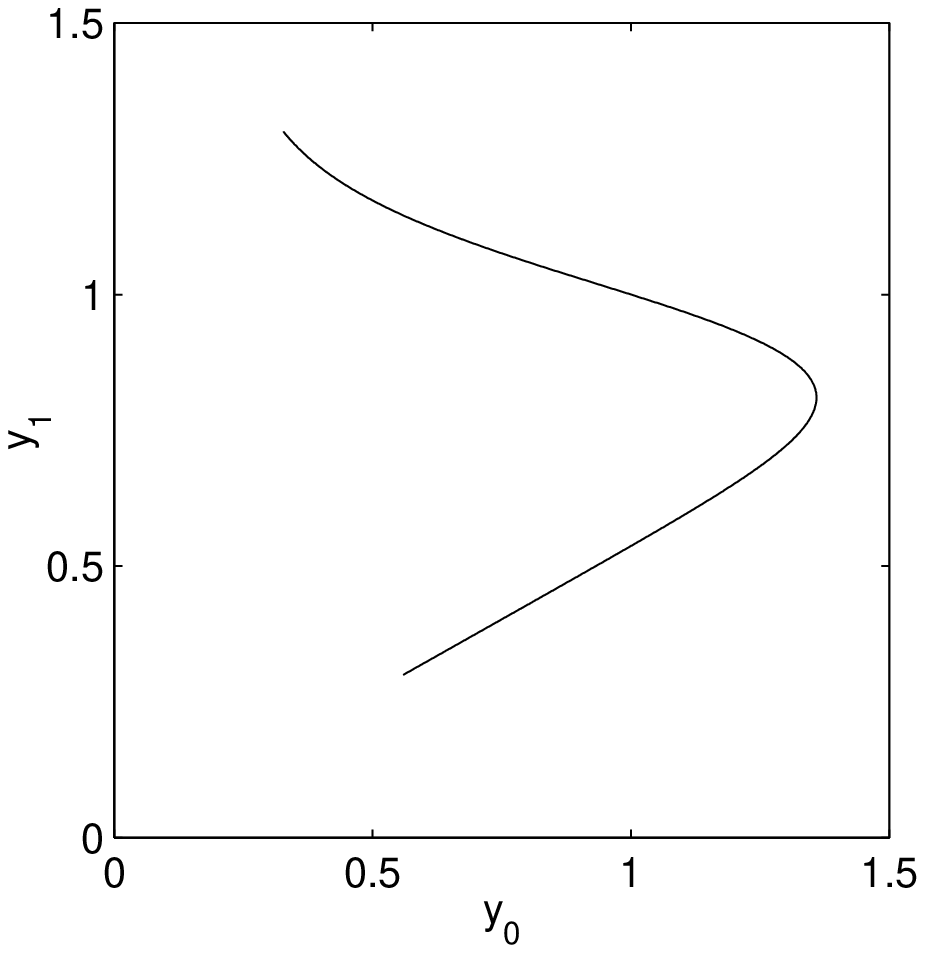}\hfill
\includegraphics[width=1.85in]{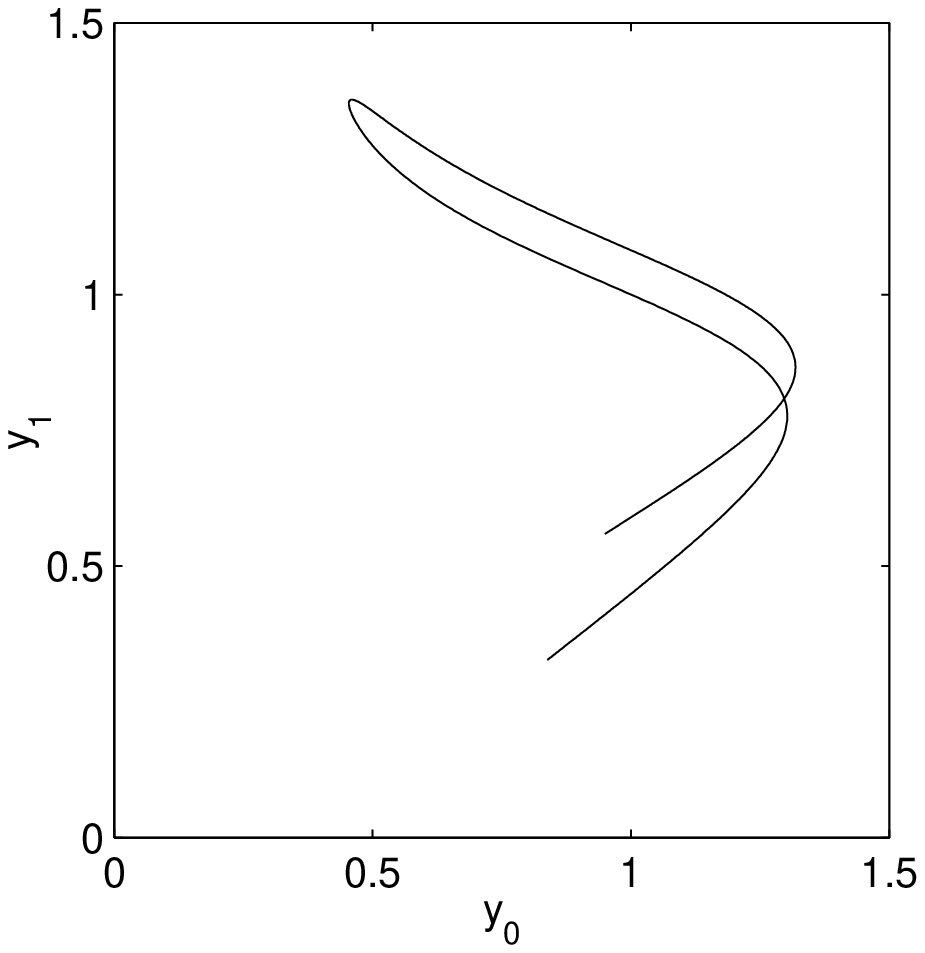}\hfill
\includegraphics[width=1.85in]{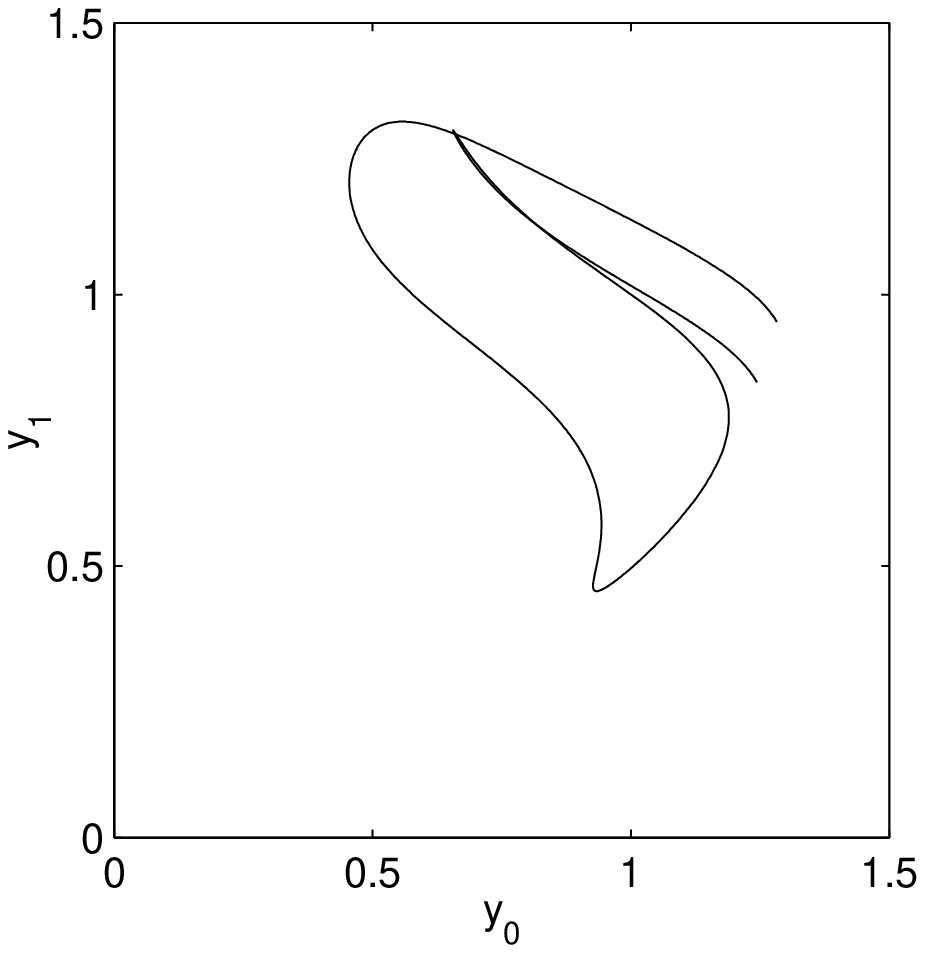}\\[0.25in]
\hspace*{0.0cm}
\includegraphics[width=1.8in]{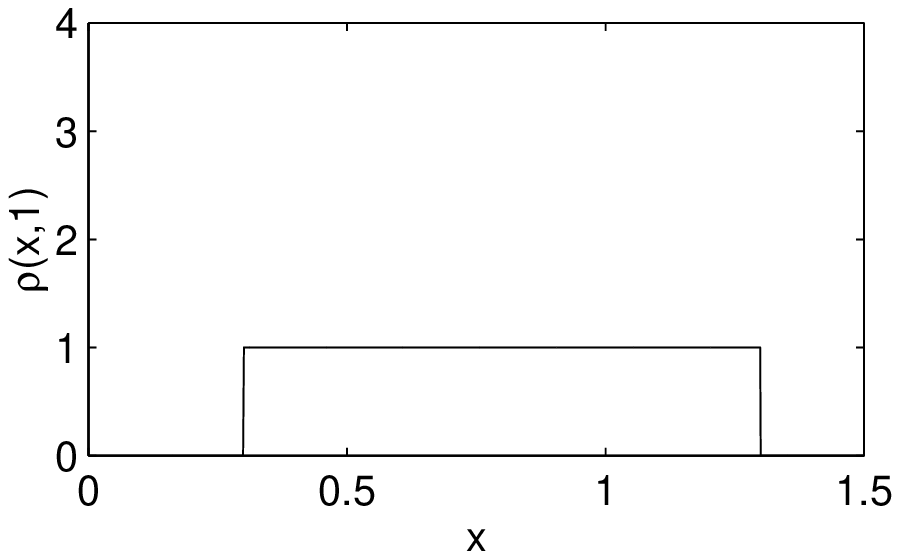}\hfill
\includegraphics[width=1.8in]{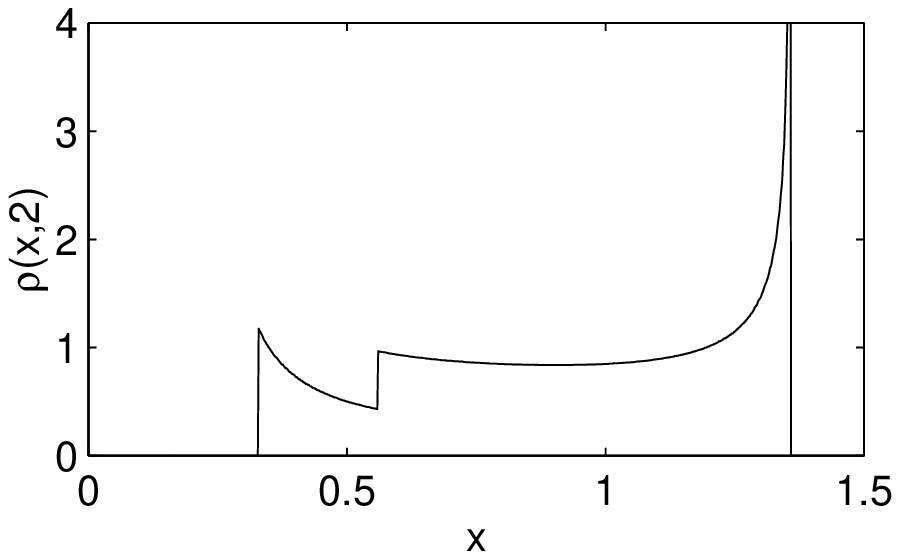}\hfill
\includegraphics[width=1.8in]{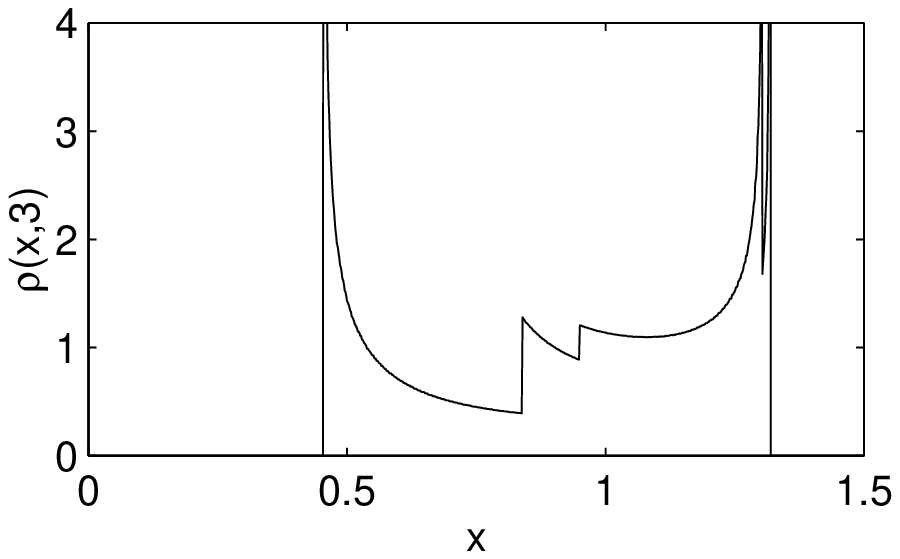}\hfill
\includegraphics[width=1.8in]{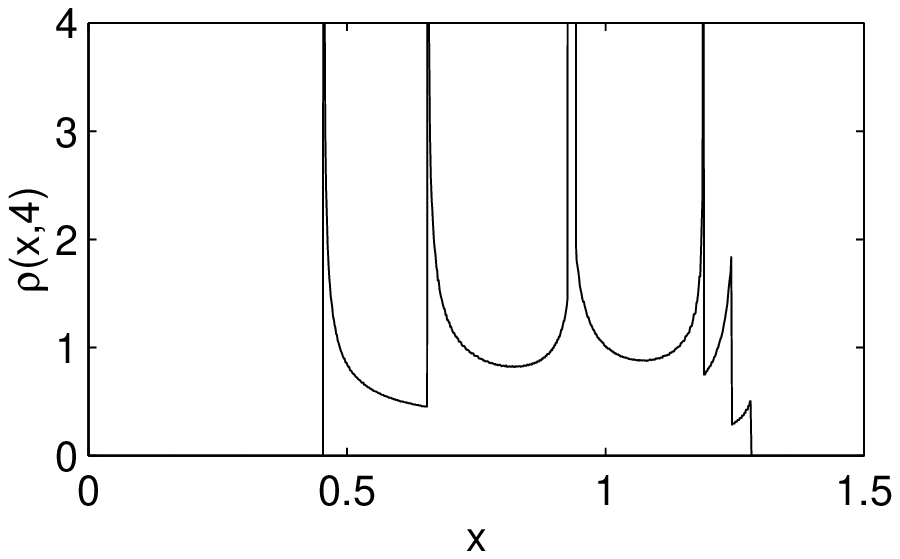}
\caption[Density support curves
for the Mackey-Glass equation restricted to constant initial
functions.]{Density support curves (projected onto the
$(y_0,y_1)$-plane) for the Mackey-Glass equation~\eqref{eq.mg2}
restricted to constant initial functions.  The second row of
graphs shows the corresponding computed densities from
Figure~\ref{fig.denscompare}.}
\label{fig.mgdenscurve}
\end{figure}
\end{landscape}
\begin{figure}
\begin{center}
\includegraphics[width=\figwidth]{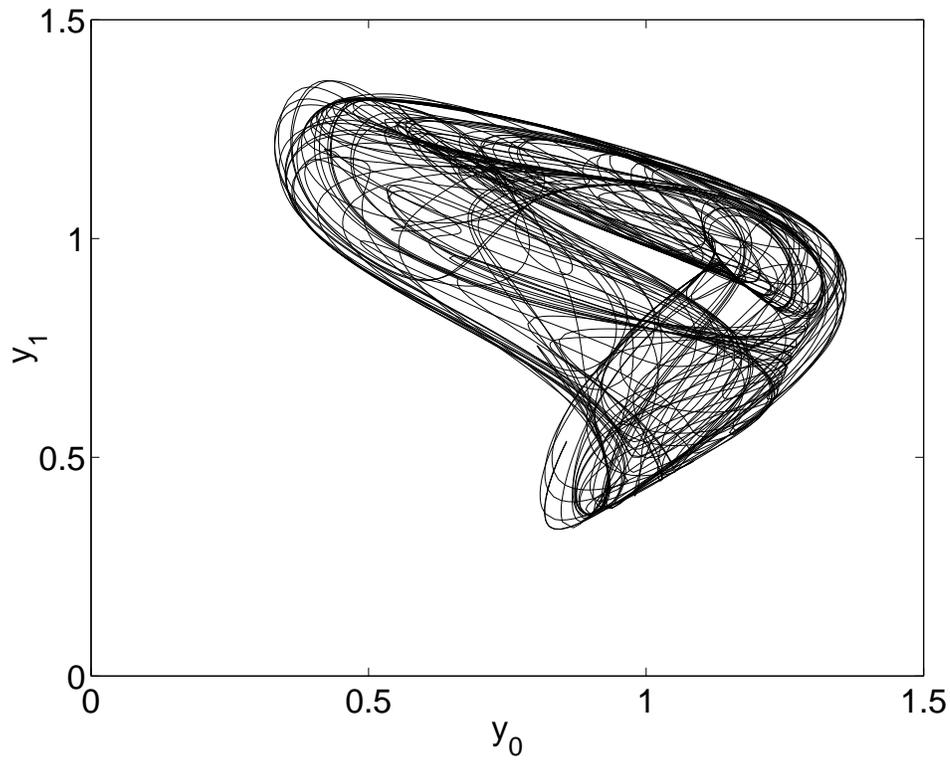}
\caption[Density support curve for the Mackey-Glass equation restricted to constant
initial functions, evolved forward in time to $t=20$.]{Density support
curve (projected onto the $(y_0,y_1)$-plane) for the Mackey-Glass
equation~\eqref{eq.mg2} restricted to constant initial functions,
evolved forward in time to $t=20$.}
\label{fig.denscurvemglong}
\end{center}
\end{figure}

For large $t$ the transformed density support curve becomes very
complicated.  Figure~\ref{fig.denscurvemglong} shows the density
support curve at time $t=20$, which follows in the sequence of
Figure~\ref{fig.mgdenscurve}.  As with the solution map shown of
Figure~\ref{fig.mgmaplong}, the complexity of this curve results from
the repeated stretching and folding that occurs under the dynamics of
the DDE.  Because of this complexity it is difficult to provide a
clear geometric interpretation of the corresponding density, as was
possible for small times as in Figure~\ref{fig.mgdenscurve}.  Also,
just as with the solution map, determining the density support curve
numerically becomes problematic (in fact impossible using finite
precision) for large times.

% =====================================================================
\clearpage
\section{Conclusions}

This chapter has developed a number of approaches to the evolution of
densities for delay equations.  To place the problem in a more
intuitive and mathematically tractable setting, we have considered
delay equations for which ensembles of initial functions are
restricted to some finite-dimensional set.  This results in a family
of measurable solution maps $S_t: \Rn \to \Rn$ for which a
Perron-Frobenius operator can be defined.  However, this restriction
destroys the semigroup structure of the family of solution maps, so
any results that are derived for such systems will have limited
application to an ergodic theory of delay equations.

At least for simple DDEs, the analytical techniques considered in
sections~\ref{sec.fpexplicit} and~\ref{sec.stepsdens} can be used to
derive explicit formulae for the Perron-Frobenius operator
corresponding to $S_t$ (examples~\ref{ex.lin}--\ref{ex.quad}).  For
more complicated equations, non-invertibility of $S_t$ makes both
methods difficult to apply.  As such, there appears to be little hope,
using these methods, of an analytical approach to the evolution of
densities for delay equations with interesting statistical properties,
such as the Mackey-Glass equation and others that exhibit chaotic
solutions.  Nevertheless, Section~\ref{sec.stepsdens} provides an
intuitive model for the evolution of densities for DDEs, in terms of
the transportation of a line mass by a sequence of flows in
$\Reals^N$.  This model gives some geometrical insight into results of
numerical approaches applied to more complicated DDEs (\eg,
Figure~\ref{fig.mgdenscurve}).

In the absence of a generally applicable analytical method, it is
desirable to have an effective computational (numerical) approach to
the evolution of densities for DDEs.  Of the numerical methods
considered (Sections~\ref{sec.histos} and~\ref{sec.approxfp}), the
simplest is the ``brute force'' method of simulating a large ensemble
of solutions, for an ensemble of initial values chosen at random in
accordance with the initial density.  Because this method relies on
adequate statistical sampling to obtain accurate results, it is
computationally intensive to the point of being impractical for many
applications.  This limitation motivates Section~\ref{sec.approxfp},
which develops a method based on a piecewise linear approximation of
the solution map $S_t$, and provides a much more efficient approach to
computing the evolution of densities.

Much of the interest in a probabilistic approach to delay equations is
with regard to their asymptotic statistical properties.  Asymptotic
densities, such as those observed for the Mackey-Glass equation in
Figure~\ref{fig.densmg2}, quantify states of statistical equilibrium.
That is, they describe the long-term equilibrium distribution of
ensembles of systems governed by DDEs.  The same densities also
characterize the long-time statistics of individual solutions.  Thus
invariant densities are important from the points of view of both
statistical mechanics and ergodic theory.

Asymptotic densities can be found by evolving an initial density
forward to large time until the asymptotic statistics become apparent.
Unfortunately none of the methods considered in this chapter, other
than the ``brute force'' ensemble simulation method, is well suited to
evolving densities to large times.  As the time increases, so does the
complexity of the solution map for the DDE (\cf\
Figures~\ref{fig.mgmaplong} and~\ref{fig.denscurvemglong}, and the
examples of Section~\ref{sec.fpexplicit}).  The dimension of the
system increases with time as well (\cf\ Section~\ref{sec.stepsdens}).
These seem to be fundamental obstacles to developing effective (\ie\
fast) numerical techniques for evolving densities to large times and
thereby obtaining asymptotic densities.  None of the methods developed
so far provides a viable alternative to the computationally intensive
ensemble simulation approach.  Other approaches to obtaining
asymptotic densities, which do not rely on evolving densities to large
times, are investigated in next chapter.

\chapter{Asymptotic Densities}\label{ch.invdens}

\begin{singlespacing}
\minitoc   % mini table of contents for this chapter
\end{singlespacing}
\newpage

% ---------------------------------------------------------------------
Chapter~\ref{ch.densev} considered density evolution for the augmented
DDE problem
\begin{equation} \label{eq.dde3}
\begin{split}
  &x'(t) = \begin{cases}
    g\big( x(t) \big) & t \in [0,1] \\
    f\big( x(t), x(t-1) \big) & t \geq 1,
          \end{cases} \\
  &x(0) = x_0,
\end{split}
\end{equation}
which determines the evolution of a quantity
$x(t)\in\Rn$.\footnote{As before, we assume that sufficient
conditions are satisfied to guarantee existence and uniqueness of
solutions, as well as continuity of solutions with respect to the
initial value.}  The role of the function $g$ here is to restrict the
DDE to a particular family of allowable initial functions parametrized
by $x_0$, each a solution of $x'=g(x)$ on $[0,1]$.  An ensemble of
initial values $x_0$ with density $\rho_0$ generates a corresponding
ensemble of solutions $x$ of~\eqref{eq.dde3}, with density $\rho(x,t)$
at time $t$.  Chapter~\ref{ch.densev} considered the problem of
determining the evolution of this density.

Numerical simulation of solution ensembles for the Mackey-Glass
equation (\cf\ Figure~\ref{fig.densmg2}, page~\pageref{fig.densmg2})
suggests that $\rho(x,t)$ approaches a limiting density $\rhostar(x)$
as $t \to \infty$.  In fact, in multiple simulations with different
initial densities and different families of allowable initial
functions (corresponding to different choices of the function $g$),
this same limiting density is observed.  It appears, then, that the
asymptotic density \rhostar\ is an intrinsic property of the DDE.
This same phenomenon can be observed in other delay equations, of
which section~\ref{sec.invex} presents some examples.  These
observations motivate the present chapter, the purpose of which is
to formulate an interpretation of asymptotic densities for DDEs and
to investigate methods for computing such densities.

Section~\ref{sec.invinterp} suggests a theoretical framework to
account for the existence of asymptotic densities for delay equations.
We find that the existence of \rhostar\ is consistent with the
existence of an SRB measure (Definition~\ref{def.SRB}), $\mu$, for the
corresponding infinite dimensional dynamical system.  In this
interpretation \rhostar\ can be seen as the projection of $\mu$ onto
the finite-dimensional space \Rn\ in which the solution variable (\ie,
the physical state) $x(t)$ is observed.

The existence of an asymptotic density has practical significance in
that it characterizes the asymptotic behavior of any ensemble governed
by a given DDE.  Moreover, \rhostar\ also appears to characterize the
asymptotic statistics of every ``typical'' solution of the DDE.  In
light of the important role played by asymptotic densities, it is
desirable to characterize and if possible compute them.
Sections~\ref{sec.ulam}--\ref{sec.approxinv} explore the problem of
computing asymptotic densities for DDEs.  Two methods are considered,
both based on previously published techniques that have proved
successful in the context of some finite-dimensional dynamical
systems.  Section~\ref{sec.ulam} presents an adaptation of the
well-studied ``Ulam's method'', while section~\ref{sec.approxinv}
develops a ``self-consistent Perron-Frobenius operator'' method.

%==========================================================================
\section[Existence of Asymptotic Densities: Ensemble Simulation]{Existence of Asymptotic Densities:\\ Ensemble Simulation}\label{sec.invex}

The simplest and most direct approach to estimating asymptotic
densities is to actually simulate large ensembles of solutions and
investigate their asymptotic statistics, as we have done already in
Figure~\ref{fig.densmg2} for the Mackey-Glass equation.  Even if
finite-precision numerical simulations do not give meaningful
predictions of the fate of individual solutions, there is reason to be
optimistic about the accuracy of statistics collected on large
ensembles of solutions (see section~\ref{sec.shadowing}).

In each of the following examples an ensemble of $10^6$ solutions has
been simulated for a given delay equation\footnote{Numerical solutions
have been performed using the numerical DDE solver DDE23~\cite{ST00}},
and histograms constructed from the ensemble of solution values $x(t)$
for some large $t$.  In each case this histogram, which approximates
the ensemble density $\rho(x,t)$, appears to approach a limiting
density $\rhostar(x)$ as $t \to \infty$.  This same limiting density
is obtained independent of the initial density and the particular
family of allowable initial functions (determined by $g$
in~\eqref{eq.dde3}).

It must be emphasized that asymptotic regularity of the ensemble
dynamics does not imply regularity of individual solutions.  On the
contrary, statistical regularity is closely tied to disordered
behavior of individual trajectories (\cf\ section~\ref{sec.ergth}).
The asymptotic density characterizes a statistical rather than a
dynamical equilibrium.  To emphasize this point, in each example below
we illustrate a single long-time solution typical of other solutions
represented in the ensemble, confirming that the solutions themselves
appear to exhibit a random character, despite the eventual invariance
the ensemble density.

\begin{example}[Mackey-Glass equation] \label{ex.mg}
Densities for the Mackey-Glass equation~\cite{MG77}
\begin{equation} \label{eq.mg3}
\begin{gathered}
  x'(t) = -\alpha x(t) + \beta \frac{x(t-1)}{1 + x(t-1)^n},\\
  \alpha=2, \quad \beta=4, \quad n=10,
\end{gathered}
\end{equation}
were estimated by ensemble simulation in Chapter~\ref{ch.densev} (\cf\
Figures~\ref{fig.densmg1}--\ref{fig.densmg2}).  The histogram
approximating the limiting density $\rho_{\ast}$, together with part
of a typical asymptotic solution, is shown in Figure~\ref{fig.invmg}
\begin{figure}
\begin{center}
\includegraphics[width=4in]{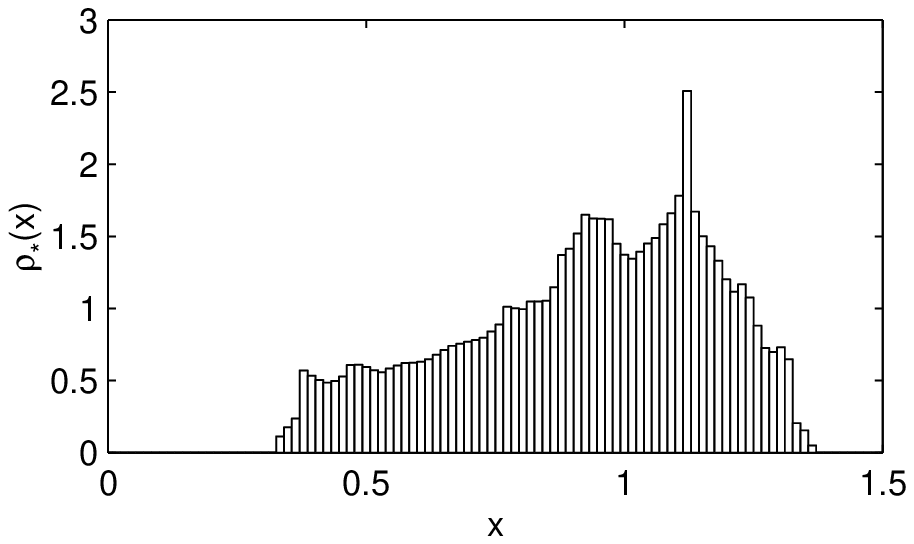}\\[0.25in]
\hspace*{0.1in}\includegraphics[width=4in]{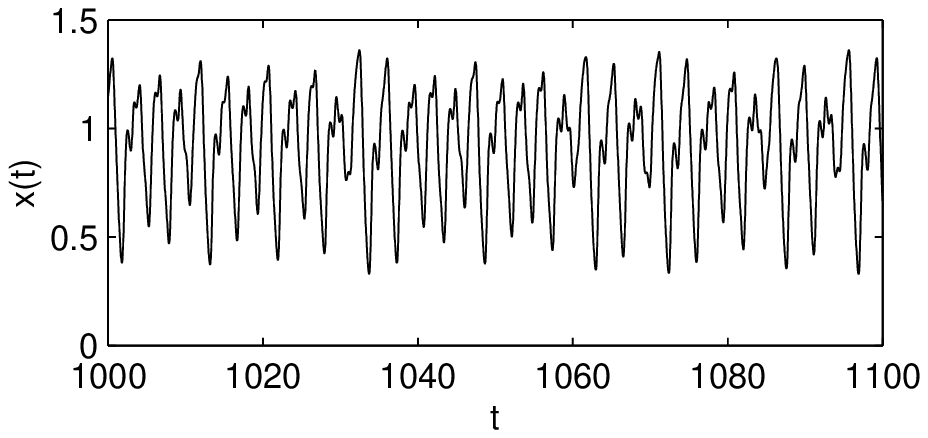}
\caption[Histogram approximating the asymptotic density for the
Mackey-Glass equation, obtained from a simulated ensemble of $10^6$
large-time numerical solutions.  Also shown is a typical solution
represented in the ensemble.]{Histogram approximating the asymptotic
density for the Mackey-Glass equation~\eqref{eq.mg3}, obtained from a
simulated ensemble of $10^6$ large-time numerical solutions.  Also
shown is a single solution typical of those represented in the
ensemble at large times.}
\label{fig.invmg}
\end{center}
\end{figure}
\end{example}

\begin{example}[Piecewise-constant nonlinearity] \label{ex.pwc}
The delay equation
\begin{equation} \label{eq.pwc}
\begin{split}
  &x'(t) = -\alpha x(t) + F\big( x(t-1) \big), \\
  &F(x) = \begin{cases} c & \text{if $x \in [x_1,x_2]$} \\
    0 & \text{otherwise}.
  \end{cases}
\end{split}
\end{equation}
has been studied previously in \cite{adH85,HM82,HW83,LML93}.  Despite
the simplicity of the piecewise-constant feedback term, solutions of
this equation are known to exhibit a wide variety of behaviors as the
parameters are varied.

Equation~\eqref{eq.pwc} can be reduced to an ordinary differential on
a sequence of intervals, on each of which it easy to construct an
analytical solution, \viz,
\begin{equation}
  x(t) = \begin{cases}
    \tfrac{c}{\alpha} +
 \big( x(t_0)-\tfrac{c}{\alpha} \big) e^{-\alpha(t-t_0)} &
 \text{if $x(t-1) \in [x_1,x_2]$} \\
 x(t_0) e^{-\alpha(t-t_0)} & \text{otherwise}.
  \end{cases}
\end{equation}
The dependence on $x(t-1)$ occurs only through the ``crossing times''
at which $x(t-1)=x_1$ or $x_2$ (where the forcing term switches on or
off).  In fact, the solution of~\eqref{eq.pwc} for $t \geq 1$ is
uniquely determined by the values $x(0)$, $x(1)$, and the set of
crossing times in the interval $[0,1]$.  This simplification
facilitates an analytical treatment, to the extent that the existence
of limit cycles, a homoclinic orbit, and chaos (in the sense of Li and
Yorke~\cite{LY75}) have been proved for certain
parameters~\cite{HM82}.  It has also been proved~\cite{adH85} that,
for certain parameters, the map governing the evolution of the
crossing times is exact (\cf\ definition~\ref{defn.exact},
page~\pageref{defn.exact}).  To date this is the only rigorous result
on strong ergodic properties for a delay differential equation.

Figure~\ref{fig.invpwc} shows the asymptotic density obtained by a
histogram of $10^6$ long-time solutions of~\eqref{eq.pwc}, with
parameter values $x_1=1$, $x_2=2$, $\alpha=6$, $c=24$ (which were also
considered in~\cite{HM82}).  The typical form of an asymptotic
solution is also shown.
\begin{figure}[p]
\begin{center}
\includegraphics[width=4in]{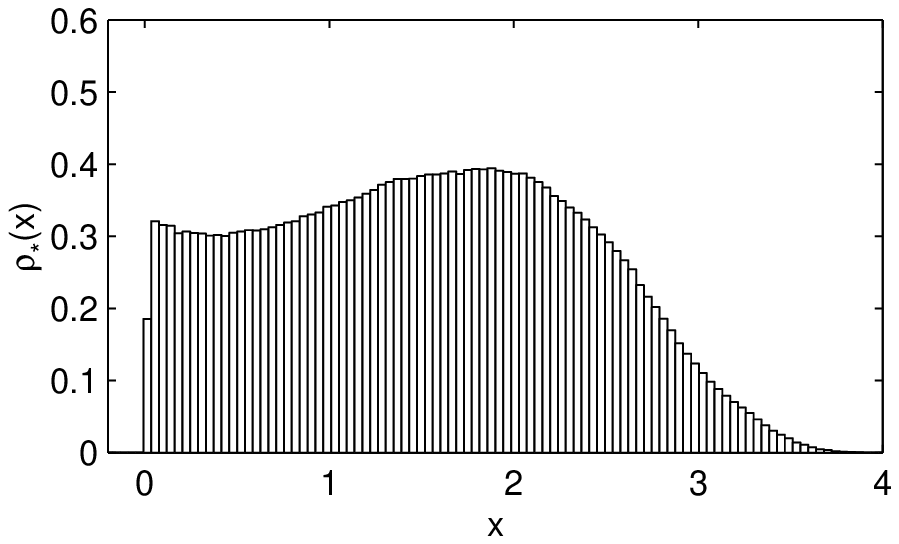}\\[0.25in]
\hspace*{0.3in}\includegraphics[width=4in]{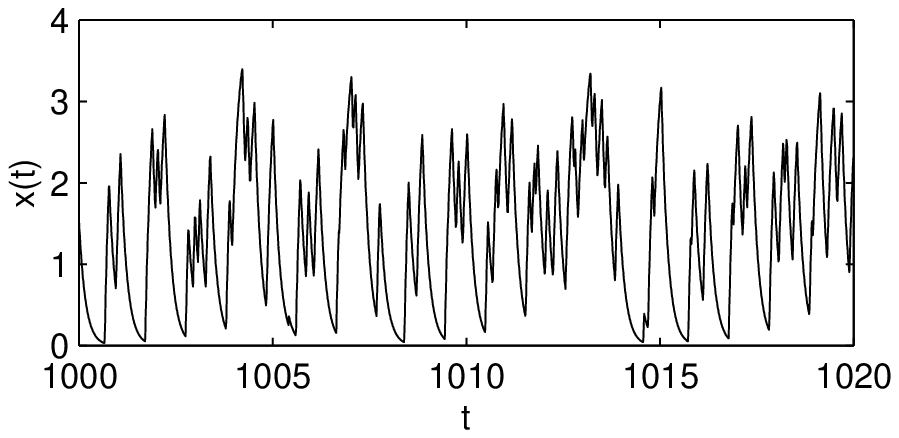}
\caption[Histogram approximating the asymptotic density for the delay
equation with piecewise-constant feedback, obtained from a simulated
ensemble of $10^6$ large-time numerical solutions.  Also shown is a
typical solution represented in the ensemble.]{Histogram approximating
the asymptotic density for the delay equation~\eqref{eq.pwc} with
piecewise-constant feedback, obtained from a simulated ensemble of
$10^6$ large-time numerical solutions.  Also shown is a typical
solution represented in the ensemble.}
\label{fig.invpwc}
\end{center}
\end{figure}
\end{example}

\begin{figure}[p]
\begin{center}
\includegraphics[width=4in]{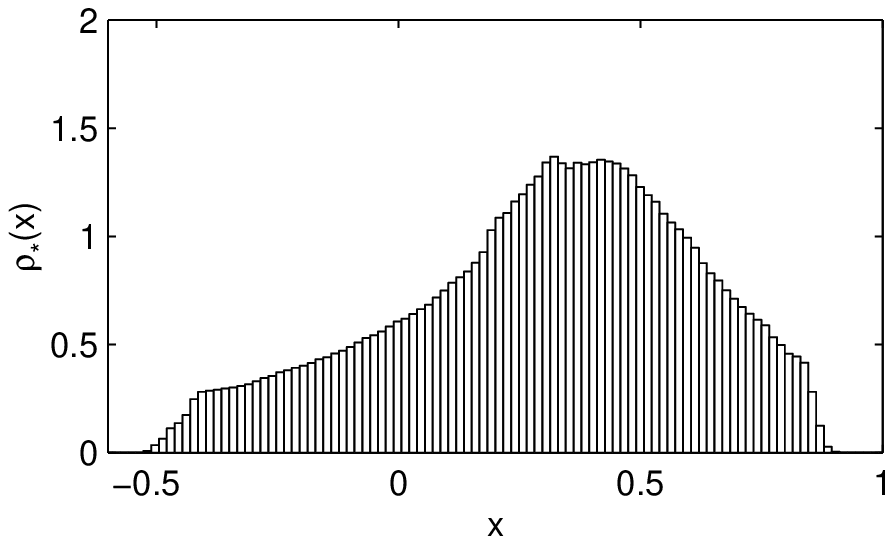}\\[0.25in]
\hspace*{0.1in}\includegraphics[width=4in]{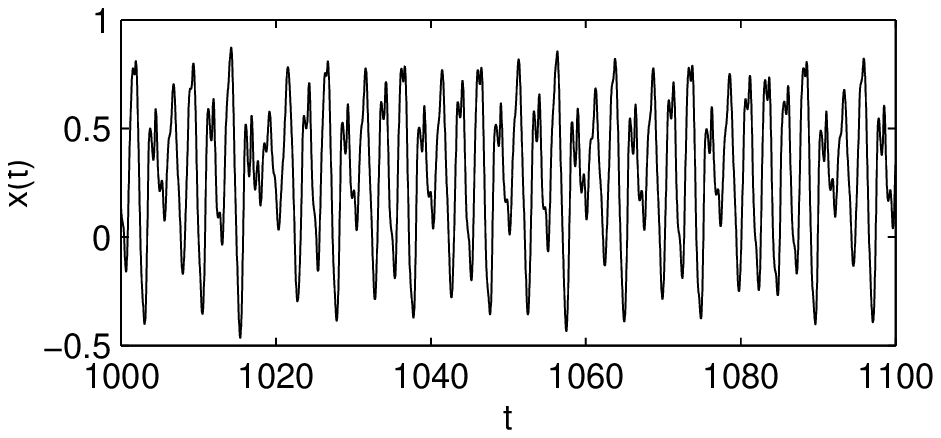}
\caption[Histogram approximating the asymptotic density for the delay
equation with piecewise-linear feedback, obtained from a simulated
ensemble of $10^6$ large-time numerical solutions.  Also shown is a
typical solution represented in the ensemble.]{Histogram approximating
the asymptotic density for the delay equation~\eqref{eq.pwl} with
piecewise-linear feedback ($\eps=0.3$), obtained from a simulated
ensemble of $10^6$ large-time numerical solutions.  Also shown is a
typical solution represented in the ensemble.}
\label{fig.invpwl}
\end{center}
\end{figure}
\begin{example}[``Tent map'' nonlinearity] \label{ex.pwl}
The delay equation
\begin{equation} \label{eq.pwl}
  \eps x'(t) = -x(t) + 1 - 1.9 |x(t-1)|.
\end{equation}
has been previously been studied from a probabilistic point of view
in~\cite{Ersh91}.  For sufficiently small $\eps$ the solutions of
DDE~\eqref{eq.pwl} appear to have chaotic solutions.
Figure~\ref{fig.invpwl} shows the asymptotic density obtained by a
histogram over $10^6$ large time solutions, together with a segment of
a typical long-time solution when $\eps=0.3$.
\end{example}

%=========================================================================
\clearpage
\section{Ergodic Theoretic Interpretation} \label{sec.invinterp}

The existence of asymptotic densities in the examples above can be
explained in terms of ergodic properties of the corresponding
dynamical system on $C([-1,0])$.  While we are not in a position to
prove that a given DDE actually possesses strong ergodic properties,
this at least provides one framework for interpreting the dynamics of
ensembles of DDE solutions.

% ------------------------------------------------------------------------
\subsection{DDE solution is the trace of a dynamical system} \label{sec.trace1}

Recall (\cf\ Chapter~\ref{ch.framework}) that the DDE
\begin{equation} \label{eq.dde4}
  x'(t) = f\big( x(t), x(t-1) \big), \quad x(t) \in \Rn, \quad t \geq 0,
\end{equation}
can be interpreted as a dynamical system on the phase space $C$ of
continuous functions from $[-1,0]$ into \Rn.  The corresponding
semigroup $\{S_t: t \geq 0\}$ of evolution operators $S_t:C \to C$
defined by
\begin{equation} \label{eq.ddesg2}
  (S_t \phi)(s) = x(t+s), \quad s \in [-1,0],
\end{equation}
where $x$ is the solution of~\eqref{eq.dde4} with initial function
\begin{equation}
  x(t) = \phi(t), \quad t \in [-1,0].
\end{equation}

For any given initial function $\phi \in C$ there corresponds a
trajectory $\{S_t\phi: t \geq 0\} \subset C$ for this dynamical
system.  Let $x_t \in C$ denote the phase point on this trajectory at
time $t$.  That is, $x_t$ is the function
\begin{equation} \label{eq.xtdef}
  x_t(s) \equiv ( S_t \phi )(s) = x(t+s), \quad s \in [-1,0].
\end{equation}
Then $x(t)$ can be expressed as
\begin{equation}
  x(t) = \pi(x_t),
\end{equation}
where the functional $\pi: C \to \Rn$ is given by
\begin{equation} \label{eq.pidef}
  \pi(u) = u(0).
\end{equation}
This gives an interpretation of the solution variable $x(t)$ as the
image under $\pi$ of the phase point of the corresponding dynamical
system.

In general if $\{S_t: t \geq 0\}$ is a dynamical system on $Y$ then a
function $h: \Reals \to X$ is called a \emph{trace} of $\{S_t\}$ if
there is a continuous map $\pi: Y \to X$ and a $y \in Y$ such that
\begin{equation}
  h(t) = \pi\big( S_t(y) \big), \quad t \geq 0.
\end{equation}
In other words, $h(t)$ is the continuous image under $\pi$ of some
trajectory of the dynamical system.  A familiar example of a trace is
the projection onto two dimensions of a trajectory of a
three-dimensional dynamical system.  See~\cite[p.\,193]{LM94} for
further discussion of this concept.

Thus we can interpret a solution $x(t)$ of the DDE~\eqref{eq.dde4} as
a trace of the corresponding dynamical system.  Indeed,
equation~\eqref{eq.ddesg2} together with~\eqref{eq.pidef} gives
\begin{equation}
  x(t) = \pi(S_t \phi ), \quad t \geq 0,
\end{equation}
so $x(t)$ is the trace of the trajectory of $\{S_t\}$ through $\phi$.
It is readily verified that $\pi$ is continuous if $C$ is equipped
with the sup norm.  With this interpretation it is straightforward to
show how various properties of the dynamical system $\{S_t\}$ are
manifested as corresponding properties of solutions of the DDE.  In
particular we have the following:
\begin{itemize}
\item If $\{S_t\}$ has an invariant measure, $\mu$, then for an
ensemble of phase points $x_t \in C$ distributed according to $\mu$,
the corresponding ensemble of DDE solutions $x=\pi(x_t)$ will be
distributed according to the probability measure
$\pi(\mu)\equiv\mu\circ\pi^{-1}$ on \Rn, and this distribution will
be invariant under the dynamics.
\item If $\{S_t\}$ has an attractor $\Lambda \subset C$, then $x(t)$
lies asymptotically on the image $\pi(\Lambda)\subset\Rn$.
\item If $\Lambda$ supports an SRB measure $\mu$, then we expect that
any solution ensemble will be asymptotically distributed according to
the measure $\pi(\mu)$.  This provides an explanation of the
convergence of ensemble histograms in the examples above.
\end{itemize}
The following sections explore these connections in greater detail.

% -----------------------------------------------------------------------
\subsection{Evidence of an invariant measure} \label{sec.ddeinv}

In the examples above, by simulating ensembles of solutions we have
found evidence for the existence of asymptotic densities for delay
equations.  That is, for a given density of initial values $x_0$ in
the DDE initial value problem~\eqref{eq.dde3}, the density $\rho(x,t)$
of the ensemble of solution values $x$ at time $t$ evolves toward a
seemingly unique density $\rhostar(x)$ as $t \to \infty$.

It is tempting to use the terminology of chapter~\ref{ch.ergth} and
call the limiting density \rhostar\ an ``invariant density''.  This
turns out to be inappropriate, since \rhostar\ is in fact not truly
invariant in the strict sense already defined: if an ensemble of
initial values $x_0$ is distributed with density \rhostar, the
subsequent evolution of the ensemble density $\rho(x,t)$ does
\emph{not} agree with \rhostar\ for all time.
Figure~\ref{fig.mgnoninv} illustrates this fact.  Here we consider the
Mackey-Glass equation~\eqref{eq.mg3} restricted to constant initial
functions (hence $g=0$ in~\eqref{eq.dde3}).  An ensemble of $10^6$
solutions has been simulated, corresponding to an ensemble of initial
values $x_0$ distributed according to the asymptotic density \rhostar\
shown in Figure~\ref{fig.invmg}.  From the resulting sequence of
histograms, which approximate $\rho(x,t)$ at times $t=0,1,2$ and
$100$, it is apparent that $\rho(x,t)$ initially diverges from
\rhostar.  Hence \rhostar\ is not appropriately described as an
invariant density.

Nevertheless, as can be seen from the histogram representing
$\rho(x,100)$ in Figure~\ref{fig.mgnoninv}, $\rho(x,t)$ does appear to
eventually converge to \rhostar\ once again.  As $t$ increases beyond
about $100$ the ensemble histograms (not shown here) agree with
\rhostar.  Thus it seems appropriate to call \rhostar\ an
``asymptotically invariant'' density.
\begin{figure}
\begin{center}
\includegraphics[height=1.6in]{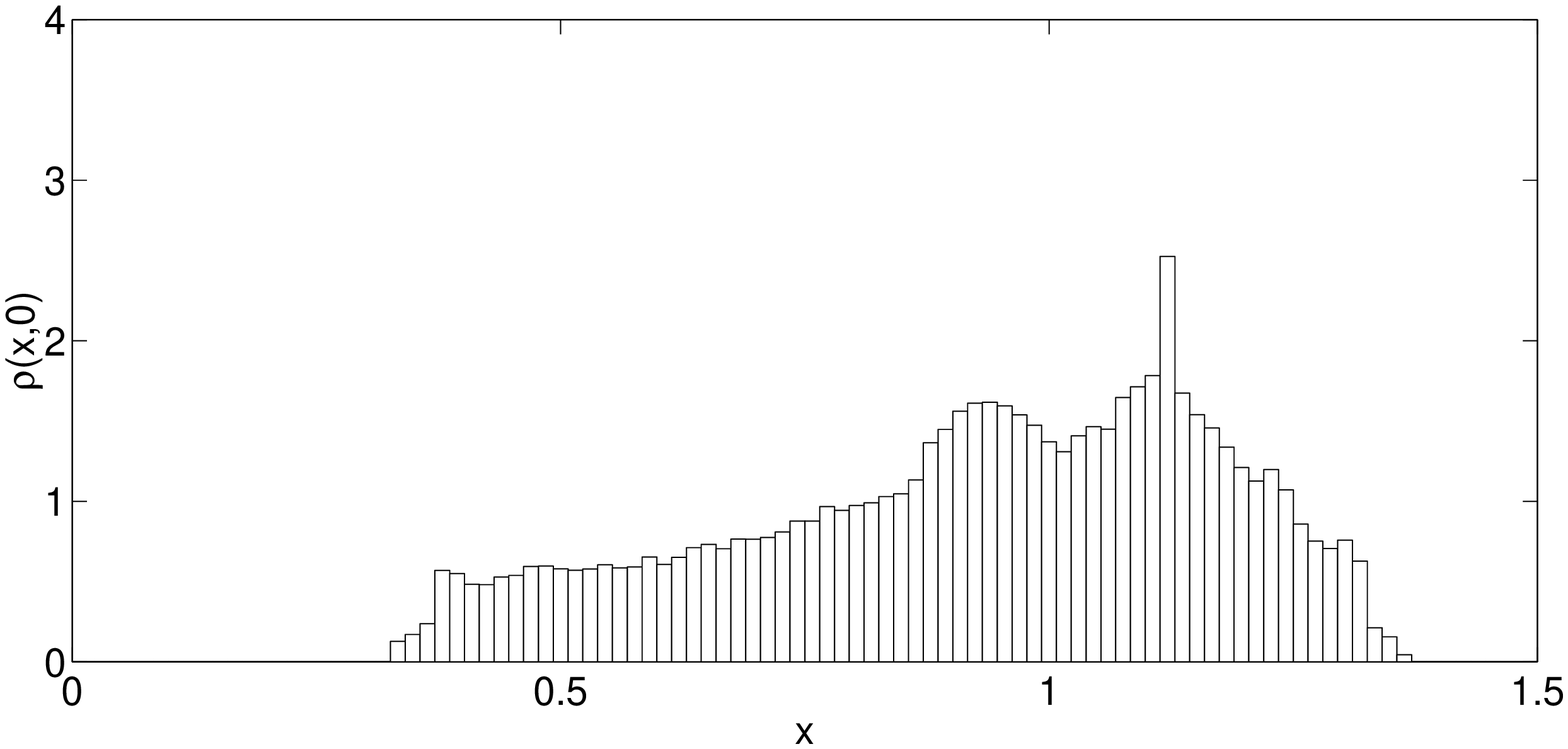} \\[0.1in]
\includegraphics[height=1.6in]{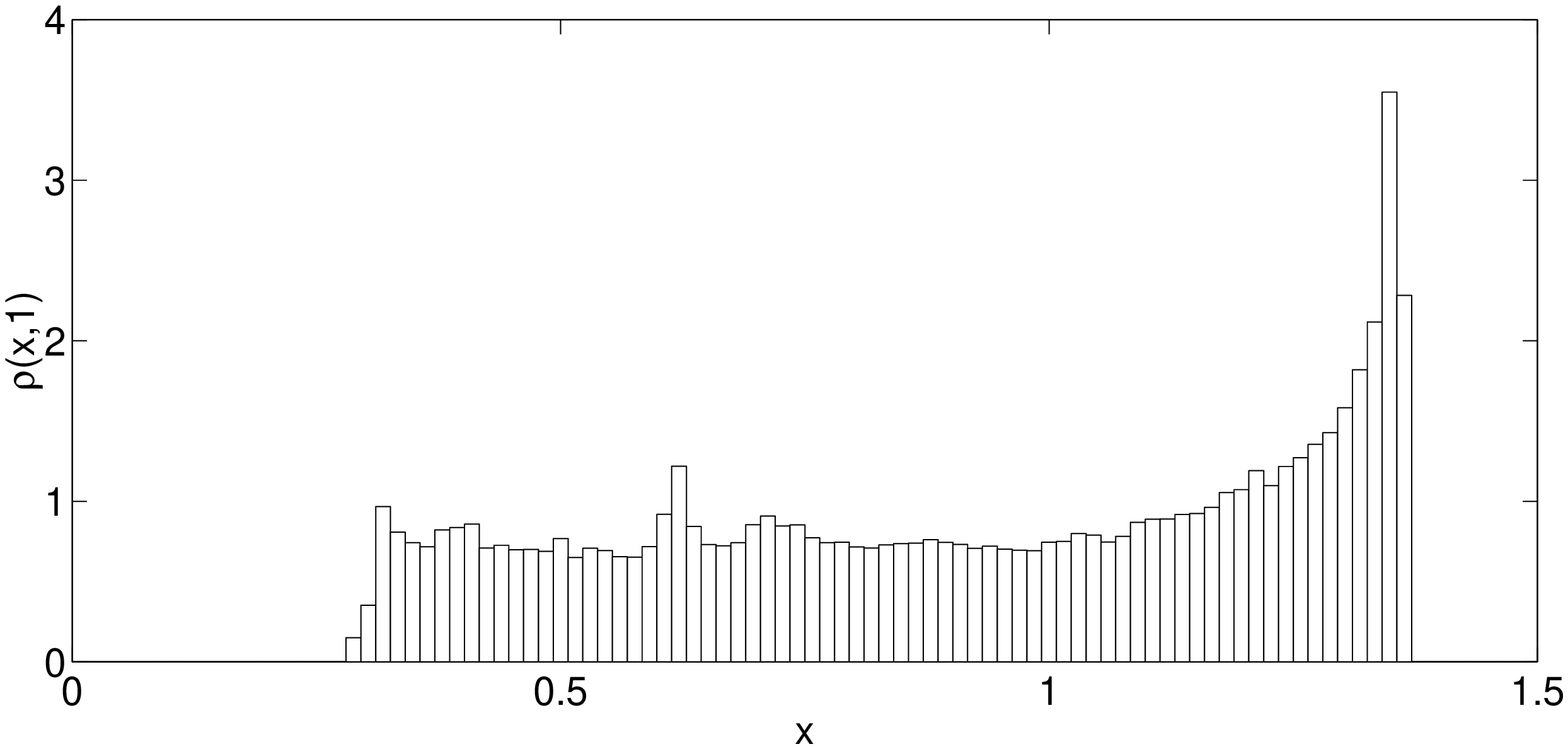} \\[0.1in]
\includegraphics[height=1.6in]{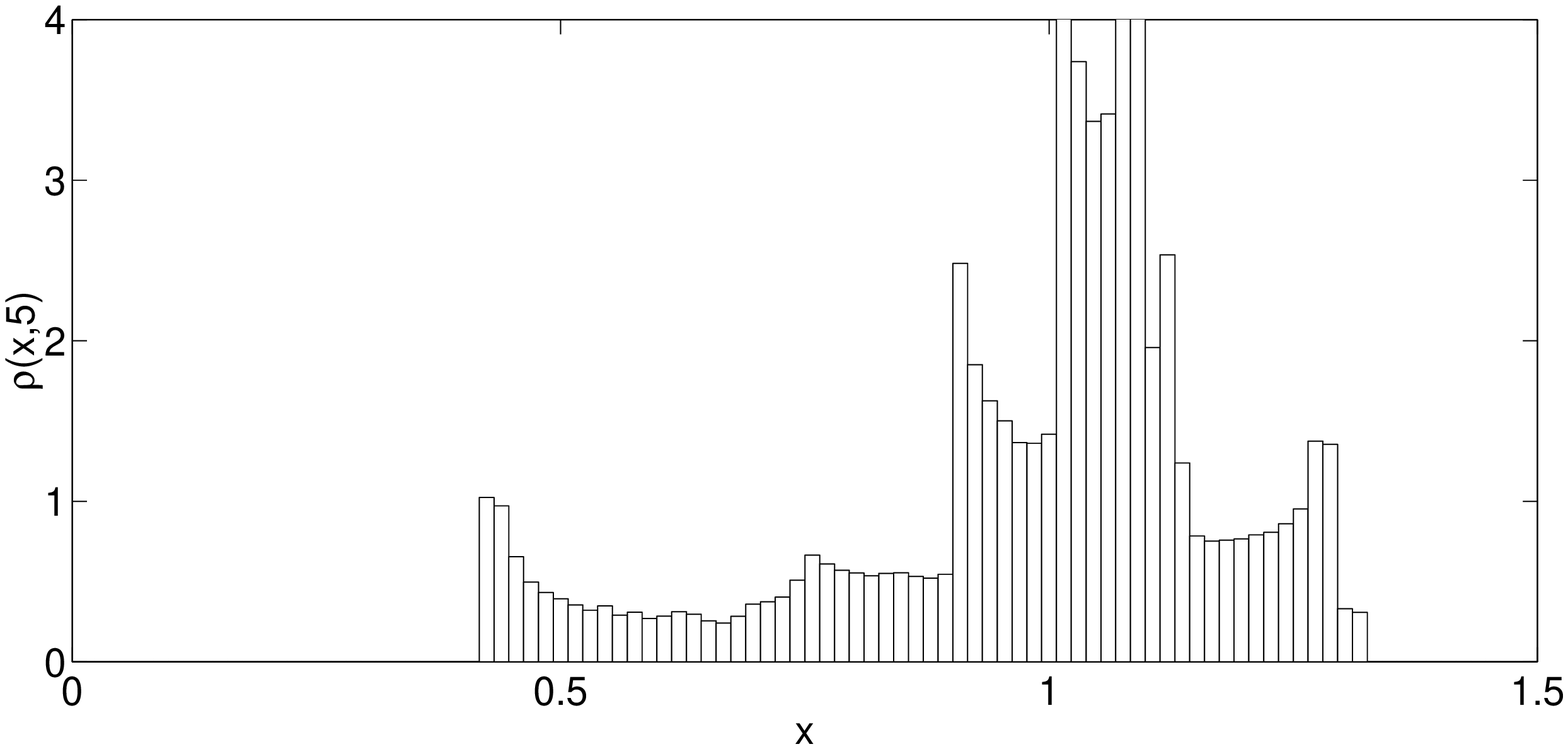} \\[0.1in]
\includegraphics[height=1.6in]{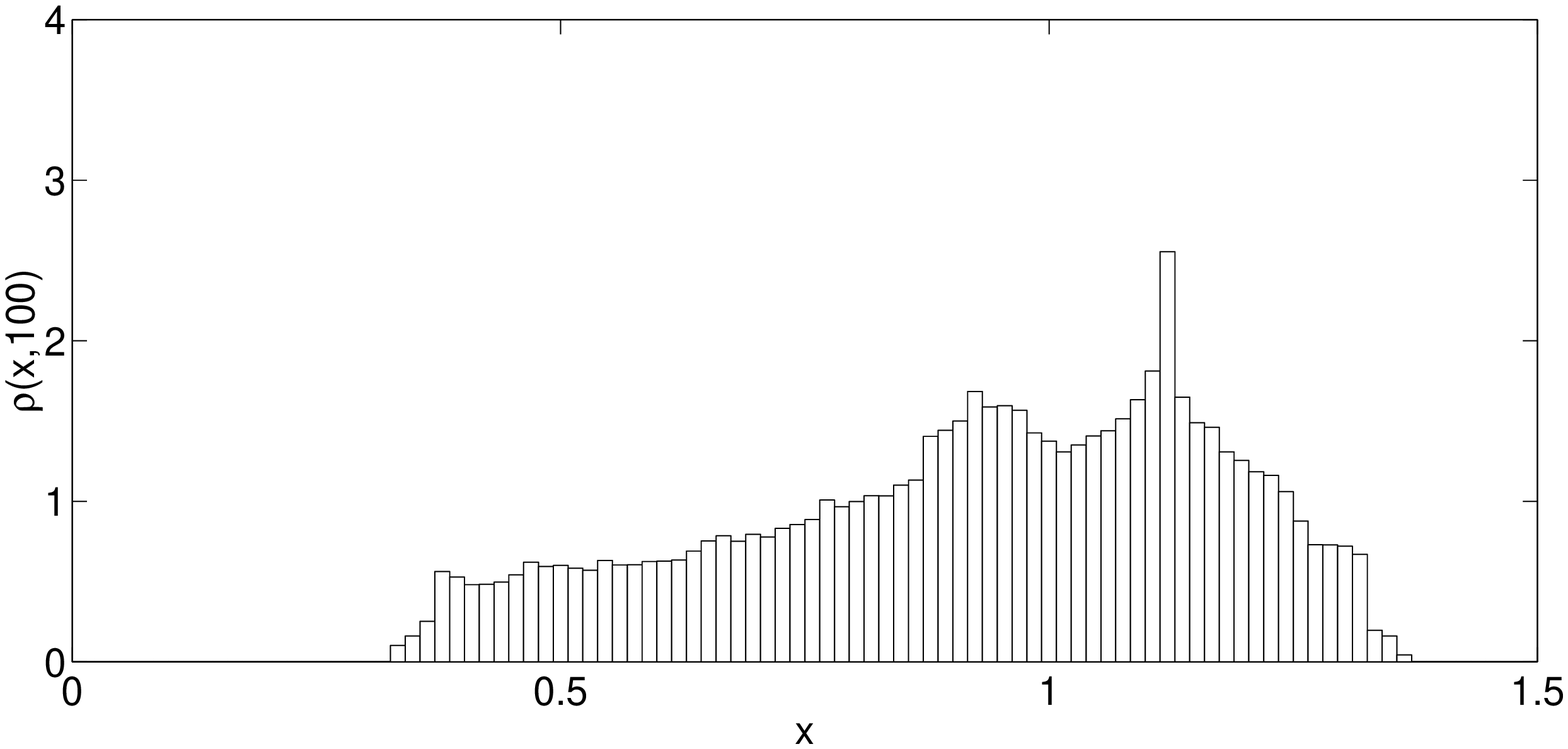}
\caption[Histograms of solution values $x(t)$ for an ensemble of $10^6$
solutions of the Mackey-Glass equation, corresponding to an ensemble
of constant initial functions with values distributed according to the
asymptotic density \rhostar.]{Histograms of solution values $x(t)$
for an ensemble of $10^6$ solutions of the Mackey-Glass
equation~\eqref{eq.mg3}, corresponding to an ensemble of constant
initial functions with values distributed according to the asymptotic
density \rhostar\ depicted in Figure~\ref{fig.invmg}.  Histograms
are shown for times $t=0,1,5,100$.}
\label{fig.mgnoninv}
\end{center}
\end{figure}

This phenomenon has a simple explanation if the infinite dimensional
dynamical system corresponding to the delay equation has an invariant
measure.  Indeed, suppose the dynamical system $\{S_t: t \geq 0\}$ (as
defined in equation~\eqref{eq.ddesg2}) has an invariant measure
\mustar, \ie,
\begin{equation}
  \mustar = S_t(\mustar) \equiv \mustar \circ S_t^{-1} \quad \forall t
  \geq 0.
\end{equation}
If an ensemble of phase points $x_t \in C$ is distributed according to
\mustar, then the corresponding ensemble of solution values $x(t) \in \Rn$
will be distributed according to the measure
\begin{equation}
  \etastar = \pi(\mustar) \equiv \mustar \circ \pi^{-1},
\end{equation}
since $x(t) = \pi(x_t)$ where the trace map $\pi: C \to \Rn$ is
defined by equation~\eqref{eq.pidef}.  Under the evolution prescribed
by the delay equation the distribution of solution values $x(t)$ does
not change with time since, under evolution by $S_t$, \etastar\
transforms to
\begin{equation}
  \pi\big( S_t (\mustar) \big) = \pi( \mustar ) = \etastar.
\end{equation}
This is just what we see with the asymptotic histograms shown in
Figures~\ref{fig.invmg}--\ref{fig.invpwl}, and suggests the
interpretation of these histograms as approximating measures \etastar\
that are just projections of \emph{invariant} measures \mustar\ on the
phase space $C$.

As discussed above, \etastar\ cannot be considered an invariant
measure for the DDE, since an ensemble of initial values distributed
according to \etastar\ does not necessarily remain distributed
according to \etastar.  This can now be understood as a consequence of
the fact that the trace map $\pi$ is not one-to-one.  That is, there
can be measures on $C$, other than \mustar, whose images under $\pi$
coincide with \etastar.  In general these measures will not be
invariant.  In particular, there is an ensemble of constant initial
functions with values distributed according to \etastar\
(corresponding to the first histogram in Figure~\ref{fig.mgnoninv}),
but this ensemble is not invariant under $S_t$.  As this ensemble
evolves, its projection under $\pi$ (corresponding to the subsequent
histograms in Figure~\ref{fig.mgnoninv}) diverge from \etastar.

The foregoing considerations show that by interpreting a DDE solution
$x(t)$ as the trace of the corresponding dynamical system $\{S_t\}$ in
$C$, we obtain a framework in which statistical properties of DDE
solutions can be understood in terms of statistical properties of
$\{S_t\}$.  In particular, existence of an invariant measure \mustar\
for $\{S_t\}$ implies existence of a corresponding measure \etastar\
on \Rn\ that is invariant under the DDE dynamics.  This suggests one
explanation for the origin of the asymptotic densities shown in
Figures~\ref{fig.invmg}--\ref{fig.invpwl}.  The following sections
show how the trace map $\pi$ carries over other properties of
trajectories in $C$ to corresponding properties of DDE solutions in
\Rn.

% -----------------------------------------------------------------------
\subsection{Existence of an attractor}

\begin{thm}
Suppose $F: \Reals \to \Reals$ maps an interval $I=(-k,k)$ into itself
for some $k$.  Let the semigroup $\{S_t: t \geq 0\}$ be defined
by~\eqref{eq.ddesg2} for the delay equation
\begin{equation} \label{eq.ddeattr}
  x'(t) = -\alpha x(t) + F\big( x(t-1) \big), \quad t \geq 1
\end{equation}
with $\alpha \geq 1$.  Then $\{S_t\}$ has an attractor, $\Lambda$.
\nomenclature[L]{$\Lambda$}{Attractor of a dynamical system}%
That is~\cite{Ruelle81}, there is a compact set $\Lambda \subset C$
and a neighborhood $U$ of $\Lambda$ such that
\begin{enumerate}[(a)]
\item For every neighborhood $V$ of $\Lambda$, $S_t(U) \subset V$ for
all sufficiently large $t$.
\item $S_t(\Lambda) = \Lambda \; \forall t \geq 0$.
\item $\Lambda = \bigcap_{t \geq 0} S_t(U)$.
\end{enumerate}
\end{thm}

\begin{proof}
(Adapted from a sketch given in~\cite{Ersh91}).
By~\cite[Prop.\,3.2]{Ruelle81} it is sufficient to show that for some
open $U \subset C$, for all $t>1$ $S_t(U)$ is relatively compact and contained in $U$.
To this end let
\begin{equation}
  U = \{u \in C: |u| < k/\alpha\}.
\end{equation}
Recall (Section~\ref{sec.explmap}) that the time-one map $S_1: C \to
C$ can be written
\begin{equation} \label{eq.expmap}
  (S_1 u)(s) = u(0) e^{-\alpha(s+1)} + \int_{-1}^s {e^{\alpha(t-s)}
   F\big( u(t) \big) \, dt}, \quad s \in [-1,0].
\end{equation}
Thus for $u \in U$ we have
\begin{equation}
\begin{split}
  \big|(S_1 u)(s)\big|
   &\leq \Big| u(1) e^{-\alpha (s+1)} \Big| + \Big| \int_{-1}^s {e^{\alpha(t-s)}
       F\big( u(t) \big) \, dt} \Big| \\
   &< \frac{k}{\alpha} e^{-\alpha (s+1)} + \int_{-1}^s {e^{\alpha(t-s)} k \, dt} \\
   &= \frac{k}{\alpha} e^{-\alpha (s+1)} + \frac{k}{\alpha} \big(1 - e^{-\alpha (s+1)}\big) \\
   &= k/\alpha.
\end{split} 
\end{equation}
By~\eqref{eq.ddesg2} we have, $\forall t \in [0,1]$,
\begin{equation}
  (S_t \phi)(s) = \begin{cases} \phi(s+t) & \text{if $s \in [-1,-t]$} \\
                 (S_1 \phi)(s+t-1) & \text{if $s \in (-t,0]$}
	       \end{cases}
\end{equation}
so that $|(S_t u)(s)| < k/\alpha$ $\forall t \in [0,1]$.  The
semigroup property then implies
\begin{equation}
  S_t(U) \subseteq U \quad \forall t \geq 0.
\end{equation}
Furthermore, if $\phi \in U$ then the solution $x(t)$
of~\eqref{eq.ddeattr} with initial function $\phi$ satisfies, for $t
\in [0,1]$,
\begin{equation}
\begin{split}
  |x'(t)| &\leq |-\alpha x(t)| + |F\big(\phi(t-1)\big)| \\
          &\leq \alpha |x(t)| + k,
\end{split}
\end{equation}
so for $u \in U$ we have
\begin{equation}
  (S_t u)'(s) \leq \alpha \frac{k}{\alpha} + k = 2 k.
\end{equation}
Thus for $t \geq 1$ every element of $S_t(U)$ has a bounded
derivative, hence the set $S_t(U)$ is equicontinuous and therefore
relatively compact (by the Arzela-Ascoli
Theorem~\cite[p.\,57]{Lang93}).  Thus by~\cite[Prop.\,3.2]{Ruelle81}
the set
\begin{equation}
  \Lambda = \bigcap_{t \geq 0} S_t(U)
\end{equation}
is compact and satisfies (a)--(c) above.
\end{proof}

In each of the examples given in Section~\ref{sec.invex} the delay
equation satisfies the conditions of the theorem above (in particular
it suffices that $F$ be bounded), so the corresponding dynamical
system has a compact attractor $\Lambda \subset C$.  Thus for any
initial function $\phi$ in some open ball $U \subset C$, the
trajectory $\{S_t \phi: t \geq 0\} \subset C$ lies asymptotically on
(or near) $\Lambda$.  The proof gives an explicit formula for the
radius of $U$ in terms of $\alpha$ and the radius of the ``maximal
invariant interval'' $I$ such that $F(I) \subseteq I$.  The full basin
of attraction of $\Lambda$, $W=\bigcup_{t \geq 0} S_t^{-1}(U)$, might
actually be much larger than $U$.

Since any initial function $\phi \in W$ has
$\text{dist}(S_t(\phi),\Lambda)
\to 0$ as $t \to \infty$, continuity of the trace map $\pi$ implies
that
\begin{equation}
  \text{dist}\big(x(t),\pi(\Lambda)\big) = \text{dist}\big(\pi(S_t
   \phi),\pi(\Lambda)\big) \to 0 \quad \text{as} \quad t \to \infty.
\end{equation}
That is, the solution $x(t)$ corresponding to the initial function $\phi$
lies asymptotically in the image of $\Lambda$ under the mapping $\pi:
C \to \Rn$.

If $S_t$ possesses an invariant measure \mustar\ describing the
asymptotic statistics of its trajectories then, since trajectories in
$C$ lie asymptotically on $\Lambda$, \mustar\ will be concentrated on
$\Lambda$.  According to the previous section, there is a
corresponding measure $\etastar = \pi(\mustar)$ on \Rn\ that is
invariant under the dynamics, and this measure will be concentrated on
$\pi(\Lambda)$.  In particular, as the following section shows, if
$\Lambda$ carries an SRB measure that characterizes the distribution
of orbits on $\Lambda$, then the image of this measure under $\pi$
describes the asymptotic statistics of typical solutions $x(t)$.

% -----------------------------------------------------------------------
\subsection{Evidence of an SRB measure} \label{sec.srbev}

Each asymptotic density \rhostar\ shown in
Figures~\ref{fig.invmg}--\ref{fig.invpwl} is constructed from an
ensemble of solutions $x(t)$ of a given DDE, evaluated at a particular
(large) time.  If we instead sample values $\{x_n = x(nh):
n=0,\ldots,N\}$ along a \emph{single} solution $x(t)$, where $h$ is
some fixed time increment (\eg\ the time step for numerical
integration), the histogram of these values approaches
\rhostar\ as $N \to \infty$.

Figure~\ref{fig.solhistpwl} illustrates this phenomenon.  Here we
consider the DDE~\eqref{eq.pwl} with piecewise-linear feedback,
restricted to constant initial functions (hence $g=0$
in~\eqref{eq.dde3}).  For an arbitrarily chosen initial value $x_0$ in
the DDE problem~\eqref{eq.dde3} we have computed a single numerical
solution and constructed histograms as described above.  Comparison
with Figure~\ref{fig.invpwl} shows good agreement between the
asymptotic histogram obtained as $N \to \infty$, and the asymptotic
histogram obtained by ensemble simulation. This behavior can also be
seen in the DDEs of examples~\ref{ex.mg} and~\ref{ex.pwc}.  Moreover,
the asymptotic histogram thus obtained seems to be independent of the
initial value $x_0$, with the exception of initial values on
equilibrium solutions of the DDE (\eg\ $x_0=0$ generates the zero
solution of the Mackey-Glass equation~\eqref{eq.mg3}, hence a trivial
histogram concentrated at the origin).
\begin{figure}
\begin{center}
\includegraphics[height=1.6in]{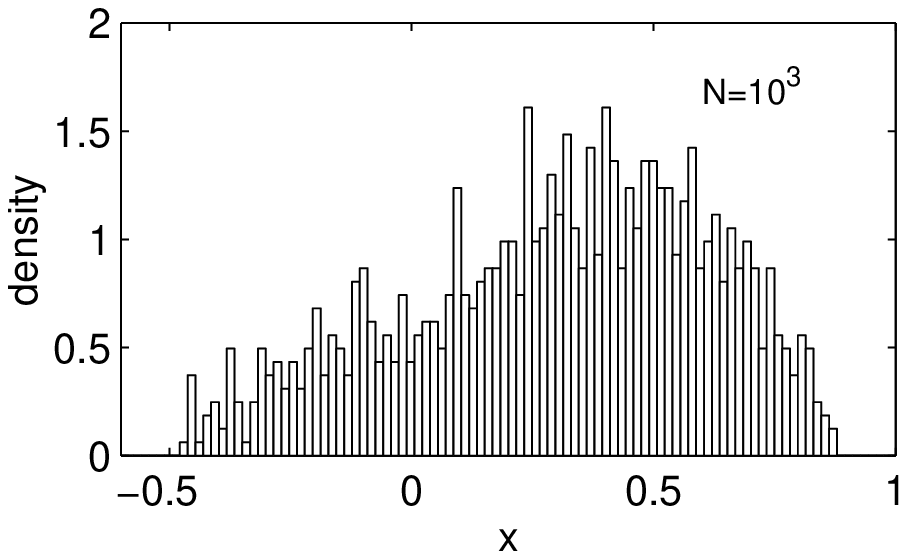} \\[0.1in]
\includegraphics[height=1.6in]{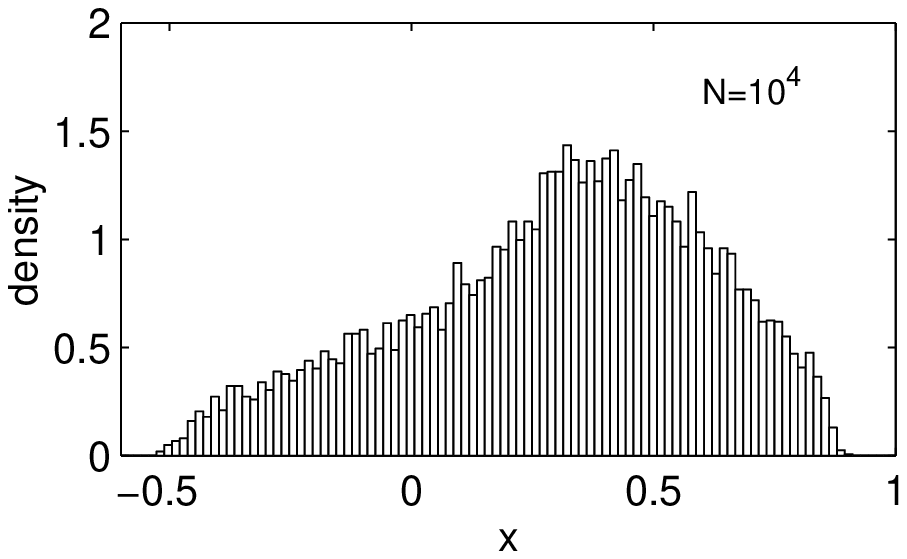} \\[0.1in]
\includegraphics[height=1.6in]{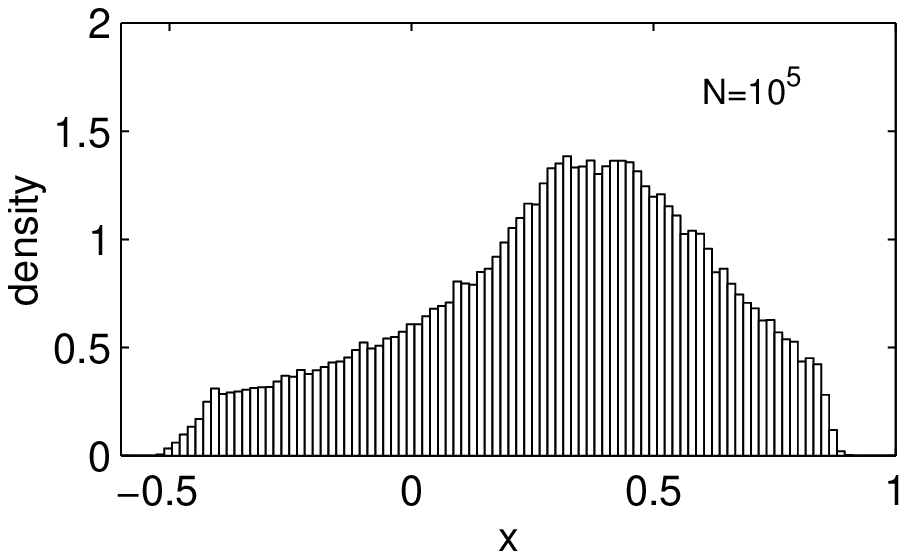} \\[0.1in]
\includegraphics[height=1.6in]{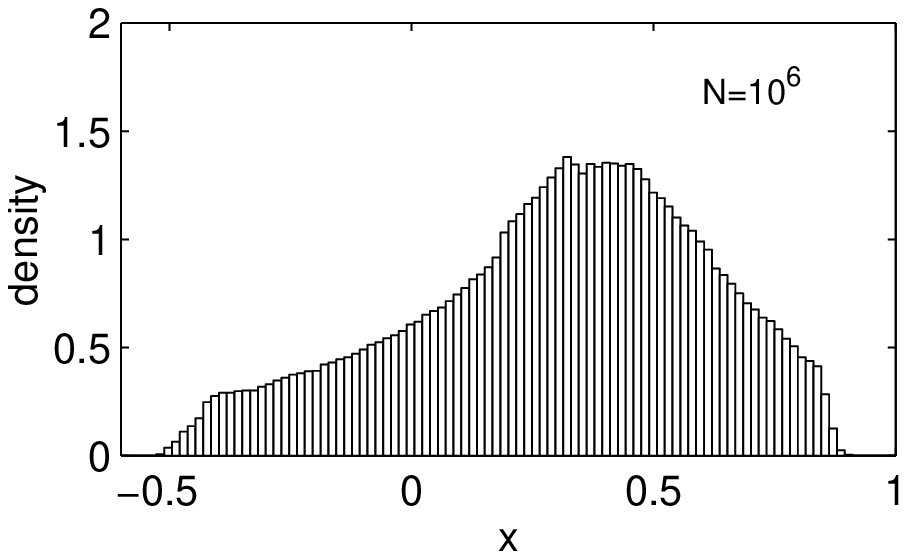}
\caption{Histograms of solution values $\{x_n = x(nh): n=0,\ldots,N\}$ for a
single numerical solution of the DDE~\eqref{eq.pwl} with fixed time
step $h$.}
\label{fig.solhistpwl}
\end{center}
\end{figure}

Convergence of a histogram along a single solution requires existence
of the time average
\begin{equation} \label{eq.ddebinav}
  \lim_{N \to \infty} \frac{\#\{n = 1,\ldots,N: x_n \in A\}}{N}
  = \lim_{N \to \infty} \frac{1}{N} \sum_{k=1}^N 1_A(x_k),
\end{equation}
for each histogram bin $A \subset \Reals$.  This behavior is indeed
expected, for solutions corresponding to $\mu$-almost every initial
function $\phi$, if the DDE's attractor supports an ergodic invariant
measure $\mu$ on $C$ (\cf\ equation~\eqref{eq.binav},
page~\pageref{eq.binav}).  However, the set of allowable (\eg\
constant) initial functions selected by~\eqref{eq.dde3} are not on the
attractor, hence not in the support of the supposed ergodic measure
$\mu$.  The fact that the time average~\eqref{eq.ddebinav} exists all
the same suggests a property stronger than ergodicity.  We conjecture
that the density \rhostar\ characterizes the asymptotic statistics of
every ``typical'' solution $x(t)$, with ``typical'' taken in the sense
of ``on a set of positive Lebesgue measure''.  This hypothesis is
similar to the existence of an SRB measure (\cf\
Section~\ref{sec.SRB}).

If the sequence $\{x_n = x(nh): n \geq 0 \}$ does in fact have a
well-defined asymptotic distribution according to a probability
measure \etastar, then convergence of histograms of $\{x_n\}$ to
\etastar\ can be expressed as.
\begin{equation} \label{eq.histconv}
  \frac{1}{N} \sum_{n=1}^N \delta_{x_n}(A) \longrightarrow
  \eta_{\ast}(A) \quad \text{as} \quad N \to \infty,
\end{equation}
for each measurable $A$, where $\delta_{x_n}$ is the probability
measure corresponding to a point mass at $x_n$:
\begin{equation}
  \delta_{x}(A) = \begin{cases}
	1 & \text{if $x \in A$} \\
        0 & \text{otherwise.}
		  \end{cases}
\end{equation}
The authors of~\cite{BB03} show that~\eqref{eq.histconv} implies
\begin{equation} \label{eq.SRBdde}
   \frac{1}{N}\sum_{i=1}^N \varphi\big(x_n\big) \longrightarrow \int \varphi \,d\eta_{\ast}.
\end{equation}
for any bounded continuous function $\varphi:\Reals\to\Reals$.  Thus
we have (supposedly) that the average of $\varphi$ along any typical
solution $x(t)$ is given by the average of $\varphi$ with respect to
$\eta_{\ast}$.  Equation~\eqref{eq.SRBdde} is just defining condition
for \etastar\ to be an SRB measure (\cf\ Section~\ref{sec.SRB}
page~\pageref{sec.SRB}), except that $\{x_n\}$ is not the orbit of a
dynamical system, but rather the trace of a dynamical system.

In fact, if the dynamical system $\{S_t\}$ corresponding to the DDE
does possess an SRB measure $\mu_{\ast}$, then $\eta_{\ast}$ is just
the image of $\mu_{\ast}$ under the trace map $\pi$.  To see this,
suppose $\mu_{\ast}$ is an SRB measure for $\{S_t\}$, \ie, for any
bounded continuous functional $\psi: C \to \Reals$,
\begin{equation} \label{eq.tavg2}
 \lim_{N \to \infty}\frac{1}{N} \sum_{k=1}^N \psi(S^k \phi) =
 \int_C \psi\,d\mu_{\ast}
\end{equation}
for all initial functions $\phi$ in a set of positive
$m$-measure.\footnote{We suppose $m$ is a measure on $C$ that provides
the relevant notion of ``almost every''.  As discussed in
Section~\ref{sec.infprob} the appropriate choice of $m$ is ambiguous,
so we leave it unspecified.}  Here $S=S_h$ is the time-$h$ map.  Let
$x$ be the DDE solution corresponding to initial function $\phi$.
Then for the particular functional $\psi=\varphi\circ\pi$, where
$\varphi: \Rn \to \Reals$ is bounded and continuous, and $\pi: C \to
\Rn$ is the trace map~\eqref{eq.pidef}, equation~\eqref{eq.tavg2}
gives
\begin{equation}
\begin{split}
 \lim_{N \to \infty}\frac{1}{N} \sum_{k=1}^N \varphi (x_n)
 &= \lim_{N \to \infty}\frac{1}{N} \sum_{k=1}^N \varphi\big( \pi(S^k(\phi) \big) \\
 &= \int_C \varphi \circ \pi\,d\mu_{\ast} \\
 &= \int_{\Rn} \varphi \,d (\mu_{\ast}\circ\pi^{-1}).
\end{split}
\end{equation}
The last line follows from measurability of $\pi$ and
Theorem~\ref{thm.cov} (page~\pageref{thm.cov}).  Comparison with
equation~\eqref{eq.SRBdde} gives (via the Riesz Representation
Theorem~\cite{Hal74})
\begin{equation}
  \eta_{\ast} = \pi(\mu_{\ast}) \equiv \mu_{\ast}\circ\pi^{-1}.
\end{equation}
That is, \etastar\ is just the image of \mustar\ under the
trace map $\pi:C \to \Rn$.  Thus the existence of an SRB measure
\mustar for $\{S_t\}$ implies the convergence of histograms, and
more generally the existence of time averages~\eqref{eq.SRBdde}, where
the asymptotic the measure \etastar\ can now be interpreted as the
projection $\pi(\mustar)$ of \mustar\ onto \Rn.

The supposed existence of an SRB measure for the infinite dimensional
dynamical system $\{S_t\}$ is only a conjecture supported by numerical
evidence.  At present a proof appears to be unattainable.  Indeed,
justifying~\eqref{eq.tavg2} is a formidable task even for
finite-dimensional dynamical systems~\cite{ER85,Via01}.  For infinite
dimensional systems, even the definition of SRB measure, and in
particular the appropriate choice of reference measure $m$, is
ambiguous.  Nevertheless, supposing the existence of an SRB measure
does provide a plausible framework that explains the apparent
existence of asymptotic densities for ensembles of solutions of some
DDEs.  This model is helpful to the discussion in the following
sections, where we consider methods of computing these asymptotic
densities.

% -----------------------------------------------------------------------
\clearpage
\subsection{Higher dimensional traces} \label{sec.trace2}

The trace map $\pi: C \to \Rn$ defined by (\cf\
Section~\ref{sec.trace1})
\begin{equation}
  \pi(u) = u(0)
\end{equation}
is a natural way to construct finite-dimensional images of objects in
$C$.  In particular if $\Lambda \subset C$ is an attractor for the
dynamical system $\{S_t\}$ corresponding to the delay
equation~\eqref{eq.dde3} then $\pi(\Lambda$) gives a
finite-dimensional picture of $\Lambda$, in the space \Rn\ of the
solution variable $x(t)$.  If $\{S_t\}$ has an invariant measure $\mu$
then the measure $\pi(\mu)$ on \Rn\ describes the statistics of an
ensemble of solutions $x(t)$, and is also invariant under the
dynamics.  Since $\pi$ is many-to-one, some information is lost in
these projections.  Indeed, if $x(t) \in \Reals$ then $\pi(\Lambda)$
is just an interval, and gives little information about the structure
of $\Lambda$.

If one was to choose a different mapping $\pi: C \to \Reals^M$ with $M
> n$ then intuitively $\pi(\Lambda)$ should give a more accurate
picture of $\Lambda$.  Ideally we would like $\pi$ to be one-to-one on
$\Lambda$, in which cased case the image $\pi(\Lambda)$ is called an
\emph{embedding} of $\Lambda$ in $\Reals^M$~\cite{HK99,SYC91}.  The
following theorem establishes that for $M$ sufficiently large, almost
\emph{any} $\pi: C \to \Reals^M$ in a certain class will yield an embedding of a
given finite-dimensional set in $C$.
\begin{thm}[after \cite{HK99}]
Let $X$ be a Banach space, $A \subset X$ a compact set with
box-counting dimension $D$.  If $M > 2D$ then almost every\footnote{In
the sense of prevalence~\cite{HSY92}, \cf\
Section~\ref{sec.prevalence}.} bounded linear function $\pi: X \to
\Reals^M$ is one-to-one on $A$.
\end{thm}

To illustrate how a multi-dimensional trace map can be used to
visualize the attractor of a DDE, we consider the mapping $\pi: C \to
\Reals^2$ defined by
\begin{equation}
  \pi(u) = \big( u(-1), u(0) \big).
\end{equation}
If $x_t = S_t(\phi) \in C$ is a phase point on a trajectory of
$\{S_t\}$ then we have (by equation~\eqref{eq.xtdef})
\begin{equation} \label{eq.trace2d}
\begin{split}
  \pi(x_t) &= \big( x_t(-1), x_t(0) \big) \\
           &= \big( x(t-1), x(t) \big)
\end{split}
\end{equation}
where $x(t)$ is the solution of the DDE with initial function $\phi$.
The image $\pi(\Lambda)$ can be approximated by computing a numerical
solution of the given DDE and plotting the set of points
$\{\big(x(t-1),x(t)\big): t \geq T\}$ in the plane (where $T$ is a
sufficiently large time for transients to die out, \ie\ for
$S_t(\phi)$ to approach $\Lambda$).  This amounts to plotting $x(t)$
vs.\ $x(t-1)$; in the literature this is occasionally done to
construct phase plots of DDE solutions in the ``pseudo phase space''
$\Reals^2$.  The results of this procedure applied to the examples in
Section~\ref{sec.invex} are shown in Figures~\ref{fig.att2dmg},
\ref{fig.att2dpwc} and~\ref{fig.att2dpwl}, respectively.  Even in
these $2$-dimensional images an intricate (presumably fractal)
structure of the attractor is apparent.

The previous section gave evidence of an SRB measure $\mu_{\ast}$
supported on $\Lambda$, and showed how the histograms depicted in
Figures~\ref{fig.invmg}--\ref{fig.invpwl} can be interpreted as
approximations of the image of $\mu_{\ast}$ under $\pi: u \mapsto
u(0)$.  Just as with the attractor $\Lambda$, a more accurate image of
$\mu_{\ast}$ is obtained under the $2$-dimensional trace
map~\eqref{eq.trace2d}.  The measure $\pi(\mu_{\ast})$ can be
approximated by computing a typical numerical solution of the given
DDE and constructing a two-dimensional histogram of the sequence of
vectors $\{\big(x(t-1),x(t)\big): t > T\}$.  The results of this
procedure applied to the examples of Section~\ref{sec.invex} are shown
in Figures~\ref{fig.inv2dmg}, \ref{fig.inv2dpwc}
and~\ref{fig.inv2dpwl}, respectively.  In each figure, part of the
$(x(t-1),x(t))$-plane is divided into a grid of rectangles (the
histogram bins).  Each rectangle is uniformly shaded with a level of
grayscale intensity that indicates the histogram height, which
approximates the measure $\pi(\mu_{\ast})$ of that rectangle.

\begin{figure}[p]
\begin{center}
\includegraphics[height=2.8in]{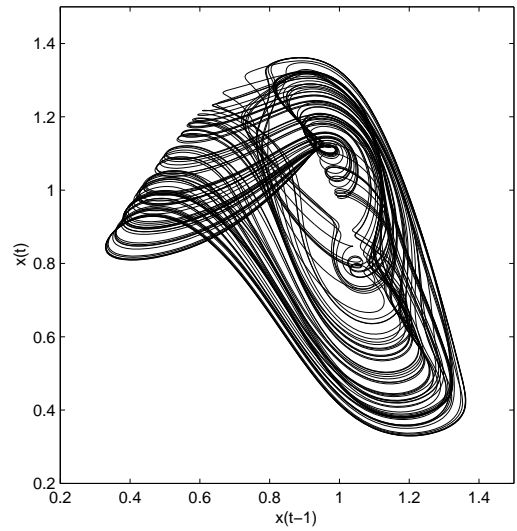}
\caption[Numerical approximation of the image $\pi(\Lambda)$ of the
attractor $\Lambda \subset C$ of the Mackey-Glass equation.]{Numerical
approximation of the image $\pi(\Lambda)$ of the attractor $\Lambda
\subset C$ of the Mackey-Glass equation~\eqref{eq.mg3}, under the
trace map~\eqref{eq.trace2d}.}
\label{fig.att2dmg}
\end{center}
\end{figure}
\begin{figure}[p]
\begin{center}
\includegraphics[height=2.8in]{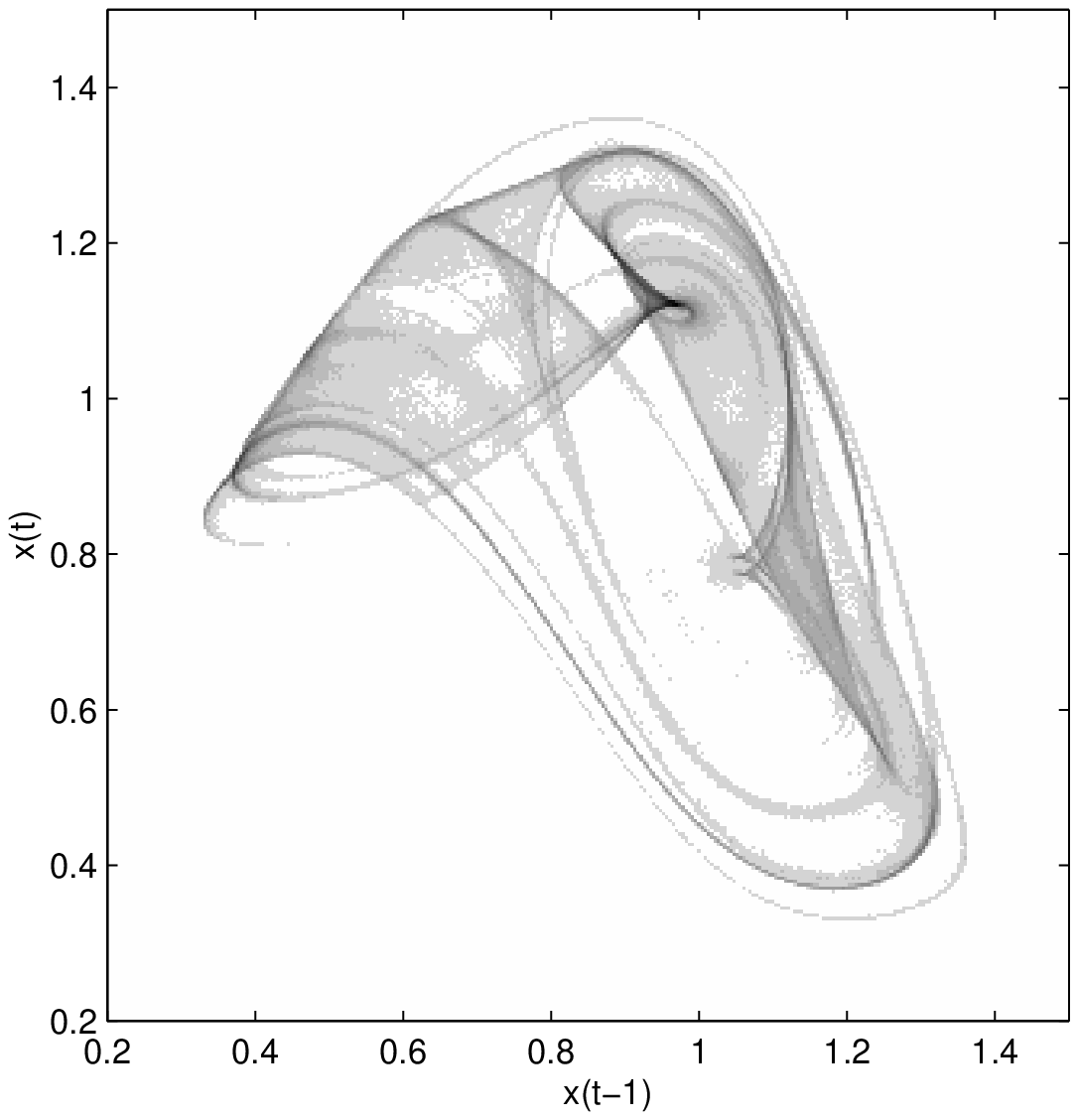}
\caption[Two-dimensional histogram approximating the projection
$\pi(\mu_{\ast})$ of the supposed SRB measure $\mu_{\ast}$ for the
Mackey-Glass equation.]{Two-dimensional histogram approximating the
projection $\pi(\mu_{\ast})$ of the supposed SRB measure $\mu_{\ast}$
for the Mackey-Glass equation~\eqref{eq.mg3}, under the trace
map~\eqref{eq.trace2d}.}
\label{fig.inv2dmg}
\end{center}
\end{figure}
\begin{figure}[p]
\begin{center}
\includegraphics[height=2.8in]{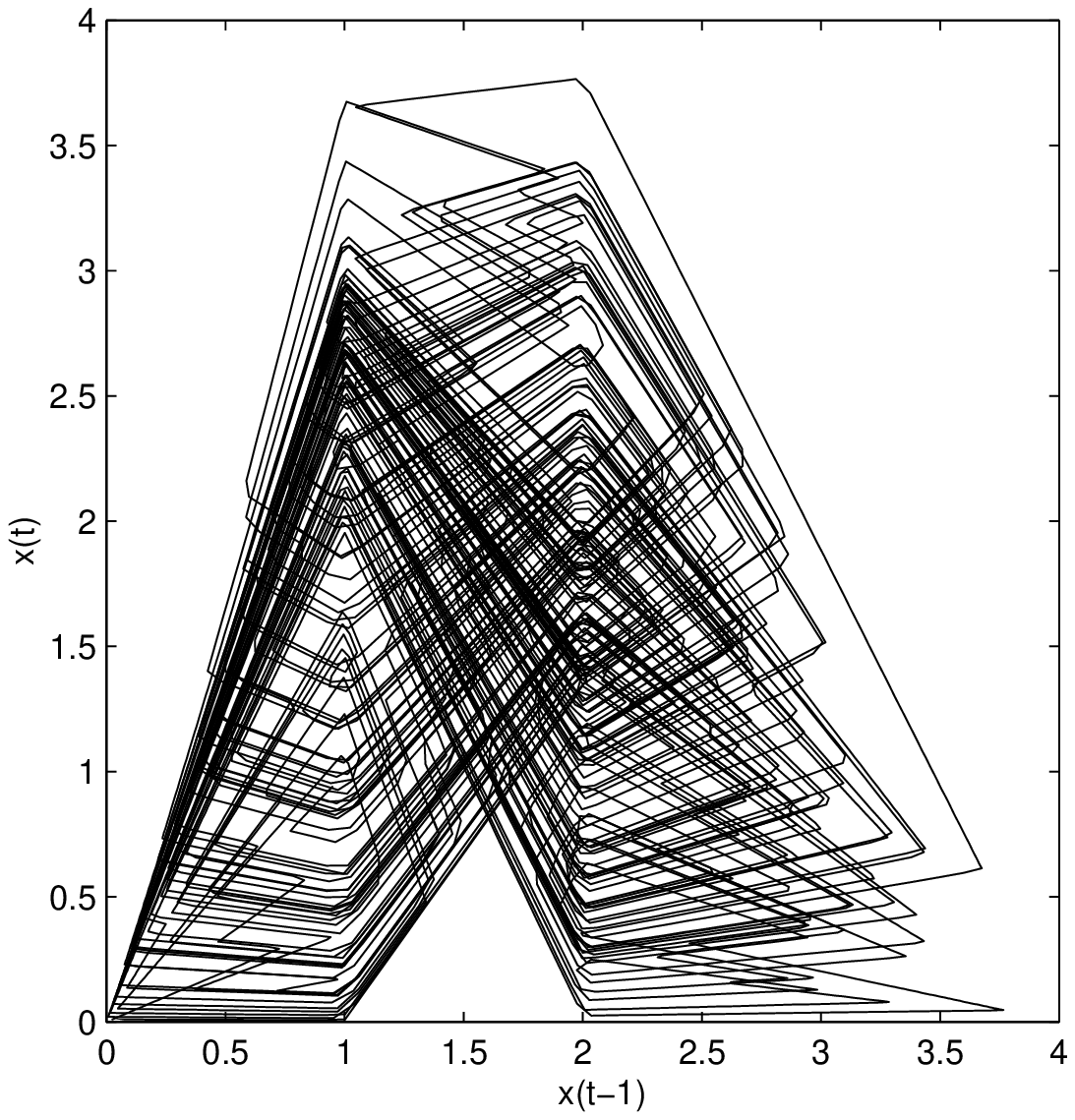}
\caption[Numerical approximation of the image $\pi(\Lambda)$ of the
attractor $\Lambda \subset C$ for the delay
equation~\eqref{eq.pwc}.]{Numerical approximation of the image
$\pi(\Lambda)$ of the attractor $\Lambda \subset C$ for the delay
equation~\eqref{eq.pwc}, under the trace map~\eqref{eq.trace2d}.}
\label{fig.att2dpwc}
\end{center}
\end{figure}
\begin{figure}[p]
\begin{center}
\includegraphics[height=2.8in]{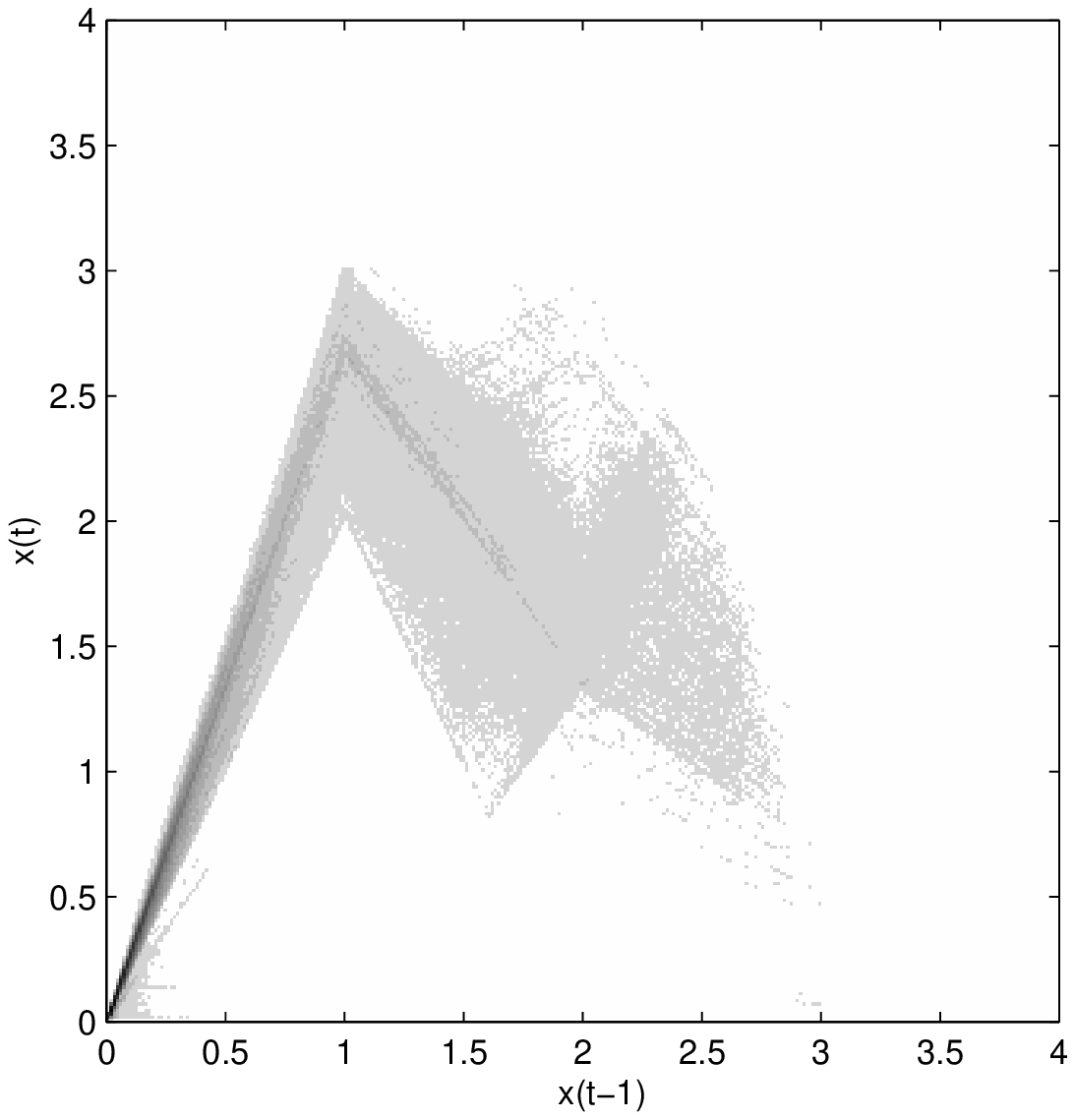}
\caption[Two-dimensional histogram approximating the projection
$\pi(\mu_{\ast})$ of the supposed SRB measure $\mu_{\ast}$ for the
delay equation~\eqref{eq.pwc}.]{Two-dimensional histogram
approximating the projection $\pi(\mu_{\ast})$ of the supposed SRB
measure $\mu_{\ast}$ for the delay equation~\eqref{eq.pwc}, under the
trace map~\eqref{eq.trace2d}.}
\label{fig.inv2dpwc}
\end{center}
\end{figure}
\begin{figure}[p]
\begin{center}
\includegraphics[height=2.8in]{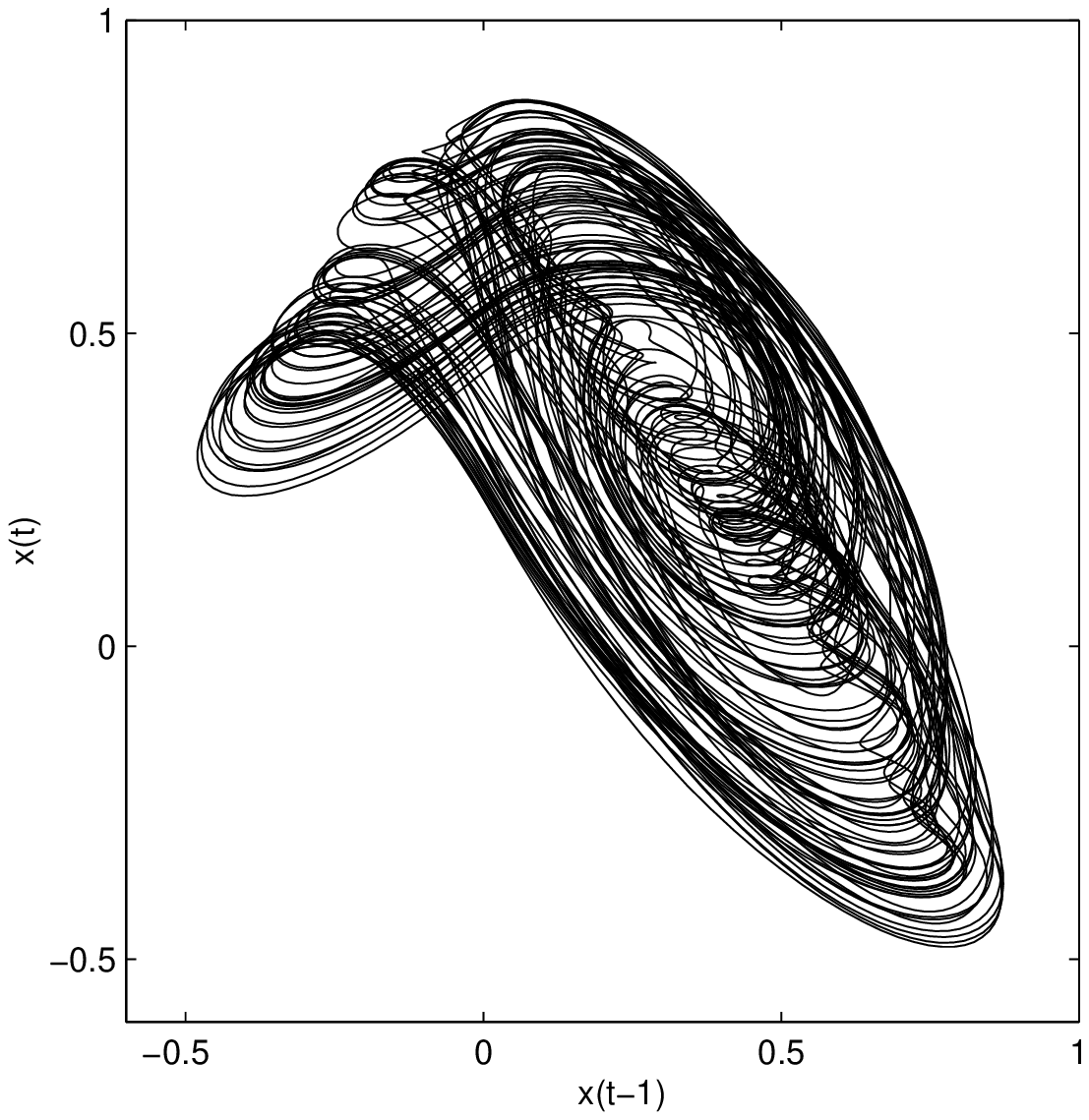}
\caption[Numerical approximation of the image $\pi(\Lambda)$ of the
attractor $\Lambda \subset C$ of the delay
equation~\eqref{eq.pwl}.]{Numerical approximation of the image
$\pi(\Lambda)$ of the attractor $\Lambda \subset C$ of the delay
equation~\eqref{eq.pwl}, under the trace map~\eqref{eq.trace2d}.}
\label{fig.att2dpwl}
\end{center}
\end{figure}
\begin{figure}[p]
\begin{center}
\includegraphics[height=2.8in]{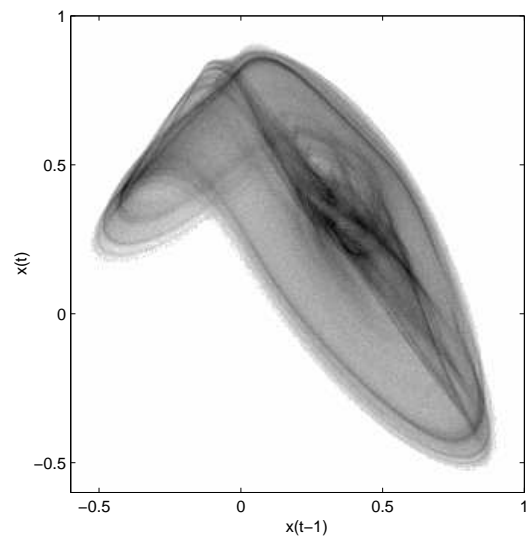}
\caption[Two-dimensional histogram approximating the projection
$\pi(\mu_{\ast})$ of the supposed SRB measure $\mu_{\ast}$ for the
delay equation~\eqref{eq.pwl}.]{Two-dimensional histogram
approximating the projection $\pi(\mu_{\ast})$ of the supposed SRB
measure $\mu_{\ast}$ for the delay equation~\eqref{eq.pwl}, under the
trace map~\eqref{eq.trace2d}.}
\label{fig.inv2dpwl}
\end{center}
\end{figure}

% =====================================================================
\clearpage
\section{Ulam's Method}\label{sec.ulam}

The remainder of this chapter is concerned with methods for computing
asymptotic densities for delay equations.  The methods considered so
far are of the ``brute force'' type: simulating large ensembles of
solutions (Section~\ref{sec.invex}) and computing statistics on a
single long solution (Section~\ref{sec.trace2}).  The utility of these
methods is limited by their large sampling requirements, owing to the
slow $O(1/\sqrt{N})$ rate of convergence of histograms as the number
$N$ of solutions is increased (see p.\,\pageref{sec.sampling}).  In
the following sections we seek more efficient methods.

A preliminary observation is in order, related to the remark at the
beginning of Section~\ref{sec.stepsdens}.  The obvious approach to
finding asymptotic densities is to begin with an evolution equation
for $\rho(x,t)$, of the form
\begin{equation}
  \frac{d}{dt}\rho = \{\text{some operator}\}(\rho).
\end{equation}
Invariance of \rhostar\ could then be characterized by setting the
left-hand side equal to zero, resulting in the equation
\begin{equation}
  \{\text{some operator}\}(\rhostar) = 0,
\end{equation}
which hopefully could be solved, at least approximately, for \rhostar.
This approach fails since, for reasons discussed in
Section~\ref{sec.stepsdens}, $\rho(x,t)$ cannot be described by an
evolution equation of the desired form (or, for that matter, by any
evolution equation in terms of $\rho(x,t)$ alone).  For that matter,
\rhostar\ cannot even be considered invariant in the above sense
(\cf\ Section~\ref{sec.ddeinv}).  At best, we can only apply this
approach to an \emph{approximate} evolution equation for $\rho$.  One
such possibility is considered in Section~\ref{sec.approxinv}.

An alternative approach, and probably the best known technique for
approximating invariant measures for dynamical systems, is Ulam's
method~\cite{Ulam60,Li76,DL91}.  The following section presents the
basic idea, after which we consider how the method might be adapted to
delay differential equations.

% ---------------------------------------------------------------------
\subsection{Stochastic approximation of dynamical systems}

Let a discrete-time dynamical system be defined by iterates of a map
$S: X \to X$ ($X \subset \Rn$) and suppose $S$ has an invariant
measure $\mu$ (for continuous time systems, the following applies to a
suitable discrete-time map, \eg\ the time-one map).  For a given
partition $\Acal=\{A_1,\ldots,A_n\}$ of $X$, let $p_i(k)$ denote the
probability at time $k$ that the system state $x(k) = S^k(x) \in A_i$.

Ulam's method approximates the evolution of the probability vector
$p(k) =$\linebreak $(p_1(k),\ldots,p_n(k))$ by a Markov chain
\begin{equation} \label{eq.markov}
  p(k+1) = Pp(k),
\end{equation}
where $P$ is a transition matrix that models the dynamics of $S$.  The
basic idea is to ignore the details of the dynamics within each $A_i$,
and instead consider the ``coarse-grained'' dynamics with respect to
the partition.  Thus, given that $x(k) \in A_j$, we suppose that
$x(k)$ is equally likely to be anywhere in $A_j$, \ie, $x(k)$ is
distributed according to normalized Lebesgue measure $\lambda$ on
$A_j$.  Then the ``transition probability'' that $x(k+1) \in A_i$ is
\begin{equation} \label{eq.transmat}
  P_{ij} = \frac{ \lambda(A_j \cap S^{-1} A_i) }{ \lambda(A_j) },
\end{equation}
\ie, the fraction (with respect to Lebesgue measure) of $A_j$ that
is mapped by $S$ into $A_i$.  This defines the $n \times n$ transition
matrix $P$ for the Markov chain~\eqref{eq.markov}, a stochastic
process that (hopefully) approximates the probabilistic dynamics of
$S$, in the following sense.

Each probability vector $p=p(k)$ defines a probability measure $\nu$
on $X$ such that $\nu(A_i)=p_i$, and in general
\begin{equation} \label{eq.nuapprox}
  \nu(A) = \sum_{i=1}^n \frac{\lambda(A \cap A_i)}{\lambda(A_i)}\cdot p_i.
\end{equation}
This measure has piecewise constant density
\begin{equation} \label{eq.densapprox}
  \rho(x) = \frac{p_i}{\lambda(A_i)} \quad \text{if $x \in A_i$}.
\end{equation}
Thus the evolution equation~\eqref{eq.markov} for $p(k)$ implicitly
defines a sequence of piecewise constant densities, approximating a
sequence of densities evolving under the action of $S$.  To be more
precise, $P$ can be interpreted as a projection of the
Perron-Frobenius operator $\tilde{P}$ corresponding to $S$ onto the
space of piecewise constant densities (with respect to the partition
\Acal).  Formally, $P$ can be written in terms of $\tilde{P}$
as~\cite{Li76}
\begin{equation}
  P (Q f) = Q (\tilde{P} f) \quad \forall f \in L^1(X),
\end{equation}
where the operator $Q$ projects $f$ to the piecewise constant function
$Qf\in L^1(X)$ given by
\begin{equation}
  (Q f)(x) = \sum_{i=1}^n\frac{\int_{A_i}
  f\,d\lambda}{\lambda(A_i)}\cdot 1_{A_i}(x).
\end{equation}

Since $P$ approximates the Perron-Frobenius operator $\tilde{P}$, it
can be hoped that a fixed point of $\tilde{P}$ (\ie, an invariant
density for $S$) can be approximated by a fixed point of $P$.  That
is, suppose that $p$ satisfies
\begin{equation}
  p = P p
\end{equation}
\ie, $p$ is an eigenvector of $P$ with eigenvalue $1$, normalized so
that $\sum p_i = 1$.  The corresponding piecewise constant density
defined by equation~\eqref{eq.densapprox} then approximates the density
of the invariant measure $\mu$.  Ulam conjectured~\cite{Ulam60} that
as the partition is refined ($\lambda(A_i) \to 0$), the sequence of
approximations obtained in this way should converge to the fixed point
of $\tilde{P}$, \ie, the true invariant density.  This conjecture has
in fact been proved for particular cases, such as piecewise expanding
transformations on intervals~\cite{Li76} and on rectangles
in~\Rn~\cite{DZ96}

% ---------------------------------------------------------------------
\subsection{Application to delay equations}

In the case of delay differential equations, we are concerned with an
invariant measure $\mu$ for a dynamical system $S_t$ on the space
$C([-1,0])$.  If $\mu$ has support in some bounded set $X \subset C$
(\eg, when $\mu$ is an SRB measure supported on a compact attractor
$\Lambda$) then in principle Ulam's method could be applied to the
time-one map $S=S_1$ and for some partition $\Acal$ of $X$.  However,
it is unclear (\cf\ Chapter~\ref{ch.framework}) what measure on $C$ is
an appropriate analog of Lebesgue measure in the definition of the
transition matrix $P$ (equation~\eqref{eq.transmat}).  For now,
suppose we do have some such reference measure, $m$, such that $0 <
m(A_i) < \infty$ for each $i$.  Then, in analogy
with~\eqref{eq.transmat} the matrix $P$ with elements
\begin{equation} \label{eq.transmat2}
  P_{ij} = \frac{ m(A_j \cap S^{-1} A_i) }{ m(A_j) }
\end{equation}
defines a Markov chain $p(k+1) = P p(k)$ that hopefully models the
dynamics of $S$.  In particular, we can hope that a fixed point $p$ of
$P$ approximates the invariant measure $\mu$ on \Acal, \ie, $p_i
\approx \mu(A_i)$.

Suppose, in accordance with the framework of
Section~\ref{sec.invinterp}, that $\etastar=\pi(\mustar)$ is the
measure on \Rn\ corresponding to the asymptotic density \rhostar\
(where $\pi: C \to \Rn$ is defined by~\eqref{eq.pidef}) and that the
support of \etastar\ is contained in some bounded $B \subset \Rn$.
Then, for a given partition $\mathcal{B}=\{B_1,\ldots,B_r\}$ of $B$, a
careful choice of \Acal\ permits an interpretation $p$ as a piecewise
constant approximation of \rhostar\ with respect to $\mathcal{B}$.
The following shows how this can be done.

Note that \etastar\ has support in the bounded set $B=\pi(X)$.  Let
$\mathcal{B}=\{B_1,\ldots,B_r\}$ be a partition of $B$ on which we
wish to approximate \etastar\ by a piecewise constant density.  Define
a partition $\Acal$ of $X$ by
\begin{equation}
  A_i = \pi^{-1}(B_i), \quad i=1,\ldots,r,
\end{equation}
or more explicitly,
\begin{equation}
  A_i = \{u \in C: u(0) \in B_i\}.
\end{equation}
Then, since $\etastar = \mustar\circ\pi^{-1}$, we have
\begin{equation}
  \mustar(A_i) = \etastar(B_i).
\end{equation}
Thus if we use Ulam's method to obtain a vector $p=Pp$ of
probabilities $p_i \approx \mustar(A_i)$ then we also have $p_i \approx
\etastar(B_i)$, yielding a piecewise constant approximation
\begin{equation} \label{eq.rhoastapprox}
  \rhostar(x) \approx \frac{p_i}{\lambda(B_i)} \quad \text{if $x
  \in B_i$}
\end{equation}
of the density \rhostar.

The immediate difficulties in implementing this method are in
evaluating the elements of the transition matrix
(equation~\eqref{eq.transmat2}), where we must compute the pre-images
$S^{-1}(A_i)$ under the infinite dimensional map $S$ (\eg, as given
explicitly in equation~\eqref{eq.expmap}) and evaluate the reference
measure $m$ of $S^{-1}(A_i)$.  Both of these difficulties can be
circumvented if we take as the reference measure $m$ the invariant
measure \mustar.  If \mustar\ is an SRB measure (which seems to be the
case in the examples of Section~\ref{sec.invex}) then for any bounded
continuous $\psi: C \to \Reals$ we have
\begin{equation} \label{eq.ddeSRB2}
  \lim_{N \to \infty}\frac{1}{N} \sum_{k=1}^N \psi\big(S^k(\phi)\big) =
  \int \psi\,d\mustar
\end{equation}
for every initial function $\phi$ in some set of positive $m$-measure.
In particular, for $\psi=1_A$ we have\footnote{There is a technical
difficulty here: $1_A$ is not continuous so~\eqref{eq.ddeSRB2} does
not strictly apply with $\psi=1_A$.  However, we can approximate $1_A$
from below by continuous functions $\{\psi_n\}$ with $\psi_n \to 1_A$
$\mustar$-almost everywhere, so $\int \psi_n \, d\mustar \to \int 1_A \, d\mustar$
by the Lebesgue dominated convergence theorem~\cite[p.\,22]{LM94}.
See \cite[p.\,134]{KH95} for further details.}
\begin{equation}
\begin{split}
  \frac{1}{N} \sum_{k=1}^N 1_A (S^k(\phi)) &= \frac{1}{N} \# \{k \in
  1,\ldots,N: S^k(\phi) \in A\} \\ &\xrightarrow{N \to \infty} \int
  1_A\,d\mustar = \int_A d\mustar = \mustar(A).
\end{split}
\end{equation}
Then for sufficiently large $N$ and for any typical initial function
$\phi$, we can approximate $\mustar(A_i)$ by
\begin{equation}
\begin{split}
  \mustar(A_i) &\approx \frac{1}{N} \# \{k \in 1,\ldots,N: S^k(\phi) \in A_i \} \\
  &= \frac{1}{N} \# \{k \in 1,\ldots,N: (S^k \phi)(0) \in B_i \} \\
  &= \frac{1}{N} \# \{k \in 1,\ldots,N: x_k \in B_i \},
\end{split}
\end{equation}
where $x_k=\big(S^k(\phi)\big)(0)=x(k)$ and $x(t)$ is the solution of the DDE
with initial function $\phi$.  Thus, to construct the transition matrix
$P$, we compute (\eg, numerically) a long sequence $\{x_k = x(k):
k=1,\ldots,N\}$ along a single solution $x(t)$ of the DDE, and
approximate $P_{ij}$ in equation~\eqref{eq.transmat2} by (see also
\eg\ \cite{Bollt})
\begin{equation} \label{eq.transmat3}
  P_{ij} = \frac{\# \{ k: x_k \in B_j \text{ and } x_{k+1} \in B_i\}}
              {\#\{k: x_k \in B_j \}},
\end{equation}
assuming $\mathcal{B}$ has been chosen so that $\{x_k\} \cap B_j \neq
\emptyset$ $\forall j$.

The transition matrix $P$ defines a Markov chain that approximates the
asymptotic probabilistic dynamics of the given DDE.  In particular, we
hope that a fixed point $p$ of $P$ will provide a piecewise constant
approximation of the asymptotic density $\rho_{\ast}$,
via~\eqref{eq.rhoastapprox}.  Indeed, when applied to the examples of
Section~\ref{sec.invex} this method identically reproduces the
asymptotic density found by computing a histogram of $\{x_k\}$ (\ie\
as in Section~\ref{sec.srbev}, Figure~\ref{fig.solhistpwl}).

It turns out that this must be the case.  Let $p=(p_1,\ldots,p_r)$
represent the normalized histogram of $\{x_k\}$ taken with respect to
the partition $\mathcal{B}$, \ie,
\begin{equation}
  p_i = \frac{1}{N}\#\{k \in 1,\ldots,N : x_k \in B_i\}
\end{equation}
A simple calculation shows that
\begin{equation}
\begin{split}
  (P p)_i &= \sum_{j=1}^r P_{ij} p_j \\
          &=\sum_{j=1}^r \frac{\#\{k: x_k \in B_j \text{ and } x_{k+1} \in B_i\}}
                   {\#\{k: x_k \in B_j \}} \cdot \frac{\#\{k: x_k \in B_j\}}{N} \\
          &=\sum_{j=1}^r \frac{\#\{k: x_k \in B_j \text{ and } x_{k+1} \in B_i\}}{N}\\
          &= \frac{\#\{k: x_{k+1} \in B_i\}}{N} \\
          &= \frac{\#\{k: x_k \in B_i\} \pm 1}{N} \\
          &= p_i \pm \frac{1}{N} \; \longrightarrow \; p_i \; \text{as} \; N \to \infty.
\end{split}
\end{equation}
Thus $p$ is (almost) a fixed point of $P$, and Ulam's method as
applied here simply reproduces the results of computing a histogram
along a single solution of the given DDE.  Essentially, the
construction is circular and no new information about the asymptotic
density is gained.

For lack of any other reasonable reference measure $m$ on $C$ that we
can evaluate (or approximate), our formulation of Ulam's method for
DDEs does not provide an independent estimate of the asymptotic
density.  However, the Markov chain $p \mapsto Pp$ defined
by~\eqref{eq.transmat3} is interesting in its own right, as a simple
model of the asymptotic probabilistic dynamics of the given DDE.  An
intuitive way to represent such a Markov chain is to graph the matrix
of transition probabilities $P$.
Figures~\ref{fig.markovmg}--\ref{fig.markovpwl} give examples of such
plots, with $P$ computed as in equation~\eqref{eq.transmat3}, for each
of the delay equations considered in the examples of
Section~\ref{sec.invex}.  For each figure the support of the
asymptotic density in \Reals\ has been partitioned into $100$
intervals $B_i$, $i=1,\ldots,100$, of equal length.  Each rectangle
$B_j \times B_i$ in the plane is shaded uniformly with grayscale level
indicating the probability $P_{ij}$ of transition from $x(t)
\in B_j$ to $x(t+1) \in B_i$.  Thus darker rectangles indicate likely transitions;
white regions indicate transitions that never occur, at least
asymptotically.

Figures~\ref{fig.markovmg}--\ref{fig.markovpwl} are remarkably similar
to Figures~\ref{fig.inv2dmg}, \ref{fig.inv2dpwc}
and~\ref{fig.inv2dpwl}.  There is in fact an intimate connection
between these figures, owing to the fact that the numerator in
equation~\eqref{eq.transmat3} is equivalent to a two-dimensional
histogram of the sequence of vectors $\{\big(x(t),x(t-1)\big):
t=0,1,\ldots\}$, with bins $\{B_j \times B_i: i, j=1,\ldots,N\}$.  The
entries of the matrix $P$ are therefore identical to the heights of
the corresponding two-dimensional histograms in
Figures~\ref{fig.inv2dmg}, \ref{fig.inv2dpwc} and~\ref{fig.inv2dpwl},
except that each column of $P$ is normalized so that $\sum_{i=1}^r
P_{ij} = 1$.
\begin{figure}
\begin{center}
\includegraphics[width=3in]{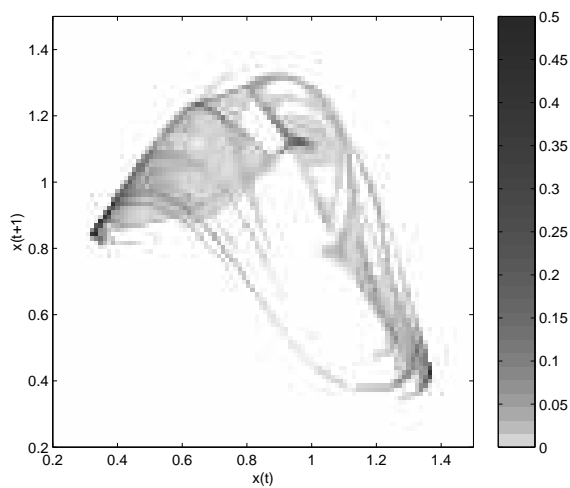}
\caption[Graphical representation of the matrix of transition
probabilities for the Markov chain approximating the asymptotic
dynamics of the Mackey-Glass equation.]{Graphical representation of
the matrix $P$ of transition probabilities (defined by
equation~\eqref{eq.transmat3}) for the Markov chain approximating the
asymptotic dynamics of the Mackey-Glass equation~\eqref{eq.mg3}.}
\label{fig.markovmg}
\end{center}
\end{figure}
\begin{figure}
\begin{center}
\includegraphics[width=3in]{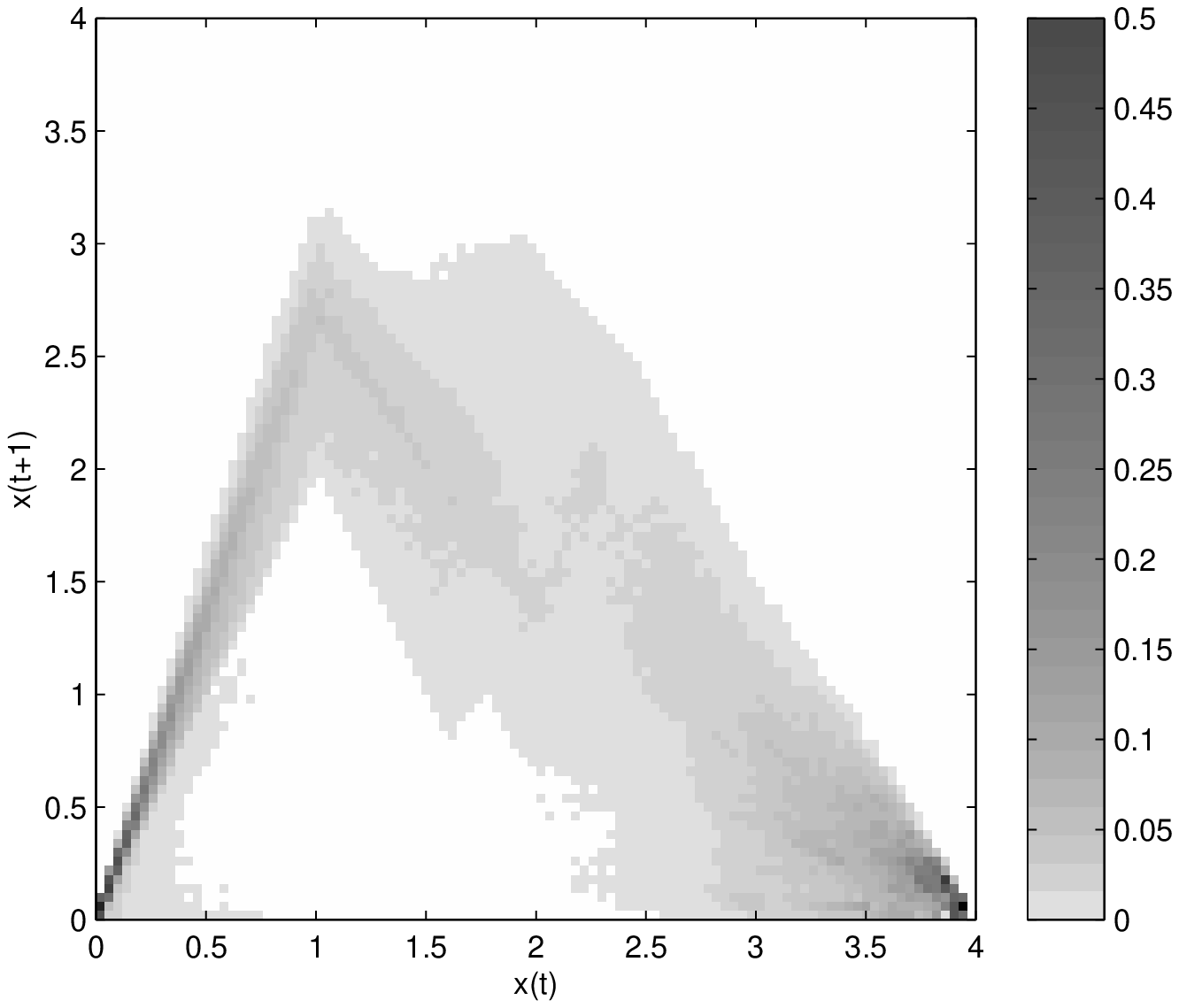}
\caption[Graphical representation of the matrix of transition
probabilities for the Markov chain approximating the asymptotic
dynamics of the delay equation~\eqref{eq.pwc}.]{Graphical
representation of the matrix $P$ of transition probabilities (defined by
equation~\eqref{eq.transmat3}) for the Markov chain approximating the
asymptotic dynamics of the delay equation~\eqref{eq.pwc}.}
\label{fig.markovpwc}
\end{center}
\end{figure}
\begin{figure}
\begin{center}
\includegraphics[width=3in]{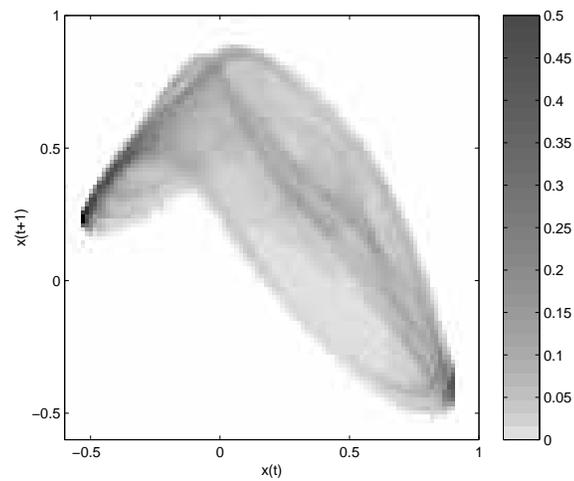}
\caption[Graphical representation of the matrix of transition
probabilities for the Markov chain approximating the asymptotic
dynamics of the delay equation~\eqref{eq.pwl}.]{Graphical
representation of the matrix $P$ of transition probabilities (defined
by equation~\eqref{eq.transmat3}) for the Markov chain approximating
the asymptotic dynamics of the delay equation~\eqref{eq.pwl}.}
\label{fig.markovpwl}
\end{center}
\end{figure}

% =====================================================================
\clearpage
\section{Fixed Points of Approximate Markov Operators}\label{sec.approxinv}

Ulam's method and its generalizations are the only techniques we are
aware of for approximating invariant measures of general dynamical
systems.  However, Lepri \etal~\cite{Lep93} present a method
specifically aimed at approximating invariant measures for
discrete-time systems with delayed dynamics, based on a similar method
for coupled map lattices~\cite{Kan89}.  This seems like a promising
approach to asymptotic densities for delay differential equations.  In
this section we consider the application of the method in~\cite{Lep93}
to a discretized approximation of a particular class of DDEs.

% ---------------------------------------------------------------------------
\subsection{Approximate Markov operator}

Consider the DDE\footnote{The ideas of this section generalize in a
straightforward way to DDEs in \Rn, but this requires more complex
notation.  To simplify the presentation we will restrict our attention
to DDEs in one dimension only.}
\begin{equation} \label{eq.dde5}
  x'(t) = -\alpha x(t) + f\big( x(t-1) \big), \quad t \geq 0, \quad x(t) \in \Reals,
\end{equation}
with initial function
\begin{equation} \label{eq.ic5}
  x(t) = \phi(t), \quad t \in [-1,0].
\end{equation}
Euler discretization of~\eqref{eq.dde5} with time step $h=1/N$ yields
\begin{equation}
  \frac{x(t+h)-x(t)}{h} \approx -\alpha x(t) + f\big( x(t-1) \big),
\end{equation}
which gives the explicit formula
\begin{equation} \label{eq.euler1}
  x_{n+1} = (1-\alpha h) x_n + h f( x_{n-N} )
\end{equation}
for the approximate solution $x_n = x(t_n)$ at the ``mesh points''
$t_n = n h$, $n=0,1,2,\ldots$.  Together with initial values $x_i =
\phi(t_i)$, $i=-N,\ldots,0$, this formula can be iterated to construct a
sequence $\{x_n: n=0,1,2,\ldots\}$ that approximates the solution
of~\eqref{eq.dde5}--\eqref{eq.ic5}.

The problem we consider here is to estimate the asymptotic density
$\rhostar(x)$ (supposing one exists) of an ensemble of systems
evolving under~\eqref{eq.euler1}.  This is just the problem considered
in~\cite{Lep93} for more general discrete-time systems with delayed
dynamics.  Their approach is easily adapted to the particular
system~\eqref{eq.euler1} as follows.

With
\begin{equation}
  y_n = x_{n-N}
\end{equation}
equation~\eqref{eq.euler1} becomes
\begin{equation}
  x_{n+1} = T (x_n, y_n),
\end{equation}
where
\begin{equation} \label{eq.Tdef}
  T: (x,y) \mapsto (1-\alpha h) x + h f(y).
\end{equation}
If $f: \Reals \to \Reals$ is measurable then $T$ is a measurable,
nonsingular transformation from $\Reals^2$ into \Reals.  Let
$\rho_2(x,y;n)$ be the density at time $n$ of the pair $(x_n,y_n)$,
for an ensemble of sequences governed by~\eqref{eq.euler1}.  Then in
analogy with the definition of the Perron-Frobenius operator, the
ensemble of values $x_{n+1}=T(x_n,y_n)$ will be distributed with
one-dimensional density $\rho_1(x;n+1)$ satisfying
\begin{equation}
  \int_A{\rho_1(x;n+1)\,dx} = \int_{T^{-1}(A)}{\rho_2(x,y;n)\,dx\,dy}
	\qquad \forall \text{ Borel } A \subset \Reals.
\end{equation}
For $A=(-\infty,s)$ we have
\begin{equation}
  T^{-1}(A) = \Big\{(x,y) \in \Reals^2: x < \frac{s - h f(y)}{1-\alpha h}\Big\},
\end{equation}
so that
\begin{equation}
  \int_{-\infty}^s{\rho_1(x;n+1)\,dx} =
	\int_{-\infty}^{\infty}\int_{-\infty}^{\frac{s-h f(y)}{1-\alpha h}}
	{\rho_2(x,y;n)\,dx\,dy}.
\end{equation}
Differentiating with respect to $s$ yields the explicit formula
\begin{equation} \label{eq.expl}
  \rho_1(x;n+1) =  \frac{1}{1-\alpha h} \int_{-\infty}^{\infty}
  \rho_2\Big( \frac{x - h f(y)}{1-\alpha h},y ; n\Big) \, dy.
\end{equation}

If the densities $\rho_{1\ast}(x)$ and $\rho_{2\ast}(x,y)$ are
invariant under the process defined by~\eqref{eq.euler1} then we can
drop the dependence on $n$ to yield
\begin{equation} \label{eq.expinv1}
  \rho_{1\ast}(x) =  \frac{1}{1-\alpha h} \int_{-\infty}^{\infty}
  \rho_{2\ast} \Big( \frac{x - h f(y)}{1-\alpha h},y\Big) \, dy.
\end{equation}
As it stands, this equation cannot be used on its own to determine
$\rho_{1\ast}$, since it requires prior knowledge of the
$2$-dimensional density $\rho_{2\ast}$.  With a similar approach it is
possible write an analogous equation defining $\rho_{2\ast}$, but this
in turn requires knowledge of the $3$-dimensional density
$\rho_{3\ast}$ of the triple $(x_n,x_{n-N},x_{n-2N})$.  In general,
the evolution equation for the density $\rho_{k}$ of the $k$-tuple
$(x_n,x_{n-N},\ldots,x_{n-(k-1)N})$ requires knowledge of
$\rho_{(k+1)}$, so that we obtain an open recursion relation for the
corresponding invariant densities $\rho_{k\ast}$.

To close this open recursion relation so it can be solved for
$\rho_{1\ast}$, one can make an approximation whereby for some $k$,
$\rho_{k\ast}$ can be expressed in terms of the $\{\rho_{j\ast}: j
\leq k\}$.  The simplest such approximation is the factorization
\begin{equation} \label{eq.factor}
  \rho_{2\ast}(x,y) = \rho_{1\ast}(x) \rho_{1\ast}(y).
\end{equation}
This amounts to assuming that $x$ and $y$ are independent, \ie,
uncorrelated.  It is a ``self-consistent'' approximation, in that $x$
and $y$ are both supposed to be distributed according to
$\rho_{1\ast}$.  This is exactly what we expect if $\rho_{1\ast}$ is
an invariant density: if $x_n$ is distributed according to
$\rho_{1\ast}$ for all sufficiently large $n$, then so must be
$y_n=x_{n-N}$.  This approximation closes equation~\eqref{eq.expinv1},
which becomes
\begin{equation} \label{eq.Qdef}
  \rho_{1\ast}(x) = \frac{1}{1-\alpha h} \int_{-\infty}^{\infty}
  \rho_{1\ast} \Big( \frac{x - h f(y)}{1-\alpha h}\Big) \,
  \rho_{1\ast}(y) \, dy \equiv (Q \rho_{1\ast})(x).
\end{equation}

By its construction $Q$ maps densities to densities, but it is not a
Markov operator since it is nonlinear.  Nevertheless, in~\cite{Kan89}
the operator analogous to $Q$ is called a ``self-consistent
Perron-Frobenius operator''.  $Q$ can be interpreted intuitively as
the operator that effects the evolution of densities under the action
of~\eqref{eq.euler1} with the assumption that at each time step,
$x_{n-N}$ is a random variable independent of $x_n$ and distributed
with the same density as $x_n$.

Since $Q$ approximates, in some sense, the probabilistic dynamics of
the Euler discretization of the given DDE, it is hoped that a density
that is invariant under $Q$ (\ie, a fixed point of $Q$) will
approximate the asymptotic density \rhostar, \eg\ as observed in
Figures~\ref{fig.invmg}--\ref{fig.invpwl}.  One approach to
approximating the solution of the operator equation
$\rhostar=Q\rhostar$ is fixed point iteration: if a sequence of
densities $\{u_{n+1}=Q u_n\}$ can be found that converges in $L^1$,
then the limit $\rhostar=\lim_{n\to\infty}u_n$ furnishes a solution
of~\eqref{eq.Qdef}.\footnote{Actually this requires continuity of $Q$,
which seems to require restricting $Q$ to $L^{\infty}$.  So far we
have not found a satisfactory proof.}  In practice this iteration is
carried out numerically.

% -----------------------------------------------------------------------
\subsection{Numerical implementation}

The integral in~\eqref{eq.Qdef} resembles a convolution.  In fact by
the change of variables
\begin{gather} \label{eq.vw}
  v(x) = \frac{1}{1-\alpha h}\: u\Big(\frac{x}{1-\alpha h}\Big) \\
  w(x) = \sum_{s \in f^{-1}\{x/h\}} \frac{u(s)}{h |f'(s)|},
\end{gather}
$Q(u)$ can be written as a convolution integral,
\begin{equation} \label{eq.conv}
  (Qu)(x) = \int_{-\infty}^{\infty} v(x-z) w(z) dz.
\end{equation}
This integral can be approximated numerically as follows.

Let the densities $u$, $Q(u)$, and $w$ be approximated by the
corresponding vectors of values they assume on a uniform grid
$\{x_1,\ldots,x_M\}$,
\begin{equation}
  x_i = x_1 + (i-1) \Delta x, \quad i = 1,\ldots,M,
\end{equation}
and designate a vector of weights $\{\alpha_j\}$ appropriate for
numerical quadrature on this grid (\eg, by Simpson's
Rule~\cite[p.\,134]{NumRec92}).  Then $(Qu)(x_i)$ can be approximated
by
\begin{equation} \label{eq.disc}
\begin{split}
  (Qu)_i = (Qu)(x_i) &\approx \sum_{j=1}^M v(x_i - x_j) w(x_j)
	        \alpha_j \\ &= \sum_{j=1}^M v\big( (i-j) \Delta x
	        \big) w(x_j) \alpha_j \\ &= \sum_{j=1}^M v(y_{i-j+M})
	        w(x_j) \alpha_j \\ &= \sum_{j=1}^M v_{i+M-j} w_j
	        \alpha_j,
\end{split}
\end{equation}
where
\begin{equation}
  y_k = (k - M) \Delta x, \quad k = 1,\dots,2M-1,
\end{equation}
and the vectors
\begin{equation}
\begin{split}
 v_k &= v(y_k) \\ w_j &= w(x_j).
\end{split}
\end{equation}
are evaluated according to~\eqref{eq.vw}, using interpolation of the
$u_i=u(x_i)$.  The final line of~\eqref{eq.disc} is a discrete
convolution, representing a moving average of length-$M$ windows of
$\{v_k\}$ with respect to the vector of weights $\{w_j \alpha_j\}$.
Using standard techniques~\cite[ch.\,12]{NumRec92}, the $(Qu)_i$ can
then be evaluated efficiently using a Fast Fourier Transform.

% ----------------------------------------------------------------------
\subsection{Case study}

The delay equation~\eqref{eq.dde5} with piecewise linear feedback term
\begin{equation} \label{eq.pwlinf}
\begin{gathered}
  f(x) = g(x)/\eps \\
  g(x) = 1 - 1.9|x|
\end{gathered}
\end{equation}
and $\alpha=1/\eps$ was found in Example~\ref{ex.pwl} to exhibit an
asymptotic density \rhostar\ when $\eps=0.3$ (\cf\
Figure~\ref{fig.invpwl}).  However, when the method described above is
applied to this equation, it does not yield an approximation of
\rhostar.  Instead, for any initial density $u$, iterating $u \mapsto
Qu$ results in convergence toward a point mass concentrated at
$x_{\ast}=1/2.9$, which is readily seen to be an unstable fixed point
of the map $x \mapsto g(x)$.

Given the success of this method for estimating invariant densities
for other systems with delayed dynamics~\cite{Lep93} this result is
surprising, but a partial explanation can be advanced as follows.  The
assumption implicit in the factorization~\eqref{eq.factor} essentially
removes any explicit delay from the dynamics of the
discretization~\eqref{eq.euler1}, since the delayed coordinate
$y_n=x_{n-N}$ is always assumed to have the same density as $x_n$.  In
other words, the delayed coordinate is being modeled by a stochastic
variable distributed like $x_n$.  Iteration of $u \mapsto Qu$ gives a
probabilistic description of the map $(x,y) \mapsto T(x,y)$ where at
every iteration $y$ is assumed to have the same distribution as $x$.
This is sort of (but not quite) like evolving a density under the
one-dimensional map $x \mapsto T(x,x)$.  From equation~\eqref{eq.Tdef}
we have
\begin{equation}
  T(x,x) = x + \frac{h}{\eps}\big( g(x) - x \big),
\end{equation}
so that $x \mapsto T(x,x)$ has the same fixed points as $g$.  In
particular the stability of the fixed point $x_{\ast}=1/2.9$ is
determined by
\begin{equation}
\begin{split}
  \left|\frac{d}{dx_{\ast}}T(x_{\ast},x_{\ast})\right|
        &= \left| 1 + \frac{h}{\eps} \big( g'(x_{\ast}) - 1 \big) \right| \\
        &= \left| 1 + \frac{h}{\eps} ( -1.9 - 1 ) \right| \\
        &= \left| 1 - 2.9\frac{h}{\eps} \right| < 1 \quad \text{for $h \ll \eps$}.
\end{split}
\end{equation}
Thus the fixed point $x_{\ast}=1/2.9$ of $T$ is stable, hence the
apparent convergence of densities to a point mass $\delta_{x_{\ast}}$.

A similar phenomenon occurs when fixed-point iteration is attempted to
solve~\eqref{eq.Qdef} in the case of either the Mackey-Glass
equation~\eqref{eq.mg3} or the delay equation~\eqref{eq.pwc} with
piecewise constant feedback.  In neither case does the method yield an
approximation of the asymptotic density found by ensemble simulation
(\cf\ Figures~\ref{fig.invmg} and~\ref{fig.invpwc}), but yields rather
a point mass concentrated as a fixed point of $T(x,x)$.

Evidently the instability of the dynamics of the DDE~\eqref{eq.dde5}
is delay-induced.  Indeed, with zero delay and $f$ defined as
in~\eqref{eq.pwlinf} the DDE becomes
\begin{equation}
  x' = F(x) \equiv -\frac{1}{\eps} x + \frac{1}{\eps} \big(1 - 1.9|x|\big),
\end{equation}
which has a single stable fixed point $x_{\ast}=1/2.9$.  It is easily
shown that the situation is similar with the other DDEs considered in
examples~\ref{ex.mg} and~\ref{ex.pwc}: with zero delay the asymptotic
dynamics are trivial.  Instability is essential to the existence of a
nontrivial invariant density.  Since the assumption
in~\eqref{eq.factor} effectively removes the delay, the resulting
condition~\eqref{eq.Qdef} reasonably does not provide an approximation
of the observed asymptotic density \rhostar.  The examples considered
in~\cite{Lep93} exhibit chaotic behavior even for zero delay; this
helps explain why they did not encounter the difficulties we find
here.

% ----------------------------------------------------------------------------
\subsection{Second-order method} \label{sec.2ord}

A less restrictive assumption than~\eqref{eq.factor}, specifically one
that retains the essential delay in the dynamics, might yield an
effective method of approximating \rhostar.  One possibility,
suggested in~\cite{Lep93} as a more accurate variant on the original
method, is to make a different approximation that truncates the
recursion relationship for the $\rho_{k\ast}$ at some $k > 2$.  The
following gives a sketch of how this might be done.

With $y_n=x_{n-N}$ and $z_n=x_{n-2N}$, the discretization~\eqref{eq.euler1} becomes
\begin{equation}
\begin{split}
  (x_{n+1},y_{n+1}) &= \big( (1-\alpha h) x_n + h f( y_n ),
                             (1-\alpha h) y_n + h f( z_n ) \big) \\
  &= \big( T(x_n,y_n),T(y_n,z_n) \big) \\
  &\equiv \tilde{T}(x_n,y_n,z_n).
\end{split}
\end{equation}
Then $\tilde{T}$ is a measurable, nonsingular transformation from
$\Reals^3$ into $\Reals^2$ .Let $\rho_3(x,y,z;n)$ be the density of
the triple $(x_n,y_n,z_n)$, for an ensemble of sequences governed
by~\eqref{eq.euler1}.  The ensemble of pairs $(x_{n+1},y_{n+1})$ will
be distributed with $2$-dimensional density $\rho_2(x,y;n+1)$
satisfying
\begin{equation} \label{eq.PFTtilde}
\int_A{\rho_2(x,y;n+1)\,dx\,dy} =
  \int_{\tilde{T}^{-1}(A)}{\rho_3(x,y,z;n)\,dx\,dy\,dz} \qquad \forall
	\text{ Borel } A \subset \Reals^2.
\end{equation}
Since
\begin{equation}
  \tilde{T}(x,y,z) \in (-\infty,r] \times (-\infty,s] \implies
	\begin{cases}
	  x < (r - h f(y)) / (1-\alpha h) \\
		y < (s - h f(z)) / (1-\alpha h),
  \end{cases}		
\end{equation}
equation~\eqref{eq.PFTtilde} yields (with $A = (-\infty,r] \times (-\infty,s]$)
\begin{equation}
  \int_{-\infty}^s \int_{-\infty}^r \rho_2(x,y;n+1)\,dx\,dy =
	\int_{-\infty}^{\infty}
	\int_{-\infty}^{\frac{s - h f(z)}{1-\alpha h}}
	\int_{-\infty}^{\frac{r - h f(y)}{1-\alpha h}}
	\rho_3(x,y,z;n) \, dx\,dy\,dz.
\end{equation}
Differentiating with respect to $r$ and $s$ yields the explicit formula
\begin{equation} \label{eq.Q2def}
\begin{split}
  \rho_2(x,y;n+1) &= \frac{1}{(1-\alpha h)^2} \int_{-\infty}^{\infty} \rho_3\Big(
	  \frac{x-h f( \frac{y - h f(z)}{1-\alpha h} )}{1-\alpha h},
    \frac{y-h f(z)}{1-\alpha h}, z; n \Big) \,dz  \\
  &= \frac{1}{(1-\alpha h)^2} \int_{-\infty}^{\infty} \rho_3\Big(
	G\big(x, G(y,z)\big), G(y,z), z; n \Big) \, dz
\end{split}
\end{equation}
where
\begin{equation}
  G(y,z) = (y - h f(z))/(1-\alpha h).
\end{equation}
If the densities $\rho_{2\ast}(x,y)$ and $\rho_{3\ast}(x,y,z)$ are
invariant under the dynamics then we can drop the dependence on $n$.
To close this equation requires an approximation whereby
$\rho_{3\ast}$ can be written in terms of $\rho_{2\ast}$.  One
possibility, suggested in~\cite{Lep93}, is the factorization
\begin{equation} \label{eq:factor2}
  \rho_{3\ast}(x,y,z) = \frac{\rho_{2\ast}(x,y) \;
  \rho_{2\ast}(y,z)}{\rho_{1\ast}(y)},
\end{equation}
where
\begin{equation} \label{eq.rho1proj}
  \rho_{1\ast}(y) = \int \rho_{2\ast}(x,y)\,dx.
\end{equation}
Thus both $(x_n,x_{n-N})$ and $(x_{n-N},x_{n-2N})$ are distributed
with the same density $\rho_{2\ast}$; this is consistent with
$\rho_{2\ast}$ being invariant under the dynamics.  Substituting these
relationships into~\eqref{eq.Q2def} yields a nonlinear operator
equation
\begin{equation} \label{eq.Qsolve}
  \rho_{2\ast} = Q \rho_{2\ast}
\end{equation}
for the two-point density $\rho_{2\ast}$.  Solving this equation by
fixed point iteration approximates the evolution of a two-point
density for the system~\eqref{eq.euler1}, where at each iteration the
pair $(x_{n-N},x_{n-2N})$ is assumed to have the same density as
$(x_n,x_{n-N})$.  Since this method retains the inherent delay in the
dynamics, it is possible that~\eqref{eq.Qsolve} will have a solution
$\rho_{2\ast}$, whereby equation~\eqref{eq.rho1proj} gives an
approximation of the invariant density $\rho_{\ast}=\rho_{1\ast}$.
However, the computational complexity is much greater than in the
previous method, and to date we have not developed an implementation.

% -------------------------------------------------------------------------------------
\subsection{Continuous-time formulation}

\newcommand{\ut}{\ensuremath{u_t}}
\newcommand{\ustar}{\ensuremath{u_{\ast}}}
\newcommand{\vstar}{\ensuremath{v_{\ast}}}

In the previous sections we considered a discretized version of a
given delay differential equation, to which the method developed
in~\cite{Lep93} could be directly applied.  However, it seems more
natural to avoid the discretization step altogether.  It is in fact
possible to formulate a continuous-time approach analogous to the
discrete time method in~\cite{Lep93}, as illustrated below.

Consider the delay equation
\begin{equation} \label{eq.dde6}
  x'(t) = f\big( x(t), x(t-1) \big), \quad x(t) \in \Rn, \quad t \geq 0,
\end{equation}
written as
\begin{equation} \label{eq.ode2n}
  x'(t) = f\big( x(t), y(t) \big)
\end{equation}
where $y(t)=x(t-1)$.  For any given value of $y=y(t)$ this is just an
ordinary differential equation prescribing the flow of $x(t)$ along a
vector field $x \mapsto f(x,y)$, in which $y$ acts as a fixed
parameter.  Of course this interpretation is valid only
instantaneously, \ie, at a particular time, since $y(t)$ itself
changes in time.  If $\rho_1(x;t)$ is the density at time $t$ of an
ensemble of solutions $x(t)$ that all share the \emph{same} value of
$y(t)$, then the transportation of $\rho_1$ along this flow is described
by a continuity equation (\cf\ Section~\ref{sec.fpode}),
\begin{equation} \label{eq.yfix}
  \frac{\partial \rho_1(x;t)}{\partial t} = -\nabla_x \cdot
  (\rho_1(x;t) f(x,y)),
\end{equation}
where $\nabla_x = (\partial/\partial x_1,\ldots,\partial/\partial
x_n)$.

Suppose now that we have an ensemble of solutions without any
restriction on $y(t)$, and that the ensemble of pairs
$\big(x(t),y(t)\big)$ is distributed with $2$-dimensional density
$\rho_2(x,y;t)$.  It is helpful to think of this ensemble as being
partitioned into sub-ensembles according to the value of $y(t)$.  Then
as the ensemble evolves under the action of~\eqref{eq.dde6}, each
sub-ensemble contributes an increment to $\rho_1$ according
to~\eqref{eq.yfix}.  Proceeding heuristically, the total increment to
$\rho_1$ can be found by summing equation~\eqref{eq.yfix} over these
sub-ensembles, \ie,
\begin{equation} \label{eq.heurist}
  \frac{\partial \rho_1(x;t)}{\partial t} =
    -\int{\nabla_x \cdot \big( \rho_2(x,y;t) \, f(x,y) \big) \, dy}.
\end{equation}

A less heuristic derivation of this result is as follows.  Suppose
that vectors $(x,y)\in\Reals^{2n}$ are distributed with density
$\rho_2(x,y;t)$, and that $x$ evolves according to
equation~\eqref{eq.ode2n}.  Let
\begin{equation}
  \rho_1(x;t) = \int \rho(x,y;t) \, dy
\end{equation}
denote the ``collapsed'' density of the ensemble of values $x$.  For
the purposes of determining the instantaneous increment of
$\rho_1(x;t)$ under the action of~\eqref{eq.ode2n}, the dynamics of
$y(t)$ are a second-order effect and can be ignored.  Thus we can make
the simplifying assumption that $y'(t) = 0$, and write
\begin{equation}
  \frac{d}{dt} \begin{bmatrix} x(t) \\
                               y(t)
               \end{bmatrix} =
               \begin{bmatrix} f\big( x(t), y(t) \big) \\
                               0
               \end{bmatrix}
  \equiv H\big( x(t), y(t) \big),
\end{equation}
which gives the flow of $(x(t),y(t)) \in \Reals^{2n}$ along the vector
field $H: \Reals^{2n} \to \Reals^{2n}$.  Under transportation by this
flow, the density $\rho_2(x,y;t)$ evolves according to the continuity
equation
\begin{equation}
\begin{split}
  \frac{\partial \rho_2(x,y;t)}{\partial t} &= -\nabla \cdot \big( \rho_2(x,y;t) H(x,y) \big) \\
  &= -\nabla_x \cdot \big( \rho_2(x,y;t) f(x,y) \big)
    -\nabla_y \cdot \big( \rho_2(x,y;t) \cdot  0 \big) \\
  &= -\nabla_x \cdot \big( \rho_2(x,y;t) f(x,y) \big),
\end{split}
\end{equation}
where $\nabla_y = (\partial/\partial y_1,\ldots,\partial/\partial
y_n)$.  Then we have
\begin{equation} \label{eq.drho1}
\begin{split}
  \frac{\partial}{\partial t} \rho_1(x;t) &= \int \frac{\partial}{\partial t} \rho_2(x,y;t)\,dy\\
  &= -\int \nabla_x \cdot \big( \rho_2(x,y;t) f(x,y) \big) \, dy,
\end{split}
\end{equation}
in agreement with~\eqref{eq.heurist}.

If the one- and two-point densities $\rho_{1\ast}(x)$ and
$\rho_{2\ast}(x,y)$ are invariant under the dynamics then
equation~\eqref{eq.drho1} gives
\begin{equation} \label{eq.invcond}
  \int{\nabla_x \cdot \big( \rho_{2\ast}(x,y) f(x,y) \big) \, dy} = 0.
\end{equation}
Assuming again that $\rho_{2\ast}$ can be factored as
in~\eqref{eq.factor}, this yields the condition
\begin{equation}
  \int{\rho_{1\ast}(y) \;\nabla_x \cdot \big( \rho_{1\ast}(x) f(x,y) \big) \, dy} = 0,
\end{equation}
which is analogous to the discrete-time condition~\eqref{eq.Qdef}.  It
is unclear whether this relationship uniquely determines (within a
constant multiple) a unique approximate invariant density
$\rho_{1\ast}$.

However, here again we have resorted to an assumption
(equation~\eqref{eq.factor}) that effectively removes the explicit
delay from the dynamics.  As discussed in the previous section, this
assumption leads to trivial asymptotic dynamics.  Without a more
sophisticated approach that retains the essential delay, it seems
unwarranted to pursue these ideas further.

% -------------------------------------------------------------------------------------
\clearpage
\section{Conclusions}

For a variety of delay differential equations, numerically computed
solution ensembles appear to converge to an asymptotic distribution,
described by an asymptotic measure $\eta$ on \Rn.  This phenomenon can
be understood in terms of ergodic properties of the associated
infinite dimensional dynamical system $\{S_t\}$ on $C$.  In the
examples considered, $\{S_t\}$ is known to possess a compact attractor
$\Lambda$: any ensemble of trajectories starting in the basin of
attraction of $\Lambda$ will, asymptotically, be distributed on
$\Lambda$.  The numerical evidence supports the existence of a natural
invariant probability measure (\ie, an SRB measure) $\mu_{\ast}$
supported on $\Lambda$.  This serves to explain the convergence of
solution ensembles to a particular asymptotic distribution $\eta$, as
well as the fact that averages along ``typical'' individual solutions
of the DDE coincide with spatial averages or expectations with respect
to this same measure.

The practical and theoretical importance of invariant measures, and
SRB measures especially, makes the computation of invariant measures
for DDEs a desirable goal.  However, an effective solution to this
problem remains elusive.  Previously published methods of estimating
invariant measures for dynamical systems do not adapt well to delay
equations.

Ulam's method---the most widely known technique for estimating
invariant measures---can be formulated for DDEs in such a way that it
yields an approximation of the asymptotic measure $\eta$.  This
approximation turns out to be identical to the histogram of a time
series generated by a typical solution of the given DDE.  Thus, at
least in our formulation, Ulam's method \emph{per se} is not a useful
approach to DDEs but merely points to the fact that if the desired
invariant measure is an SRB measure, then it can be estimated by
computing a histogram along a single long-time solution.  This is, in
fact, a far more efficient method than the ensemble simulation
approach, in which on the order of $10^6$ individual solutions must be
computed.

An alternative approach to estimating invariant measures for DDEs is
the ``self-consistent Perron-Frobenius operator'' method
of~\cite{Lep93}.  A suitable discretization of a given DDE yields a
discrete-time system of the type to which this method applies.
Somewhat surprisingly, however, a straightforward application of the
method fails to generate the desired approximate invariant density.

\chapter{Transient Chaos}\label{ch.transient}

\begin{singlespacing}
\minitoc   % mini table of contents for this chapter
\end{singlespacing}
\newpage

% ---------------------------------------------------------------------
\section{Introduction}

Studies of chaotic dynamical systems have focused mainly on
persistent, or attracting chaos---\ie, on systems that possess a
chaotic attractor.  The phenomenon of
\emph{transient} chaos has aroused less interest despite its
ubiquity~\cite{KG85}.  Systems exhibiting transient chaos have the
distinguishing feature that their evolutions are very irregular
(chaotic) during a transient period, but eventually become periodic.
This behavior is seen in many physical systems, including fluid
dynamics
\cite{Ahl80,Darby95} and chaotic scattering
\cite{scattering}, as well as in mathematical dynamical systems such
as the H{\'e}non map \cite{GOY83,Hen76}, the Lorenz system
\cite{Lor63,YY79}, and the forced damped pendulum \cite{BGOY88}.

There appear to be a number of underlying universal features of
transient chaos, despite the diversity of its manifestations; see
\cite{Tel90} for a review.  The central notion is the existence in
phase space of an unstable invariant set on which the dynamics are
chaotic (\eg, in the sense of Li and Yorke~\cite{LY75}).  Such a set
is called a \emph{strange repeller} or \emph{chaotic saddle}, since
the instability is typically of saddle type.  Except for its
instability, this set plays a role similar to that of a strange
attractor: the dynamics on the repeller are closely related to the
irregularity of nearby trajectories.

The generally accepted model of the phase space dynamics underlying
transient chaos is as follows.  A typical initial phase point is
attracted, under the system dynamics, along the stable manifold of the
chaotic saddle.  The trajectory subsequently wanders in a neighborhood
of the saddle for some time, during which it exhibits the dynamics
associated with the saddle.  Eventually it exits along the saddle's
unstable manifold, and arrives asymptotically at one of the system's
attractors (typically a periodic orbit or equilibrium point, but
possibly a chaotic attractor).  Because the saddle is the dynamical
invariant that determines the behavior of trajectories during their
transient phase, its structure and the dynamics on it are the objects
of primary interest in the analysis of transient chaos.

To date there has been no published account of transient chaos in
delay differential equations.\footnote{\cite{KG85} references
unpublished work by P.\ Grassberger and I.\ Procaccia.}  However, a
number of published results would suggest that transient chaos occurs
in some DDEs.  The results that motivated the present study were the
observations in \cite{LML93} and \cite{AE87}\footnote{In the DDE
studied in~\cite{AE87} fractal basins are present even in the absence
of a delay.  Our primary interest here is in systems where the chaotic
dynamics are ``delay induced'', \ie, arise only in the presence of an
intrinsic delay in the dynamics.}\ of fractal basins of attraction (a
hallmark of transient chaos~\cite{Tel90}) in delay equations.
Transverse homoclinic orbits, which imply the existence of a chaotic
saddle and Yorke-type chaos~\cite{Guck83}, have been proved to exist
in some DDEs~\cite{HM82, HW83, Wal81}.  Hale and
Sternberg~\cite{Hale88} found numerical evidence of transverse
homoclinic orbits in the Mackey-Glass equation~\cite{MG77}.  These
results have all been presented amid discussions of attracting chaos,
whereas transient chaos in DDEs has not been specifically
investigated.

Aside from the importance of delay equations in describing natural and
industrial processes, transient chaos in DDEs has special relevance to
the study of infinite dimensional dynamical systems.  Because
numerical integration of DDEs is relatively easy as compared, for
example, with partial differential equations, Farmer \cite{Farm82}
pointed out that DDEs make convenient prototypical models for the
study of attractors of infinite dimensional systems.  In the same
spirit, DDEs could serve as simple models for the study of
\emph{transient} chaos in infinite dimensional systems.

The following section provides numerical evidence of transient chaos
in delay equations of the form
\begin{equation} \label{eq.ddetype}
  x'(t) = -\alpha x(t) + F\big( x(t-1) \big).
\end{equation}
We extend the results of \cite{LML93} and show how the existence of
fractal basins of attraction can be used to find solutions with
long-lived chaotic transients in a first-order DDE having only
periodic attractors.  In addition (and in contradiction with the
negative result of~\cite{LML93}) we find parameter sets of the
Mackey-Glass equation that yield fractal basins of attraction, and we
illustrate the existence of long-lived chaotic transients for this
system.

Numerical analysis of transient chaos (\eg, approximation of the
saddle, its dimension, Lyapunov exponents, entropy, \etc.)\ requires a
method for computing trajectories on (or very near) the chaotic
saddle.  Various methods have been proposed and applied to
finite-dimensional systems~\cite{BGOY88,KG85,Nusse89,Sweet01}.  There
are no published accounts of attempts to apply these methods to
infinite dimensional systems such as DDEs.  In section
\ref{sec.analysis} we develop an adaptation of the
``stagger-and-step'' method~\cite{Sweet01} and apply it to the DDEs
for which we have found evidence of transient chaos.  Having
constructed a numerical approximation of the saddle, we illustrate
graphical methods for visualizing the saddle, and characterize its
geometry quantitatively using standard methods for estimating ergodic
parameters such as Lyapunov exponents and fractal dimensions.

% ---------------------------------------------------------------------
\section{Evidence of Transient Chaos} \label{sec.transev}

% ------------------------------------------------------------------------
\subsection{Fractal basins of attraction}

The (necessarily open) set of all initial phase points eventually
asymptotic to a given attractor is called that attractor's \emph{basin
of attraction}.  In a system with multistability, \ie\ one that
possesses more than one attractor, the points that do not lie in any
basin of attraction constitute the \emph{basin boundary}.  The basin
boundary is necessarily invariant under the system dynamics.  If this
set is fractal (\ie\ has non-integer capacity dimension) then the
dynamics in a neighborhood of the boundary exhibits sensitivity to
initial conditions~\cite{MGOY85}, hence transient chaos.  This
happens, for example, if there is a ``horseshoe''~\cite{KH95,Smale67}
in the dynamics (\eg\ if there is a transverse homoclinic
orbit~\cite{Guck83}); in this case the basin boundary has a
Cantor-like structure.  However, multistability and existence of a
fractal basin boundary are not necessary for transient chaos: \eg\ the
H{\'e}non map~\cite{Hen76} for some parameter values exhibits
transient chaos but has only a single attractor (at
infinity)~\cite{Nusse91}.

Unstable invariant sets such as unstable fixed points and chaotic
saddles and their stable manifolds must lie within the basin boundary,
since they are not in any basin of attraction.  The basin boundary can
consist entirely of the stable manifolds of unstable invariant sets,
but this need not be the case~\cite{MGOY85}.

Multistability and fractal basin boundaries for delay differential
equations are reported in \cite{AE87,LML93}.  However, the connection
of this observation to the possible existence of transient chaos has
not been investigated.  In the following we reproduce the fractal
basin boundaries for the DDE considered in~\cite{LML93}, and also
provide evidence of fractal basins in the much-studied Mackey-Glass
equation~\cite{MG77}.

% ------------------------------------------------------------------------
\subsubsection{DDE with piecewise-constant feedback}

In \cite{LML93}, Losson \etal\ studied the delay equation
\begin{equation} \label{eq.pwc2}
\begin{split}
  &x'(t) = -\alpha x(t) + F\big( x(t-1) \big), \\
  &F(x) = \begin{cases} c & \text{if $x \in [x_1,x_2]$} \\
    0 & \text{otherwise}.
  \end{cases}
\end{split}
\end{equation}
For parameters $\alpha=3.25$, $c=20.5$, $x_1=1$ and $x_2=2$, they
found\footnote{There seems to be an error in~\cite{LML93}, which gives
$\alpha=3.75$.}\ three coexisting attracting periodic solutions.
Figure~\ref{fig.persolspwc} illustrates these solutions, together with
the zero solution which is also attracting, found by numerically
integrating\footnote{Using the Fortran code DKLAG6~\cite{Thomp}.}\
equation~\eqref{eq.pwc2} to large $t$ with different initial
functions.
\begin{figure}
\begin{center}
\includegraphics[width=\figwidth]{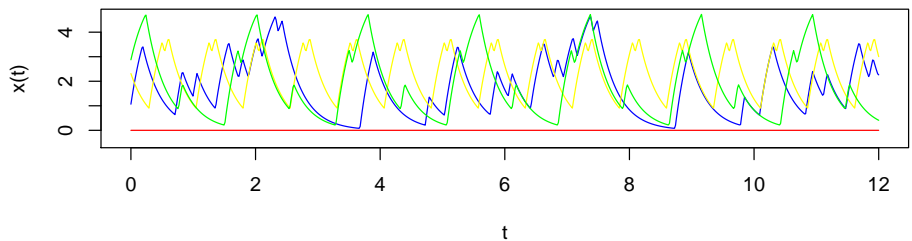}
\caption[Four coexisting attracting periodic solutions of the delay
equation~\eqref{eq.pwc2}.]{Four coexisting attracting periodic
solutions (including the trivial solution) of the delay
equation~\eqref{eq.pwc2} with parameters $\alpha=3.25$, $c=20.5$,
$x_1=1$ and $x_2=2$.}
\label{fig.persolspwc}
\end{center}
\end{figure}

Recall that the phase space of equation~\eqref{eq.pwc2}, considered as
a dynamical system $\{S_t: t \geq 0\}$ (\cf\
Chapter~\ref{ch.framework}), is the space $C$ of continuous
real-valued functions on the interval $[-1,0]$.  The phase point $x_t
\in C$ at time $t$ is the solution history,
\begin{equation} \label{eq.transstate}
  x_t(s) = x(t+s), \quad s \in [-1,0].
\end{equation}
Each of the solutions shown in Figure~\ref{fig.persolspwc} corresponds
to a periodic orbit in $\Gamma \subset C$.  Each such $\Gamma$ has a
corresponding basin of attraction, which is the set of initial
functions in $C$ that are asymptotic to $\Gamma$ under the action of
$S_t$ as $t \to \infty$.  Thus basins of attraction for DDEs are
subsets of the infinite dimensional space $C$.  Consequently,
visualizing the basins of attraction and their boundaries is
problematic.

One possibility for visualizing the basins of attraction of the
DDE~\eqref{eq.pwc2} is to visualize ``cross-sections'' through $C$.
Consider for example the subspace of $C$ spanned by the functions $s
\mapsto 1$ and $s \mapsto s$, that is, functions of the form
\begin{equation} \label{eq.icsubspace1}
  \phi(s) = A + B s, \quad s \in [-1,0].
\end{equation}
These constitute a two-dimensional subspace $\Sigma$ of $C$,
parametrized by coordinates $(A,B) \in \Reals^2$.  For a given
attracting periodic orbit $\Gamma$ of~\eqref{eq.pwc2}, the
intersection of its basin of attraction with $\Sigma$ can be
approximated numerically and visualized by the following method. For a
given point in the $(A,B)$-plane, numerically
integrate~\eqref{eq.pwc2} with the corresponding initial function
$\phi$ of the form~\eqref{eq.icsubspace1}.  If the resulting orbit is
asymptotic to $\Gamma$ then $\phi$ is in the basin of attraction of
$\Gamma$, so plot the point $(A,B)$.  Repeating this procedure for a
grid of points in the $(A,B)$-plane yields a picture approximating
part of the basin of attraction's intersection with $\Sigma$.

If a different color is associated with each of the various basins of
attraction then all four basins of attraction can be visualized on a
single graph, as in Figure \ref{fig.pwcbasins1}.  This figure shows
the results of the procedure above, carried out for the delay
equation~\eqref{eq.pwc2}.  Here we have made a slight change from
equation~\eqref{eq.icsubspace1} and taken initial functions on the
interval $[0,1]$ of the form
\begin{equation}
  x(t) = A + B t, \quad t \in [0,1],
\end{equation}
or equivalently
\begin{equation}
  \phi(s) = A + B (s + 1), \quad s \in [-1,0].
\end{equation}
Figure~\ref{fig.pwcbasins1} corresponds to Figure 11 of~\cite{LML93},
except for a different choice of axes.  As already pointed out
in~\cite{LML93}, the resulting image suggests that the basin boundary
is a fractal set with Cantor-like structure.  This is supported by
evidence given in~\cite{Hale88} of the presence of a transverse
homoclinic orbit, and implies the existence of transiently chaotic
solutions.  Indeed in section~\ref{sec.longtrans} we are able, with
the help of Figure~\ref{fig.pwcbasins1}, to find solutions
of~\eqref{eq.pwc2} that exhibit long chaotic transients.

\begin{figure}
\begin{center}
\includegraphics[height=3.5in]{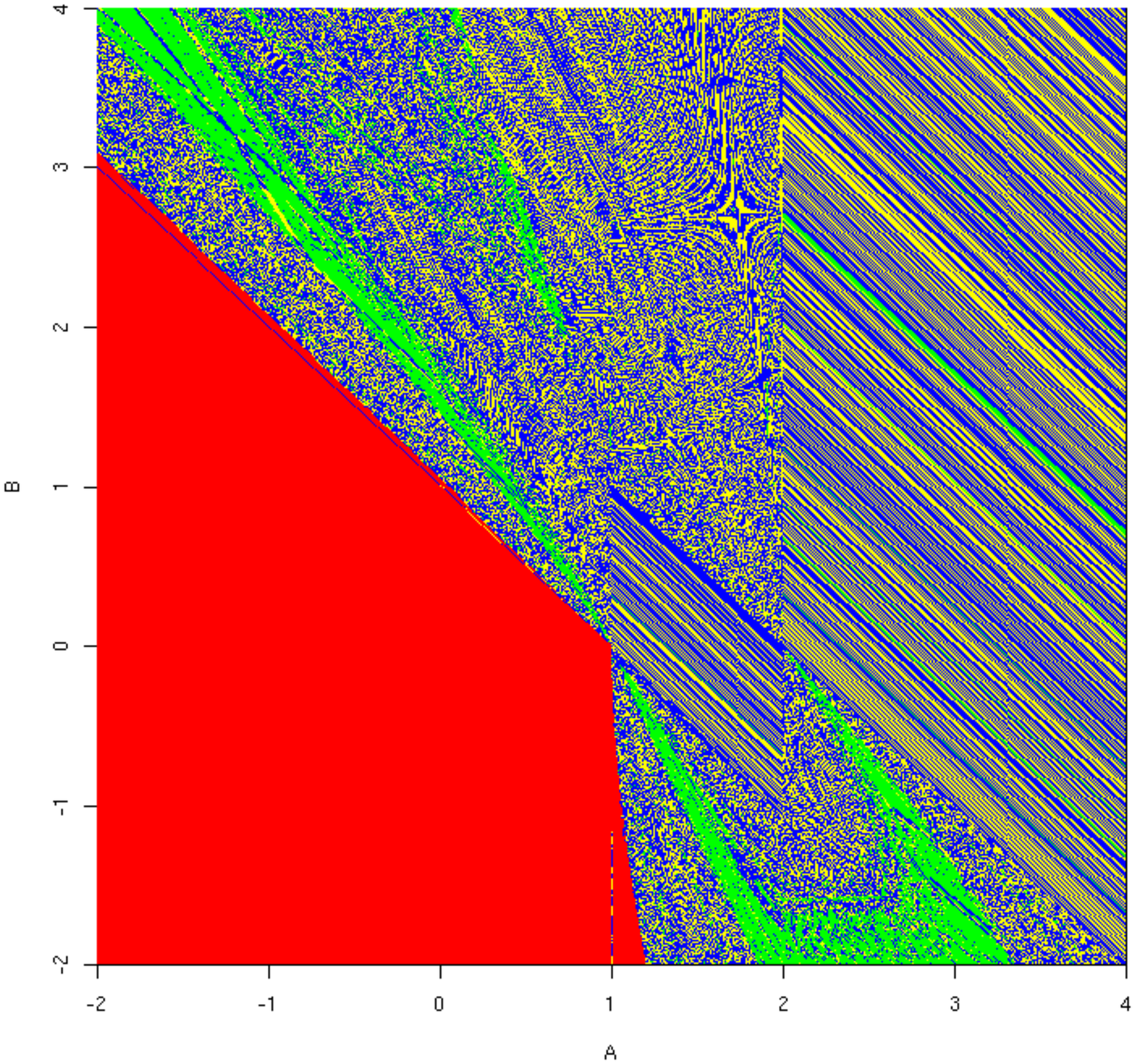}\\
\includegraphics[height=3.5in]{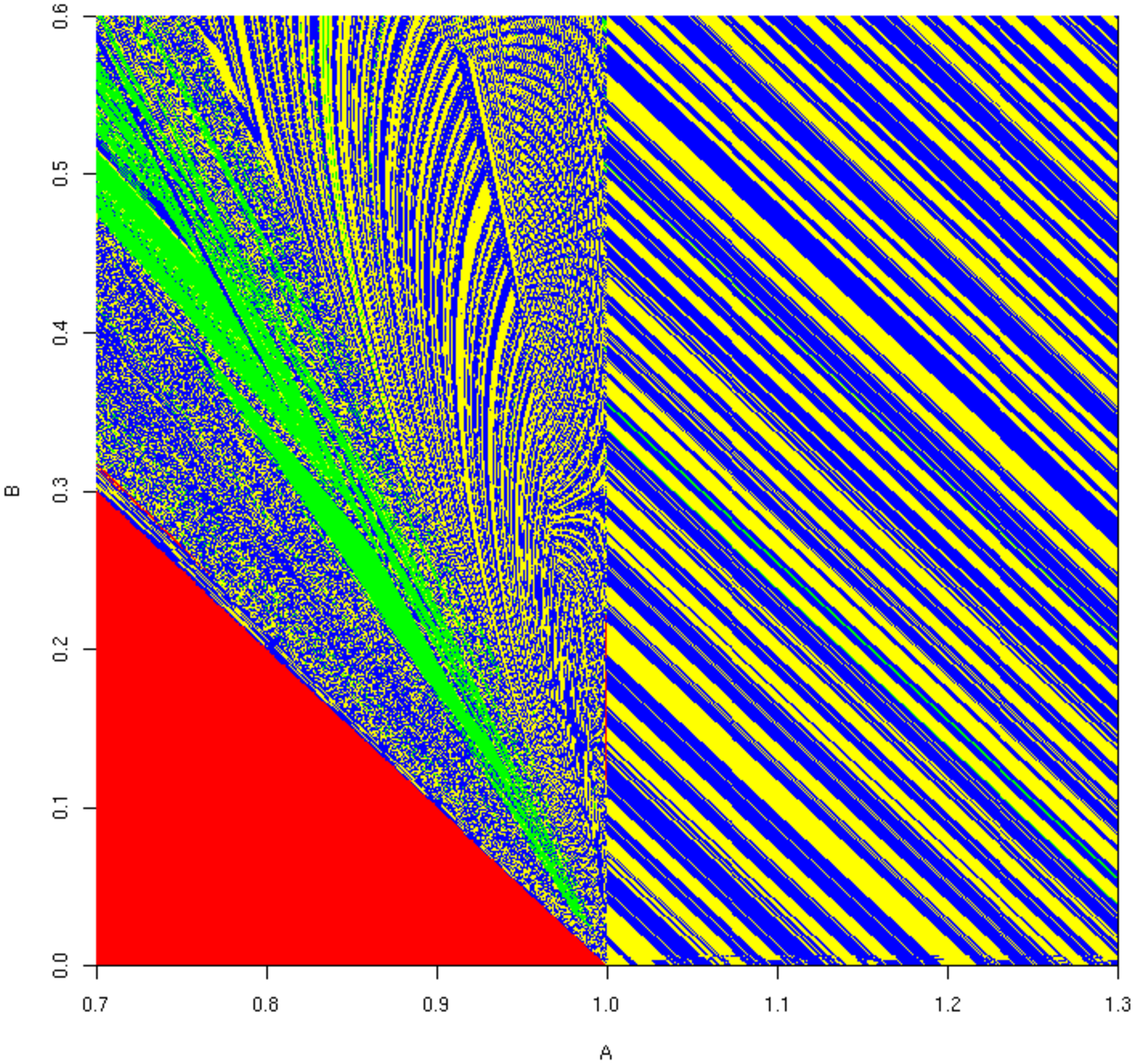}
\caption[Basins of attraction for the delay
equation~\eqref{eq.pwc2} for initial functions of the form $t \mapsto
A + B t$.]{Basins of attraction for the delay equation~\eqref{eq.pwc2}
for initial functions of the form $t \mapsto A + B t$, $t \in [0,1]$.
The basin colors are those of the corresponding solutions shown in
Figure~\ref{fig.persolspwc}.  The second image shows an enlargement,
by a factor of $10$, of part of the first image.}
\label{fig.pwcbasins1}
\end{center}
\end{figure}

To illustrate that there is nothing very special about the subspace
$\Sigma$ spanned by functions of the form~\eqref{eq.icsubspace1},
Figure~\ref{fig.pwcbasins2} shows basins of attraction for the
DDE~\eqref{eq.pwc2} for initial functions of the form
\begin{equation} \label{eq.icsubspace2}
  x(t) = A \cos (2 \pi t) + B \sin (2 \pi t), \quad t \in [0,1].
\end{equation}
This figure exhibits a fractal structure similar to that seen in
Figure~\ref{fig.pwcbasins1}.
\begin{figure}
\begin{center}
\includegraphics[height=3.5in]{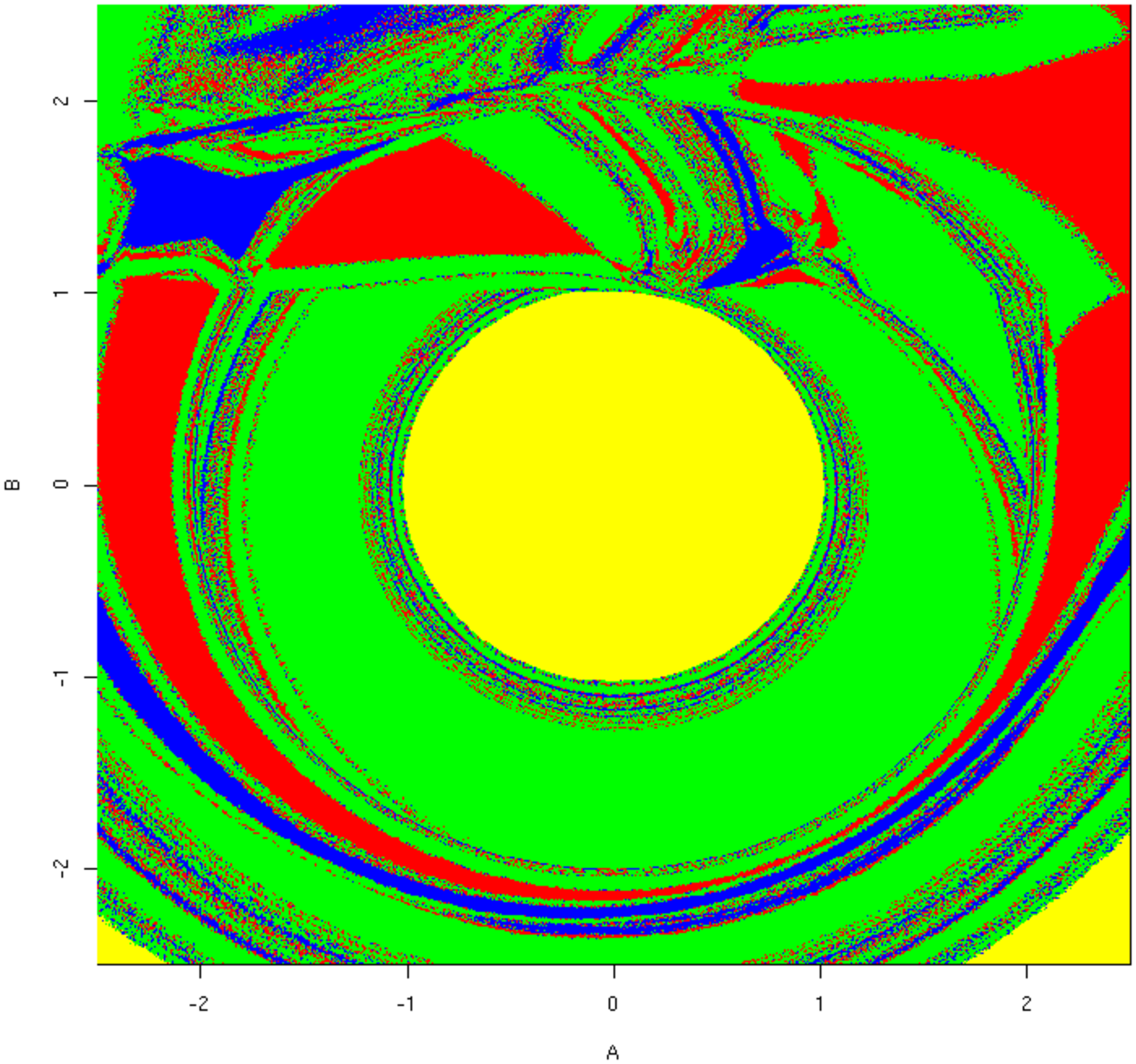} \\
\includegraphics[height=3.5in]{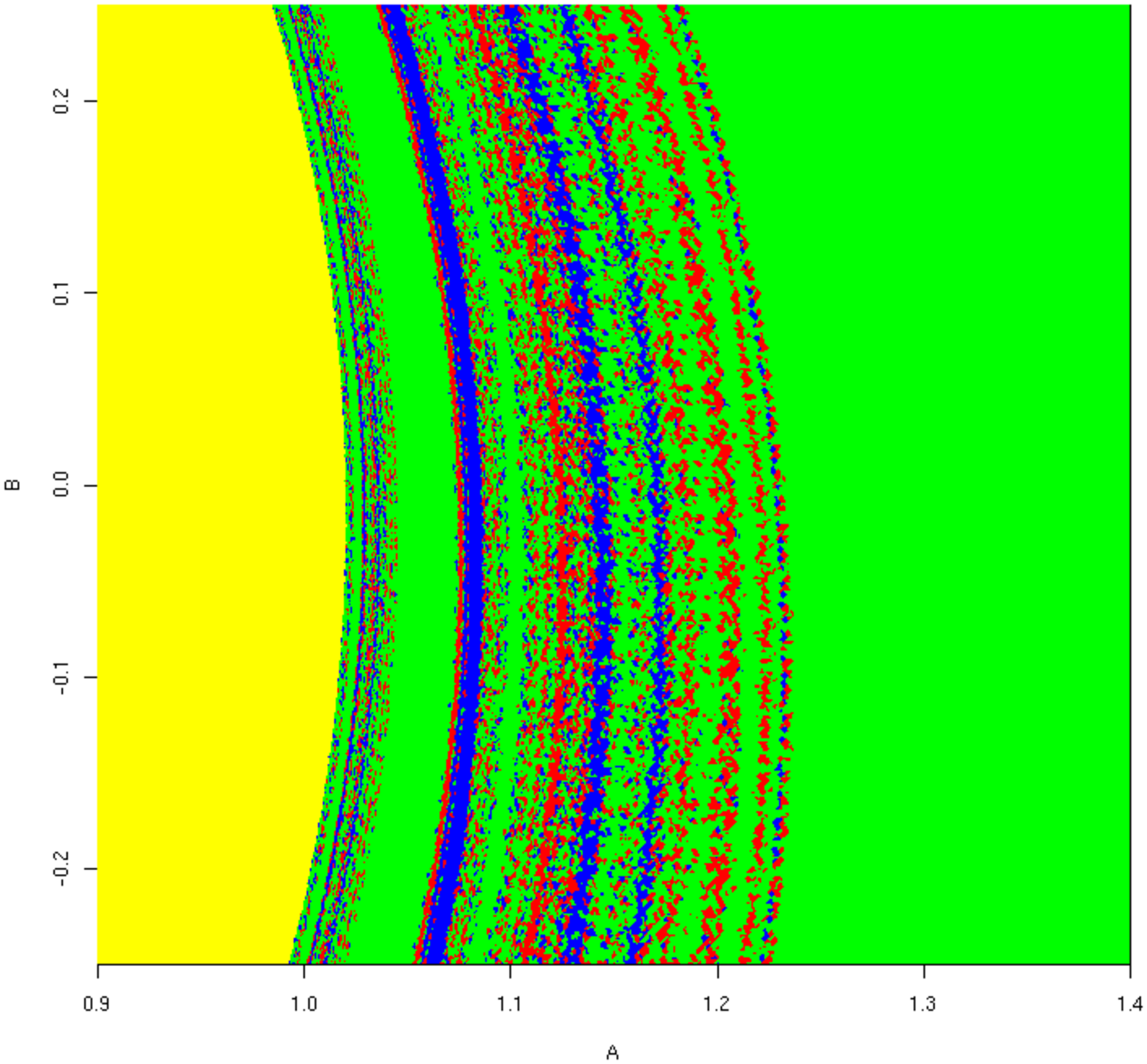}
\caption[As Figure~\ref{fig.pwcbasins1}, except with initial functions
of the form $t \mapsto A \sin(2 \pi t) + B \cos(2 \pi t)$.]{As
Figure~\protect\ref{fig.pwcbasins1}, except with initial functions of
the form $t \mapsto A \sin(2 \pi t) + B \cos(2 \pi t)$.  The second
image shows an enlargement, by a factor of $10$, of part of the first
image (the granularity is a result of limited numerical accuracy).}
\label{fig.pwcbasins2}
\end{center}
\end{figure}

% --------------------------------------------------------------------
\clearpage
\subsubsection{Mackey-Glass equation} \label{sec.transmg}

The Mackey-Glass delay differential equation \cite{MG77},
\begin{equation} \label{eq.mg4}
  x'(t) = -\alpha x(t) + \beta \frac{x(t-1)}{1 + x(t-1)^{10}},
\end{equation}
was originally introduced to model oscillations in neutrophil
populations.  It has subsequently been the subject of much study
because of the variety of dynamical phenomena it exhibits.

Losson \etal\ \cite{LML93} found multistability (coexistence of two
attracting periodic orbits) in the DDE~\eqref{eq.mg4}, but for the
parameter sets they considered they found basins of attraction with
only simple, non-fractal boundaries.  We have made a more thorough
search of parameter space, sampling the rectangle $[0.5,10] \times
[0,200]$ in $(\alpha,\beta)$-space at $20 \times 20$ resolution.  For
each sample, $100$ different numerical solutions of~\eqref{eq.mg4}
were computed.  In this way we found numerous parameter values for
which equation~\eqref{eq.mg4} has higher-order multi-stability and
basins of attraction with apparently fractal boundaries.  For example,
at $\alpha=1/0.1625$, $\beta=12/0.1625$, there are the four coexisting
attracting periodic solutions shown in Figure \ref{fig.persolsmg}.
These solutions occur in symmetric positive and negative pairs, due to
the invariance of equation~\eqref{eq.mg4} under the transformation
$x(t) \mapsto -x(t)$.
\begin{figure}
\begin{center}
\includegraphics[height=2in]{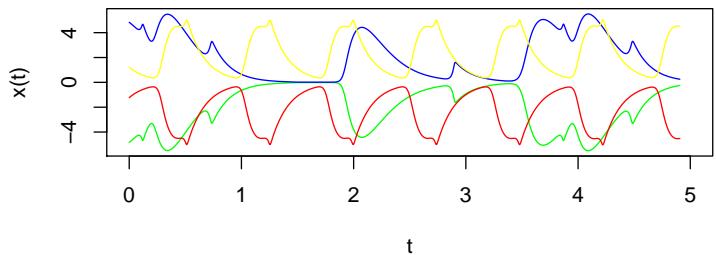}
\caption[Four coexisting attracting periodic solutions of the
Mackey-Glass equation.]{Four coexisting attracting periodic solutions
of the Mackey-Glass equation~\eqref{eq.mg4} with parameters
$\alpha=1/0.1625$, $\beta=12/0.1625$.}
\label{fig.persolsmg}
\end{center}
\end{figure}

Figure~\ref{fig.mgbasins1} shows the basins of attraction for
equation~\eqref{eq.mg4}, or rather, part of the intersection of these
basins with the subspace of $C$ spanned by functions of the
form~\eqref{eq.icsubspace1}.  Also shown is a sequence of
magnifications, spanning four orders of magnitude, that suggest the
basin boundaries are fractal sets with Cantor-like structure.  The
reflection symmetry in the first figure is due to invariance of the
DDE under the transformation $x(t) \mapsto -x(t)$, or equivalently
$(A,B) \mapsto (-A,-B)$.
\begin{figure}
\begin{center}
\setlength{\unitlength}{0.84in}
\begin{picture}(6.5,6.5)
\put(0,3.25){ \includegraphics[width=2.73in]{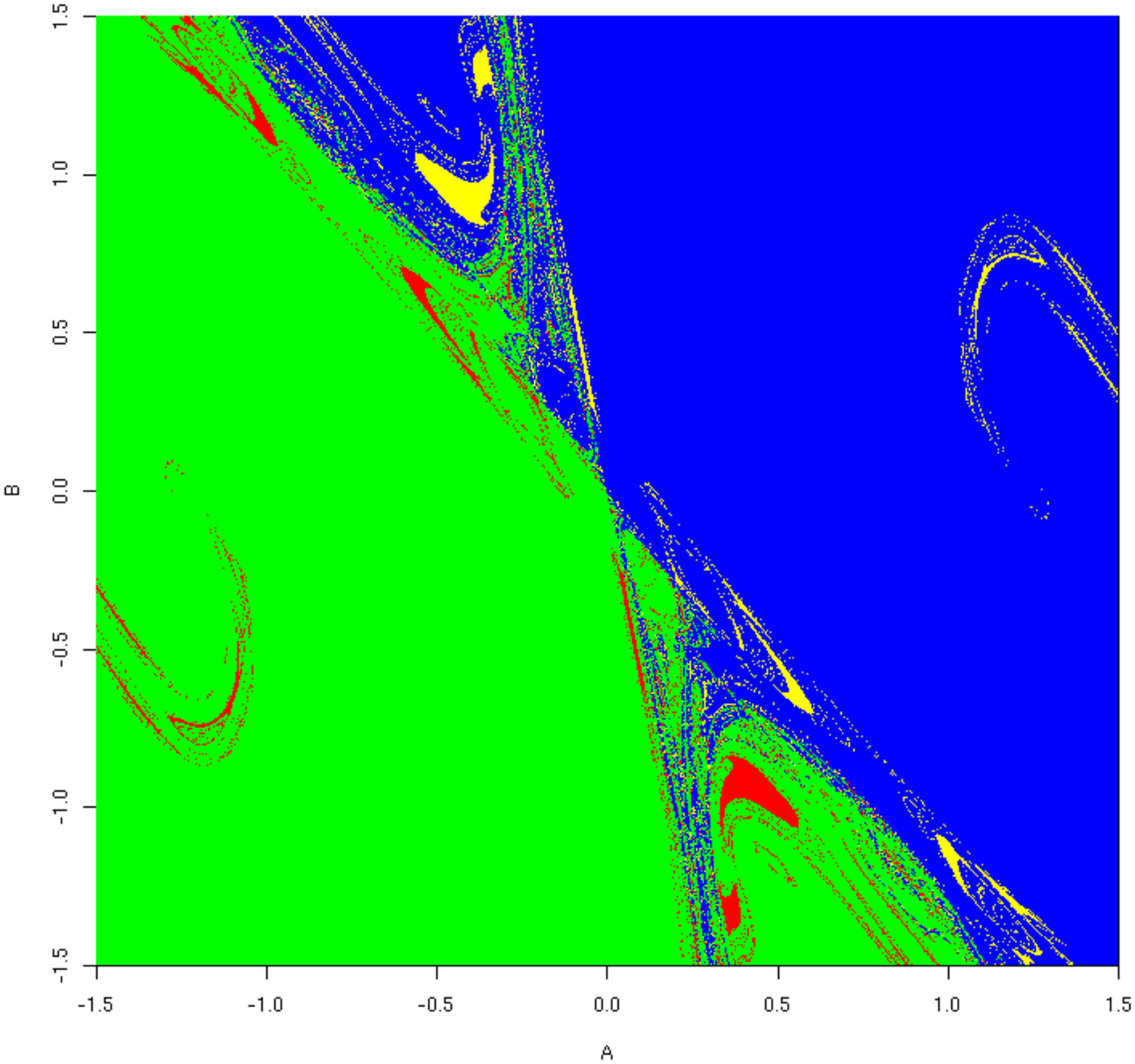} }
\put(3.25,3.25){ \includegraphics[width=2.73in]{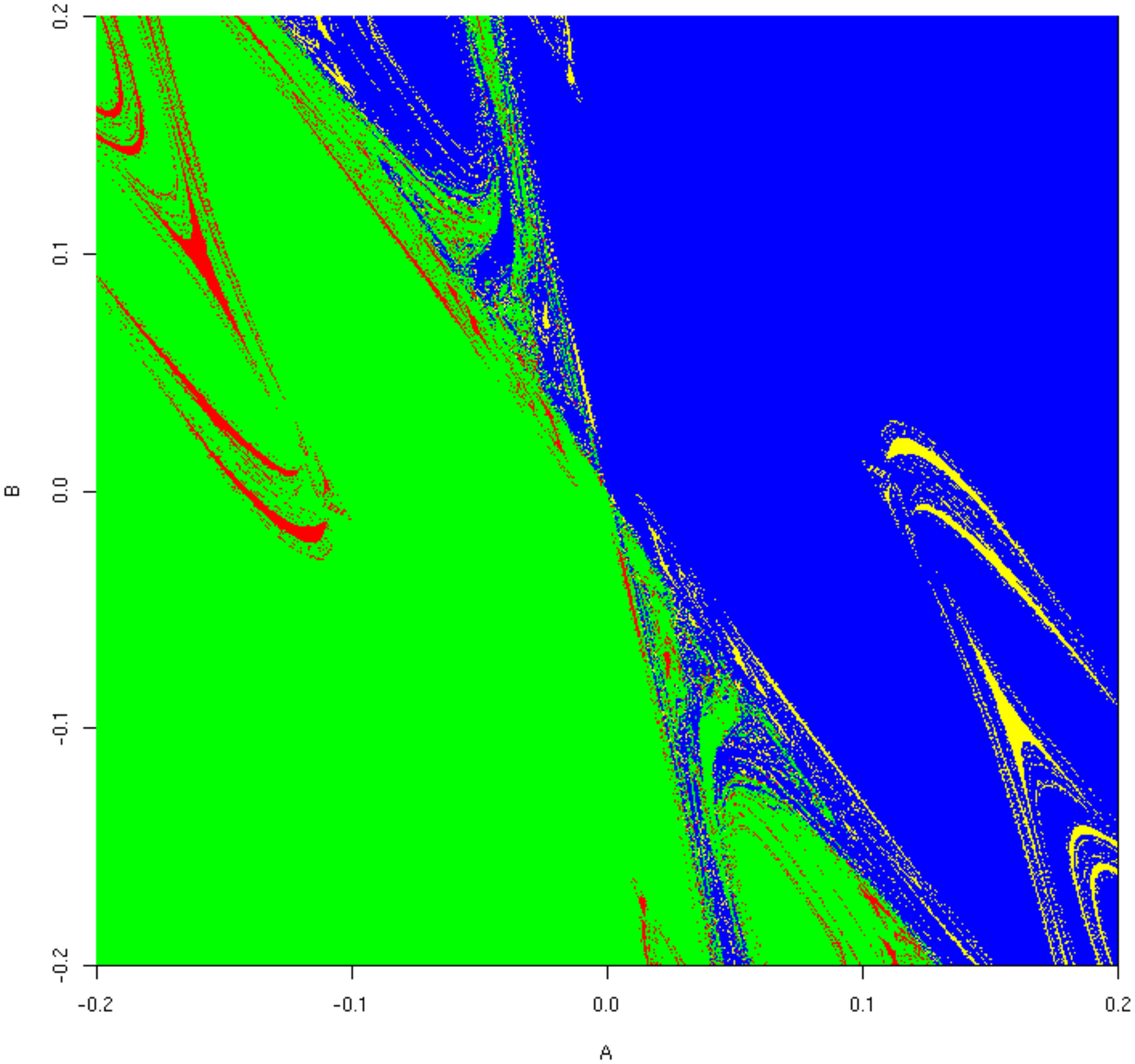} }
\put(0,0){ \includegraphics[width=2.73in]{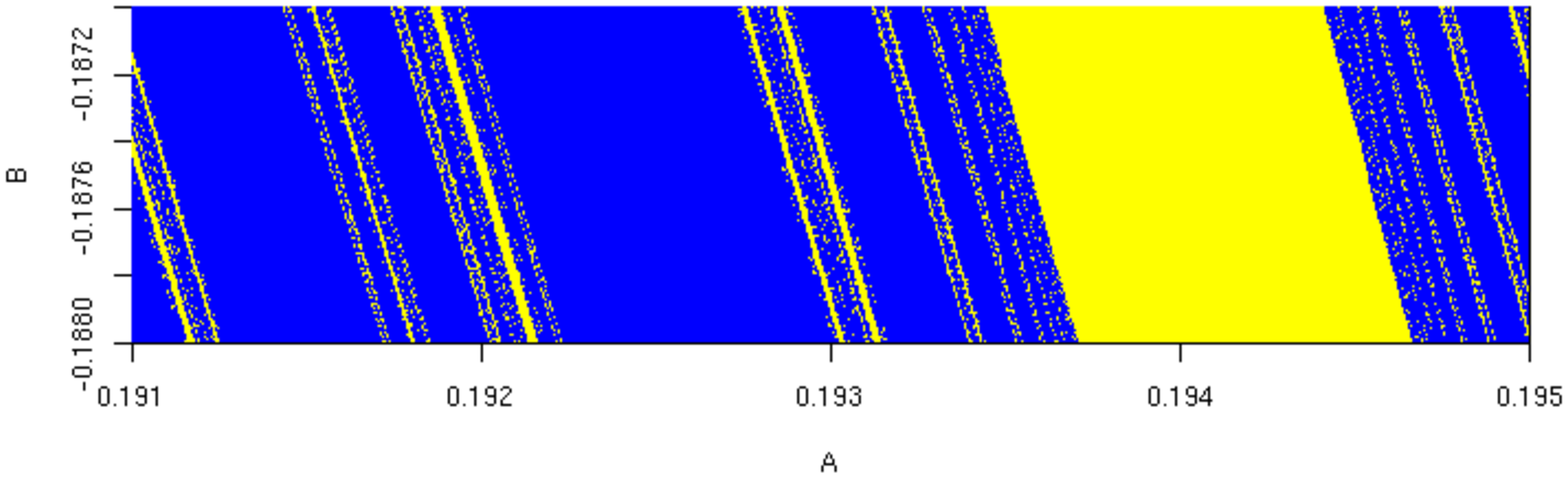} }
\put(0,1.6){ \includegraphics[width=2.73in]{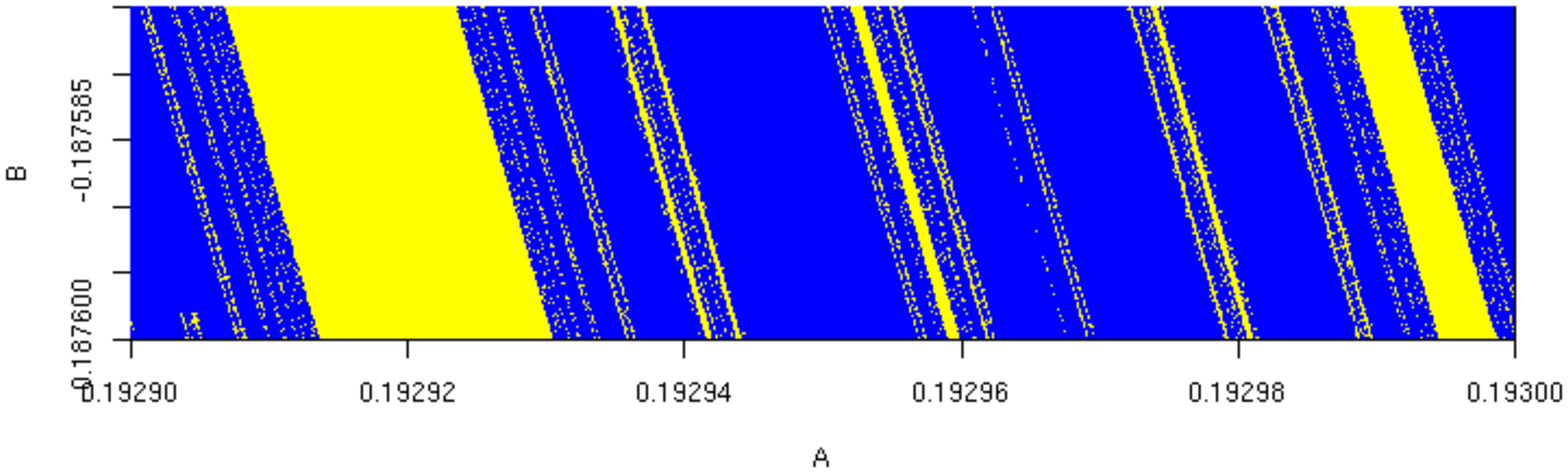} }
\put(3.25,0){ \includegraphics[width=2.73in]{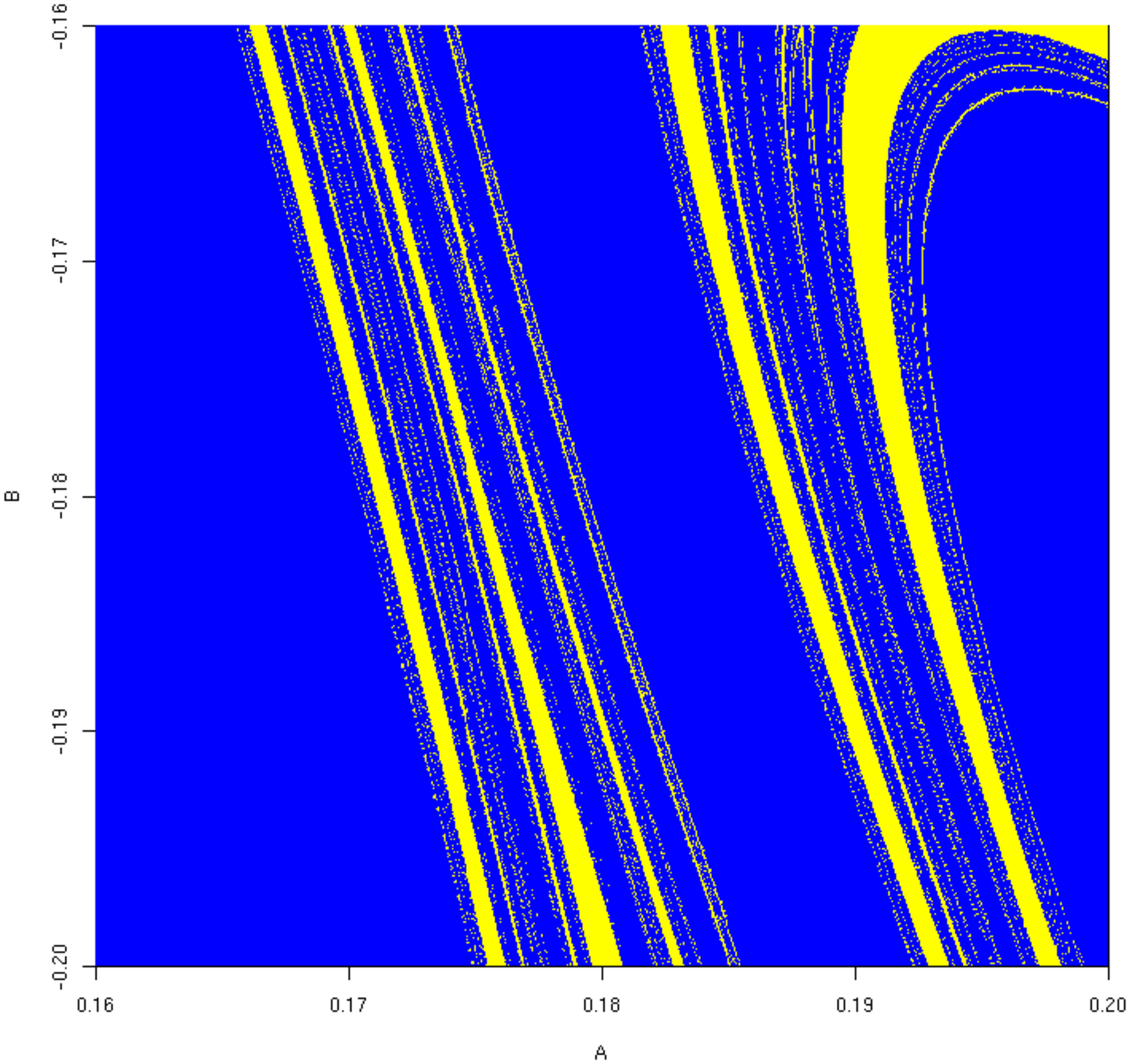} }
\put(1.57,4.75){\framebox(0.35,0.35){}}
\dashline{0.05}(1.57,4.75)(3.58,3.59)
\dashline{0.05}(1.57,5.1)(3.58,6.22)
\put(6.13,3.6){\framebox(0.28,0.28){}}
\dashline{0.05}(6.13,3.88)(3.58,2.99)
\dashline{0.05}(6.42,3.88)(6.42,2.99)
\put(5.8,1.09){\framebox(0.28,0.07){}}
\dashline{0.05}(6.09,1.08)(3.16,0.34)
\dashline{0.05}(6.09,1.16)(3.16,1.03)
\put(1.67,0.6){\framebox(0.07,0.02){}}
\dashline{0.05}(1.67,0.6)(0.33,1.94)
\dashline{0.05}(1.74,0.6)(3.16,1.94)
\end{picture}\\
\caption[Basins of attraction of the Mackey-Glass
equation for initial functions of the form $t \mapsto A
+ B t$.]{Basins of attraction of the Mackey-Glass
equation~\eqref{eq.mg4} for initial functions of the form $t \mapsto A
+ B t$, $t \in [0,1]$.  The basin colors are those of the
corresponding solutions shown in Figure~\ref{fig.persolsmg}.  A
sequence of magnifications over $4$ orders of magnitude is shown to
highlight the fractal structure of the basin boundaries.}
\label{fig.mgbasins1}
\end{center}
\end{figure}

%% Since high-order multi-stability in the Mackey-Glass has not been
%% reported before, it is of interest to quantify how ubiquitous this
%% phenonenon is.  That is, we would like to know how large is the
%% $(\alpha,\beta)$ parameter set of equation \eqref{eq.mg4} for which
%% multistability of order higher than two occurs.  Figure
%% \ref{MGparsets} shows the results of our search of the parameter
%% space, where on a fine grid in the $(\alpha,\beta)$-plane we have
%% plotted a pixel at every point for which we found more than two
%% distinct attractors.  For each set of parameter values, the search for
%% distinct attractors was done by computing long-time solutions of
%% \eqref{eq.mg4} for $10,000$ different initial functions of the form
%% \eqref{eq:ics}, and counting the number of distinct asymptotic
%% solutions that occurred.\footnote{I haven't made this figure yet, but
%% one interesting possibility is to include information (using color)
%% about \emph{how many} distinct attractors there are for each set of
%% parameter values.  The boundaries between different colors would then
%% be bifurcation sets.}
%% \begin{figure}[p]
%% \begin{center}
%% %\includegraphics{mgparsets}
%% \end{center}
%% \caption{Subsets of $(\alpha,\beta)$ parameter space for which the
%% Mackey-Glass equation~\eqref{eq.mg4} exhibits multi-stability.}
%% \label{MGparsets}
%% \end{figure}
%% 

% ------------------------------------------------------------------------
\clearpage
\subsection{Chaotic transients} \label{sec.longtrans}

If transient chaos does occur in the delay equations~\eqref{eq.pwc2}
and~\eqref{eq.mg4} then it should be possible to find solutions with
long chaotic transients.  If there is a chaotic saddle, the basin
boundary will contain the saddle and its stable manifold.  Thus
chaotic transients should be found for initial functions close to the
basin boundary.  Such initial functions, of the form $x(t) = A + B t$,
$t \in [0,1]$, can simply be read off Figures~\ref{fig.pwcbasins1}
and~\ref{fig.mgbasins1} by choosing a point $(A,B)$ near a boundary
between basins of attraction.  This point can be refined, using a
bisection algorithm, to obtain a point $(A,B)$ and corresponding
initial function $\phi$ arbitrarily close (within numerical precision)
to the basin boundary.  Integration of the DDE forward from this
initial function is then expected to yield a solution with a long
chaotic transient.

Figure~\ref{fig.pwclongtrans} shows a numerical solution of
equation~\eqref{eq.pwc2} corresponding to an initial function found in
this way.  This solution does indeed appear to exhibit aperiodic
behavior for a considerable duration (about $50$ time units) before
settling down to one of the attracting periodic solutions.
Figure~\ref{fig.mglongtrans} shows a solution of the Mackey-Glass
equation~\ref{eq.mg4} obtained in the same manner, also with a long
chaotic transient that settles down to an attracting periodic orbit
after about $50$ time units.

\begin{figure}
\begin{center}
\includegraphics[width=5.5in]{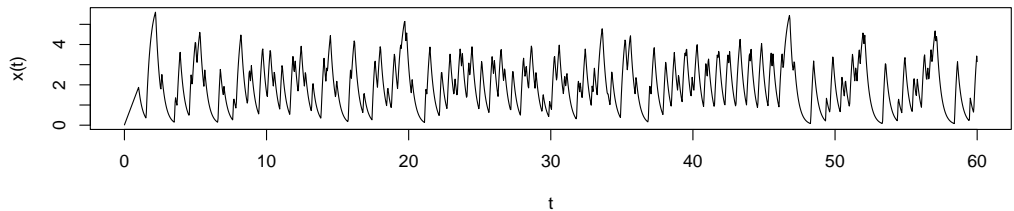}
\caption[A transient chaotic solution of the delay equation
\eqref{eq.pwc2}.]{A transient chaotic solution of the delay equation
\eqref{eq.pwc2}, corresponding to the initial function $x(t) =
0.0131614 + 1.870858 t$, $t \in [0,1]$.}
\label{fig.pwclongtrans}
\end{center}
\end{figure}
\begin{figure}
\begin{center}
\includegraphics[width=5.5in]{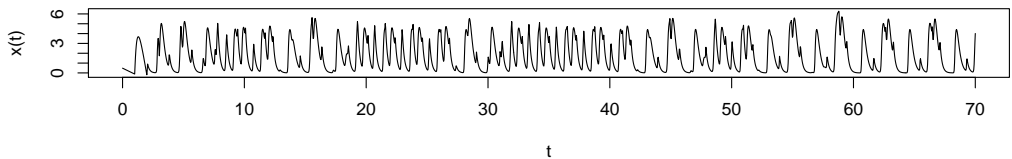}
\caption[A transient chaotic solution of the Mackey-Glass
equation.]{A transient chaotic solution of the Mackey-Glass
equation~\eqref{eq.mg4}, corresponding to the initial function $x(t) =
0.482417 - 0.599163 t$, $t \in [0,1]$.}
\label{fig.mglongtrans}
\end{center}
\end{figure}

The observation of long chaotic transients, together with the
existence of seemingly fractal basins of attraction, suggests the
presence of chaotic saddles in the delay equations~\eqref{eq.pwc2}
and~\eqref{eq.mg4}.  It is of interest to approximate these saddles
numerically.  This problem is considered in the following sections.

% ------------------------------------------------------------------------
\section{Numerical Analysis} \label{sec.algorithms}

Analysis of transient chaos---that is, finding a numerical
approximation of the saddle and characterizing it quantitatively, \eg\
in terms of ergodic properties such as Lyapunov exponents and fractal
dimensions---requires a method for computing arbitrarily long orbits
on or very near the chaotic saddle.  Because the saddle is unstable,
this is not straightforward: rather than being attracted to the
saddle, any numerical trajectory starting near the saddle, no matter
how close, will eventually depart from it.

A variety of algorithms have been developed for computing numerical
trajectories very near a chaotic saddle, including the
``straddle-orbit''\cite{BGOY88}, ``PIM'' (Proper Interior Maximum)
\cite{Nusse89}, and ``stagger-and-step'' \cite{Sweet01} methods.  Each
of these algorithms uses the following strategy.  Starting with an
initial phase point near the saddle (or its stable manifold), evolve
this point forward under the system dynamics, occasionally applying a
perturbation of size less than \eps\ to keep the trajectory within a
small neighborhood of the saddle.  By construction the resulting
trajectory $\{x_n\}$ satisfies
\begin{equation}
  |x_{n+1} - S(x_n)| < \eps,
\end{equation}
so $\{x_n\}$ is an $\eps$-pseudo-orbit of $S$.  It is presumed that if
\eps\ is sufficiently small then $\{x_n\}$ approximates a true orbit on the saddle
(\eg\ by the Shadowing Lemma, \cf\ Section~\ref{sec.shadowing}).  The
important differences between the various algorithms, which we outline
below, lie in the methods used to find appropriate perturbations.

Throughout the following, let $\{S^n: n \in \Zplus\}$ be a
discrete-time dynamical system defined by iterates of a transformation
$S: X \to X$.

% -----------------------------------------------------------------------
\subsection{Straddle orbit method}

The goal of the straddle orbit method is to construct a perturbed
orbit $\{x_n\}$ that follows the basin boundary.  Between iterates of
$S$, the straddle orbit method~\cite{BGOY88} employs a bisection
algorithm to perturb an orbit to within a small neighborhood of the
basin boundary.  The algorithm is as follows.

\begin{center}
\fbox{\parbox{5in}{
\begin{enumerate}
\item Choose phase points $x_A, x_B \in X$ that lie in different basins.
\item Let $x_C$ be the midpoint of the segment $x_Ax_B$; determine which basin $x_C$ is in,
\eg\ by iteration of $S$.
\item If $x_C$ is in the same basin as $x_A$, let $x_A:=x_C$, else let $x_B:= x_C$.
\item If $|x_A-x_B| \geq \eps$, return to 2.
\item Let $x_A:= S(x_A)$, $x_B:= S(x_B)$.
\item If $|x_A-x_B| \geq \eps$, return to 2, else return to 5.
\end{enumerate}
}}
\end{center}

The perturbation phase of the algorithm is carried out in steps 2--4,
which use a bisection algorithm to isolate the basin boundary between
points $x_A$ and $x_B$ separated by a distance less than \eps.  Points
$x_A$ and $x_B$ are both evolved under iterates of $S$, and the
bisection algorithm is repeated whenever $x_A$ and $x_B$ differ by
more than \eps.

The straddle orbit $\{x_n\}$ is the sequence of points $x_A$
(alternatively $x_B$) obtained from step 5.  By construction, the
straddle orbit remains within distance \eps\ of the basin boundary.
It is presumed that for \eps\ sufficiently small the straddle orbit
approximates an orbit on the boundary.  If the boundary consists of
the stable manifold of the chaotic saddle then the straddle orbit
should follow this stable manifold and, after an initial transient
phase, should remain with
\eps\ of the saddle itself.

% -----------------------------------------------------------------------
\subsection{PIM method}

The goal of the PIM method~\cite{Nusse89,Nusse91} is to construct a
perturbed orbit $\{x_n\}$ that remains indefinitely within some
neighborhood $R$ of the chaotic saddle.  This is done by perturbing
the orbit so as to increase the escape time $T(x_n)$, which is the
number of iterates of $S$ required to take $x_n$ out of $R$.

Suppose the dynamical system $S: X \to X$ has a chaotic saddle
$\Lambda$.  Let $R \subset X$ be a \emph{transient region}, such that
$\Lambda \subset R$ and $R$ contains no attractor.  For $x
\in R$, define the \emph{escape time} $T(x)$ by
\begin{equation}
\begin{split}
  T(x) &= \min\{n > 0: S^n(x) \notin R\} \\
     ( &= \infty \text{ if } S^n(x) \in R \; \forall n>0 ).
\end{split}
\end{equation}

In the following, an ordered set $(x_a,x_b,x_c)$ of points $x_a, x_b,
x_c \in X$ is said to be a PIM triple if $x_b$ lies on the segment
$x_a x_c$ and $T(x_b) > \max\big( T(x_a),T(x_c)\big)$.

\begin{center}
\fbox{\parbox{5in}{
\begin{enumerate}
\item Choose a PIM triple $(x_a,x_b,x_c)$ such that $x_a$ and $x_c$ lie in different basins.
\item Choose $N $ equally spaced points on the segment $x_a x_c$.  From these
choose a PIM triple $(\bar{x}_a,\bar{x}_b,\bar{x}_c)$ such that the
segment $\bar{x}_a \bar{x}_c$ is a proper subset of the segment $x_a
x_b$.
\item Let $x_a:= \bar{x}_a$, $x_c:= \bar{x}_c$.
\item If $|x_a - x_c| \geq \eps$, return to 2.
\item Let $x_a:= S(x_a)$, $x_c:= S(x_c)$.
\item If $|x_a-x_c| \geq \eps$, return to 2, else return to 5.
\end{enumerate}
}}
\end{center}

The PIM orbit $\{x_n\}$ is the sequence of points $x_a$ (alternatively
$x_b$) obtained from step 5.  By construction the PIM orbit $\{x_n\}$
satisfies $T(x_n) > 0$ for all $n$, so $\{x_n\}$ remains within the
transient region $R$ for all time.

% -----------------------------------------------------------------------
\subsection{Stagger-and-step method}

The stagger-and-step method~\cite{Sweet01} is also based on seeking
perturbations that increase the escape time of the orbit, but
perturbations are not restricted to a particular line segment.  Let
the escape time $T(x)$ be defined as in the PIM method above, and let
$\Tstar > 0$ (the minimum allowed escape time).  Then a
stagger-and-step trajectory $\{x_n\}$ results from iterating the
following algorithm.

\begin{center}
\fbox{\parbox{5in}{
\begin{enumerate}
\item Choose $x_0 \in X$ with $T(x_0) > \Tstar$.
\item If $T(x_n) \leq \Tstar$, find a random perturbation $r \in X$, $|r|<\eps$,
such that $T(x_n+r) > \Tstar$, and let $x_n:= x_n + r$.
\item Let $x_{n+1}:= S(x_n)$.
\item Return to 2.
\end{enumerate}
}}
\end{center}

% -----------------------------------------------------------------------
\section{Application to Delay Equations} \label{sec.analysis}

\subsection{Approximate discrete-time map} \label{sec.approxmap}

The algorithms described above are formulated in the context of
discrete-time dynamical systems.  For a continuous-time system they
can be applied to an appropriate discretized version of the dynamics,
for example iterates of the time-one map.  In the case of a delay
equation~\eqref{eq.ddetype} the time-one map is a transformation $S =
S_1: C \to C$.  In a numerical simulation this transformation cannot
be represented exactly, and we must resort to a finite dimensional
approximation.  For this purpose the phase point $x_t \in C$ can be
represented by the vector $\mathbf{u}(t) =
\big(u_0(t),\ldots,u_N(t)\big)$ of values
\begin{equation} \label{eq.uidef}
  u_i(t) = x_t(s_i) = x(t+s_i),
\end{equation}
that $x_t$ takes on a uniform grid
\begin{equation}
  s_i = -1 + i/N, \quad i=0,\ldots,N,
\end{equation}
typically with $N$ of the order $10^2$ or $10^3$.  The time-one map
$S(\phi)$ can be carried out by constructing a numerical solution $x$
with initial function $\phi$ to time $t=1$, and evaluating the vector
of solution values $u_i=x(1+s_i)$, $i=0,\ldots,N$.

To be more precise, in numerical integration of the
DDE~\eqref{eq.ddetype} with time step $0 < h \ll 1$, $\mathbf{u}(t+h)$
is approximated by
\begin{equation}
  \mathbf{u}(t) \mapsto \mathbf{u}(t+h) \approx T\big( \mathbf{u}(t)
  \big)
\end{equation}
for some transformation $T: \Reals^{N+1} \to \Reals^{N+1}$, the
details of which depend on the choice of integration scheme.  With
$h=1/N$ we have
\begin{equation}
\begin{split}
  \mathbf{u}(t+1) &= \mathbf{u}(t + Nh) \\
                  &\approx T^N\big( \mathbf{u}(t) \big),
\end{split}
\end{equation}
so that
\begin{equation} \label{eq.Stild}
  \tilde{S} = T^N
\end{equation}
approximates the time-one map $S$.  The transformation $\tilde{S}:
\Reals^{N+1} \to \Reals^{N+1}$ is readily implemented using any of the
various codes available for numerical integration of
DDEs.\footnote{Throughout this chapter, numerical integration is
performed using the Fortran code DKLAG6~\cite{Thomp} interfaced with
the R language~\cite{R}.}

% ----------------------------------------------------------------------
\subsection{Computing escape times}

Implementing the PIM and stagger-and-step methods requires a method of
computing the escape time function $T$.  This in turn requires a
practical method of describing the transient region $R$.  One method
is to define $R$ to be the set
\begin{equation}
  R = \{x \in X: \text{dist}(x,A_i) > \delta \; \forall i \},
\end{equation}
for some $\delta > 0$, where the $A_i$ are the attractors of the
system.  Presumably these are already known or approximated, so that
$\text{dist}(x,A_i)$ can be computed.  This approach was used
in~\cite{Nusse89} for low-dimensional systems.

Although it has not received mention elsewhere, it seems clear that
the $A_i$ must also include periodic orbits whose stability is of
saddle type.  Otherwise, the presence of such an orbit implies the
existence of trajectories with arbitrarily large escape times, but
which do not lie near the chaotic saddle (\eg, trajectories that
follow the stable manifold of the unstable periodic orbit).  If a
particular such orbit $A$ is not counted among the $A_i$ then the PIM
and stagger-and-step methods fail, with the trajectory $\{x_n\}$
converging to $A$.\footnote{As a side-effect that might be exploited,
the PIM and stagger-and-step methods appear to be novel ways to find
and approximate saddle type periodic orbits.}  In this case the
resulting trajectory $\{x_n\}$ automatically yields an approximation
of $A$, which can then be included among the $A_i$ and the method
applied again.

% ------------------------------------------------------------------------
\subsection{Failure of existing algorithms} \label{sec.transfail}

In principle, any of the methods in~\cite{BGOY88,Nusse89,Sweet01}
could be used to approximate the chaotic saddle (if one exists) for
the finite dimensional map $\tilde{S}: \Reals^{N+1} \to
\Reals^{N+1}$ defined in Section~\ref{sec.approxmap}.  We have
implemented all of these methods and applied them to the delay
equations~\eqref{eq.pwc2} and~\eqref{eq.mg4} for which we have found
evidence of transient chaos.  We find that each of these methods does
in fact generate an aperiodic trajectory of some duration, but that
eventually this trajectory converges to an unstable (saddle type)
periodic orbit, regardless of the method used.  The unstable periodic
solutions found in this way for equations~\eqref{eq.pwc2}
and~\eqref{eq.mg4} are shown in Figures~\ref{fig.unstablepwc}
and~\ref{fig.unstablemg}, respectively.  This mode of failure is
interesting since it has not been reported before.
\begin{figure}
\begin{center}
\includegraphics[width=\figwidth]{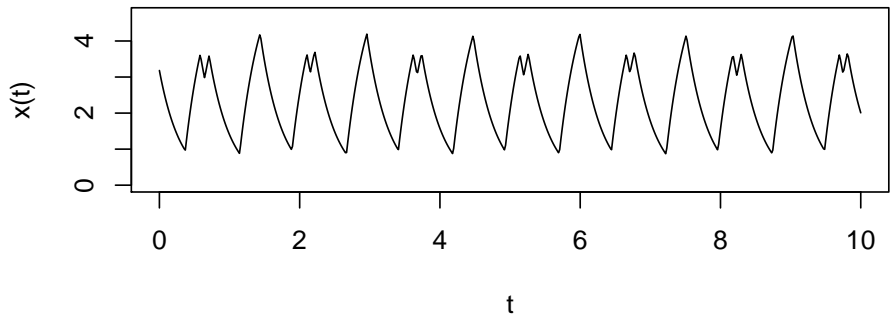}
\caption{An unstable (saddle type) periodic solution of the
delay equation~\eqref{eq.pwc2}.}
\label{fig.unstablepwc}
\end{center}
\end{figure}
\begin{figure}
\begin{center}
\includegraphics[width=\figwidth]{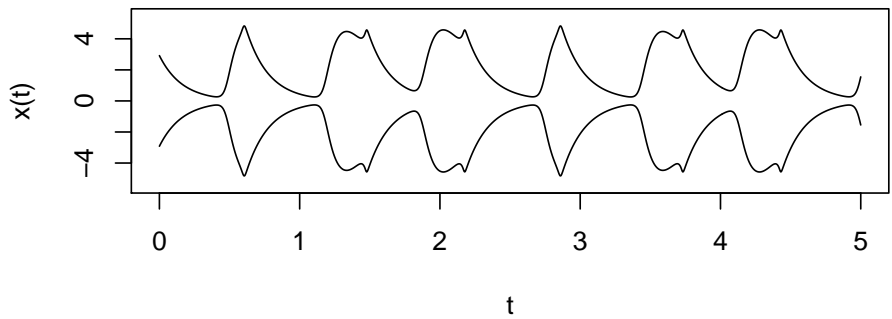}
\caption[Unstable (saddle type) periodic solutions of the
Mackey-Glass equation.]{Unstable (saddle type) periodic solutions of
the Mackey-Glass equation~\eqref{eq.mg4}.}
\label{fig.unstablemg}
\end{center}
\end{figure}

There is a simple plausible explanation for the failure of the
straddle orbit method.  For each of the DDEs under consideration there
appears to be a saddle type periodic orbit.  The stable manifold of
any such orbit is contained within the basin boundary.  A straddle
orbit, which by construction lies close to the basin boundary, is
eventually perturbed onto (or near) this stable manifold.  The orbit
thereafter follows this stable manifold rather than the stable
manifold of the chaotic saddle, and consequently the straddle orbit is
asymptotic to the periodic orbit.

In both the PIM and stagger-and-step methods the unstable periodic
orbit can be avoided by excluding it from the transient region $R$, as
described in the previous section.  This prevents these methods from
converging to the periodic orbit, since an orbit that enters a
neighborhood of an unstable periodic orbit is deemed to have left the
transient region.  However, with this provision both algorithms
eventually stall at a point where they are unable to find a
perturbation that yields an increased escape time.  The reason for
this failure seems to be the mechanism illustrated in
Figure~\ref{fig.transfail}.  Here the orbit through the phase point
$x$ exits the transient region $R$ by entering a neighborhood $\Omega$
of an unstable periodic orbit, represented by the point $P$.  The
orbit through the perturbed phase point $\tilde{x}$ narrowly misses
$\Omega$, so that $T(\tilde{x}) > T(x)$.  Since it increases the
escape time of the trajectory, $\tilde{x}$ is taken as a
``successful'' perturbation in either the PIM or stagger-and-step
method.  However, this perturbation is spurious since $\tilde{x}$ is
carried by successive iterations of $S$ into a region where further
successful perturbations, spurious or otherwise, do not exist.  Both
algorithms come to a halt when the trajectory through $\tilde{x}$
reaches a point where all phase points within distance $\eps$ exit $R$
via $\Omega$, and all have the same escape time.
\begin{figure}
\begin{center}
% RT 4.9.2019: use standalone figure for submission to arxiv
%\psfrag{x}{$x$}
%\psfrag{xt}{$\tilde{x}$}
%\psfrag{O}{$\Omega$}
%\psfrag{R}{$R$}
%\psfrag{P}{$P$}
%\includegraphics{transfail}
\includegraphics{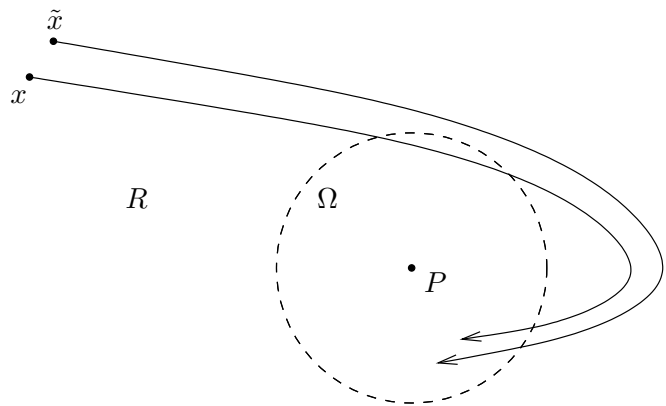}
\caption[Mechanism of failure of the PIM and stagger-and-step methods.]{Mechanism
of failure of the PIM and stagger-and-step methods.  The phase point
$x$ exits the transient region $R$ via a neighborhood $\Omega$ of an
unstable periodic orbit $P$.  The perturbed phase point $\tilde{x}$
just misses $\Omega$, so $T(\tilde{x}) > T(x)$.}
\label{fig.transfail}
\end{center}
\end{figure}

It has been pointed out~\cite{Sweet01} that the PIM method is expected
to fail if the chaotic saddle has more than one unstable direction.
This would suggest an alternative explanation for the failure of the
method here.  However, the results of Section~\ref{sec.translyap}
provide evidence that the chaotic saddle has only one unstable
direction, for both of the DDEs~\eqref{eq.pwc2} and~\eqref{eq.mg4}.
Therefore the failure of the method does not seem to be caused by the
presence of multiple unstable directions.

% ---------------------------------------------------------------------------
\subsection{Modified stagger-and-step method} \label{sec.ssmod}

The following slight modification of the stagger-and-step method
circumvents the difficulty described above, and proves effective for
the delay equations considered here.  Before applying it to DDEs, we
present the modified algorithm in its general form.

The idea behind the method is to make a more aggressive search for a
``successful stagger'', \ie\ a perturbation that moves $x_n$ onto a
nearby trajectory that remain within a neighborhood of the saddle for
a longer period of time $T(x_n)$.  In the original stagger-and-step
method a perturbation is sought only when $T(x_n) =
\Tstar$.  For $T(x_n) > \Tstar$, $x_n$ is simply iterated forward under $S$
until $T(x_n)=\Tstar$.  In our modification, up to $N>0$ random
perturbations are selected in an attempt to find a successful stagger
at \emph{each} iteration.

For a given $\Tstar>0$ (the minimum allowed escape time) and small
$\eps > 0$, a trajectory $\{x_n\}$ results from iteration of the
following algorithm.
\begin{center}
\fbox{\parbox{5in}{
\begin{enumerate}
\item Choose $x_0 \in X$ with $T(x_0) > \Tstar$.
\item Set $j:= 1$.
\item \label{it.first}Choose a random perturbation $r \in X$, $|r| < \eps$.
\item If $T(x_n+r) > T(x_n)$ then set $x_n:= x_n+r$ and go to 6.
\item If $T(x_n) \leq \Tstar$ or $j < N$ then set $j:= j+1$
and return to \ref{it.first}.
\item Let $x_{n+1}:= S (x_n)$.
\item Return to 2.
\end{enumerate}
}}
\end{center}

Thus each iteration consists of a possible successful stagger of the
current phase point, followed by an iteration of $S$.  Because $T(x_n)
\ge \Tstar > 0$ for all $n$, we have $x_n \in R$ for all time; that is, the trajectory
never leaves the transient region.  If the trajectory is aperiodic, it
is taken (after removal of an initial transient where it follows the
stable manifold of the chaotic saddle) as an approximation of the
chaotic saddle.

The parameters \eps, \Tstar\ and $N$ are adjustable.  As discussed
in~\cite{Sweet01}, the success of the stagger-and-step method can
depend on a careful choice of \Tstar; this observation applies also to
the modified algorithm.  If \Tstar\ is outside a range of suitable
values, a very large and possibly infinite number of ``stagger
attempts'' are necessary before a successful stagger (one that
increases $T(x_n)$ above \Tstar) is found.

Parameter $N$ is the maximum number of ``stagger attempts'' at each
iteration if $T(x_n)>\Tstar$.  The original stagger-and-step method is
similar but not identical to the case $N=1$.  In this case only one
stagger attempt is made for phase points with $T(x_n)>\Tstar$.
Usually this attempt fails.  After each such failure $T(x_n)$
decreases by one, until $T(x_n)=\Tstar$ where an exhaustive search for
a successful stagger is made.

In applying the method to delay equations we find, regardless of the
choice of \Tstar, that very frequently a successful stagger cannot be
found when $T(x_n)=\Tstar$, and the algorithm ``gets stuck''.  This
phenomenon appears to be related to the instability of trajectories in
a neighborhood of the saddle: under iteration of $\tilde{S}$, the
potential successful staggers within an \eps-ball at $x_n$ quickly
diverge to a distance greater than \eps\ from the current phase point.
If this divergence occurs in fewer than the number of iterations
required for $T(x_n)$ to fall to $\Tstar$, a successful stagger will
fail to exist when $T(x_n)=\Tstar$ is reached.  With our modification,
taking $N>1$ ($N=5$ has worked well in practice) greatly reduces the
frequency of this outcome, since a more thorough search for a
successful stagger is made at each iteration.

Even with this approach it occasionally happens that a phase point is
reached where $T(x_n)=\Tstar$ and where it is impossible to find a
successful stagger.  In such cases an effective remedy is to revert to
a previous iteration, make a thorough search until a successful
stagger is found, and continue the algorithm from this point.

% ----------------------------------------------------------------------
\section{Numerical Analysis Results}

\subsection{Visualizing the saddle and its invariant measure} \label{sec.transvis}

The stagger-step algorithm described above, when applied to the
approximate discrete-time map $\tilde{S}$ as defined for the delay
equations~\eqref{eq.pwc2} and~\eqref{eq.mg4}, yields a numerical
trajectory $\{\mathbf{u}(n) \in \Reals^{N+1}: n=0,1,2,\ldots\}$ that
approximates a trajectory in $C$ near the presumed chaotic saddle.
The corresponding solution of the DDE can be constructed from
equation~\eqref{eq.uidef}.  Figures~\ref{fig.saddlesolpwc}
and~\ref{fig.saddlesolmg} show chaotic solutions, computed in this
manner, for the delay equations~\eqref{eq.pwc2} and~\eqref{eq.mg4},
respectively.
\begin{figure}
\begin{center}
\includegraphics[width=\figwidth]{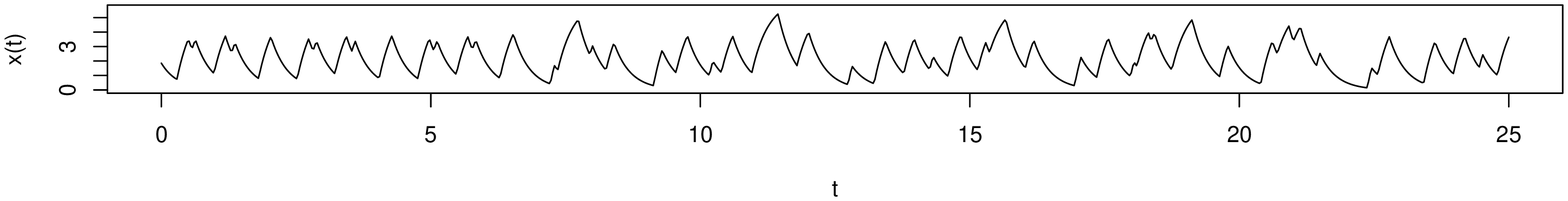} \\
\includegraphics[width=\figwidth]{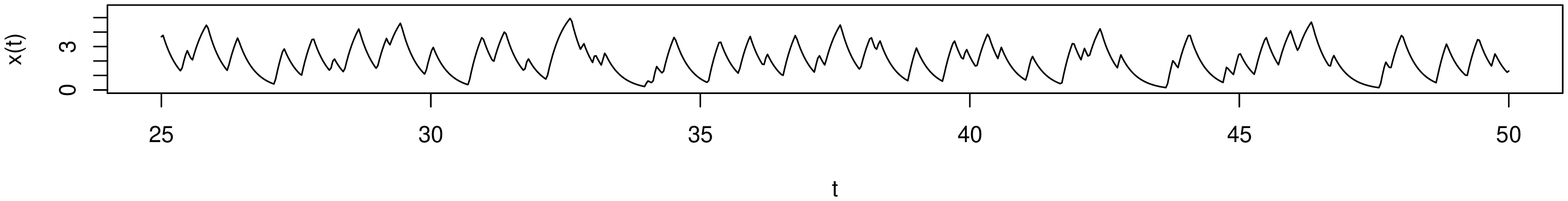} \\
\includegraphics[width=\figwidth]{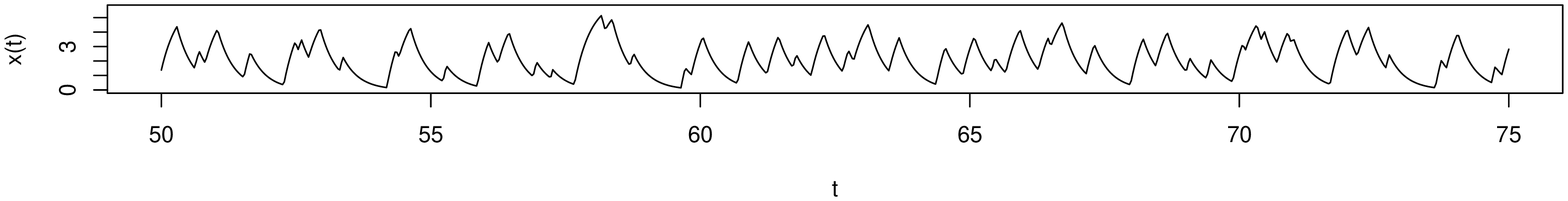} \\
\includegraphics[width=\figwidth]{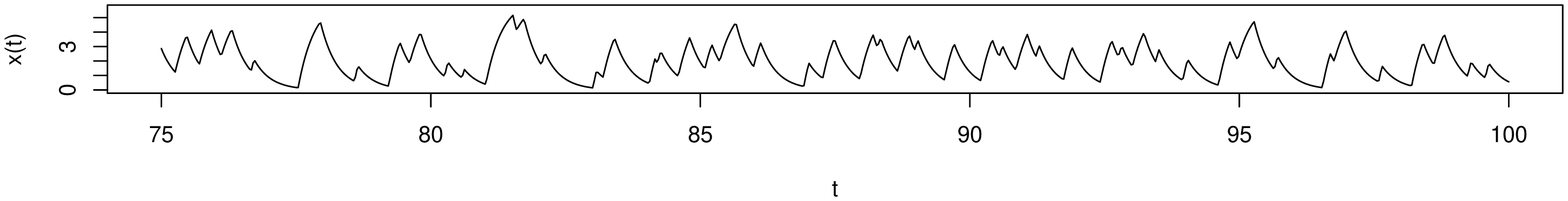} \\
\includegraphics[width=\figwidth]{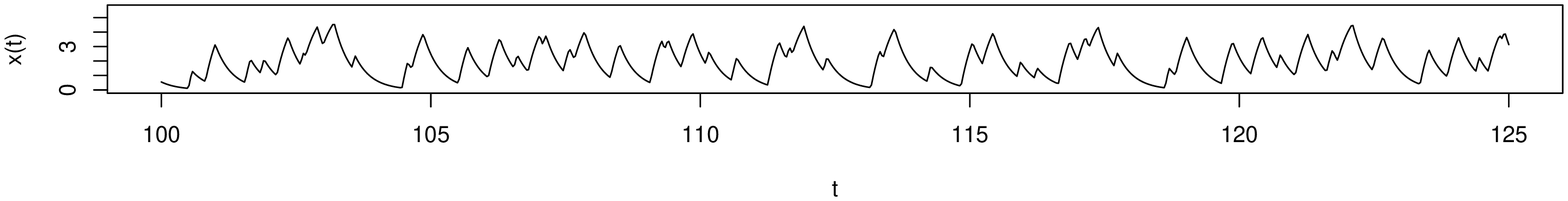} \\
\includegraphics[width=\figwidth]{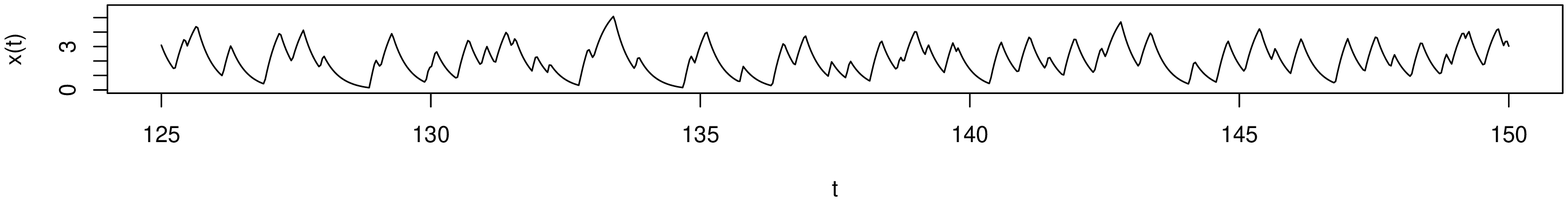}
\caption{A segment of a numerical solution near the chaotic saddle of
the delay equation~\eqref{eq.pwc2}, computed using the modified
stagger-and-step algorithm.}
\label{fig.saddlesolpwc}
\end{center}
\end{figure}
\begin{figure}
\begin{center}
\includegraphics[width=\figwidth]{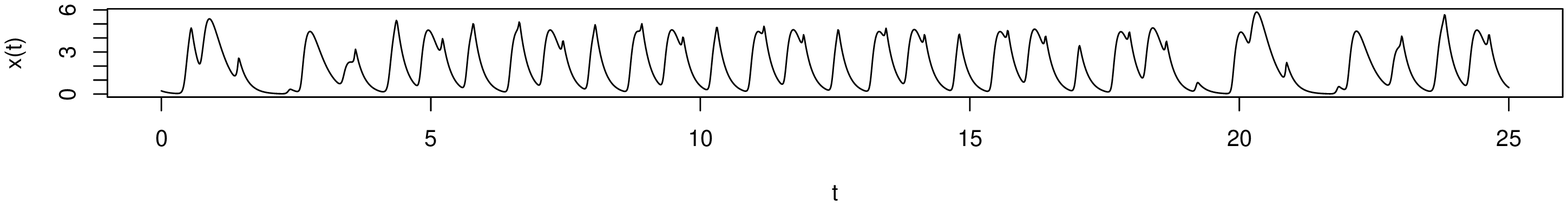} \\
\includegraphics[width=\figwidth]{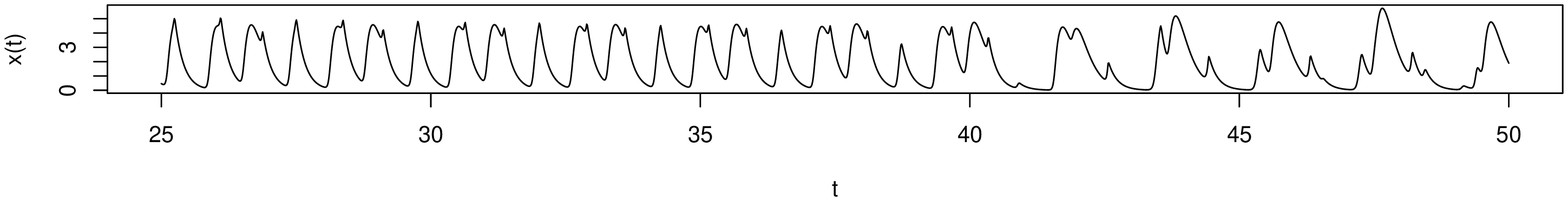} \\
\includegraphics[width=\figwidth]{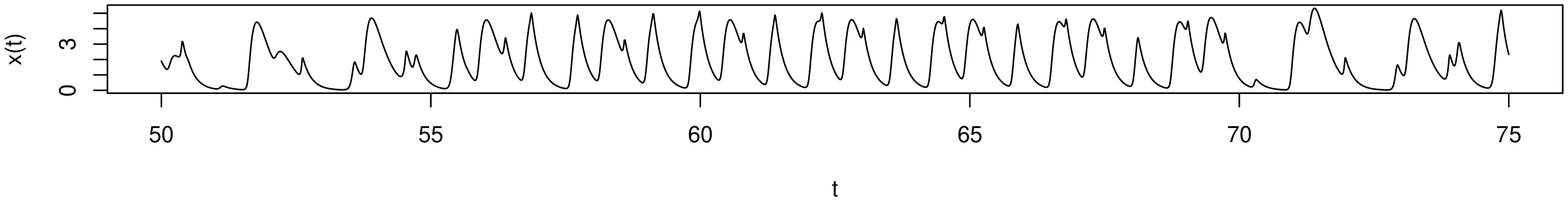} \\
\includegraphics[width=\figwidth]{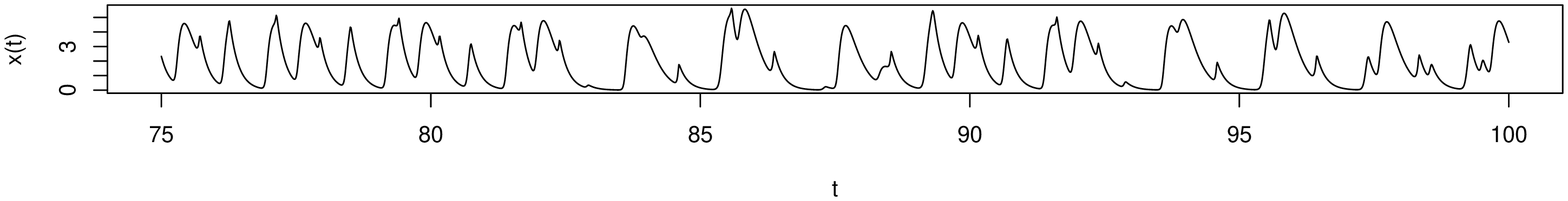} \\
\includegraphics[width=\figwidth]{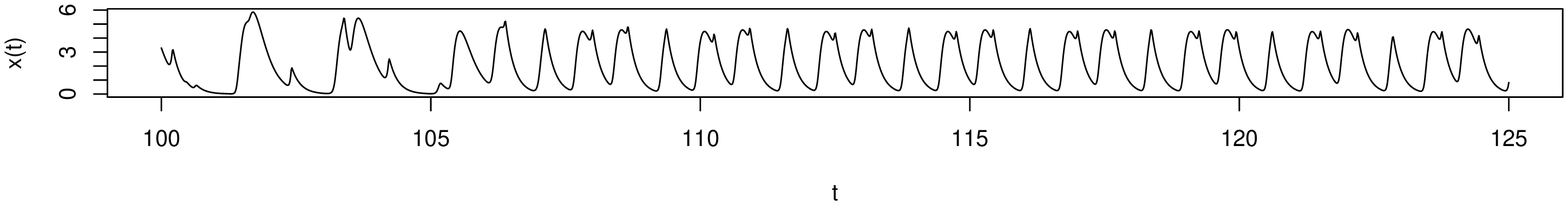} \\
\includegraphics[width=\figwidth]{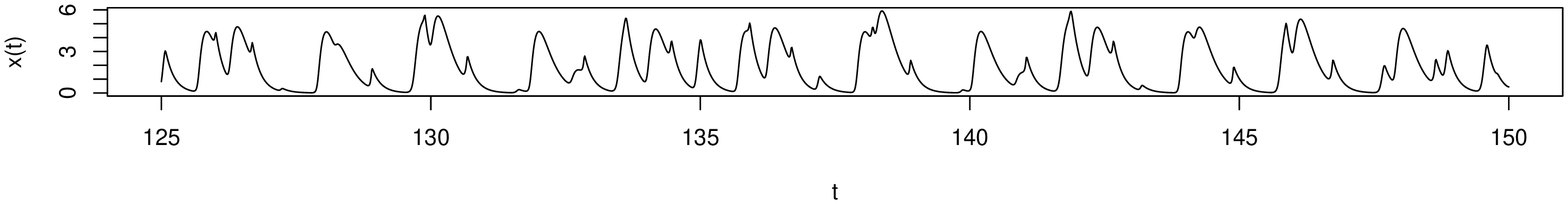}
\caption[A segment of a numerical solution near the chaotic saddle of
the Mackey-Glass equation, computed using the modified
stagger-and-step algorithm.]{A segment of a numerical solution near
the chaotic saddle of the Mackey-Glass equation~\eqref{eq.mg4},
computed using the modified stagger-and-step algorithm.}
\label{fig.saddlesolmg}
\end{center}
\end{figure}

With the trajectory $\{\mathbf{u}(n)\}$ in hand, it is possible to
visualize and otherwise characterize the geometry of the saddle.  The
phase space dimension of this trajectory is finite but large, so that
visualizing it (and hence the chaotic saddle it approximates) is
difficult.  As discussed in Sections~\ref{sec.trace1}
and~\ref{sec.trace2}, one way to visualize an object in the infinite
dimensional phase space $C$ is to plot its image under a suitable
``trace map''
\begin{equation}
  \pi: C \to \Reals^M.
\end{equation}
For delay equations like those considered here, a common practice is
to use the map
\begin{equation} \label{eq.transtrace}
  \pi: u \mapsto \big( u(-1), u(0) \big),
\end{equation}
\ie\ to plot $x(t)$ vs.\ $x(t-1)$.  With respect to the finite-dimensional
vector $\mathbf{u}$ that approximates $u$, $\pi$ acts according to
\begin{equation} \label{eq.transtrace2}
  \pi: \mathbf{u} \mapsto (u_0,u_N).
\end{equation}

Figure~\ref{fig.pwcsaddle} shows a two-dimensional image, computed in
the manner described above, of the chaotic saddle for the delay
equation~\eqref{eq.pwc2}.  Similarly, Figure~\ref{fig.mgsaddle}
presents an image of the chaotic saddle of the Mackey-Glass
equation~\eqref{eq.mg4}.  These images provide a ``flattened'' view of
a trajectory on the saddle.  Of course, the trajectory itself lives in
the phase space $C$ (or, at least, the numerical approximation of the
trajectory lives in $\Reals^{N+1}$).

\begin{figure}[p]
\begin{center}
\includegraphics[height=2.8in]{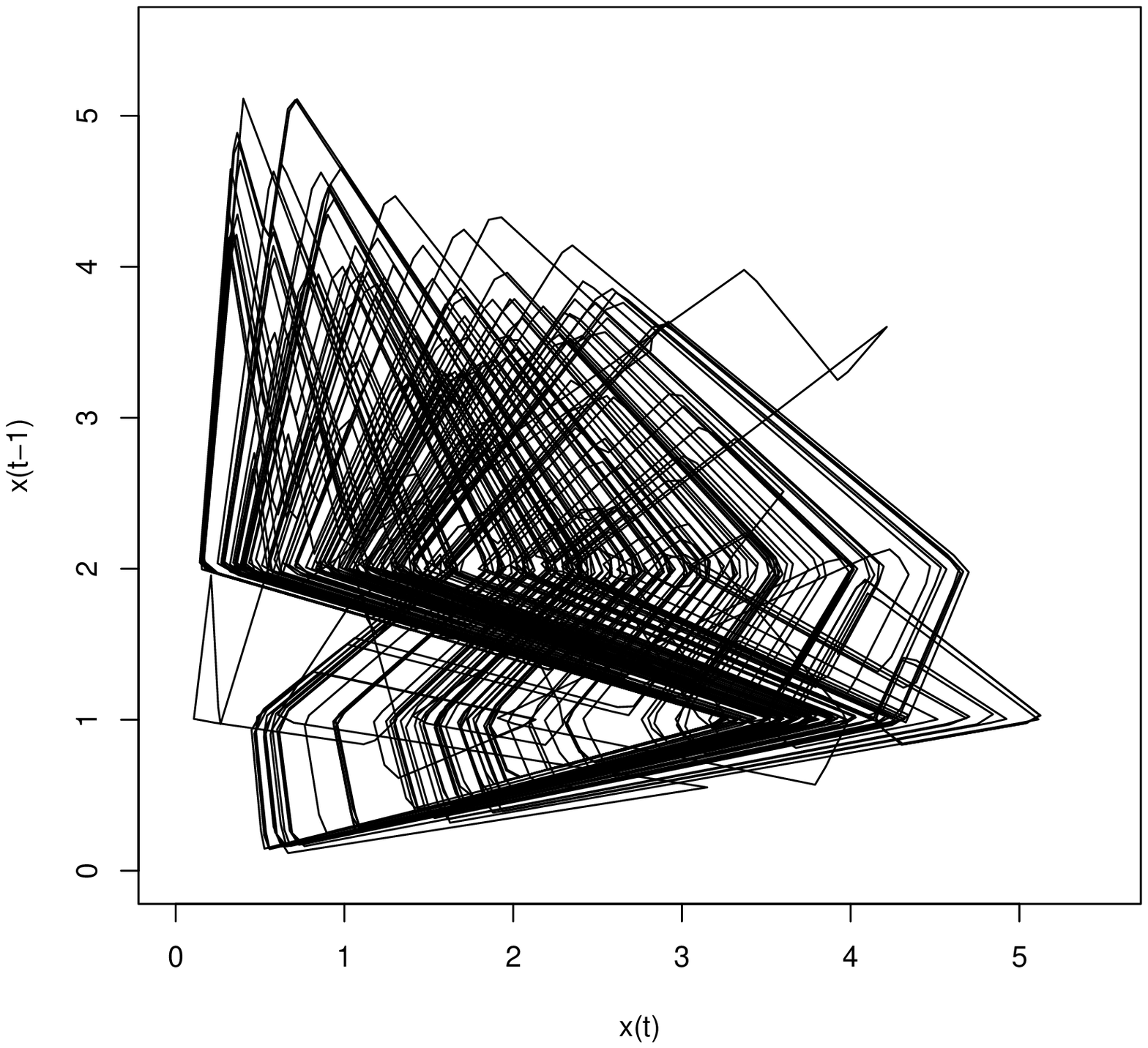}
\caption[Two-dimensional projection of part of a trajectory on the
chaotic saddle of the delay equation~\eqref{eq.pwc2}.]{Two-dimensional
projection, under the trace map~\eqref{eq.transtrace}, of part of a
trajectory on the chaotic saddle of the delay
equation~\eqref{eq.pwc2}.}
\label{fig.pwcsaddle}
\end{center}
\end{figure}
\begin{figure}[p]
\begin{center}
\includegraphics[height=2.8in]{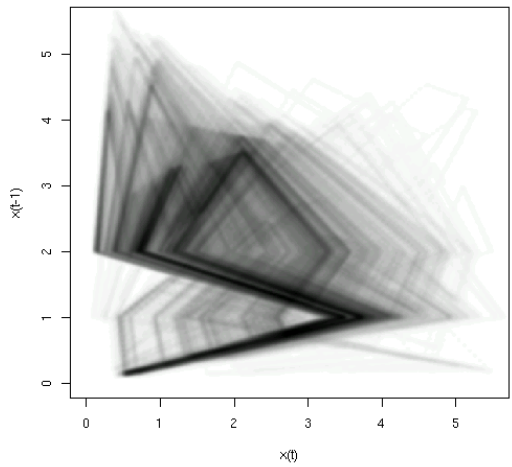}
\caption[Two-dimensional projection of the natural invariant measure
on the chaotic saddle of the delay
equation~\eqref{eq.pwc2}.]{Two-dimensional projection, under the trace
map~\eqref{eq.transtrace}, of the natural invariant measure on the
chaotic saddle of the delay equation~\eqref{eq.pwc2}.}
\label{fig.saddledenspwc2d}
\end{center}
\end{figure}

\begin{figure}[p]
\begin{center}
\includegraphics[height=2.8in]{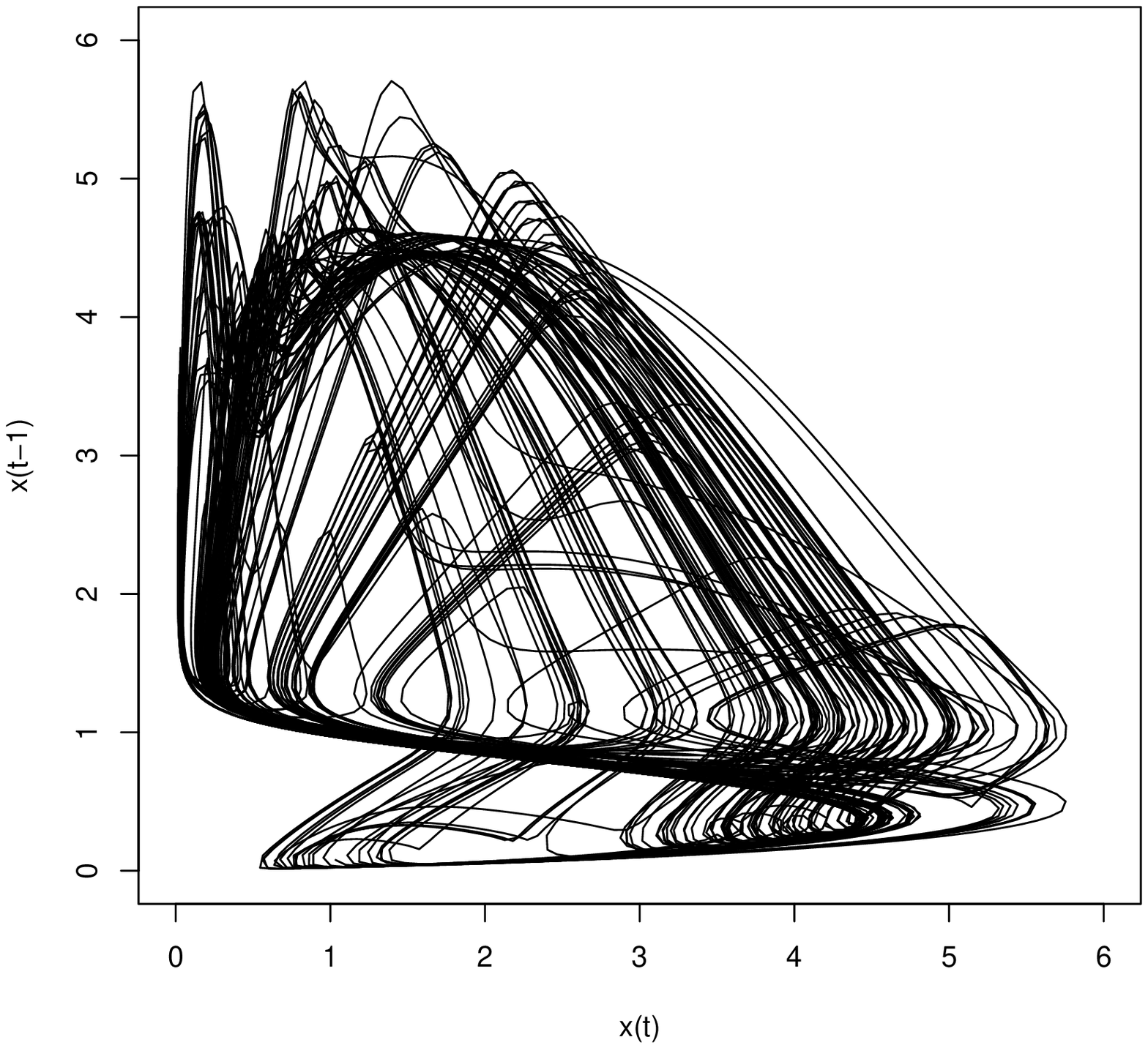}
\caption[Two-dimensional projection of part of a trajectory on the
chaotic saddle of the Mackey-Glass equation.]{Two-dimensional
projection, under the trace map~\eqref{eq.transtrace}, of part of a
trajectory on the chaotic saddle of the Mackey-Glass
equation~\eqref{eq.mg4}.}
\label{fig.mgsaddle}
\end{center}
\end{figure}
\begin{figure}[p]
\begin{center}
\includegraphics[height=2.8in]{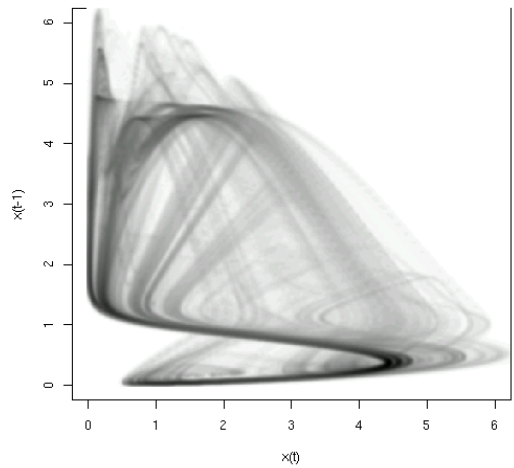}
\caption[Two-dimensional projection of the natural invariant measure
on the chaotic saddle of the Mackey-Glass equation.]{Two-dimensional
projection, under the trace map~\eqref{eq.transtrace}, of the natural
invariant measure on the chaotic saddle of the Mackey-Glass
equation~\eqref{eq.mg4}.  Density is indicated by grayscale
intensity.}
\label{fig.saddledensmg2d}
\end{center}
\end{figure}

The asymptotic statistics of a trajectory on the saddle induces an
invariant measure.  That is, for a given initial phase point $u \in C$
on the saddle, the measure $\mu$ on $C$ defined by
\begin{equation}
  \mu(A) = \lim_{M \to \infty} \frac{1}{M} \sum_{n=1}^M 1_A\big( S^n(u) \big),
\end{equation}
provided the limit exists, is invariant under the dynamics $S$ (\cf\
Section~\ref{sec.statreg}).  Using a stagger-step trajectory
$\{\mathbf{u}(n): n=1,2,\ldots,M\}$ of large length $M$, this measure
can be approximated by the measure
\begin{equation}
  \tilde{\mu}(A) = \frac{1}{M} \sum_{n=1}^M 1_A\big(\mathbf{u}(n)\big).
\end{equation}
on $\Reals^{N+1}$.

As discussed in Section~\ref{sec.trace2}, it is possible to visualize
$\tilde{\mu}$ by computing its two-dimensional image under the trace
map $\pi$ (equation~\eqref{eq.transtrace2}).  This amounts to simply
computing a two-dimensional histogram (\ie, density) of pairs
$(u_0,u_N)$ along the numerical trajectory $\{\mathbf{u}(n)\}$.  The
resulting images of the invariant measures for the delay
equations~\eqref{eq.pwc2} and~\eqref{eq.mg4} are shown in
Figures~\ref{fig.saddledenspwc2d} and~\ref{fig.saddledensmg2d},
respectively.

In principle there may be an uncountable number of distinct invariant
measures on the saddle.  In practice, however, any stagger-step
trajectory appears to yield the same approximate invariant measure,
for both of the delay equations considered here.  This suggests the
existence of a unique ``natural'' invariant measure for these systems,
analogous to SRB measure in that it is the invariant measure naturally
selected by numerical simulations and (presumably) physical
experiments.  See~\cite{Dham99,Dham01} and references therein for
discussion of how the notion of natural invariant measure should be
defined in the context of transient chaos.

In terms of the solution $x(t)$, the asymptotic statistics on the
chaotic saddle are described by the measure $\nu$ on \Reals\ given by
\begin{equation}
  \nu(A) = \lim_{T \to \infty} \frac{1}{T} \int_0^T 1_A\big(x(t))
  \,dt,
\end{equation}
or in terms of the discrete-time map $S$,
\begin{equation}
  \nu(A) = \lim_{M \to \infty} \frac{1}{M} \sum_{n=1}^M 1_A((S^n
  u)(1)).
\end{equation}
This is just the one-dimensional projection, under the trace map $\pi:
u \mapsto u(0)$, of the natural invariant measure $\mu$.  The measure
$\nu$ has the practical significance of describing where a solution on
the saddle spends most of its time.  It is more intuitively understood
in terms of its density (the so-called ``invariant density''), which
can be approximated by a histogram of values $x(t)$ along a long
solution near the saddle.  Invariant densities found in this way for
the delay equations~\eqref{eq.pwc2} and~\eqref{eq.mg4} are shown in
Figures~\ref{fig.saddledenspwc1d} and~\ref{fig.saddledensmg1d},
respectively.

\begin{figure}
\begin{center}
\includegraphics[width=\figwidth]{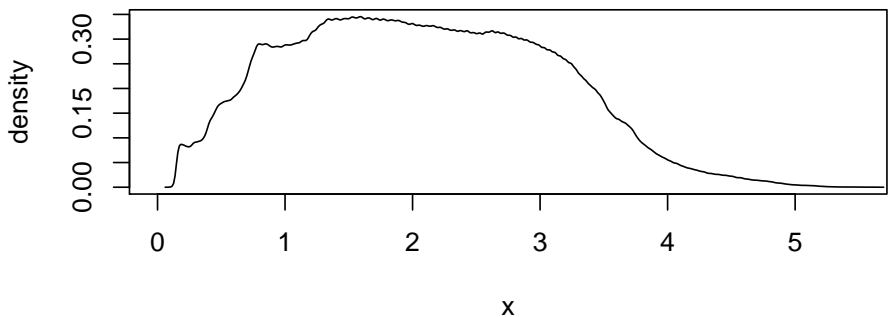}
\caption{One-dimensional invariant density for the chaotic saddle of
the delay equation~\eqref{eq.pwc2}.}
\label{fig.saddledenspwc1d}
\end{center}
\end{figure}

\begin{figure}
\begin{center}
\includegraphics[width=\figwidth]{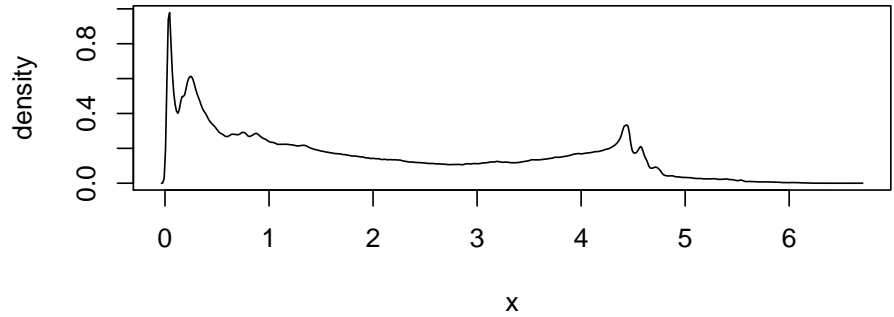}
\caption[One-dimensional invariant density for the chaotic saddle of
the Mackey-Glass equation.]{One-dimensional invariant density for the
chaotic saddle of the Mackey-Glass equation~\eqref{eq.mg4}.}
\label{fig.saddledensmg1d}
\end{center}
\end{figure}

% ------------------------------------------------------------------------
\subsection{Quantitative characterization: Ergodic parameters} \label{sec.translyap}

It is considered \emph{de rigeur}, whenever a new chaotic invariant
set is found, to characterize its geometry and the dynamics on it by
computing its fractal dimensions and Lyapunov exponents, and possibly
other ergodic parameters (\cf\ Section~\ref{sec.dimexp}).  For the
chaotic saddles of the delay equations considered here, for which we
are able to compute arbitrarily long numerical trajectories
$\{\mathbf{u}(n)\}$ approximating the saddle, these quantities can be
found by straightforward computation using the algorithms referenced
in Section~\ref{sec.dimexp}.

Table~\ref{tab.transdimexp} summarizes the results of computations of
the largest five Lyapunov exponents $\lambda_i$~\cite{Ben80a,Ben80b},
the Lyapunov dimension~\cite{Fred83}
\begin{equation}
  d_L = j + \frac{\sum_{i=1}^j \lambda_i}{|\lambda_{j+1}|} \quad \text{where}
  \quad j = \max_k \sum_{i=1}^k \lambda_i > 0,
\end{equation}
and the correlation dimension~\cite{GP83} of the chaotic saddles for
each of the delay equations~\eqref{eq.pwc2} and~\eqref{eq.mg4}.  Both
DDEs have a zero Lyapunov exponent.  This is to be expected for
continuous-time systems in general, with the zero Lyapunov exponent
corresponding to neutral expansion along the direction tangent to the
flow~\cite{Hak83}.

Note that both DDEs exhibit just a single positive Lyapunov exponent.
The existence of a positive exponent confirms that there is
sensitivity to initial conditions in a neighborhood of the saddle.
The fact that there is only one single positive exponent suggests that
the saddle has only one unstable direction.  Thus the failure of the
PIM method for these DDEs cannot be due to the presence of multiple
unstable directions.

\begin{landscape}
\begin{table}
\begin{center}
{\small
\begin{tabular}{c|c|*{5}{c}|c|c}
System & Parameters
&
\multicolumn{5}{c|}{Lyapunov spectrum (bits/time)}
&
\begin{minipage}[t]{0.75in}
Lyapunov dimension
\end{minipage}
&
\begin{minipage}[t]{0.75in}
Correlation dimension
\end{minipage} \\
& & $\lambda_1$ & $\lambda_2$ & $\lambda_3$ & $\lambda_4$ & $\lambda_5$ & & \\
\hline
\begin{minipage}[c]{1.75in}
\begin{equation*}
F(x) = \begin{cases}
          c & \text{if $x \in [x_1,x_2]$} \\
          0 & \text{otherwise}
       \end{cases}
\end{equation*}
\end{minipage}
&
\begin{minipage}[c]{1in}
\begin{equation*}
\begin{cases}
\alpha=3.25 \\ c=20.5 \\ x_1=1 \\ x_2=2
\end{cases}
\end{equation*}
\end{minipage}
&
$0.54$ & $0.00$ & $-1.5$ & $-8.2$ & $-12$ & $2.36$ & $1.96$ \\
\begin{minipage}[c]{1.75in}
\begin{equation*}
  F(x) = \beta \frac{x}{1 + x^{10}}
\end{equation*}
\end{minipage}
&
\begin{minipage}[c]{1in}
\begin{equation*}
\begin{cases}
\alpha=1/0.1625 \\ \beta=12/0.1625
\end{cases}
\end{equation*}
\end{minipage}
&
$0.60$ & $0.00$ & $-0.50$ & $-3.1$ & $-3.8$ & 3.03 & 2.24
\end{tabular}
}
\vspace*{0.5in}
\caption{Lyapunov exponents and fractal dimensions for chaotic saddles
of delay differential equations having the form $x'(t) = -\alpha x(t)
+ F\big(x(t-1)\big)$.}
\label{tab.transdimexp}
\end{center}
\end{table}
\end{landscape}

% ------------------------------------------------------------------------
\section{Conclusions}

Transient chaos in delay differential equations, although anticipated
by numerous published results~\cite{AE87,HM82,HW83,Hale88,LML93,
Wal81}, has not specifically been investigated before.  Multistability
with fractal basins of attraction, a key signature of transient chaos,
has been reported previously for the delay equation~\eqref{eq.pwc2}.
We have also found multistability and fractal basins in the
Mackey-Glass equation~\eqref{eq.mg4}.  Existence of fractal basins
suggests the existence of transiently chaotic trajectories and the
presence of a chaotic saddle~\cite{MGOY85}.  Indeed, by computing
basins of attraction for these two delay equations we have been able
to illustrate, numerically, the existence of solutions with long
chaotic transients.

The published methods for approximating chaotic saddles fail when
applied to the delay equations considered here.  The PIM
method~\cite{Nusse89} in particular is expected to fail if the saddle
has more than one unstable direction~\cite{Sweet01}.  However, our
computations of Lyapunov exponents point to the presence of only a
single unstable direction, so this explanation is inadequate.
Instead, the failure of the known methods stems from the presence of
saddle type unstable periodic orbits, to which each of the methods
eventually converge.  This mode of failure has not been observed
before.

A slightly modified version of the stagger-and-step method
(Section~\ref{sec.ssmod}) avoids unstable periodic orbits, and appears
to be an effective method for approximating chaotic saddles for delay
equations.  Using this method we are able to compute chaotic numerical
solutions of arbitrarily long duration, for both of
equations~\eqref{eq.pwc2} and~\eqref{eq.mg4}.

Despite the fact that the saddle is embedded in an infinite
dimensional phase space (or at least a finite- but high-dimensional
phase space used for numerical approximation), it is possible to go
some way toward visualizing it by graphing its projections onto two
dimensions.  We have done this for the delay equations considered
here, projecting the numerical approximation of the chaotic saddle
onto two dimensions by applying a particular ``trace map'' commonly
used for visualizing the dynamics of delay equations.

The distribution of orbits on the chaotic saddle can be characterized
by an invariant measure.  This too can be approximated numerically
from a stagger-and-step trajectory for a given DDE, and projected onto
two dimensions for the purpose of visualization.

Chaotic invariant sets are typically characterized in terms of their
ergodic parameters such as Lyapunov exponents and dimensions.  These
can be found by the standard algorithms, involving straightforward
computations on the numerical trajectory resulting from the
stagger-and-step algorithm.  For each of the delay
equations~\eqref{eq.pwc2} and~\eqref{eq.mg4} we have applied these
techniques to the numerical trajectory approximating the saddle, and
thereby estimated its Lyapunov spectrum, Lyapunov dimension, and
correlation dimension, the results of which are summarized in
Table~\ref{tab.transdimexp}.  Two interesting results follow.  Firstly,
despite the infinite dimensionality of the phase space of the delay
equation, the saddle itself has quite low dimension, of the order $2
\sim 3$.  The same observation has been made with regard to chaotic
attractors of delay equations~\cite{Farm82}.  Second, for both of the
DDEs considered the saddle has only one positive Lyapunov exponent,
hence only one unstable direction.  This lends support to our
hypothesis that the previously published algorithms for approximating
the saddle fail for some reason other than the existence of multiple
unstable directions.

\chapter{Conclusion}

\begin{singlespacing}
\minitoc  % mini table of contents for this chapter
\end{singlespacing}

\vspace{0.4in}

% --------------------------------------------------------------------
\section{Summary of Conclusions}

The probabilistic approach to evolutionary delay differential
equations suffers from the chief difficulty, encountered in various
guises throughout this thesis, that the phase space of a delay equation
is infinite dimensional.  This difficulty might explain the absence in
the literature of a thorough discussion of how such an approach might
be developed.  In Chapter~\ref{ch.framework} we have attempted to
bridge this gap.

The phase space that arises naturally in the formulation of delay
equations as dynamical systems is the space $C$ of continuous
functions from the interval $[-1,0]$ into \Rn.  Any solution of a
given DDE can be identified with the evolution of a corresponding
phase point in $C$.  Within this context Chapter~\ref{ch.framework}
develops the basic framework for the application of ergodic concepts
to delay equations, and explores the implications of this framework.
Its main conclusions, mostly negative, are consequences of the
infinite dimensionality of $C$.

An ergodic approach to delay equations entails an adequate theory of
probability on the phase space $C$, for which measure theoretic
probability provides a sufficiently abstract setting.  However, there
are peculiarities of probability in infinite dimensions that make an
ergodic approach to DDEs problematic.  The main analytical tool of
applied ergodic theory is the Perron-Frobenius operator~\cite{LM94},
which prescribes the evolution of probability densities under the
action of a dynamical system.  The Perron-Frobenius operator formalism
has been very successful in the analysis of finite dimensional
systems.  However, infinite dimensional systems cannot be expected to
have well-defined densities, owing to singularity of the evolution
operator (Section~\ref{sec.infdimdens}).  Consequently one cannot
define a Perron-Frobenius operator corresponding to a delay equation.
The absence of a well-defined Perron-Frobenius operator precludes the
application of the major part of the ergodic theoretic toolbox, \eg\
in~\cite{LM94}.

An ergodic approach to delay equations will also require a theory of
integration with respect to measures on $C$.  For example, the
Birkhoff Ergodic Theorem (\cf\ Section~\ref{sec.ergmixexact}) allows
one to express a time average in terms of an integral (\ie\ spatial
average or expectation) with respect to an ergodic measure.  The
evaluation of such integrals requires an adequate theory of
integration on function space.  This theory is lacking, except in the
special case of Wiener measure.  Because Wiener measure is invariant
under the quantum field equations, the theory of integration with
respect to Wiener measure has been well developed in the physics
literature.  However, we cannot expect invariance of Wiener measure
under delay differential equations, except perhaps in special cases.
The lack of a general theory of integration on function space is a
serious barrier to developing an ergodic theory of delay equations.

In the application of ergodic concepts to physical systems, the notion
of SRB measure plays a central role.  An SRB measure characterizes the
asymptotic statistics of almost every orbit of a dynamical system.
For finite dimensional systems Lebesgue measure provides a natural and
essentially unique translation-invariant notion of ``almost every''.
However, in infinite dimensions there is no measure analogous to
Lebesgue measure.  In particular, there is no non-trivial
translation-invariant measure on $C$ to provide the requisite notion
of ``almost every''.  Consequently the definition of SRB measure for
delay equations is ambiguous.  One way to resolve this ambiguity is to
substitute the notion of prevalence~\cite{HSY92} in place of
``Lebesgue almost every'' in the definition of SRB measure.  This is
consistent with the present definition of SRB measure for finite
dimensional systems and therefore provides a natural extension to the
infinite dimensional case.

% ----------------------------------------------------------------------

\bigskip

The difficulties associated with the infinite dimensionality of DDEs
can be avoided either by finite dimensional approximation or by
removing the requirement that DDEs be treated in a dynamical systems
context (thereby precluding any discussion of ergodic concepts, since
the notion of an evolution semigroup underlies all of ergodic theory).

Despite the infinite dimensionality of the phase space, a delay
equation nevertheless prescribes the evolution of some finite
dimensional quantity $x(t) \in \Rn$.  In Chapter~\ref{ch.densev} we
have investigated probabilistic approaches to this evolution problem.
In order that a DDE prescribes a finite dimensional evolutionary
process, it is necessary to restrict the set of allowable initial
functions to some finite dimensional subset of $C$, \eg\ the subspace
of constant functions.  If the set of allowable initial functions is
$n$-dimensional then one can define a family of solution maps $S_t:
\Rn \to \Rn$.  For some simple DDEs for which the solution map can be
found analytically, one can derive an explicit formula for the
Perron-Frobenius operator corresponding to $S_t$ and thereby
analytically solve the density evolution problem
(Section~\ref{sec.fpexplicit}).  For more complicated equations this
method is impractical, due both to the difficulty of finding $S_t$
analytically and to the non-invertibility of $S_t$ once it has been
found.  In particular this method fails to provide an analytical
approach to the evolution of densities for chaotic delay equations
with interesting statistical properties.

In the absence of a generally applicable analytical method, it is
desirable to have an effective computational approach to the evolution
of densities for DDEs.  Of the numerical methods considered, the
simplest is the ``brute force'' method of simulating large ensembles
of solutions and compiling histograms to approximate densities
(Section~\ref{sec.histos}).  Because this method relies on adequate
statistical sampling to obtain accurate results it is computationally
intensive, to the point of being impractical for many applications.
Nevertheless, due to general results on the reliability of statistics
computed from numerical simulations~\cite{Ben78a}, one has reason to
hope that ensemble simulation is a robust means of estimating ensemble
densities.

As an alternative to ensemble simulation we have developed a numerical
method for computing the evolution of densities for DDEs, based on
piecewise linear approximation of the solution map
(Section~\ref{sec.approxfp}).  This approximate solution map is easily
inverted, and leads to an approximate version of the analytical
approach considered in Section~\ref{sec.fpexplicit}.  This method is
much more efficient for computing the evolution of densities.
However, its effectiveness is limited for DDEs with chaotic dynamics,
for which the solution map for large times can be extremely complex
and difficult to approximate.  Thus for chaotic DDEs this numerical
approach is not suitable for evolving densities to arbitrarily large
times.

The method of steps, frequently used to find analytical solutions of
delay equations, can be used to express a delay equation as a sequence
of ordinary differential equations.  With appropriate modifications to
this method, a DDE can be expressed as a system of simultaneous ODEs.
One can then formulate the density evolution problem for a DDE in
terms of a corresponding ODE system, and use methods for ODEs to write
an evolution equation for the density (Section~\ref{sec.stepsdens}).
For simple DDEs this equation can be solved analytically by the method
of characteristics, reproducing the analytical results obtained in
Section~\ref{sec.fpexplicit}.  For more complicated equations an
analytical solution is not possible, but the method retains some
intuitive appeal because it provides a model for the evolution of
densities for DDEs, in terms of the transportation of a line mass by a
flow in $\Reals^N$.  A numerical implementation of this model provides
geometrical insight into the results of our other numerical approaches
to density evolution (Section~\ref{sec.geom}).  However, the model has
limited utility when densities are evolved to large time, both because
the dimension $N$ increases with time, and, for chaotic DDEs, the
solution map becomes very complicated and the resulting distribution
of the line mass in $\Reals^N$ becomes difficult to approximate and in
any case loses its intuitive appeal.

% ----------------------------------------------------------------------

\bigskip

For a variety of delay differential equations, numerically computed
solution ensembles appear to converge to a unique asymptotic
distribution, characterized by an asymptotic probability measure
\etastar, with density \rhostar, on \Rn\ (Section~\ref{sec.invex}).
The same density is observed if one constructs a histogram of solution
values along a \emph{single} solution of the DDE.  This phenomenon can
be understood in terms of (hypothetical) ergodic properties of the
associated infinite dimensional dynamical system $\{S_t\}$ on $C$
(Section~\ref{sec.invinterp}).  In the examples considered, $\{S_t\}$
is known to possess a compact attractor $\Lambda$: any ensemble of
trajectories starting in the basin of attraction of $\Lambda$ will,
asymptotically, be distributed on $\Lambda$.  The numerical evidence
supports the existence of a natural invariant probability measure
(\ie, an SRB measure) \mustar\ supported on $\Lambda$.  The asymptotic
measure \etastar\ can be interpreted as the image of \mustar\ under a
suitable projection from $C$ into \Rn.

The practical and theoretical importance of invariant measures, and
SRB measures especially, makes the computation of asymptotic measures
for DDEs a desirable goal.  The simplest approach to this problem is
to evolve an initial density forward in time until the asymptotic
statistics become apparent.  However, of the methods discussed above
for computing the evolution of densities for DDEs, none is well suited
to evolving densities to large times except the ``brute force''
ensemble simulation method.  The increase in time of both the
complexity of the solution map and the dimension of the system are
fundamental obstacles to evolving densities to large times.  In
Chapter~\ref{ch.invdens} we have investigated alternative approaches
to estimating asymptotic densities.

For some DDEs we have hypothesized the existence of an SRB measure,
since a histogram of a single solution reproduces the asymptotic
density found by ensemble simulation.  For DDEs with this property the
asymptotic density can be approximated more efficiently by simulating
a single long solution, rather than the many (\eg\ $10^6$) different
solutions required for ensemble simulation (Section~\ref{sec.srbev}).
While this approach can be used to quickly estimate the asymptotic
density, it falls into the class of ``brute force'' methods that
provide no insight into underlying the mechanism.

Ulam's method~\cite{Ulam60}---the most widely known technique for
estimating invariant measures---can be formulated for DDEs so that it
yields an approximation of the asymptotic measure \etastar\
(Section~\ref{sec.ulam}).  However, this approximation turns out to be
identical (by definition) to the histogram of a time series generated
by a typical solution of the given DDE, because our construction is
somewhat circular.  Thus, at least in our formulation, Ulam's method
does not provide an independent estimate of the invariant density.

The ``self-consistent Perron-Frobenius operator'' method
of~\cite{Kan89,Lep93} is a promising approach to estimating
asymptotic densities for DDEs.  A suitable discretization of a given
DDE yields a discrete-time system to which this method can be applied.
However, a straightforward implementation
(Section~\ref{sec.approxinv}) fails to generate the desired
approximate invariant density.  Instead, the method converges to an
unstable fixed point of the DDE.  This failure can be explained, in
part, by the fact that in the delay equations we consider, the
instability (hence chaotic behavior and nontrivial statistical
properties) is delay-induced, \ie\ is not present in the absence of a
non-zero delay.  In the systems considered in~\cite{Lep93}, chaotic
behavior is present even in the case of zero delay.  Adapting the
approach of~\cite{Lep93} to systems with delay-induced instability
remains an open problem.

A solution to the problem of estimating asymptotic densities for DDEs
remains elusive.  Previously published methods of estimating invariant
measures for dynamical systems do not adapt well to delay equations.
At present the only effective methods are based on computing
statistics on numerical solutions.

% ----------------------------------------------------------------------

\bigskip

Transient chaos in delay differential equations, although anticipated
by a number of publications, has not been investigated before.  We
have found multistability with fractal basins of attraction, a key
signature of transient chaos, in a delay equation with
piecewise-constant nonlinearity (also reported in~\cite{LML93}) as
well as the Mackey-Glass equation (Section~\ref{sec.transev}).
Existence of fractal basins suggests the existence of transiently
chaotic trajectories and the presence of a chaotic saddle.  Indeed, by
computing basins of attraction for these two delay equations we have
been able to illustrate, numerically, the existence of solutions with
long chaotic transients.

Previously published numerical methods for approximating unstable
chaotic sets (\ie\ chaotic saddles) all fail when applied to the delay
equations considered here.  The PIM method~\cite{Nusse89} in
particular is expected to fail if the saddle has more than one
unstable direction.  However, our calculations of Lyapunov exponents
point to the presence of only a single unstable direction, so this
explanation is inadequate.  Instead, the failure of the known
algorithms appears to stem from the presence of saddle type unstable
periodic orbits, to which each of the algorithms eventually converge.
This mode of failure has not been observed before.

We have developed a modified version of the stagger-and-step method
(Section~\ref{sec.ssmod}), aimed at avoiding unstable periodic orbits.
With this method, aperiodic numerical trajectories of arbitrarily long
duration can be found for both of the delay equations considered in
our examples.  It is presumed that such a trajectory approximates the
chaotic saddle.

Despite the fact that the saddle is embedded in an infinite
dimensional phase space (or at least a finite- but high-dimensional
phase space used for numerical approximation), it is possible to go
some way toward visualizing the saddle by graphing its projections
onto two dimensions.  We have done this for the delay equations
considered here, projecting the numerical approximation of the chaotic
saddle onto two dimensions by applying a particular ``trace map''
commonly used for visualizing the dynamics of delay equations
(Section~\ref{sec.transvis}).

The distribution of orbits on the chaotic saddle can be characterized
by an invariant measure.  This too can be approximated numerically
from a stagger-and-step trajectory, and projected onto two dimensions
for the purpose of visualization.

Chaotic invariant sets are typically characterized in terms of their
ergodic parameters such as Lyapunov exponents and dimensions.  These
can be found using standard algorithms, involving straightforward
computations on the numerical trajectory resulting from the
stagger-and-step algorithm.  For each of the delay equations
considered we have applied these techniques and thereby estimated the
Lyapunov spectrum, Lyapunov dimension, and correlation dimension of
the respective chaotic saddles (Section~\ref{sec.translyap}).  Despite
the infinite dimensionality of the phase space of the delay equation,
the saddle itself has quite low dimension, of the order $2 \sim 3$.
The same observation has been made with regard to chaotic
\emph{attractors} of delay equations~\cite{Farm82}.  For both of the
DDEs considered, the saddle has only one positive Lyapunov exponent,
hence only one unstable direction.  This lends support to our
hypothesis that the previously published algorithms for approximating
the saddle fail for some reason other than the existence of multiple
unstable directions.

% ----------------------------------------------------------------------
\section{Directions for Further Research}

A number of the investigations undertaken in this thesis suggest
avenues for further research.  Some of the more promising directions
are outlined below, along with some new ideas that have not been
sufficiently developed to warrant inclusion in the foregoing.

In Section~\ref{sec.infprob} we showed that the notion of physical or
SRB measure is problematic for infinite dimensional systems such as
DDEs, owing to the lack (or ambiguity) of an appropriate notion of
``almost every'' in infinite dimensions.  Prevalence~\cite{HSY92} is
one such notion that appears not to have been investigated in the
context of physical measures.  Its translation invariance makes
prevalence a good candidate for what might be considered a ``natural''
sense of almost every, which is what we desire in the identification
of a natural or physical invariant measure.  Since prevalence is
equivalent to ``Lebesgue almost every'' in finite dimensional spaces,
using prevalence in the definition of SRB measure would be consistent
with the already-accepted but inadequate definition.  Prevalence of a
given set also happens to be relatively easy to prove, thanks to a
number of results established in~\cite{HSY92}.  Especially as SRB
measure is expected to play an important role in an ergodic theoretic
understanding of such infinite dimensional phenomena as \eg\
turbulence in fluids, further investigation of the definition of SRB
measure for infinite dimensional systems, and the potential role of
the notion of prevalence in particular, is warranted.

Many results in ergodic theory are expressed in terms of integrals
over the phase space.  For this reason it is argued above that an
ergodic theoretic understanding of delay equations will require a
theory of integration on the function space $C$.  Such a theory has
been developed if the measure of integration is Wiener measure, but we
have argued that this is inadequate for an understanding of delay
equations because, in general, Wiener measure is not expected to be
invariant under a given DDE.  This position is dissatisfying, if only
because the physics literature has amassed such a wealth of theory and
tools (\eg\ Feynman diagrams) for integrating functionals with respect
to Wiener measure.  The possibility of exploiting these tools in the
context of delay equations has not been adequately explored.  One
avenue, for example, might follow the approach
of~\cite{BK84,Rud85,Rud87,Rud88}, where exactness is proved for a
class of partial differential equations by establishing a conjugacy
with a certain dynamical system for which Wiener measure is invariant.
It would be interesting to seek a class of delay equations for which a
similar argument could be made.  This would provide the first rigorous
result on strong ergodic properties of a delay equation as a dynamical
system in $C$.

In Section~\ref{sec.approxfp} we develop a numerical method for
approximating the evolution of densities, based on piecewise-linear
approximation of the solution map for a given DDE.  The method itself
is not limited to delay equations, and could be applied to any system
for which a solution map can be approximated.  It would be
interesting, for example, to apply this approach to estimating
invariant densities for ordinary differential equations.

Ulam's method may yet turn out to be an effective approach to
estimating asymptotic densities for delay equations.  The Euler
discretization with time step $h$, considered in
Section~\ref{sec.approxinv}, defines a discrete-time dynamical system
$S: \Reals^N \to \Reals^N$.  For some delay equations $S$ is known to
have strong ergodic properties~\cite{LM95}, making it a good candidate
for the application of Ulam's method.  That is, Ulam's method might be
used to approximate the $N$-dimensional natural $S$-invariant measure.
This would require a partition \Acal\ of $\Reals^N$ and the definition
of a transition probability matrix relative to \Acal\
(equation~\eqref{eq.transmat}, with Lebesgue measure $\lambda$ on
$\Reals^N$).  At first this seems impractical, since even a relatively
coarse partition would contain on the order of $100^N$ cells (with $N$
large, say $N \gtrsim 100$, so that the discretization accurately
models the DDE).  However, attractors for delay equations have
generally been found to be low-dimensional (\eg, $D \approx
3$)~\cite{Farm82}, and therefore should be resolvable with a carefully
chosen partition containing on the order of $100^3$ cells.  This is
within the limits of practical computation.  An ``automatic
refinement'' method for choosing an optimal partition, coupled with
efficient construction of the transition matrix $P$, has recently been
implemented in the software package GAIO~\cite{DJ98,DJ99,Junge01}.
Some exciting recent developments in this area show that numerical
simulations can be used to prove rigorous results about dynamical
systems~\cite{Tucker00,Tucker02} and infinite dimensional systems in
particular~\cite{Day03}.  It seems reasonable that this software could
be applied to the estimation of invariant measures for delay
equations.

In Section~\ref{sec.approxinv} we attempted to develop a method for
estimating asymptotic densities for DDEs, based on the ideas
in~\cite{Kan89,Lep93}.  The results of this investigation were
disappointing and somewhat surprising.  Using the same ideas, the
authors of~\cite{Kan89} and~\cite{Lep93} have successfully estimated
``collapsed'' (\ie, one- and two-dimensional) invariant densities for
some other high-dimensional systems.  We have hypothesized that this
discrepancy is due to the fact that for the systems considered
in~\cite{Lep93}, chaotic behavior occurs even in the absence of an
explicit delay in the dynamics, whereas for the DDEs we consider the
instability is inherently delay-induced.  This explanation is not
entirely satisfactory.  It seems that some variation on this theme
should be effective for delay differential equations.  In particular,
the second-order method outlined in Section~\ref{sec.2ord} warrants
further investigation.

In Chapter~\ref{ch.transient}, previously published algorithms for
approximating chaotic saddles failed when applied to the delay
equations we have considered.  This failure is apparently due to the
presence of saddle type periodic orbits, which capture the numerical
trajectory that would otherwise approximate the saddle.  This
mechanism should not be specific to delay equations.  It would be
interesting to investigate the behavior of these algorithms for
simpler (\eg\ finite dimensional) systems that also possess
saddle-type periodic orbits, both to test our hypotheses regarding
this mode of failure and to investigate more effective remedies.

Our modification of the stagger-and-step method~\cite{Sweet01}
successfully avoids the difficulty that otherwise occurs in the
presence of saddle type periodic orbits.  An obvious extension of this
work would be an attempt at a similar modification of the PIM method.
Our numerical results suggest that the chaotic saddle has only one
unstable direction, at least for the DDEs we have considered.  Thus in
principle the PIM method~\cite{Nusse89} should be applicable, but for
the difficulty presented by the presence of saddle type periodic
orbits.  A successful modification of this method would be a valuable
tool not just for delay equations but for the numerical analysis of
transient chaos in general.

A theme that arises in various contexts throughout this thesis is that
the analysis of DDEs is complicated by their infinite dimensionality.
Chaos in infinite-dimensional systems is not a new subject: the field
has a rich literature, with a focus mainly on chaotic partial
differential equations (PDEs).  However, little has been said about
PDEs in an ergodic theoretic context; consequently this thesis
provides only scant discussion of the PDE literature.  Nevertheless, a
number of fruitful avenues for further research are suggested by the
problem of adapting and applying to DDEs the techniques that have been
developed for PDEs.

One such collection of techniques is motivated by the possible
existence of \emph{inertial manifolds} for DDEs.  An inertial manifold
(IM), a subset of phase space, is a smooth, finite-dimensional
invariant manifold that exponentially attracts all trajectories in a
certain neighborhood (see \eg\ \cite{Tem88,Tem90,Tem97} and references
therein).  Inertial manifolds have been shown to exist for a broad
class of dissipative dynamical systems including some PDEs.  If an
inertial manifold $\mathcal{M}$ exists then any global attractor will
be contained in $\mathcal{M}$; thus the subsystem obtained by
restricting the original dynamical system to $\mathcal{M}$ faithfully
reproduces the asymptotic dynamics.  This subsystem is itself
equivalent to a certain \emph{finite} system of ordinary differential
equations.  This drastic simplification makes it possible to analyze the
asymptotic dynamics using methods for finite dimensional systems.
This approach has led, for example, to rigorous bounds on the
dimension of the attractor for some PDEs~\cite{Tem00}.  Techniques
have also been developed for approximating inertial manifolds
numerically, and results have been obtained on the persistence of
inertial manifolds under perturbation and numerical
approximation~\cite{DT94,JRT00,JS95,JST98,JT96}.

We are unaware of any investigations of inertial manifolds for chaotic
DDEs (in the case of a DDE we are interested in finite-dimensional
invariant manifolds in the function space $C$).  However, the
observations in~\cite{Farm82} and elsewhere of low-dimensional
attractors for some DDEs suggest the presence of an inertial manifold.
In a preliminary investigation in this direction, we have considered
the special case of seeking a flat inertial manifold, \ie\ an
attracting invariant linear subspace of $C$.  If such an invariant
subspace exists it should be readily observable, \eg\ by employing a
Gram-Schmidt procedure to show (numerically) the existence of a finite
basis for points on the attractor.  We have carried out this procedure
for the chaotic DDEs considered in Section~\ref{sec.invex}, with
results indicating that there is no finite basis for the attractor,
hence no flat inertial manifold.  The possible existence of curved
inertial manifolds warrants further investigation.  The finite
dimensional reduction such an invariant manifold would afford might be
instrumental in circumventing some of the difficulties encountered in
this thesis.

As mentioned in Section~\ref{sec.ddetheory}, delay differential
equations are a special case of a more general class of retarded
functional differential equations (RFDEs).  Throughout this thesis we
have restricted our attention to delay equations having the particular
form of equation~\eqref{eq.dde1} (\eg\ with a single, fixed delay).
At times we have been able to exploit the relative simplicity of the
special form of this equation (\eg\ Section~\ref{sec.stepsdens}) to
obtain a desired result.  However, many of the results of this thesis
do not rely on the special form of this equation.  An interesting and
obvious avenue for further research is to extend the present work,
where possible, to more general RFDEs.

%
%----------------------------------------------------------------------
% END MATTER
%----------------------------------------------------------------------
%
\bibliographystyle{amsplain}
\bibliography{thesisbib}
%
% uncomment this to include nomenclature:
%\cleardoublepage
%\addstarredchapter{Nomenclature}
%\markboth{}{NOMENCLATURE}
%\printnomenclature
%
\end{document}